\documentclass[11pt,reqno]{amsart} % AMS article style with right equation numbering
\pdfoutput=1 % ensures that arXiv recognizes microtype package

\usepackage[mathlines]{lineno}
\usepackage{microtype}

\usepackage{pbox}

\usepackage[top=1.00in,bottom=1.00in,left=1.00in,right=1.00in]{geometry}

\usepackage{xcolor}

\definecolor{my-red}{rgb}{0.5,0.0,0.0}
\definecolor{my-blue}{rgb}{0.0,0.0,0.6}
\definecolor{my-green}{rgb}{0.0,0.5,0.0}
\definecolor{light-gray}{gray}{0.6}

\usepackage{enumitem}

\usepackage{tocvsec2}

\usepackage[labelfont=bf]{caption}

\usepackage[font={small},labelfont=md]{subfig}

\usepackage[noadjust,nosort]{cite}

\usepackage{amsmath, amsthm, amssymb, tikz, mathtools, array, mathrsfs, tensor, ifthen, xparse, graphicx}

\usepackage[foot]{amsaddr}

\usepackage{polynom}

\usepackage{esint}

\usepackage{multicol}

\usepackage{dsfont}
	\newcommand{\one}{\mathds{1}}

\usepackage{framed}

\usepackage{scalerel}
 
\usepackage{stackengine}

\usepackage{upref}

\usepackage[pagebackref=true, colorlinks=true, urlcolor=light-gray, linkcolor=my-blue, citecolor=my-green]{hyperref}

\usepackage{orcidlink}

\numberwithin{equation}{section}


\newcommand{\eq}[1]{\begin{linenomath}\postdisplaypenalty=0\begin{align*} #1 \end{align*}\end{linenomath}}

\NewDocumentCommand{\eeq}{om}{\begin{linenomath}\postdisplaypenalty=0\begin{align} \IfNoValueTF{#1}{}{\tag{#1}} \begin{split} #2 \end{split} \end{align}\end{linenomath}}

\newcommand{\stackref}[2]{
\readlist*\mylist{#1}
\stackrel{\mbox{\footnotesize\foreachitem\x\in\mylist[]{\ifnum\xcnt=1\else,\fi\eqref{\x}}}}{#2}
}
\newcommand{\stackrefp}[2]{
\readlist*\mylist{#1}
\stackrel{\hphantom{\mbox{\footnotesize\foreachitem\x\in\mylist[]{\ifnum\xcnt=1\else,\fi\eqref{\x}}}}}{#2}
}
\newcommand{\stackrefpp}[3]{
\readlist*\mylist{#1}
\readlist*\mylistt{#2}
\stackrel{\parbox{\widthof{\footnotesize\foreachitem\x\in\mylistt[]{\ifnum\xcnt=1\else,\fi\eqref{\x}}}}{\centering\footnotesize\foreachitem\x\in\mylist[]{{\ifnum\xcnt=1\else,\fi\eqref{\x}}}}}{#3}
}

\newcommand{\eps}{\varepsilon}

\newcommand{\vphi}{\varphi}

\newcommand{\E}{\mathbb{E}}

\renewcommand{\P}{\mathbb{P}}
\newcommand{\Q}{\mathbb{Q}}
\newcommand{\R}{\mathbb{R}}

\newcommand{\Z}{\mathbb{Z}}

\newcommand{\cC}{\mathcal{C}}
\newcommand{\cD}{\mathcal{D}}

\newcommand{\cF}{\mathcal{F}}

\newcommand{\cH}{\mathcal{H}}
\newcommand{\cI}{\mathcal{I}}

\newcommand{\cK}{\mathcal{K}}

\newcommand{\cM}{\mathcal{M}}

\newcommand{\cP}{\mathcal{P}}
\newcommand{\cQ}{\mathcal{Q}}
\newcommand{\cR}{\mathcal{R}}
\newcommand{\cS}{\mathcal{S}}
\newcommand{\cT}{\mathcal{T}}
\newcommand{\cU}{\mathcal{U}}

\newcommand{\cX}{\mathcal{X}}

\newcommand{\cZ}{\mathcal{Z}}

\newcommand{\sL}{\mathscr{L}}

\newcommand{\sP}{\mathscr{P}}

\newcommand{\nA}{\mathsf{A}}
\newcommand{\nB}{\mathsf{B}}

\newcommand{\nP}{\mathsf{P}}
\newcommand{\nQ}{\mathsf{Q}}

\newcommand{\nX}{\mathsf{X}}

\newcommand{\fC}{\mathfrak{C}}

\newcommand{\fa}{\mathfrak{a}}

\newcommand{\fc}{\mathfrak{c}}

\newcommand{\fp}{\mathfrak{p}}
\newcommand{\fq}{\mathfrak{q}}
\newcommand{\fm}{\mathfrak{m}}
\newcommand{\fn}{\mathfrak{n}}
\newcommand{\fu}{\mathfrak{u}}

\newcommand*{\qedrem}{\hfill\ensuremath{\square}}

\newcommand{\wt}[1]{\widetilde{#1}}
\newcommand{\wh}[1]{\widehat{#1}}

\DeclareMathOperator{\id}{id}

\DeclareMathOperator{\e}{e} % exponential
\DeclareMathOperator*{\argmax}{arg\,max}

\newcommand{\givenk}[3][]{#1[ #2 \: #1| \: #3 #1]} % with brackets
\newcommand{\givenp}[3][]{#1( #2 \: #1| \: #3 #1)} % with parentheses
\newcommand{\dd}{\mathrm{d}} % for differentials
\newcommand{\sing}{\mathsf{s}} % for ''singular'' superscripts
\newcommand{\ac}{\mathsf{ac}} % for ''absolutely continuous'' superscripts

\DeclarePairedDelimiter\ceil{\lceil}{\rceil}
\DeclarePairedDelimiter\floor{\lfloor}{\rfloor}

            \makeatletter
            \DeclareFontFamily{OMX}{MnSymbolE}{}
            \DeclareSymbolFont{MnLargeSymbols}{OMX}{MnSymbolE}{m}{n}
            \SetSymbolFont{MnLargeSymbols}{bold}{OMX}{MnSymbolE}{b}{n}
            \DeclareFontShape{OMX}{MnSymbolE}{m}{n}{
                <-6>  MnSymbolE5
               <6-7>  MnSymbolE6
               <7-8>  MnSymbolE7
               <8-9>  MnSymbolE8
               <9-10> MnSymbolE9
              <10-12> MnSymbolE10
              <12->   MnSymbolE12
            }{}
            \DeclareFontShape{OMX}{MnSymbolE}{b}{n}{
                <-6>  MnSymbolE-Bold5
               <6-7>  MnSymbolE-Bold6
               <7-8>  MnSymbolE-Bold7
               <8-9>  MnSymbolE-Bold8
               <9-10> MnSymbolE-Bold9
              <10-12> MnSymbolE-Bold10
              <12->   MnSymbolE-Bold12
            }{}
            
            \let\llangle\@undefined
            \let\rrangle\@undefined
            \DeclareMathDelimiter{\llangle}{\mathopen}%
                                 {MnLargeSymbols}{'164}{MnLargeSymbols}{'164}
            \DeclareMathDelimiter{\rrangle}{\mathclose}%
                                 {MnLargeSymbols}{'171}{MnLargeSymbols}{'171}
            \makeatother
            
    \DeclareFontFamily{U}{matha}{\hyphenchar\font45}
    \DeclareFontShape{U}{matha}{m}{n}{ <-6> matha5 <6-7> matha6 <7-8>
    matha7 <8-9> matha8 <9-10> matha9 <10-12> matha10 <12-> matha12 }{}
    \DeclareSymbolFont{matha}{U}{matha}{m}{n}
    \DeclareFontFamily{U}{mathx}{\hyphenchar\font45}
    \DeclareFontShape{U}{mathx}{m}{n}{ <-6> mathx5 <6-7> mathx6 <7-8>
    mathx7 <8-9> mathx8 <9-10> mathx9 <10-12> mathx10 <12-> mathx12 }{}
    \DeclareSymbolFont{mathx}{U}{mathx}{m}{n}
    
    \DeclareMathDelimiter{\llbrack} {4}{matha}{"76}{mathx}{"30}
    \DeclareMathDelimiter{\rrbrack} {5}{matha}{"77}{mathx}{"38}

\renewcommand{\le}{\leqslant}

\renewcommand{\ge}{\geqslant}

\newcommand{\Tri}{\mathsf{Tri}}

\newcommand{\unif}{\mathsf{unif}}
\newcommand{\Slope}{\mathrm{Slope}}
\newcommand{\Area}{\mathrm{Area}}
\newcommand{\Interior}{\mathrm{Interior}}
\newcommand{\irr}{\mathsf{irr}}
\newcommand{\reg}{\mathsf{reg}}
\newcommand{\far}{\mathsf{far}}

\newtheorem{theorem}{Theorem}[section]
\newtheorem{proposition}[theorem]{Proposition}
\newtheorem{corollary}[theorem]{Corollary}
\newtheorem{lemma}[theorem]{Lemma}
\newtheorem{claim}[theorem]{Claim}

\theoremstyle{definition} % uncomment to make the following environments non-italicized
\newtheorem{definition}[theorem]{Definition}

\newtheorem{remark}[theorem]{Remark}

\newenvironment{proofclaim}[1][Proof]
	{\begin{proof}[#1]}
	{\end{proof}}

\title[Longest convex chains]{Longest convex chains with i.i.d.\ points}



\author[E. Bates]{Erik Bates$^*$\,\orcidlink{0000-0002-3472-036X}}
\address{$^*$Department of Mathematics, North Carolina State University, \texttt{ebates@ncsu.edu}}
\author[A. Sen]{Arnab Sen$^\dagger$}
\address{$^\dagger$School of Mathematics, University of Minnesota, \texttt{arnab@umn.edu}} 
\begin{document}

% changes footnote labeling back to numbers
%\renewcommand{\thefootnote}{\arabic{footnote}} \setcounter{footnote}{0}

%: ABSTRACT
\begin{abstract}
Sample $n$ i.i.d.\ points from a triangle, according to some continuous density function.
Given two vertices $\nA,\nB$ of the triangle, what is the maximum number of samples that form a convex chain with initial point $\nA$ and terminal point $\nB$?
We show that to leading order, the answer is $cn^{1/3}$, generalizing a result of Ambrus and B\'{a}r\'{a}ny \cite{ambrus_barany09} that considered uniformly distributed points.
Furthermore, we express the constant $c$ using a variational formula whose maximizer (if unique) gives the limiting curve formed by the longest convex chain.
By comparison, for $n$ i.i.d.\ samples from the unit square, the length of the longest monotone chain is asymptotically $c'n^{1/2}$.
Despite the difference in scale, our formula is nicely connected to one established for $c'$ by Deuschel and Zeitouni \cite{deuschel_zeitouni95}.
\end{abstract}

\maketitle
%\setcounter{tocdepth}{1} %1 for just sections, 2 for subsections
%\vspace{-2\baselineskip}
\tableofcontents

% uncomment to display line numbers
%\linenumbers

\section{Introduction}
Let $\cT$ denote the triangle with vertices $(0,0)$, $(1,1)$, and $(1,0)$.
A finite set 
\eq{
\cC=\{(x_1,y_1),\ldots,(x_k,y_k)\}\subset\cT\setminus\{(0,0),(1,1)\}
}
is said to be a \textit{convex chain} of \textit{length} $k$ if $\cC\cup\{(0,0),(1,1)\}$ is in convex position, meaning none of the $k+2$ points lie in the convex hull of the others.
Equivalently, if the elements of $\cC$ are labeled such that $x_1\le\cdots\le x_k$, then the following sequence of slopes is strictly increasing (see Lemma~\ref{lem_position_slopes}):
\eeq{ \label{convex_chain}
\frac{y_1-0}{x_1-0} < \frac{y_2-y_1}{x_2-x_1} < \cdots < \frac{y_k-y_{k-1}}{x_k-x_{k-1}} < \frac{1-y_k}{1-x_k}.
}
Figure~\ref{fig_chain_illustration} (right) provides an illustration.

In this paper, we investigate the longest convex chains within a random collection of points.
Namely, we consider $n$ i.i.d.\ samples from $\cT$, and study the longest convex chain found as a subset of these samples.
Of primary interest are the length and shape of such a chain.
We will show that in the large-$n$ limit, these two quantities satisfy a variational relationship (Theorem~\ref{thm_var_formula}).
Moreover, the optimizers of the variational expression describe the limiting shape of longest convex chains (Theorem~\ref{thm_limit_curves}).
In fact, this last statement holds even when conditioning on the event that all $n$ points form a convex chain (Theorem~\ref{thm_conditional}).

The problem of longest convex chains is a natural extension of the problem of longest monotone chains in the unit square $[0,1]^2$, illustrated in Figure~\ref{fig_chain_illustration} (left).
Moreover, when the $n$ i.i.d.~samples are uniform in $[0,1]^2$, the longest monotone chain corresponds to the longest increasing subsequence of a uniformly random permutation on $n$ elements.
The latter problem has enjoyed a rich mathematical story \cite{romik15}, and together with its infinite-space version called Poissonian last-passage percolation \cite{aldous_diaconis95,baik_deift_johansson99,johansson00b}, has been an influential prototype for the Kardar--Parisi--Zhang (KPZ) universality class \cite{corwin16,dauvergne_virag21_arxiv,baik23}.
As far as we know, there is yet to be any understanding of---or even speculation on---the fluctuations theory for longest convex chains.
The present paper addresses the first-order behavior, setting the stage for future investigations in this direction.

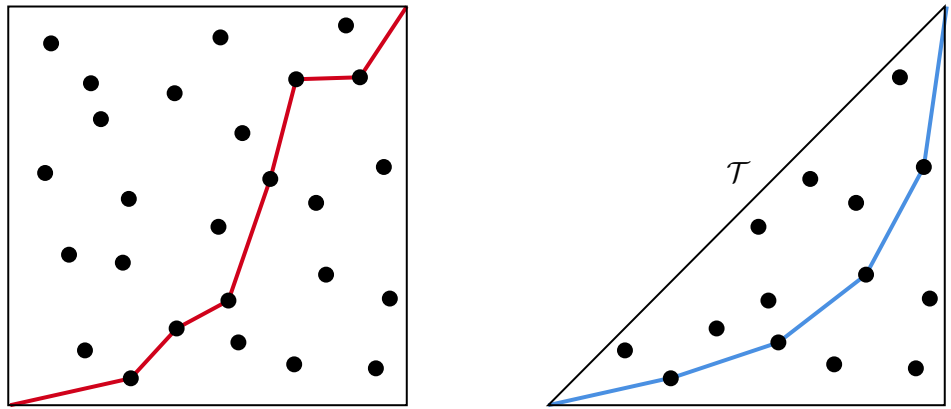
\begin{figure}[h]

\tikzset{every picture/.style={line width=0.75pt}} %set default line width to 0.75pt        

\begin{tikzpicture}[x=0.75pt,y=0.75pt,yscale=-1,xscale=1]
%uncomment if require: \path (0,234); %set diagram left start at 0, and has height of 234

%Straight Lines [id:da9214667924512325] 
\draw [color={rgb, 255:red, 208; green, 2; blue, 27 }  ,draw opacity=1 ][line width=1.5]    (91,210) -- (151.5,196.5) ;
%Straight Lines [id:da09333397853114078] 
\draw [color={rgb, 255:red, 208; green, 2; blue, 27 }  ,draw opacity=1 ][line width=1.5]    (151.5,196.5) -- (174.5,171.5) ;
%Straight Lines [id:da3006669536268972] 
\draw [color={rgb, 255:red, 208; green, 2; blue, 27 }  ,draw opacity=1 ][line width=1.5]    (221.5,96.5) -- (234.5,46.5) ;
%Straight Lines [id:da23815827366184816] 
\draw [color={rgb, 255:red, 208; green, 2; blue, 27 }  ,draw opacity=1 ][line width=1.5]    (200.5,157.5) -- (221.5,96.5) ;
%Straight Lines [id:da4895776239447486] 
\draw [color={rgb, 255:red, 208; green, 2; blue, 27 }  ,draw opacity=1 ][line width=1.5]    (174.5,171.5) -- (200.5,157.5) ;
%Straight Lines [id:da9674630190559386] 
\draw [color={rgb, 255:red, 208; green, 2; blue, 27 }  ,draw opacity=1 ][line width=1.5]    (234.5,46.5) -- (266.5,45.5) ;
%Straight Lines [id:da21578866604766833] 
\draw [color={rgb, 255:red, 208; green, 2; blue, 27 }  ,draw opacity=1 ][line width=1.5]    (266.5,45.5) -- (290,10) ;
%Shape: Circle [id:dp39894076628191677] 
\draw  [fill={rgb, 255:red, 0; green, 0; blue, 0 }  ,fill opacity=1 ] (271,191.5) .. controls (271,189.57) and (272.57,188) .. (274.5,188) .. controls (276.43,188) and (278,189.57) .. (278,191.5) .. controls (278,193.43) and (276.43,195) .. (274.5,195) .. controls (272.57,195) and (271,193.43) .. (271,191.5) -- cycle ;
%Shape: Circle [id:dp7611187641720868] 
\draw  [fill={rgb, 255:red, 0; green, 0; blue, 0 }  ,fill opacity=1 ] (148,196.5) .. controls (148,194.57) and (149.57,193) .. (151.5,193) .. controls (153.43,193) and (155,194.57) .. (155,196.5) .. controls (155,198.43) and (153.43,200) .. (151.5,200) .. controls (149.57,200) and (148,198.43) .. (148,196.5) -- cycle ;
%Shape: Circle [id:dp5003200953300776] 
\draw  [fill={rgb, 255:red, 0; green, 0; blue, 0 }  ,fill opacity=1 ] (218,96.5) .. controls (218,94.57) and (219.57,93) .. (221.5,93) .. controls (223.43,93) and (225,94.57) .. (225,96.5) .. controls (225,98.43) and (223.43,100) .. (221.5,100) .. controls (219.57,100) and (218,98.43) .. (218,96.5) -- cycle ;
%Shape: Circle [id:dp8919244289704852] 
\draw  [fill={rgb, 255:red, 0; green, 0; blue, 0 }  ,fill opacity=1 ] (192,120.5) .. controls (192,118.57) and (193.57,117) .. (195.5,117) .. controls (197.43,117) and (199,118.57) .. (199,120.5) .. controls (199,122.43) and (197.43,124) .. (195.5,124) .. controls (193.57,124) and (192,122.43) .. (192,120.5) -- cycle ;
%Shape: Circle [id:dp8495533438982282] 
\draw  [fill={rgb, 255:red, 0; green, 0; blue, 0 }  ,fill opacity=1 ] (230,189.5) .. controls (230,187.57) and (231.57,186) .. (233.5,186) .. controls (235.43,186) and (237,187.57) .. (237,189.5) .. controls (237,191.43) and (235.43,193) .. (233.5,193) .. controls (231.57,193) and (230,191.43) .. (230,189.5) -- cycle ;
%Shape: Circle [id:dp7148908853532814] 
\draw  [fill={rgb, 255:red, 0; green, 0; blue, 0 }  ,fill opacity=1 ] (197,157.5) .. controls (197,155.57) and (198.57,154) .. (200.5,154) .. controls (202.43,154) and (204,155.57) .. (204,157.5) .. controls (204,159.43) and (202.43,161) .. (200.5,161) .. controls (198.57,161) and (197,159.43) .. (197,157.5) -- cycle ;
%Shape: Circle [id:dp9399546418933054] 
\draw  [fill={rgb, 255:red, 0; green, 0; blue, 0 }  ,fill opacity=1 ] (263,45.5) .. controls (263,43.57) and (264.57,42) .. (266.5,42) .. controls (268.43,42) and (270,43.57) .. (270,45.5) .. controls (270,47.43) and (268.43,49) .. (266.5,49) .. controls (264.57,49) and (263,47.43) .. (263,45.5) -- cycle ;
%Shape: Circle [id:dp7271170775739273] 
\draw  [fill={rgb, 255:red, 0; green, 0; blue, 0 }  ,fill opacity=1 ] (125,182.5) .. controls (125,180.57) and (126.57,179) .. (128.5,179) .. controls (130.43,179) and (132,180.57) .. (132,182.5) .. controls (132,184.43) and (130.43,186) .. (128.5,186) .. controls (126.57,186) and (125,184.43) .. (125,182.5) -- cycle ;
%Shape: Circle [id:dp3735228536137528] 
\draw  [fill={rgb, 255:red, 0; green, 0; blue, 0 }  ,fill opacity=1 ] (278,156.5) .. controls (278,154.57) and (279.57,153) .. (281.5,153) .. controls (283.43,153) and (285,154.57) .. (285,156.5) .. controls (285,158.43) and (283.43,160) .. (281.5,160) .. controls (279.57,160) and (278,158.43) .. (278,156.5) -- cycle ;
%Shape: Circle [id:dp21754725510069362] 
\draw  [fill={rgb, 255:red, 0; green, 0; blue, 0 }  ,fill opacity=1 ] (171,171.5) .. controls (171,169.57) and (172.57,168) .. (174.5,168) .. controls (176.43,168) and (178,169.57) .. (178,171.5) .. controls (178,173.43) and (176.43,175) .. (174.5,175) .. controls (172.57,175) and (171,173.43) .. (171,171.5) -- cycle ;
%Shape: Circle [id:dp9197423949664333] 
\draw  [fill={rgb, 255:red, 0; green, 0; blue, 0 }  ,fill opacity=1 ] (241,108.5) .. controls (241,106.57) and (242.57,105) .. (244.5,105) .. controls (246.43,105) and (248,106.57) .. (248,108.5) .. controls (248,110.43) and (246.43,112) .. (244.5,112) .. controls (242.57,112) and (241,110.43) .. (241,108.5) -- cycle ;
%Shape: Circle [id:dp46497147885609147] 
\draw  [fill={rgb, 255:red, 0; green, 0; blue, 0 }  ,fill opacity=1 ] (246,144.5) .. controls (246,142.57) and (247.57,141) .. (249.5,141) .. controls (251.43,141) and (253,142.57) .. (253,144.5) .. controls (253,146.43) and (251.43,148) .. (249.5,148) .. controls (247.57,148) and (246,146.43) .. (246,144.5) -- cycle ;
%Shape: Circle [id:dp6933498355223352] 
\draw  [fill={rgb, 255:red, 0; green, 0; blue, 0 }  ,fill opacity=1 ] (275,90.5) .. controls (275,88.57) and (276.57,87) .. (278.5,87) .. controls (280.43,87) and (282,88.57) .. (282,90.5) .. controls (282,92.43) and (280.43,94) .. (278.5,94) .. controls (276.57,94) and (275,92.43) .. (275,90.5) -- cycle ;
%Straight Lines [id:da8532986746618081] 
\draw [color={rgb, 255:red, 74; green, 144; blue, 226 }  ,draw opacity=1 ][line width=1.5]    (361,210) -- (421.5,196.5) ;
%Straight Lines [id:da24409245213865427] 
\draw [color={rgb, 255:red, 74; green, 144; blue, 226 }  ,draw opacity=1 ][line width=1.5]    (476.5,178.5) -- (520.5,144.5) ;
%Straight Lines [id:da7450861494633589] 
\draw [color={rgb, 255:red, 74; green, 144; blue, 226 }  ,draw opacity=1 ][line width=1.5]    (520.5,144.5) -- (549.5,90.5) ;
%Straight Lines [id:da08646463120144066] 
\draw [color={rgb, 255:red, 74; green, 144; blue, 226 }  ,draw opacity=1 ][line width=1.5]    (549.5,90.5) -- (561,10) ;
%Shape: Right Triangle [id:dp26180706924311115] 
\draw   (560,10) -- (361,210) -- (560,210) -- cycle ;
%Shape: Circle [id:dp9849648263113367] 
\draw  [fill={rgb, 255:red, 0; green, 0; blue, 0 }  ,fill opacity=1 ] (202,178.5) .. controls (202,176.57) and (203.57,175) .. (205.5,175) .. controls (207.43,175) and (209,176.57) .. (209,178.5) .. controls (209,180.43) and (207.43,182) .. (205.5,182) .. controls (203.57,182) and (202,180.43) .. (202,178.5) -- cycle ;
%Straight Lines [id:da023682792255474205] 
\draw [color={rgb, 255:red, 74; green, 144; blue, 226 }  ,draw opacity=1 ][line width=1.5]    (422.5,196.5) -- (476.5,178.5) ;
%Shape: Square [id:dp162894337611024] 
\draw   (90,10) -- (290,10) -- (290,210) -- (90,210) -- cycle ;
%Shape: Circle [id:dp9878751115424632] 
\draw  [fill={rgb, 255:red, 0; green, 0; blue, 0 }  ,fill opacity=1 ] (231,46.5) .. controls (231,44.57) and (232.57,43) .. (234.5,43) .. controls (236.43,43) and (238,44.57) .. (238,46.5) .. controls (238,48.43) and (236.43,50) .. (234.5,50) .. controls (232.57,50) and (231,48.43) .. (231,46.5) -- cycle ;
%Shape: Circle [id:dp8024432176927736] 
\draw  [fill={rgb, 255:red, 0; green, 0; blue, 0 }  ,fill opacity=1 ] (204,73.5) .. controls (204,71.57) and (205.57,70) .. (207.5,70) .. controls (209.43,70) and (211,71.57) .. (211,73.5) .. controls (211,75.43) and (209.43,77) .. (207.5,77) .. controls (205.57,77) and (204,75.43) .. (204,73.5) -- cycle ;
%Shape: Circle [id:dp3630068964191788] 
\draw  [fill={rgb, 255:red, 0; green, 0; blue, 0 }  ,fill opacity=1 ] (193,25.5) .. controls (193,23.57) and (194.57,22) .. (196.5,22) .. controls (198.43,22) and (200,23.57) .. (200,25.5) .. controls (200,27.43) and (198.43,29) .. (196.5,29) .. controls (194.57,29) and (193,27.43) .. (193,25.5) -- cycle ;
%Shape: Circle [id:dp4219952630545679] 
\draw  [fill={rgb, 255:red, 0; green, 0; blue, 0 }  ,fill opacity=1 ] (133,66.5) .. controls (133,64.57) and (134.57,63) .. (136.5,63) .. controls (138.43,63) and (140,64.57) .. (140,66.5) .. controls (140,68.43) and (138.43,70) .. (136.5,70) .. controls (134.57,70) and (133,68.43) .. (133,66.5) -- cycle ;
%Shape: Circle [id:dp7689317688709627] 
\draw  [fill={rgb, 255:red, 0; green, 0; blue, 0 }  ,fill opacity=1 ] (117,134.5) .. controls (117,132.57) and (118.57,131) .. (120.5,131) .. controls (122.43,131) and (124,132.57) .. (124,134.5) .. controls (124,136.43) and (122.43,138) .. (120.5,138) .. controls (118.57,138) and (117,136.43) .. (117,134.5) -- cycle ;
%Shape: Circle [id:dp8922176385586739] 
\draw  [fill={rgb, 255:red, 0; green, 0; blue, 0 }  ,fill opacity=1 ] (108,28.5) .. controls (108,26.57) and (109.57,25) .. (111.5,25) .. controls (113.43,25) and (115,26.57) .. (115,28.5) .. controls (115,30.43) and (113.43,32) .. (111.5,32) .. controls (109.57,32) and (108,30.43) .. (108,28.5) -- cycle ;
%Shape: Circle [id:dp5924537471927082] 
\draw  [fill={rgb, 255:red, 0; green, 0; blue, 0 }  ,fill opacity=1 ] (128,48.5) .. controls (128,46.57) and (129.57,45) .. (131.5,45) .. controls (133.43,45) and (135,46.57) .. (135,48.5) .. controls (135,50.43) and (133.43,52) .. (131.5,52) .. controls (129.57,52) and (128,50.43) .. (128,48.5) -- cycle ;
%Shape: Circle [id:dp8076467714615707] 
\draw  [fill={rgb, 255:red, 0; green, 0; blue, 0 }  ,fill opacity=1 ] (256,19.5) .. controls (256,17.57) and (257.57,16) .. (259.5,16) .. controls (261.43,16) and (263,17.57) .. (263,19.5) .. controls (263,21.43) and (261.43,23) .. (259.5,23) .. controls (257.57,23) and (256,21.43) .. (256,19.5) -- cycle ;
%Shape: Circle [id:dp45388744681255266] 
\draw  [fill={rgb, 255:red, 0; green, 0; blue, 0 }  ,fill opacity=1 ] (170,53.5) .. controls (170,51.57) and (171.57,50) .. (173.5,50) .. controls (175.43,50) and (177,51.57) .. (177,53.5) .. controls (177,55.43) and (175.43,57) .. (173.5,57) .. controls (171.57,57) and (170,55.43) .. (170,53.5) -- cycle ;
%Shape: Circle [id:dp6586442084443611] 
\draw  [fill={rgb, 255:red, 0; green, 0; blue, 0 }  ,fill opacity=1 ] (147,106.5) .. controls (147,104.57) and (148.57,103) .. (150.5,103) .. controls (152.43,103) and (154,104.57) .. (154,106.5) .. controls (154,108.43) and (152.43,110) .. (150.5,110) .. controls (148.57,110) and (147,108.43) .. (147,106.5) -- cycle ;
%Shape: Circle [id:dp575140655618047] 
\draw  [fill={rgb, 255:red, 0; green, 0; blue, 0 }  ,fill opacity=1 ] (105,93.5) .. controls (105,91.57) and (106.57,90) .. (108.5,90) .. controls (110.43,90) and (112,91.57) .. (112,93.5) .. controls (112,95.43) and (110.43,97) .. (108.5,97) .. controls (106.57,97) and (105,95.43) .. (105,93.5) -- cycle ;
%Shape: Circle [id:dp28791670716833573] 
\draw  [fill={rgb, 255:red, 0; green, 0; blue, 0 }  ,fill opacity=1 ] (144,138.5) .. controls (144,136.57) and (145.57,135) .. (147.5,135) .. controls (149.43,135) and (151,136.57) .. (151,138.5) .. controls (151,140.43) and (149.43,142) .. (147.5,142) .. controls (145.57,142) and (144,140.43) .. (144,138.5) -- cycle ;
%Shape: Circle [id:dp9104491449508699] 
\draw  [fill={rgb, 255:red, 0; green, 0; blue, 0 }  ,fill opacity=1 ] (542,191.5) .. controls (542,189.57) and (543.57,188) .. (545.5,188) .. controls (547.43,188) and (549,189.57) .. (549,191.5) .. controls (549,193.43) and (547.43,195) .. (545.5,195) .. controls (543.57,195) and (542,193.43) .. (542,191.5) -- cycle ;
%Shape: Circle [id:dp05674030791799678] 
\draw  [fill={rgb, 255:red, 0; green, 0; blue, 0 }  ,fill opacity=1 ] (419,196.5) .. controls (419,194.57) and (420.57,193) .. (422.5,193) .. controls (424.43,193) and (426,194.57) .. (426,196.5) .. controls (426,198.43) and (424.43,200) .. (422.5,200) .. controls (420.57,200) and (419,198.43) .. (419,196.5) -- cycle ;
%Shape: Circle [id:dp47767892774960774] 
\draw  [fill={rgb, 255:red, 0; green, 0; blue, 0 }  ,fill opacity=1 ] (489,96.5) .. controls (489,94.57) and (490.57,93) .. (492.5,93) .. controls (494.43,93) and (496,94.57) .. (496,96.5) .. controls (496,98.43) and (494.43,100) .. (492.5,100) .. controls (490.57,100) and (489,98.43) .. (489,96.5) -- cycle ;
%Shape: Circle [id:dp34129070332707223] 
\draw  [fill={rgb, 255:red, 0; green, 0; blue, 0 }  ,fill opacity=1 ] (463,120.5) .. controls (463,118.57) and (464.57,117) .. (466.5,117) .. controls (468.43,117) and (470,118.57) .. (470,120.5) .. controls (470,122.43) and (468.43,124) .. (466.5,124) .. controls (464.57,124) and (463,122.43) .. (463,120.5) -- cycle ;
%Shape: Circle [id:dp48170061514771945] 
\draw  [fill={rgb, 255:red, 0; green, 0; blue, 0 }  ,fill opacity=1 ] (501,189.5) .. controls (501,187.57) and (502.57,186) .. (504.5,186) .. controls (506.43,186) and (508,187.57) .. (508,189.5) .. controls (508,191.43) and (506.43,193) .. (504.5,193) .. controls (502.57,193) and (501,191.43) .. (501,189.5) -- cycle ;
%Shape: Circle [id:dp7752593026189388] 
\draw  [fill={rgb, 255:red, 0; green, 0; blue, 0 }  ,fill opacity=1 ] (468,157.5) .. controls (468,155.57) and (469.57,154) .. (471.5,154) .. controls (473.43,154) and (475,155.57) .. (475,157.5) .. controls (475,159.43) and (473.43,161) .. (471.5,161) .. controls (469.57,161) and (468,159.43) .. (468,157.5) -- cycle ;
%Shape: Circle [id:dp5143902564994175] 
\draw  [fill={rgb, 255:red, 0; green, 0; blue, 0 }  ,fill opacity=1 ] (534,45.5) .. controls (534,43.57) and (535.57,42) .. (537.5,42) .. controls (539.43,42) and (541,43.57) .. (541,45.5) .. controls (541,47.43) and (539.43,49) .. (537.5,49) .. controls (535.57,49) and (534,47.43) .. (534,45.5) -- cycle ;
%Shape: Circle [id:dp17325062368008226] 
\draw  [fill={rgb, 255:red, 0; green, 0; blue, 0 }  ,fill opacity=1 ] (396,182.5) .. controls (396,180.57) and (397.57,179) .. (399.5,179) .. controls (401.43,179) and (403,180.57) .. (403,182.5) .. controls (403,184.43) and (401.43,186) .. (399.5,186) .. controls (397.57,186) and (396,184.43) .. (396,182.5) -- cycle ;
%Shape: Circle [id:dp6083909073279076] 
\draw  [fill={rgb, 255:red, 0; green, 0; blue, 0 }  ,fill opacity=1 ] (549,156.5) .. controls (549,154.57) and (550.57,153) .. (552.5,153) .. controls (554.43,153) and (556,154.57) .. (556,156.5) .. controls (556,158.43) and (554.43,160) .. (552.5,160) .. controls (550.57,160) and (549,158.43) .. (549,156.5) -- cycle ;
%Shape: Circle [id:dp8340236741960988] 
\draw  [fill={rgb, 255:red, 0; green, 0; blue, 0 }  ,fill opacity=1 ] (442,171.5) .. controls (442,169.57) and (443.57,168) .. (445.5,168) .. controls (447.43,168) and (449,169.57) .. (449,171.5) .. controls (449,173.43) and (447.43,175) .. (445.5,175) .. controls (443.57,175) and (442,173.43) .. (442,171.5) -- cycle ;
%Shape: Circle [id:dp8590169408513914] 
\draw  [fill={rgb, 255:red, 0; green, 0; blue, 0 }  ,fill opacity=1 ] (512,108.5) .. controls (512,106.57) and (513.57,105) .. (515.5,105) .. controls (517.43,105) and (519,106.57) .. (519,108.5) .. controls (519,110.43) and (517.43,112) .. (515.5,112) .. controls (513.57,112) and (512,110.43) .. (512,108.5) -- cycle ;
%Shape: Circle [id:dp6091611881054928] 
\draw  [fill={rgb, 255:red, 0; green, 0; blue, 0 }  ,fill opacity=1 ] (517,144.5) .. controls (517,142.57) and (518.57,141) .. (520.5,141) .. controls (522.43,141) and (524,142.57) .. (524,144.5) .. controls (524,146.43) and (522.43,148) .. (520.5,148) .. controls (518.57,148) and (517,146.43) .. (517,144.5) -- cycle ;
%Shape: Circle [id:dp7622250265419095] 
\draw  [fill={rgb, 255:red, 0; green, 0; blue, 0 }  ,fill opacity=1 ] (546,90.5) .. controls (546,88.57) and (547.57,87) .. (549.5,87) .. controls (551.43,87) and (553,88.57) .. (553,90.5) .. controls (553,92.43) and (551.43,94) .. (549.5,94) .. controls (547.57,94) and (546,92.43) .. (546,90.5) -- cycle ;
%Shape: Circle [id:dp38857915819004496] 
\draw  [fill={rgb, 255:red, 0; green, 0; blue, 0 }  ,fill opacity=1 ] (473,178.5) .. controls (473,176.57) and (474.57,175) .. (476.5,175) .. controls (478.43,175) and (480,176.57) .. (480,178.5) .. controls (480,180.43) and (478.43,182) .. (476.5,182) .. controls (474.57,182) and (473,180.43) .. (473,178.5) -- cycle ;

% Text Node
\draw (449,86.4) node [anchor=north west][inner sep=0.75pt]    {$\mathcal{T}$};

\end{tikzpicture}
\caption{
\textit{Left}: A monotone chain of length $6$ among $26$ sample points.
\textit{Right}: A convex chain of length $4$ among $14$ sample points.
It is more natural to consider convex chains in the triangle $\cT$ than in the unit square $[0,1]^2$, since any sequence $\{(x_i,y_i)\}_{i=1}^k$ in $[0,1]^2$ that satisfies \eqref{convex_chain} is necessarily confined to $\cT$.
See also Remark~\ref{rem_convex_position}.
}
\label{fig_chain_illustration}
\end{figure}

\subsection{Definitions and assumptions} \label{subsec_defs_assumptions}
Let $\fp\colon\cT\to[0,\infty)$ be a probability density function on the triangle $\cT$.
Consider a collection $\cS_n = \{(X_1,Y_1),\ldots,(X_n,Y_n)\}$ of $n$ independent samples from $\fp$.
For convenience we assume $(\cS_n)_{n\ge1}$ are defined on the same probability space, but we make no assumptions about the coupling.
The greatest length of a convex chain in $\cS_n$ will be denoted by
\eq{
\sL_n \coloneqq \max\{\#\cC:\, \text{$\cC\subseteq\cS_n$ and $\cC$ is a convex chain}\}.
}
We make the following assumptions about the density function $\fp$:
\begin{gather}
\text{$\fp$ is continuous;} \label{p_cont} \tag{H1} \\
\text{there exists a constant $\fc>0$ such that $\inf \fp \ge 2\fc$.} \label{p_lower} \tag{H2}
\end{gather}
In some places we will be able to replace \eqref{p_cont} with a weaker assumption:
\begin{align} \label{p_upper}
\text{there exists a constant $\fC<\infty$ such that $\sup \fp \le 2\fC$.} \tag{H0}
\end{align}
Uniform sampling corresponds to the constant density function $\fp \equiv 1/\Area(\cT) = 2$, in which case one can take $\fc=\fC=1$.
In general, the constants $\fc$ and $\fC$ control how much $\fp$ differs from being uniform.

In order to capture the possible trajectories of convex chains, we consider the following class of functions:
\eeq{ \label{F_def}
\mathcal{F} \coloneqq \{f\colon[0,1]\to[0,1]:\,\text{$f$ is convex and continuous, $f(0) = 0$}\}.
}
The only reason for the continuity condition in \eqref{F_def} is to ensure that $f(1) = \lim_{x\nearrow1}f(x)$; after all, continuity in $[0,1)$ is automatic from convexity and $f(0)=0$.
Furthermore, it is a classical fact that every convex function is twice differentiable almost everywhere.
Therefore, we can define a functional $J\colon\mathcal{F}\to[0,\infty)$ by
\eeq{ \label{J_def1}
J(f) = J_{\fp}(f) \coloneqq \int_0^1 \big[f''(x)\cdot \fp(x,f(x))\big]^{1/3}\ \dd x.
}
This functional will be central in our main results.
We will say more about its derivation in Section~\ref{subsec_ideas}, but it also has an interpretation from affine differential geometry.
Namely, the integral $\int_0^1 f''(x)^{1/3}\, \dd x$ is the equi-affine arclength of the graph of $f$, so $J(f)$ may be viewed as a \textit{density-weighted} affine arclength.
As we will soon see, longest convex chains tend to resemble curves that maximize this arclength.

\subsection{Main results} \label{subsec_main_results}

Our first result gives the leading asymptotic behavior of $\sL_n$.

\begin{theorem}[Growth rate of longest convex chain] \label{thm_var_formula}
Assume \eqref{p_cont}.
There exists a universal constant $\alpha\in(0,\infty)$ (not depending on the density function $\fp$) such that
\eeq{ \label{variational_formula}
\lim_{n\to\infty}\frac{\sL_n}{n^{1/3}} = \frac{\alpha}{2}\sup_{f\in\cF}J_\fp(f) \quad \text{almost surely and in $L^p$ for every $p\in[1,\infty)$.}
}
\end{theorem}

The intuition behind the variational formula \eqref{variational_formula} is given in Section~\ref{variational_formula}.
While establishing convergence requires considerable effort (indeed, that is the central goal of this paper), identifying the scale $n^{1/3}$ does not.
For instance, it is not difficult to obtain a crude upper bound (e.g.\ Proposition~\ref{prop_general}\ref{prop_general_tail}) using an exact expression from \cite{barany_rote_steiger_zhang00} for the likelihood that $n$ independent uniform points form a convex chain.
This expression is a cousin of Valtr's formula \cite{valtr96} and is quoted as \eqref{0o0nc}.
We give a short alternative proof in Remark~\ref{rem_alternative}.

The constant $\alpha$ comes from the uniform case $\fp\equiv2$, which was considered by Ambrus and B\'{a}r\'{a}ny \cite{ambrus_barany09}.
In that case, the right-hand side of \eqref{variational_formula} is equal to $\alpha$; see Theorem~\ref{thm_uniform_known} and the discussion that follows.
The novelty here is to extend that result to non-uniform densities and give a variational expression \eqref{variational_formula} for the growth constant. 

For convenience, we denote the value of the supremum in \eqref{variational_formula} by 
\eeq{ \label{Jstar_def}
J_\star = J_\star(\fp) \coloneqq \sup_{f\in\cF}J(f)
= \sup_{f\in\cF}\int_0^1\big[f''(x)\cdot\fp(x,f(x))\big]^{1/3}\ \dd x.
}
The set of maximizers will be denoted by
\eeq{ \label{argmax_def}
\argmax J \coloneqq \{f\in\cF:\, J(f) = J_\star\}.
}
Analyzing this set---even showing that it is nonempty---is made difficult by the fact that $f\mapsto J(f)$ is not continuous.
Nevertheless, we verify in Proposition~\ref{prop_usc} that $J$ is upper semicontinuous in a certain topology on $\cF$, from which it follows that $\argmax J$ is nonempty (Proposition~\ref{prop_maximizers}).
Furthermore, we prove that every maximizer is strictly convex and has an absolutely continuous first derivative (Proposition~\ref{prop_maximizer_properties}).
The supremum in \eqref{variational_formula} is unchanged if we restrict to $C^\infty$ functions (see Lemma~\ref{lem_smooth_f}), but genuine maximizers might not be so nice.

Following Theorem~\ref{thm_var_formula}, one expects that any longest convex chain should ``look like'' a maximizer of $J$.
This is the content of the next result.
Given $\eps>0$, define the quantity
\eeq{ \label{L_neps_def}
\sL_n^\far(\eps) \coloneqq \max\Big\{\#\cC:\, \text{$\cC\subseteq\cS_n$, $\cC$ is a convex chain,  $\displaystyle\inf_{f\in\argmax J}\max_{(x,y)\in\cC}|f(x)-y|\ge\eps$}\Big\}.
}
In words, $\sL_n^\far(\eps)$ is the maximum size of a convex chain that is $\eps$-far (at one or more points) from every optimizer $f$.
So we always have $\sL_n^\far(\eps)\le\sL_n$, and equality occurs if and only if there exists a longest convex chain that is $\eps$-far from every optimizer.

\begin{theorem}[Longest convex chains concentrate near optimizers] \label{thm_limit_curves}
Assume \eqref{p_cont} and \eqref{p_lower}.
Then for every $\eps>0$, there exists $\delta=\delta(\eps,\fp)>0$ such that
\eeq{ \label{limit_curves_prob}
\limsup_{n\to\infty}\frac{1}{n^{1/3}}\log\P\big(\sL_n^\far(\eps)\ge\sL_n-\delta n^{1/3}\big) < 0. 
}
In particular, the following statement holds with probability one.
For any choice of convex chain $\cC_n\subseteq\cS_n$ such that $\#\cC_n \ge \sL_n-\delta n^{1/3}$, we have
\eeq{ \label{410bxw}
\limsup_{n\to\infty}\inf_{f\in\argmax J}\max_{(x,y)\in\cC_n}|f(x)-y| \le \eps.
}
\end{theorem}

If $J$ has a unique maximizer, then \eqref{410bxw} gives convergence to a single deterministic curve.
For instance, this is the case when $\fp$ is uniform (see Proposition~\ref{prop_unif_case}) or more generally when $\fp$ is concave in its second coordinate (Proposition~\ref{prop_unique_y_concave}).
In general, however, there may be multiple maximizers (see Proposition~\ref{prop_multiple_optimizers}), perhaps even uncountably many.
Consequently, it is not immediately clear that the random quantity $\sL_n^\far(\eps)$ in \eqref{L_neps_def} is measurable, but thankfully measurability follows from the compactness of $\argmax J$ (see Lemma~\ref{lem_measurable}).
Other properties of $\argmax J$ are given in Propositions~\ref{prop_maximizer_properties} and \ref{prop_equicontinuity}.

Note that Theorem~\ref{thm_limit_curves} is a result about ``typical'' longest convex chains, in the sense that we assume nothing about the value of $\sL_n$.
If we condition on the event $\{\sL_n=n\}$ (meaning the full sample set $\cS_n$ is a convex chain), then we obtain a stronger concentration result where the first factor of $n^{1/3}$ in \eqref{limit_curves_prob} is replaced by $n$ as in \eqref{thm_conditional_eq}.
This type of conditioning was also considered in \cite{barany_rote_steiger_zhang00,brosset25}.

\begin{theorem}[Conditioning on a full convex chain] \label{thm_conditional}
Assume \eqref{p_cont} and \eqref{p_lower}.
Then for every $\eps>0$, we have
\eeq{ \label{thm_conditional_eq}
\limsup_{n\to\infty}\frac{1}{n}\log\P\givenp[\Big]{\inf_{f\in\argmax J}\max_{(x,y)\in\cS_n}|f(x)-y| \ge \eps}{\text{$\cS_n$ is a convex chain}}<0.
}
\end{theorem}

It should be stressed that conditioning on $\{\sL_n=n\}$ places the model in a qualitatively different regime.
The fluctuations of $\sL_n$ are trivially $0$ in this conditional regime, whereas unconditional fluctuations presumably diverge as $n\to\infty$ (no faster than $n^{1/6+o(1)}$ according to Proposition~\ref{prop_general}\ref{prop_general_concentration}).
Deviations of the longest chains from the limiting curves should also be different under the conditioning.
At least in the uniform case $\fp\equiv 2$, these conditional deviations admit a Gaussian central limit theorem at scale $n^{-1/2}$ \cite[Theorem~3]{barany_rote_steiger_zhang00}.
One expects the corresponding unconditional deviations to be asymptotically much larger, since the extreme upper tail event $\{\sL_n = n\}$ is most readily achieved by concentrating many points near the maximizer.
Despite these expected differences, Theorems~\ref{thm_limit_curves} and Theorem~\ref{thm_conditional} establish concentration to the same set of curves, namely the maximizers of $J$.

The proofs of Theorems~\ref{thm_var_formula} and \ref{thm_limit_curves} (unconditional regime) and Theorem~\ref{thm_conditional} (conditional regime) share many common elements, so our presentation is mostly unified.
In some ways the unconditional regime is more difficult, such as accounting for fluctuations on various scales and proving that sample points lying in certain regions are very unlikely to participate in a longest convex chain.
In other ways the conditional regime is more challenging, arising from the fact that we condition on a very rare event.
In fact, our arguments allow us to determine the rarity up to subexponential correction:

\begin{theorem}[Likelihood of a full convex chain] \label{thm_probability}
Assume \eqref{p_cont} and \eqref{p_lower}.
Then
\eeq{ \label{thm_probability_eq}
\lim_{n\to\infty}\frac{1}{n}\log\Big[\frac{(3n)!}{n!}\P(\text{$\cS_n$ is a convex chain})\Big] = \log\Big(\frac{27J_\star^3}{4}\Big),
}
where $J_\star$ is defined in \eqref{Jstar_def}.
\end{theorem}

In the uniform case $\fp\equiv2$, one can compute $J_\star = 2$ using Proposition~\ref{prop_unif_case}\ref{prop_unif_case_a} with boundary data $(a,b,f_a,f_b,s_a,s_b) = (0,1,0,1,0,\infty)$.
In general, it follows from \eqref{p_upper} and \eqref{p_lower} that
\eeq{ \label{J_star_trivial}
2\fc^{1/3} \le J_\star \le 2\fC^{1/3}.
}

\begin{remark}[The underlying triangle does not matter] \label{rem_convex_position}
Given any nondegenerate triangle $\triangle$ with two distinguished vertices $\nA,\nB\in\R^2$, we can consider the affine transformation $\vphi$ such that $\vphi(\cT) = \triangle$, $\vphi(0,0) = \nA$, and $\vphi(1,1) = \nB$. 
Since being in convex position is invariant under bijective affine transformations, we have
\eq{
\text{$\{(0,0),(1,1)\}\cup\cC$ is in convex position}
\quad \iff \quad
\text{$\{\nA,\nB\}\cup\vphi(\cC)$ is in convex position}.
}
Moreover, $\vphi(\cS_n)$ is a set of $n$ independent samples from the density function
\eq{
\wh\fp \coloneqq \frac{\Area(\cT)}{\Area(\triangle)}\fp\circ\vphi^{-1} \quad \text{on $\triangle$.}
}
Therefore, we could have started with a set $\wh\cS_n$ of $n$ independent samples from $\wh \fp$, and defined
\eq{
\wh\sL_n \coloneqq \max\{\#\cC:\, \text{$\cC\subseteq\wh\cS_n$ and $\{\nA,\nB\}\cup\cC$ is in convex position}\}.
}
Since the random variables $\sL_n$ and $\wh\sL_n$ have the same law, any result we prove for convex chains on $\cT$ can be immediately transferred to general triangles.
The advantage of working on $\cT$ is that the relevant convex function theory is much easier to state.

At the same time, the proof strategies for Theorems~\ref{thm_var_formula}--\ref{thm_probability} will require us to work on certain sub-triangles of $\cT$.
Consequently, some intermediate results are worth stating in a general setting.
For the reader's convenience, we have placed all such results in Section~\ref{sec_general_triangles}.
The quantities $R$ and $r$ in \eqref{ratio_assumption} play the same role as $\fC$ and $\fc$ in \eqref{p_upper} and \eqref{p_lower}.
\qedrem
\end{remark}

\subsection{Related literature} \label{subsec_literature}

Our proof of Theorems~\ref{thm_var_formula} and \ref{thm_limit_curves} will use the following special case, which supplies the constant $\alpha$ in \eqref{variational_formula}.

\begin{theorem}[Longest convex chain in uniform case, {\cite[Theorem~1.1]{ambrus_barany09}}] \label{thm_uniform_known}
Let $\cU_n$ be a set of $n$ independent uniform samples from $\cT$.
Denote the longest convex chain in $\cU_n$ by
\eq{
\sL_n^\unif \coloneqq \max\{\#\cC:\, \text{$\cC\subseteq\cU_n$ and $\cC$ is a convex chain}\}.
}
There exists a constant $\alpha\in(0,\infty)$ such that
\eq{
\lim_{n\to\infty}\frac{\E\sL_n^\unif}{n^{1/3}} = \alpha.
}
\end{theorem}

Identifying the exact value of $\alpha$ remains an interesting open problem.
The following bounds were shown in \cite[Section~4]{ambrus_barany09}:
\eq{
0.6863 \approx \frac{1-\e^{-2}}{2^{1/3}} \le \alpha \le 2^{1/3}\e \approx 3.4248.
}
The same article provides a sketch for improving the lower bound to $\approx 1.5772$.
Based on numerical experiments \cite[Section~9]{ambrus_barany09}, it was conjectured that $\alpha=3$.
For our purposes it is enough to know $0 < \alpha < \infty$, and we will elaborate on our use of Theorem~\ref{thm_uniform_known} in Section~\ref{subsec_ideas}.

Our variational formula \eqref{variational_formula} is similar to one proved by Deuschel and Zeitouni \cite{deuschel_zeitouni95} for longest monotone chains.
They assume a continuously differentiable and strictly positive probability density $\bar\fp$ on $[0,1]^2$, and proceed to show in \cite[Theorem~2(i)]{deuschel_zeitouni95} that the longest monotone chain among $n$ independent samples from $\bar\fp$ is, to leading order,
\eeq{ \label{dz_formula}
2n^{1/2}\sup_{\substack{f\colon[0,1]\to[0,1] \\ \text{$f$ nondecreasing}}}\int_0^1\big[f'(x)\cdot\bar\fp(x,f(x))\big]^{1/2}\ \dd x, 
}
where monotonicity guarantees that $f'$ exists at almost every $x\in(0,1)$.
The constant $2$ is analogous to $\alpha$ from Theorem~\ref{thm_uniform_known}, coming from the fact that the length of the longest increasing subsequence in a \textit{uniformly} random permutation on $n$ elements is asymptotically $2n^{1/2}$ \cite{vershik_kerov77,logan_shepp77}.
An extension of \eqref{dz_formula} can be found in \cite{calder_esedoglu_hero14}, where a Hamilton--Jacobi equation is used to express a simultaneous law of large numbers for monotone chains with different endpoints.

Under the additional assumption that \eqref{dz_formula} admits only finitely many optimizers, it was shown in \cite[Theorems~2(ii) and 1]{deuschel_zeitouni95} that longest monotone chains converge in probability to the set of optimizers of \eqref{dz_formula}, similarly to our Theorems~\ref{thm_limit_curves} and \ref{thm_conditional} respectively.
While we were inspired by the results of \cite{deuschel_zeitouni95}, the proofs are quite different.
In a few places this is by choice---with the goal, for instance, of assuming only continuity of $\fp$ instead of continuous differentiability, or to strengthen convergence statements with the help of Lemma~\ref{lem_talagrand}---but mostly this is by necessity, owing to fundamental differences between monotone chains and convex chains.
For example, monotone chains are invariant under 1-to-1 monotone maps applied to either coordinate.  
That is, only the ordering of sample points matters rather than their actual values, which is not true for convex chains.

Prior to the uniform case studied by \cite{ambrus_barany09}, a lattice problem with some similar features was considered in \cite{barany95} and \cite{vershik94}.
These works studied the behavior of a uniformly random convex polygon whose vertices lie in $(\tfrac1n\Z^2)\cap[0,1]^2$, establishing a deterministic limit shape as $n\to\infty$.
The key realization was that the limit shape is the convex set that maximizes the so-called affine perimeter, which is given by the functional $J$ in our setting.
This fact was generalized to other confining sets (beyond the square $[0,1]^2$) in \cite{barany97}, and a related large deviation principle was obtained in \cite{vershik_zeitouni99}.
A statistical mechanical route to the results of \cite{barany95,vershik94} was given in \cite{sinai94}, which also proved a central limit theorem for deviations from the limit curve.
This approach was generalized in \cite{bogachev_zarbaliev11} to show that the limit curve is the same for a one-parameter family of measures on convex polygonal lines; see also \cite{bogachev14}.
The scenario in which the polygonal lines have a constrained number of vertices was considered in \cite{bureaux_enriquez17}.
The same article nicely recounts the arithmetic features arising in the lattice setting, which leads to some explicit asymptotic constants involving the Riemann zeta function.

There have been various other modifications of classical last-passage percolation, each one imposing a different type of constraint on paths: slope bounds \cite{basdevant_gerin22}, area below curve \cite{basu_ganguly_hammond18}, confinement to a narrow strip \cite{dey_joseph_peled24}, gaps between points \cite{basdevant_gerin19}, path entropy \cite{berger_torri19b}, and H\"{o}lder norms \cite{berger_torri21}.
Only the last two of these change the leading order of growth, although laws of large numbers remain open; see \cite[Section~2.2]{berger_torri19b} and \cite[Section~2.3]{berger_torri21}.

\subsection{Idea behind variational formula} \label{subsec_ideas} 

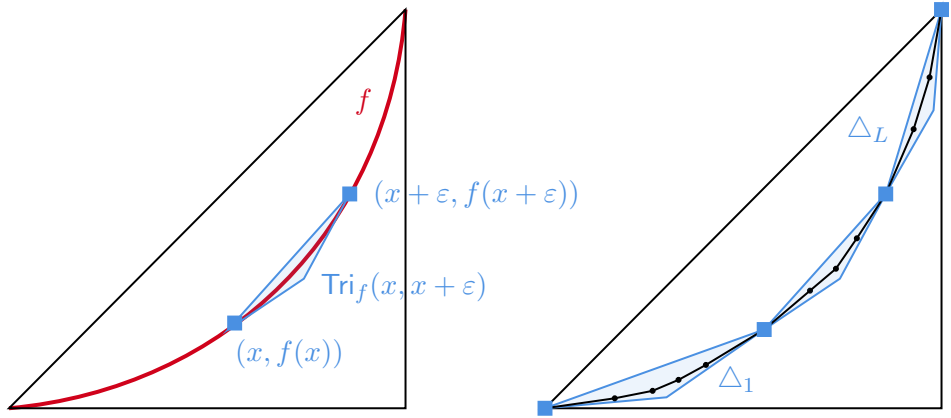
\begin{figure}[th]
\tikzset{every picture/.style={line width=0.75pt}} %set default line width to 0.75pt        

\begin{tikzpicture}[x=0.75pt,y=0.75pt,yscale=-1,xscale=1]
%uncomment if require: \path (0,231); %set diagram left start at 0, and has height of 231

%Curve Lines [id:da12223896000633405] 
\draw [color={rgb, 255:red, 208; green, 2; blue, 27 }  ,draw opacity=1 ][line width=1.5]    (91,210) .. controls (201,200) and (281,120) .. (290,10) ;
%Shape: Circle [id:dp9349295619561607] 
\draw  [fill={rgb, 255:red, 0; green, 0; blue, 0 }  ,fill opacity=1 ] (492,151) .. controls (492,150.45) and (492.45,150) .. (493,150) .. controls (493.55,150) and (494,150.45) .. (494,151) .. controls (494,151.55) and (493.55,152) .. (493,152) .. controls (492.45,152) and (492,151.55) .. (492,151) -- cycle ;
%Shape: Right Triangle [id:dp26180706924311115] 
\draw   (559,10) -- (360,210) -- (559,210) -- cycle ;
%Shape: Polygon [id:ds182810395302484] 
\draw  [color={rgb, 255:red, 74; green, 144; blue, 226 }  ,draw opacity=1 ][fill={rgb, 255:red, 74; green, 144; blue, 226 }  ,fill opacity=0.1 ] (508,145) -- (470,170.5) -- (531,102.5) -- cycle ;
%Shape: Polygon [id:ds40077320249509096] 
\draw  [color={rgb, 255:red, 74; green, 144; blue, 226 }  ,draw opacity=1 ][fill={rgb, 255:red, 74; green, 144; blue, 226 }  ,fill opacity=0.1 ] (421,204.5) -- (360,210) -- (470,170.5) -- cycle ;
%Shape: Polygon [id:ds034665720020992774] 
\draw  [color={rgb, 255:red, 74; green, 144; blue, 226 }  ,draw opacity=1 ][fill={rgb, 255:red, 74; green, 144; blue, 226 }  ,fill opacity=0.1 ] (555,60.5) -- (531,102.5) -- (559,10) -- cycle ;
%Shape: Right Triangle [id:dp7058843382455791] 
\draw   (290,10) -- (91,210) -- (290,210) -- cycle ;
%Shape: Square [id:dp05017782518730518] 
\draw  [color={rgb, 255:red, 74; green, 144; blue, 226 }  ,draw opacity=1 ][fill={rgb, 255:red, 74; green, 144; blue, 226 }  ,fill opacity=1 ] (201,164) -- (207.5,164) -- (207.5,170.5) -- (201,170.5) -- cycle ;
%Shape: Square [id:dp16227207360583829] 
\draw  [color={rgb, 255:red, 74; green, 144; blue, 226 }  ,draw opacity=1 ][fill={rgb, 255:red, 74; green, 144; blue, 226 }  ,fill opacity=1 ] (258.75,99.25) -- (265.25,99.25) -- (265.25,105.75) -- (258.75,105.75) -- cycle ;
%Shape: Circle [id:dp24648499125471668] 
\draw  [fill={rgb, 255:red, 0; green, 0; blue, 0 }  ,fill opacity=1 ] (505,140) .. controls (505,139.45) and (505.45,139) .. (506,139) .. controls (506.55,139) and (507,139.45) .. (507,140) .. controls (507,140.55) and (506.55,141) .. (506,141) .. controls (505.45,141) and (505,140.55) .. (505,140) -- cycle ;
%Shape: Circle [id:dp8917074211371836] 
\draw  [fill={rgb, 255:red, 0; green, 0; blue, 0 }  ,fill opacity=1 ] (515.5,124.75) .. controls (515.5,124.2) and (515.95,123.75) .. (516.5,123.75) .. controls (517.05,123.75) and (517.5,124.2) .. (517.5,124.75) .. controls (517.5,125.3) and (517.05,125.75) .. (516.5,125.75) .. controls (515.95,125.75) and (515.5,125.3) .. (515.5,124.75) -- cycle ;
%Shape: Circle [id:dp07227077041562513] 
\draw  [fill={rgb, 255:red, 0; green, 0; blue, 0 }  ,fill opacity=1 ] (544,70) .. controls (544,69.45) and (544.45,69) .. (545,69) .. controls (545.55,69) and (546,69.45) .. (546,70) .. controls (546,70.55) and (545.55,71) .. (545,71) .. controls (544.45,71) and (544,70.55) .. (544,70) -- cycle ;
%Shape: Circle [id:dp21291975455343615] 
\draw  [fill={rgb, 255:red, 0; green, 0; blue, 0 }  ,fill opacity=1 ] (552,44) .. controls (552,43.45) and (552.45,43) .. (553,43) .. controls (553.55,43) and (554,43.45) .. (554,44) .. controls (554,44.55) and (553.55,45) .. (553,45) .. controls (552.45,45) and (552,44.55) .. (552,44) -- cycle ;
%Shape: Circle [id:dp31382048342983526] 
\draw  [fill={rgb, 255:red, 0; green, 0; blue, 0 }  ,fill opacity=1 ] (394,205) .. controls (394,204.45) and (394.45,204) .. (395,204) .. controls (395.55,204) and (396,204.45) .. (396,205) .. controls (396,205.55) and (395.55,206) .. (395,206) .. controls (394.45,206) and (394,205.55) .. (394,205) -- cycle ;
%Shape: Circle [id:dp10635640851615247] 
\draw  [fill={rgb, 255:red, 0; green, 0; blue, 0 }  ,fill opacity=1 ] (413,201.25) .. controls (413,200.7) and (413.45,200.25) .. (414,200.25) .. controls (414.55,200.25) and (415,200.7) .. (415,201.25) .. controls (415,201.8) and (414.55,202.25) .. (414,202.25) .. controls (413.45,202.25) and (413,201.8) .. (413,201.25) -- cycle ;
%Shape: Circle [id:dp985133678202805] 
\draw  [fill={rgb, 255:red, 0; green, 0; blue, 0 }  ,fill opacity=1 ] (426,195.75) .. controls (426,195.2) and (426.45,194.75) .. (427,194.75) .. controls (427.55,194.75) and (428,195.2) .. (428,195.75) .. controls (428,196.3) and (427.55,196.75) .. (427,196.75) .. controls (426.45,196.75) and (426,196.3) .. (426,195.75) -- cycle ;
%Shape: Circle [id:dp5802334365284875] 
\draw  [fill={rgb, 255:red, 0; green, 0; blue, 0 }  ,fill opacity=1 ] (439.75,188.25) .. controls (439.75,187.7) and (440.2,187.25) .. (440.75,187.25) .. controls (441.3,187.25) and (441.75,187.7) .. (441.75,188.25) .. controls (441.75,188.8) and (441.3,189.25) .. (440.75,189.25) .. controls (440.2,189.25) and (439.75,188.8) .. (439.75,188.25) -- cycle ;
%Straight Lines [id:da9952273398667464] 
\draw    (360,210) -- (395,205) ;
%Straight Lines [id:da12054539549750665] 
\draw    (395,205) -- (414,201.25) ;
%Straight Lines [id:da9569388948872829] 
\draw    (427,195.75) -- (440.75,188.25) ;
%Straight Lines [id:da4571972838740216] 
\draw    (414,201.25) -- (427,195.75) ;
%Straight Lines [id:da2473763502259968] 
\draw    (440.75,188.25) -- (470,170.5) ;
%Straight Lines [id:da9119064426137456] 
\draw    (470,170.5) -- (493,151) ;
%Straight Lines [id:da24772579988558896] 
\draw    (493,151) -- (506,140) ;
%Straight Lines [id:da7467594538971492] 
\draw    (506,140) -- (516.5,124.75) ;
%Straight Lines [id:da454791871863348] 
\draw    (516.5,124.75) -- (531,102.5) ;
%Straight Lines [id:da331912433908888] 
\draw    (531,102.5) -- (545,70) ;
%Straight Lines [id:da022932760188804413] 
\draw    (545,70) -- (553,44) ;
%Straight Lines [id:da7750899889861188] 
\draw    (553,44) -- (559,10) ;
%Shape: Square [id:dp17475774320302317] 
\draw  [color={rgb, 255:red, 74; green, 144; blue, 226 }  ,draw opacity=1 ][fill={rgb, 255:red, 74; green, 144; blue, 226 }  ,fill opacity=1 ] (555.75,6.75) -- (562.25,6.75) -- (562.25,13.25) -- (555.75,13.25) -- cycle ;
%Shape: Square [id:dp40646537300646235] 
\draw  [color={rgb, 255:red, 74; green, 144; blue, 226 }  ,draw opacity=1 ][fill={rgb, 255:red, 74; green, 144; blue, 226 }  ,fill opacity=1 ] (527.75,99.25) -- (534.25,99.25) -- (534.25,105.75) -- (527.75,105.75) -- cycle ;
%Shape: Square [id:dp9365868894633166] 
\draw  [color={rgb, 255:red, 74; green, 144; blue, 226 }  ,draw opacity=1 ][fill={rgb, 255:red, 74; green, 144; blue, 226 }  ,fill opacity=1 ] (466.75,167.25) -- (473.25,167.25) -- (473.25,173.75) -- (466.75,173.75) -- cycle ;
%Shape: Square [id:dp050568371963856995] 
\draw  [color={rgb, 255:red, 74; green, 144; blue, 226 }  ,draw opacity=1 ][fill={rgb, 255:red, 74; green, 144; blue, 226 }  ,fill opacity=1 ] (356.75,206.75) -- (363.25,206.75) -- (363.25,213.25) -- (356.75,213.25) -- cycle ;
%Shape: Polygon [id:ds7751573802770028] 
\draw  [color={rgb, 255:red, 74; green, 144; blue, 226 }  ,draw opacity=1 ][fill={rgb, 255:red, 74; green, 144; blue, 226 }  ,fill opacity=0.1 ] (239,145) -- (201,170.5) -- (262,102.5) -- cycle ;

% Text Node
\draw (263,48.4) node [anchor=north west][inner sep=0.75pt]  [color={rgb, 255:red, 208; green, 2; blue, 27 }  ,opacity=1 ]  {$f$};
% Text Node
\draw (203,173.9) node [anchor=north west][inner sep=0.75pt]  [color={rgb, 255:red, 74; green, 144; blue, 226 }  ,opacity=1 ]  {$( x,f( x))$};
% Text Node
\draw (272.25,93.9) node [anchor=north west][inner sep=0.75pt]  [color={rgb, 255:red, 74; green, 144; blue, 226 }  ,opacity=1 ]  {$( x+\varepsilon ,f( x+\varepsilon ))$};
% Text Node
\draw (246,140.4) node [anchor=north west][inner sep=0.75pt]  [color={rgb, 255:red, 74; green, 144; blue, 226 }  ,opacity=1 ]  {$\Tri_{f}( x,x+\varepsilon )$};
% Text Node
\draw (445,187.78) node [anchor=north west][inner sep=0.75pt]  [color={rgb, 255:red, 74; green, 144; blue, 226 }  ,opacity=1 ]  {$\triangle _{1}$};
% Text Node
\draw (510,62.4) node [anchor=north west][inner sep=0.75pt]  [color={rgb, 255:red, 74; green, 144; blue, 226 }  ,opacity=1 ]  {$\triangle_{L}$};

\end{tikzpicture}
\caption{Heuristic for variational formula \eqref{variational_formula}.
\textit{Left}: candidate function $f$ (shown in red) together with its tangency triangle between $x$ and $x+\eps$ (shown in blue).
\textit{Right}: a convex chain realized by concatenating smaller chains within the tangency triangles defined by $f$. Sample points that do not participate in the chain are not shown.}
\label{fig_sketch}
\end{figure}

Using Theorem~\ref{thm_uniform_known}, we now give an instructive heuristic for the variational formula \eqref{variational_formula}.
Suppose we wish to identify a convex chain that closely resembles some function $f\in\cF$.
One method is to concatenate many smaller chains that live in a sequence of ``tangency triangles'' of $f$; see Figure~\ref{fig_sketch}.
The tangency triangle between $x$ and $x+\eps$, denoted by $\Tri_f(x,x+\eps)$, is bounded from below by the tangent lines of $f$ at $x$ and $x+\eps$, and bounded from above by the secant line of $f$ between $x$ and $x+\eps$; see Definition~\ref{def_triangle} for precise details.
For simplicity, let us assume $f$ is twice continuously differentiable, in which case one can calculate (see Lemma~\ref{lem_area})
\eq{
\Area\big(\Tri_f(x,x+\eps)\big) = \frac{1}{8}f''(x)\eps^3 + o(\eps^3) \quad \text{as $\eps\searrow0$}.
}
Since $\cS_n$ consists of $n$ independent samples from a continuous density $\fp$, we can estimate the number of those samples that lie in a particular tangency triangle, as follows:
\eq{
\E\Big[\#\big(\cS_n\cap \Tri_f(x,x+\eps)\big)\Big] 
&= n\int_{\Tri_f(x,x+\eps)}\fp(u,v)\ \dd u\, \dd v \\
&\approx n\fp(x,f(x))\cdot\Area\big(\Tri_f(x,x+\eps)\big)
\approx \frac{n}{8}f''(x)\cdot \fp(x,f(x))\eps^3.
}
Furthermore, when $\eps$ is small, the sample set inside $\Tri_f(x,x+\eps)$ is approximately uniform, since $\fp$ looks locally like a constant.
Therefore, Theorem~\ref{thm_uniform_known} (together with Remark~\ref{rem_convex_position} to convert from $\cT$ to $\Tri_f(x,x+\eps)$) suggests
\eq{
\E\Big[\text{maximum size of a convex chain in $\Tri_f(x,x+\eps)$}\Big]
\approx \alpha\Big[\frac{n}{8}f''(x)\cdot \fp(x,f(x))\Big]^{1/3}\eps.
}
We stress that, by definition, a ``convex chain in $\Tri_f(x,x+\eps)$'' must begin at $(x,f(x))$ and end at $(x+\eps,f(x+\eps))$, making it  possible to concatenate with a convex chain in $\Tri_f(x+\eps,x+2\eps)$, and so on (see Lemma~\ref{lem_concatenation}).
By dividing $[0,1]$ into $L$ many subintervals $[x_0,x_1],[x_1,x_2],\ldots,[x_{L-1},x_L]$ of length $\eps = \frac{1}{L}$ (where $L\gg1$), and denoting the associated tangency triangles by $\triangle_0,\triangle_1,\ldots,\triangle_{L-1}$, we obtain
\eq{
\frac{1}{n^{1/3}}\sum_{\ell=0}^{L-1} \E\Big[\text{max size of a convex chain in $\triangle_\ell$}\Big]
&\approx \frac{\alpha}{2} \sum_{\ell=0}^{L-1} \big[f''(x_\ell)\cdot\fp(x_\ell,f(x_\ell))\big]^{1/3}\cdot\frac{1}{L}\\
&\approx \frac{\alpha}{2}\int_0^1 \big[f''(x)\cdot\fp(x,f(x))\big]^{1/3}\ \dd x
= \frac{\alpha}{2}J(f).
}
By making these estimates rigorous (as we do in Sections~\ref{sec_two_sided_inputs} and \ref{sec_lower_bound}), one realizes a convex chain whose length is $n^{1/3}(\alpha/2)J(f) - o(n^{1/3})$.

To prove a matching upper bound, one would like to show that the strategy just described is in some sense the \textit{only} way to realize a convex chain.
More precisely, one wishes to find a small number of functions $f\in\cF$ such that any convex chain must lie in the tangency triangles of at least one of these functions.
Here ``small'' means small enough that the estimates given above can be performed for all $f$ simultaneously, even as $L\to\infty$.
This task is significantly more difficult than the lower bound, so much so that we did not manage to accomplish it!
Instead, we construct a small number of \textit{piecewise} convex functions $\tilde f$ that fulfill these requirements.
By carefully controlling how much each $\tilde f$ fails to be convex, we are able to recover genuinely convex functions in the limit $n\to\infty$ while also preserving $J$ values.
A more complete sketch is provided in Section~\ref{subsec_upper_outline}.

Tangency triangles were also used in \cite{ambrus_barany09}, but only for a single function, namely
\eeq{ \label{f_star_def}
f_\star(x) \coloneqq (1-\sqrt{1-x})^2.
}
It was shown in \cite[Theorem~1.3]{ambrus_barany09} that longest convex chains (in the uniform case) concentrate around the graph of $f_\star$, which is a parabola that is tangent to the boundary of $\cT$ at both $(0,0)$ and $(1,1)$.
See Figure~\ref{fig_simulation} for a simulation.
In fact, the same limit curve appears in the lattice setting \cite{barany95,vershik94,sinai94,bogachev_zarbaliev11,bogachev14,bureaux_enriquez17} mentioned in Section~\ref{subsec_literature}.
Not coincidentally, $f_\star$ is the unique maximizer of $f\mapsto\int_0^1 f''(x)^{1/3}\, \dd x$, as shown in Proposition~\ref{prop_unif_case}\ref{prop_unif_case_a} with $(a,b,f_a,f_b,s_a,s_b) = (0,1,0,1,0,\infty)$.
This fact is implicitly used in \cite{ambrus_barany09} to prove Theorem~\ref{thm_uniform_known}, although the language there is entirely in terms of triangles; see \cite[Corollary~2.1]{ambrus_barany09}.
By comparison, our Proposition~\ref{prop_unif_case} offers a more analytic understanding of why the limit curves can be explicitly identified in the uniform case.
A challenge in this paper is that for general distributions, we do not know the optimal $f$---in fact, there may be multiple optimizers---which is why it so much harder here (than in \cite[Section~4]{ambrus_barany09}) to get the upper and lower bounds to match.

\begin{figure}[h]
\centering
\includegraphics[width=0.45\linewidth]{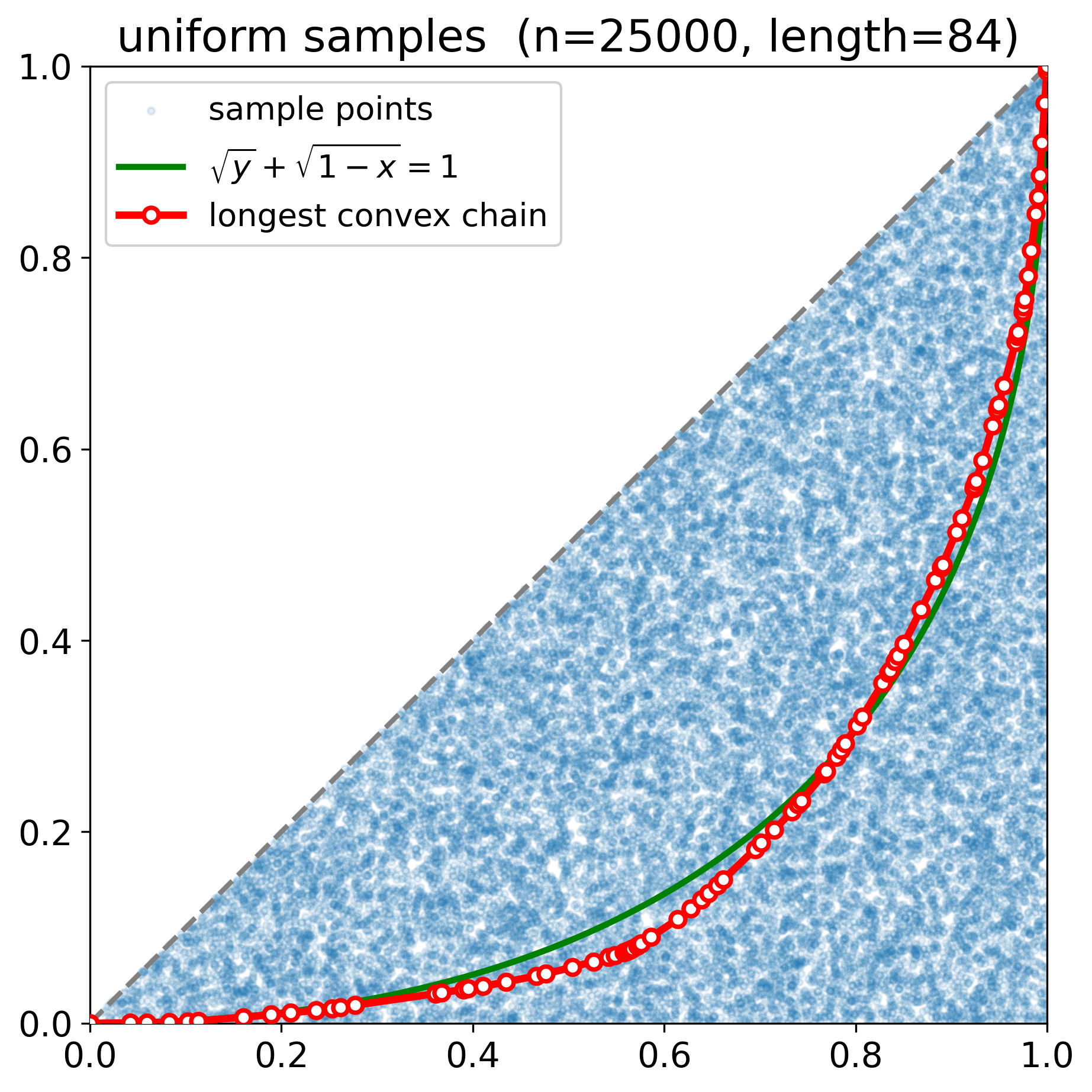}
\qquad
\includegraphics[width=0.45\linewidth]{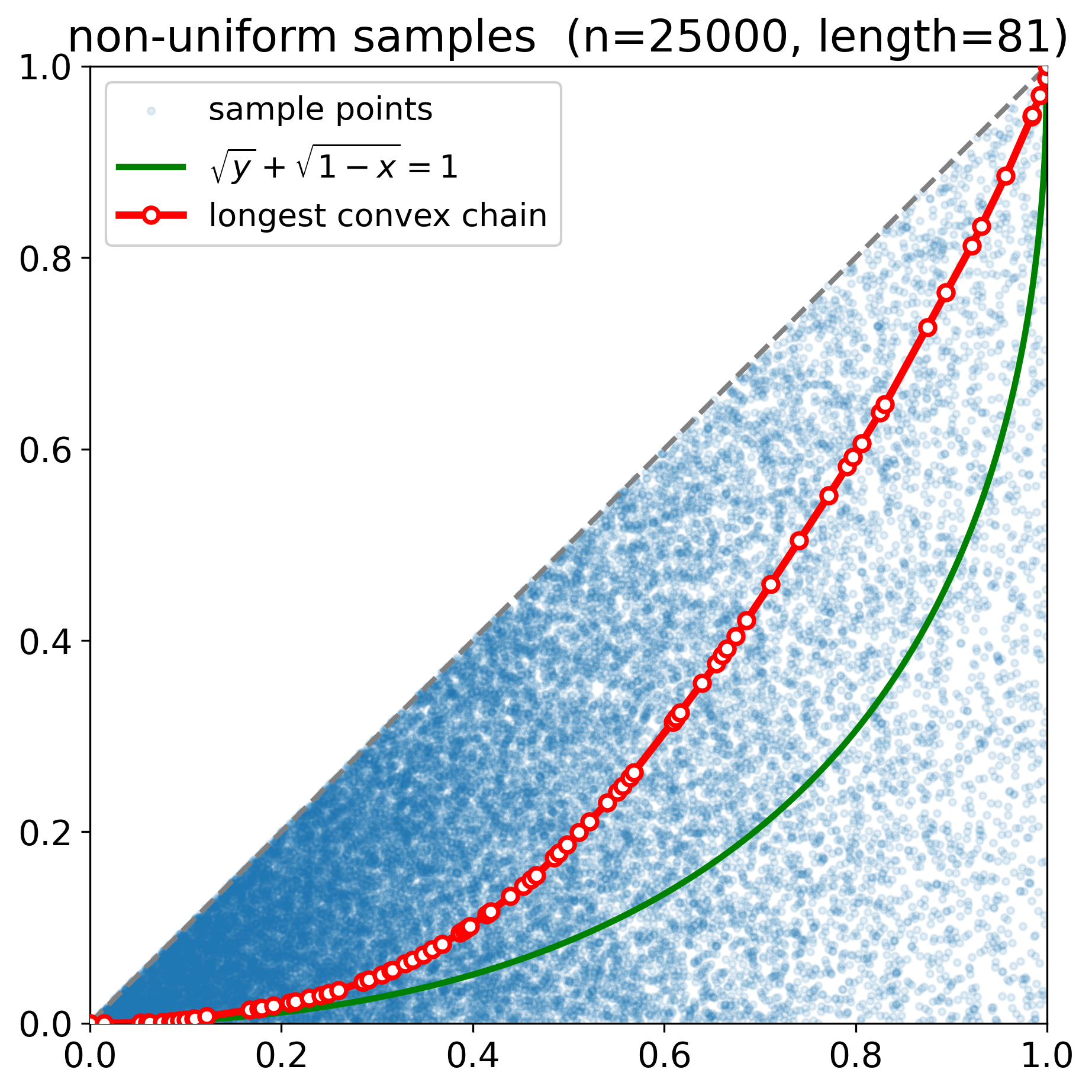}
\caption{Two realizations simulated by Nishant Ajitsaria.
\textit{Left:} $n=$ 25,000 sample points with uniform density $\fp\equiv2$, resulting in $\sL_n = 84$. The limit curve given by \eqref{f_star_def} is shown in solid green.
%\textit{Right:} $n=$ 25,000 sample points with density $\fp(x,y) = \frac{\lambda^2}{1-\e^{-\lambda}(1+\lambda)}\e^{-\lambda x}$ where $\lambda=4$, resulting in $\sL_n = 81$.
\textit{Right:} $n=$ 25,000 sample points with density $\fp(x,y) = \frac{16}{1-5\e^{-4}}\e^{-4x}$, resulting in $\sL_n = 81$.
The graph of \eqref{f_star_def} is shown only for comparison.
}
\label{fig_simulation}
\end{figure}

\subsection{Open problems} \label{subsec_open_problems}
Here we mention several questions left for future work.
\begin{enumerate}[label=\textup{\arabic*.}]

\item Is the constant $\alpha$ from Theorem~\ref{thm_uniform_known} equal to $3$, as conjectured in \cite{ambrus_barany09}?

\item Are there non-uniform densities $\fp$ for which a unique maximizer can be explicitly computed?

\item Is the number of maximizers finite under assumptions \eqref{p_cont} and \eqref{p_lower}?
Deuschel and Zeitouni \cite{deuschel_zeitouni95} believed this was true in their setting; see \cite[p.~854, Remark~1]{deuschel_zeitouni95}.

\item Given $f\in\cF$, can one find $\fp$ satisfying \eqref{p_cont} and \eqref{p_lower} such that $f\in\cF$ is a (unique) maximizer of $J_\fp$?
Proposition~\ref{prop_maximizer_properties} provides several necessary conditions on $f$; what about sufficient conditions?
For the lattice models mentioned in Section~\ref{subsec_literature}, a similar question is addressed in \cite{bogachev_zarbaliev99,bogachev_zarbaliev23}.

\item What is the order of fluctuations of $\sL_n$?  In this paper, we use Lemma~\ref{lem_talagrand} to give upper bounds on the scale $n^{1/6}$.

\item Can one compute large deviation probabilities for $\sL_n$, analogously to \cite{deuschel_zeitouni99}?

\item Can one prove a version of Theorem~\ref{thm_conditional} under the conditioning $\sL_n \approx \ell_n$ for any sequence $\ell_n$ that grows faster than $n^{1/3}$?
This is open even in the case of uniform samples, and the first theorem of \cite{bureaux_enriquez17} provides some evidence.
In the uniform case, does the central limit theorem from \cite{barany_rote_steiger_zhang00} still apply but at a potentially larger scale depending on $(\ell_n)_{n\ge1}$?
In the KPZ class, behavior of this type was identified in \cite{ganguly_hegde_zhang23_arxiv}.

\item Monotone chains require positivity of first-order differences, while convex chains require positivity of second-order differences. 
The case of $k^\text{th}$-order differences has been considered in \cite{ambrus20_arxiv, ambrus17} for uniform samples, the relevant scale being $n^{1/(k+1)}$. 
Is there a law of large numbers and associated variational formula when the samples come from a general density function?

\end{enumerate}

\subsection{Derivative notation} \label{subsec_derivative_notation}
Here we declare some notation that is used throughout the paper.
Consider a convex function $f\colon[a,b]\to\R$.
We denote its one-sided derivatives by
\eq{
\partial^+f(x) &\coloneqq \lim_{\eps\searrow0}\frac{f(x+\eps)-f(x)}{\eps} \quad \text{for $x\in[a,b)$}, \\
\partial^-f(x) &\coloneqq \lim_{\eps\searrow0}\frac{f(x)-f(x-\eps)}{\eps}
\quad \text{for $x\in(a,b]$}.
}
The functions $\partial^+ f$ and $\partial^- f$ are the right-continuous and left-continuous versions, respectively, of the same nondecreasing function.
For every $x\in(a,b)$ we have $-\infty < \partial^-f(x) \le \partial^+f(x)<\infty$, and $f$ is differentiable at $x$ if and only if $\partial^-f(x) = \partial^+f(x)$, in which case we write $f'(x) = \partial^-f(x) = \partial^+f(x)$.
It is possible that $\partial^+ f(a) = -\infty$ and/or $\partial^-f(b) = \infty$.

Since the one-sided derivatives are monotone, the following limits exist and are finite for almost every $x\in(a,b)$:
\eeq{ \label{second_deriv_def}
f''(x) = \lim_{r\to0}\frac{\partial^\pm f(x+r)-f'(x)}{r} < \infty.
}
We will only need $f''$ defined almost everywhere, so we use the notation $f''$ even when $f$ is not everywhere twice differentiable.
It may be that $f'$ is not absolutely continuous, so in general the following relationship is only an inequality:
\eeq{ \label{2_deriv_ineq}
\int_a^b f''(x)\ \dd x \le \partial^-f(b)-\partial^+f(a).
}
Nevertheless, we can characterize when equality occurs:
\eeq{ \label{2_deriv_eq}
&\int_a^b f''(x)\ \dd x = \partial^- f(b) - \partial^+ f(a) < \infty \\
&\iff
\text{$\partial^+ f$ is absolutely continuous on $[a,u]$ $\forall$ $u\in(a,b)$, and $\partial^-f(b)<\infty$} \\
&\iff
\text{$\partial^- f$ is absolutely continuous on $[u,b]$ $\forall$ $u\in(a,b)$, and $\partial^+f(a)>-\infty$}.
}
By comparison, \textit{convex} continuous functions are absolutely continuous, so we always have
\eeq{ \label{1_deriv_eq}
f(b)-f(a)
= \int_a^b \partial^+f(x)\ \dd x
= \int_a^b \partial^-f(x)\ \dd x.
}

\subsection{Organization of the paper} \label{subsec_organization}

Section~\ref{sec_analysis_of_variational_formula} is devoted to the deterministic analysis of the optimization problem in \eqref{variational_formula}, including a complete solution in the uniform case (Proposition~\ref{prop_unif_case}).
The proof of \eqref{variational_formula} is split into two parts: a lower bound (Section~\ref{sec_lower_bound}) and an upper bound (Section~\ref{sec_upper_bound}).
Some preliminary estimates used in both parts are recorded in Section~\ref{sec_two_sided_inputs}.
Finally, Section~\ref{sec_thm_proofs} synthesizes the various ingredients to prove all four theorems stated in Section~\ref{subsec_main_results}.

\section{Analysis of the variational formula} \label{sec_analysis_of_variational_formula} 

The purpose of this section is to analyze the variational problem $\sup_{f\in\cF} J_\fp(f)$.
For the constant density function $\fp\equiv2$, the problem can be solved explicitly, under any boundary conditions on $f$.
This is the content of Section~\ref{subsec_uniform}.
For general density functions, it is not clear \textit{a priori} that $J$ even admits a maximizer.
To resolve this, we equip the space $\cF$ with a suitable metric in Section~\ref{subsec_metric}.
Convergence in this metric can be characterized in terms of weak convergence of measures, as discussed in Section~\ref{subsec_space_measures}, and this characterization allows us prove upper semicontinuity of $J$ in Section~\ref{subsec_usc}.
Consequences for the set of maximizers are pursued in Section~\ref{subsec_maximizers}. 
Finally, in Section~\ref{subsec_approximation} we show that any $f\in\cF$ can be approximated by a smooth function $g$ such that $J(g)\approx J(f)$.

\subsection{Variational formula in the uniform case} \label{subsec_uniform}

The following proposition gives complete understanding of the variational problem $\sup_{f\in\cF} J_\fp(f)$ when $\fp$ is constant.
We consider general boundary conditions because they will arise in later arguments.

\begin{proposition}[Solution for uniform density]\label{prop_unif_case}
Consider any four position parameters $a,b,f_a,f_b\in\R$ and any two slope parameters $s_a,s_b\in\R\cup\{-\infty,\infty\}$ such that
\begin{equation} \label{eq:feasible}
a<b \quad \text{and} \quad s_a < \frac{f_b - f_a}{b- a} < s_b.
\end{equation}
Denote by $\cF_\partial$ the set of continuous convex $f\colon[a,b]\to\R$ with the following boundary data:
\begin{align} \label{boundary_data}
f(a)\ge f_a,\quad 
f(b)\le f_b,\quad 
\partial^+f(a)\ge s_a,\quad 
\partial^-f(b)\le s_b.
\end{align}
Denote by $\cF_\partial^=$ the set of continuous convex $f\colon[a,b]\to\R$ satisfying \eqref{boundary_data} with equalities:
\eq{
f(a)=f_a,\quad f(b)=f_b,\quad \partial^+f(a)= s_a,\quad \partial^-f(b)= s_b.
}
Then the following statements hold:
\begin{enumerate}[label=\textup{(\alph*)}]

\item \label{prop_unif_case_a}
If $-\infty<s_a<s_b=\infty$, then
%\begin{subequations} \label{eq_unif_case_a}
\begin{gather}
\label{eq_unif_case_a}
%\label{unif_case_a_sup}
\sup_{f \in \cF_\partial} \int_a^b f''(x)^{1/3}\ \dd x 
%= \sup_{f \in \cF_\partial^=} \int_a^b f''(x)^{1/3}\ \dd x =
= \big(4(b-a)\Delta_1\big)^{1/3},
%\label{Delta1_def}
\quad \text{where} \quad
\Delta_1 \coloneqq f_b-f_a-s_a(b-a) > 0.
\end{gather}
%\end{subequations}
The unique maximizer belongs to $\cF_\partial^=$ and has second derivative given by
\eeq{ \label{29cwq}
f''(x) &= \frac{\Delta_1}{2\sqrt{b-a}}(b-x)^{-3/2}.
}

\item \label{prop_unif_case_b}
If $-\infty=s_a<s_b<\infty$, then
%\begin{subequations} \label{eq_unif_case_b}
\begin{gather}
\label{eq_unif_case_b}
%\label{unif_case_b_sup}
\sup_{f \in \cF_\partial} \int_a^b f''(x)^{1/3}\ \dd x 
%= \sup_{f \in \cF_\partial^=} \int_a^b f''(x)^{1/3}\ \dd x
= \big(4(b-a)\Delta_2\big)^{1/3},
%\label{Delta2_def}
\quad \text{where} \quad
\Delta_2 \coloneqq s_b(b-a)-(f_b-f_a).
\end{gather}
%\end{subequations}
The unique maximizer belongs to $\cF_\partial^=$ and has second derivative given by
\eq{
f''(x) &= \frac{\Delta_2}{2\sqrt{b-a}}(x-a)^{-3/2}.
}

\item \label{prop_unif_case_c}
If $-\infty<s_a<s_b<\infty$, then
\eeq{ \label{unif_case_c_sup}
\sup_{f \in \cF_\partial} \int_a^b f''(x)^{1/3}\ \dd x = \Big(\frac{4\Delta_1\Delta_2}{s_b-s_a}\Big)^{1/3}.
}
The unique maximizer belongs to $\cF_\partial^=$ and has second derivative given by
\eeq{ \label{17xxih}
f''(x) = \frac{1}{3^{3/2}}\Big(u^2\cdot\frac{x-a}{b-a} + v^2\cdot\frac{b-x}{b-a}\Big)^{-3/2},
}
where $u$ and $v$ solve \eqref{both_crit}.
\end{enumerate}
\end{proposition}

\begin{proof}
Case~\ref{prop_unif_case_a}:
First we make a general observation.
For every convex function $f$, we have
\eeq{ \label{29vc4}
\int_a^b f''(x)^{1/3}\ \dd x
&\stackrel{\hphantom{\mbox{\footnotesize(H\"older)}}}{=} \int_a^b \big(f''(x)(b-x)\big)^{1/3}\cdot(b-x)^{-1/3}\ \dd x \\
&\stackrel{\mbox{\footnotesize(H\"older)}}{\le} \Big(\int_a^b f''(x)(b-x)\ \dd x\Big)^{1/3}\Big(\int_a^b (b-x)^{-1/2}\ \dd x\Big)^{2/3} \\
&\stackrel{\hphantom{\mbox{\footnotesize(H\"older)}}}{=} \Big(\int_a^b f''(x)(b-x)\ \dd x\Big)^{1/3}\big(2(b-a)^{1/2}\big)^{2/3}.
}
Furthermore, the above application of H\"older results in an equality if and only if
\begin{align} \label{2c7h4}
\text{$\exists$ a constant $\beta\ge0$ such that}
\quad f''(x) &= \beta(b-x)^{-3/2} \quad \text{for almost every $x\in(a,b)$.}
\intertext{A slightly stronger condition is}
\label{2c7h4_strong}
\text{$\exists$ a constant $\beta\ge0$ such that}
\quad f''(x) &= \beta(b-x)^{-3/2} \quad \text{for every $x\in(a,b)$.}
\end{align}
To control the integral on the last line of \eqref{29vc4}, we assume $f$ belongs to $\cF_\partial$, which implies the final inequality below:
\eeq{ \label{29vc5}
\int_a^b f''(x)(b-x)\ \dd x
&\stackrefp{1_deriv_eq}{=} \int_a^b \int_x^b f''(x)\ \dd u\, \dd x \\
&\stackrefp{1_deriv_eq}{=} \int_a^b \int_a^u f''(x)\ \dd x\, \dd u \\
&\stackref{2_deriv_ineq}{\le} \int_a^b\big(\partial^-f(u)-\partial^+f(a)\big)\ \dd u \\
&\stackref{1_deriv_eq}{=} f(b) - f(a) - (b-a)\partial^+ f(a) \\
&\stackrefpp{boundary_data}{1_deriv_eq}{\le} f_b - f_a - s_a(b-a) 
\stackref{eq_unif_case_a}{=} \Delta_1.
}
Insert \eqref{29vc5} into \eqref{29vc4} to obtain the upper bound $(\le)$ in \eqref{eq_unif_case_a}.
Note that the two inequalities in \eqref{29vc5} are equalities if and only if
\begin{subequations} \label{2c7h5}
\begin{gather}
\text{$\partial^-f$ is absolutely continuous on $[a,u]$ for every $u\in(a,b)$, and} \label{2c7h5_a} \\
f(a) = f_a, \quad
f(b) = f_b, \quad
\partial^+f(a) = s_a, \label{2c7h5_b}
\end{gather}
\end{subequations}
where in the case of \eqref{2c7h5_a} we are using \eqref{2_deriv_eq}.
Also notice that
\eq{
\text{\eqref{2c7h4} and \eqref{2c7h5_a} hold} \quad \iff \quad
\text{\eqref{2c7h4_strong} holds.}
}
Furthermore, by simply computing $\int_a^b\beta(b-x)^{-1/2}\,\dd x = 2\beta\sqrt{b-a}$, one sees that
\eq{
\text{\eqref{2c7h4_strong} holds and \eqref{29vc5} is an equality} \quad \iff \quad \text{\eqref{29cwq} holds.}
}
In summary, we have the following equivalence for any $f\in\cF_\partial$:
\eq{
\text{$f$ achieves RHS of \eqref{eq_unif_case_a}} \quad 
&\iff \quad \text{\eqref{29vc4} and \eqref{29vc5} are equalities} \\
&\iff \quad \text{\eqref{2c7h4} and \eqref{2c7h5_a} hold, \eqref{29vc5} is an equality} \\
&\iff \quad \text{\eqref{2c7h4_strong} holds and \eqref{29vc5} is an equality} \\
&\iff \quad \text{\eqref{29cwq} holds}.
}
To complete the proof of case~\ref{prop_unif_case_a}, we just need to show that \eqref{29cwq} actually generates an element of $\cF_\partial^=$. 
That is, we need to check that the initial conditions $f'(a)=s_a$ and $f(a)=f_a$ are consistent with the terminal conditions $f'(b)=\infty$ and $f(b)=f_b$.
The first requirement $f'(b) = \infty$ holds since
\eeq{ \label{qn7l4}
f'(x) &= s_a + \int_a^x f''(u)\ \dd u
\stackref{29cwq}{=} s_a + \int_a^x\frac{\Delta_1}{2\sqrt{b-a}}(b-u)^{-3/2}\ \dd u \\
&= s_a + \frac{\Delta_1}{\sqrt{b-a}}(b-u)^{-1/2}\Big|_{u=a}^{u=x}
\stackref{eq_unif_case_a}{=} 2s_a - \frac{f_b-f_a}{b-a} + \frac{\Delta_1}{\sqrt{b-a}}(b-x)^{-1/2}.
}
The second requirement $f(b) = f_b$ can then be verified as follows:
\eq{
f(b) 
= f_a + \int_a^b f'(x)\ \dd x
&\stackref{qn7l4}{=} 2f_a + 2s_a(b-a) - f_b + \frac{\Delta_1}{\sqrt{b-a}}\int_a^b (b-x)^{-1/2}\ \dd x \\
&\stackrefp{qn7l4}{=} 2f_a + 2s_a(b-a) - f_b + 2\Delta_1
\stackref{eq_unif_case_a}{=} f_b.
}

\medskip

\noindent Case~\ref{prop_unif_case_b}:
This is similar to case~\ref{prop_unif_case_a}, but in \eqref{29vc4} we replace $b-x$ with $x-a$, and we replace \eqref{29vc5} with
\eeq{ \label{29vc6}
\int_a^b f''(x)(x-a)\ \dd x \le s_b(b-a) - (f_b - f_a) \stackref{eq_unif_case_b}{=} \Delta_2.
}
Analogously to before, we note that \eqref{29vc6} is an equality if and only if
\begin{subequations} \label{2c7h6}
\begin{gather}
\text{$\partial^+f$ is absolutely continuous on $[u,b]$ for every $u\in(a,b)$, and} \\
f(a) = f_a, \quad
f(b) = f_b, \quad
\partial^-f(b) = s_b.
\end{gather}
\end{subequations}

\medskip

\noindent
Case~\ref{prop_unif_case_c}:
Consider any $f\in\cF_\partial$.
A direct computation shows that for any $\beta>0$ and $t\ge0$, we have
\begin{equation} \label{ineq:subproblem_unique_maximizer}
t^{1/3}-\beta t \le \frac{2}{3^{3/2}} \beta^{-1/2}, \quad \text{with equality if and only if 
$t = (3\beta)^{-3/2}$.}
\end{equation}
For any $u,v>0$ and $x \in (a, b)$, applying \eqref{ineq:subproblem_unique_maximizer} with $\beta =u^2\cdot\tfrac{x-a}{b-a}+v^2\cdot\frac{b-x}{b-a}>0$ and $t=f''(x) \ge 0$ yields
\begin{equation}\label{ineq:pointwise}
f''(x)^{1/3}
\le \frac{u^2(x-a)+v^2(b-x)}{b-a}f''(x)+\frac{2}{3^{3/2}}\Big(\frac{u^2(x-a)+v^2(b-x)}{b-a}\Big)^{-1/2},   
\end{equation}
with equality if and only if
\eeq{ \label{mn5x2}
f''(x)= \frac{1}{3^{3/2}}\Big(u^2\cdot\frac{x-a}{b-a}+v^2\cdot\frac{b-x}{b-a}\Big)^{-3/2}.
}
Integrating \eqref{ineq:pointwise} over $x\in(a,b)$, and using \eqref{29vc5} and \eqref{29vc6} to control the right-hand side, we obtain
\begin{align} \label{eq:ptwise_bound}
\int_a^b f''(x)^{1/3}\ \dd x
&\le \frac{\Delta_2}{b-a}u^2 + \frac{\Delta_1}{b-a}v^2 + \frac{4(b-a)}{3^{3/2}}\cdot\frac{1}{u+v} \eqqcolon G(u,v).
\end{align}
Before proceeding further, we note that
\eeq{ \label{1v7xe}
\text{\eqref{eq:ptwise_bound} is an equality}
\quad &\iff \quad \parbox{0.5\textwidth}{\centering \eqref{29vc5} and \eqref{29vc6} are equalities, and \\ \eqref{ineq:pointwise} is an equality for almost every $x\in(a,b)$} \\
&\iff \quad \parbox{0.5\textwidth}{\centering \eqref{2c7h5} and \eqref{2c7h6} hold, and \\ \eqref{mn5x2} holds for almost every $x\in(a,b)$}  \\
&\iff \quad \parbox{0.5\textwidth}{\centering $f\in\cF_\partial^=$ and \eqref{mn5x2} holds for every $x\in(a,b)$.}
}
The next step is to minimize the right-hand side of \eqref{eq:ptwise_bound}.

To this end, the critical point equations are
\begin{align} \label{eq:G_critical1}
\frac{\partial G}{\partial u} 
&=2\frac{\Delta_2}{b-a}u-\frac{4(b-a)}{3^{3/2}}\cdot\frac{1}{(u+v)^2}=0, \\
\frac{\partial G}{\partial v}
&= 2\frac{\Delta_1}{b-a}v-\frac{4(b-a)}{3^{3/2}}\cdot\frac{1}{(u+v)^2}=0. \notag %\label{eq:G_critical2}
\end{align}
Subtracting the first equation from the second allows us to solve for $v$ in terms of $u$:
\begin{subequations} \label{both_crit}
\begin{equation}\label{uv_ratio}
v
= \frac{\Delta_2}{\Delta_1}u.
\end{equation}
Plugging \eqref{uv_ratio} into \eqref{eq:G_critical1} allows us to solve for $u$:
\eeq{ \label{crit_u}
2\frac{\Delta_2}{b-a}u
=\frac{4(b-a)}{3^{3/2}}\cdot\frac{1}{\bigl(u+\tfrac{\Delta_2}{\Delta_1}u\bigr)^2}
\quad\Longrightarrow\quad
u^3=\frac{2(b-a)^2}{3^{3/2}}\cdot\frac{\Delta_1^2}{\Delta_2(\Delta_1+\Delta_2)^2}.
}
\end{subequations}
At this critical point, we have
\eeq{ \label{94jv3x}
\frac{\Delta_2}{b-a}u^2
&\stackref{crit_u}{=} \frac{\Delta_2}{b-a}\Big(\frac{2(b-a)^2}{3^{3/2}}\cdot\frac{\Delta_1^2}{\Delta_2(\Delta_1+\Delta_2)^2}\Big)^{2/3} \\
&\stackrefp{crit_u}{=} \frac{1}{3}\Big(\frac{4\Delta_1\Delta_2(b-a)}{\Delta_1+\Delta_2}\Big)^{1/3}\cdot\frac{\Delta_1}{\Delta_1+\Delta_2}.
}
Now use the relation
\eeq{ \label{slope_relation}
\Delta_1 + \Delta_2 = (s_b-s_a)(b-a)
}
to rewrite \eqref{94jv3x} as 
\eq{
\frac{\Delta_2}{b-a}u^2 = \frac{1}{3}\Big(\frac{4\Delta_1\Delta_2}{s_b-s_a}\Big)^{1/3}\cdot\frac{\Delta_1}{\Delta_1+\Delta_2}.
}
By similar arithmetic (just exchange $\Delta_1$ and $\Delta_2$), we also have
\eq{
\frac{\Delta_1}{b-a}v^2
= \frac{1}{3}\Big(\frac{4\Delta_1\Delta_2}{s_b-s_a}\Big)^{1/3}\cdot\frac{\Delta_2}{\Delta_1+\Delta_2}.
}
In addition, we can compute
\eq{
\frac{4(b-a)}{3^{3/2}}\cdot\frac{1}{u+v}
&\stackrefpp{uv_ratio}{crit_u}{=} \frac{4(b-a)}{3^{3/2}}\cdot\frac{\Delta_1}{\Delta_1+\Delta_2}\cdot\frac{1}{u} \\
&\stackref{crit_u}{=} \frac{4(b-a)}{3^{3/2}}\cdot\frac{\Delta_1}{\Delta_1+\Delta_2}\cdot\Big(\frac{3^{3/2}}{2(b-a)^2}\cdot\frac{\Delta_2(\Delta_1+\Delta_2)^2}{\Delta_1^2}\Big)^{1/3} \\
&\stackrefp{crit_u}{=} \frac{2}{3}\Big(\frac{4\Delta_1\Delta_2(b-a)}{\Delta_1+\Delta_2}\Big)^{1/3}
\stackref{slope_relation}{=} \frac{2}{3}\Big(\frac{4\Delta_1\Delta_2}{s_b-s_a}\Big)^{1/3}.
}
Now add the three previous displays to obtain the value of $G$ at the critical point:
\eeq{ \label{value_at_crit}
G(u,v) = \Big(\frac{4\Delta_1\Delta_2}{s_b-s_a}\Big)^{1/3}.
}
In light of \eqref{eq:ptwise_bound}, this proves the upper bound ($\le$) for \eqref{unif_case_c_sup}.

To achieve the lower bound ($\ge$) and identify a unique maximizer, we fix the critical point $(u,v)$ from above so that \eqref{value_at_crit} holds.
Because of \eqref{1v7xe}, we just need to verify that there is a unique $f\in\cF_\partial^=$ that satisfies \eqref{mn5x2}.
Equivalently, we need to show that the initial conditions $f'(a) = s_a$ and $f(a) = f_a$ are consistent with the terminal conditions $f'(b) = s_b$ and $f(b) = f_b$.
Indeed, $f$ has the desired slope at $b$:
\eq{
f'(b)
&\stackrefp{uv_ratio}{=} s_a + \int_a^b f''(x)\ \dd x
= s_a + \frac{1}{3^{3/2}}\int_a^b\Big(u^2\cdot\frac{x-a}{b-a}+v^2\cdot\frac{b-x}{b-a}\Big)^{-3/2}\ \dd x \\ 
&\stackrefp{uv_ratio}{=} s_a + \frac{2(b-a)}{3^{3/2}}\cdot\frac{1}{uv(u+v)} \\
&\stackref{uv_ratio}{=} s_a + \frac{2(b-a)}{3^{3/2}}\cdot\frac{1}{u^3}\cdot\frac{\Delta_1^2}{\Delta_2(\Delta_1+\Delta_2)}
\stackref{crit_u}{=} s_a + \frac{\Delta_1+\Delta_2}{b-a} 
\stackref{slope_relation}{=} s_b.
}
To see that $f$ has the desired value at $b$, we compute
\eeq{ \label{wib8x}
f(b) 
&= f(a) + (b-a)f'(a) + \int_a^b f''(x)(b-x)\ \dd x \\
&= f_a + s_a(b-a) + \frac{1}{3^{3/2}}\int_a^b\Big(u^2\cdot\frac{x-a}{b-a}+v^2\cdot\frac{b-x}{b-a}\Big)^{-3/2}(b-x)\ \dd x  \\
&= f_a + s_a(b-a) + \frac{(b-a)^2}{3^{3/2}}\int_0^1\big(u^2(1-y)+v^2y\big)^{-3/2}y\ \dd y.
}
Compute the integral on the final line using integration by parts:
\eq{
&\int_0^1\big((v^2-u^2)y + u^2\big)^{-3/2}y\ \dd y \\
&= \frac{-2\big((v^2-u^2)y + u^2\big)^{-1/2}}{v^2-u^2}\cdot y\Big|_{y=0}^{y=1}
+ 2\int_0^1\frac{\big((v^2-u^2)y + u^2\big)^{-1/2}}{v^2-u^2}\ \dd y \\
&= \frac{-2v^{-1}}{v^2-u^2} + \frac{4(v-u)}{(v^2-u^2)^2}
= \frac{2}{v(u+v)^2}
\stackref{uv_ratio}{=} \frac{1}{u^3}\cdot\frac{2\Delta_1^3}{\Delta_2(\Delta_1+\Delta_2)^2}
\stackref{crit_u}{=} \frac{3^{3/2}\Delta_1}{(b-a)^2}.
}
Inserting this computation into \eqref{wib8x}, we obtain the desired value:
\[
f(b) = f_a + s_a(b-a) + \Delta_1 \stackref{eq_unif_case_a}{=} f_b. \qedhere
\]
\end{proof}

\subsection{Metrizing the function space} \label{subsec_metric}

Recall from \eqref{F_def} the space $\cF$ of admissible functions. 
We equip $\cF$ with a metric $d$ defined as follows.
For any subinterval $[a,b]\subset[0,1)$, denote the uniform norm on $[a,b]$ by
\eq{
\|f\|_{[a,b]} \coloneqq \sup_{x\in[a,b]}|f(x)|.
}
Now define the following metric on $\cF$:
\eeq{ \label{metric_def}
d(f,g) \coloneqq \sum_{k=1}^\infty 2^{-k}(\|f-g\|_{[0,1-1/k]}\wedge 1), \quad f,g\in\cF.
}
The key point is that $d$ generates the topology of locally uniform convergence on $[0,1)$.
This topology is coarser than that of uniform convergence on $[0,1]$, but has the advantage of compactness:

\begin{lemma}[Compactness] \label{lem_compact}
$(\cF,d)$ is a compact metric space.
\end{lemma}

\begin{proof}
Checking that $d$ is a metric is straightforward and thus omitted.
Here we verify compactness.
For every $f\in\cF$ and $x<y$ in $[0,1]$, convexity implies
\eeq{ \label{trivial_deriv_bound}
\frac{f(y) - f(x)}{y-x} \le \frac{f(1) - f(x)}{1-x} \le \frac{1 - 0}{1-x}.
}
Now consider any sequence $(f_n)_{n\ge1}$ in $\cF$.
For every $t\in[0,1)$, the restrictions $\big(f_n\big|_{[0,t]}\big)_{n\ge1}$ are uniformly Lipschitz by \eqref{trivial_deriv_bound}, and thus uniformly equicontinuous.
By Arzel\`a--Ascoli and diagonalization, we can pass to a subsequence and assume that for every $k\in\Z_{\ge1}$, the sequence $(f_n)_{n\ge1}$ converges pointwise uniformly on $[0,1-1/k]$.
Denote the limit by
\eq{
f(x) \coloneqq \lim_{n\to\infty} f_n(x) \quad \text{for $x\in[0,1)$}, \qquad
f(1) \coloneqq \lim_{x\nearrow1} f(x).
}
By construction, we have
\eq{
\lim_{n\to\infty}\|f_n - f\|_{[0,1-1/k]} \quad \text{for every $k\in\Z_{\ge1}$},
}
so $d(f_n,f)\to 0$ as $n\to\infty$.
Since each $f_n$ belongs to $\cF$, the limiting function $f$ also belongs to $\cF$.
We have thus proved (sequential) compactness.
\end{proof}

Although $d$ is defined in terms of uniform norms, convergence can be phrased without any uniformity:

\begin{lemma}[Characterization of convergence] \label{lem_convergence}
For any $f\in\cF$ and any sequence $(f_n)_{n\ge1}$ in $\cF$, the following statements are equivalent:
\begin{enumerate}[label=\textup{(\alph*)}]

\item \label{lem_convergence_a}
$d(f_n,f)\to0$ as $n\to\infty$.

\item \label{lem_convergence_b}
$f_n(x)\to f(x)$ as $n\to\infty$, for every $x\in[0,1)$.

\item \label{lem_convergence_c}
$\partial^+f_n(t)\to f'(t)$ as $n\to\infty$, for every differentiability point $t\in(0,1)$ of $f$.

\end{enumerate}
\end{lemma}

\begin{proof}
\ref{lem_convergence_a}$\implies$\ref{lem_convergence_b}: This implication is straightforward and does not rely on convexity:
\eq{
\lim_{n\to\infty}d(f_{n},f) = 0 \quad 
\iff \quad &\lim_{n\to\infty}\|f_{n}-f\|_{[0,1-1/k]} = 0 \quad \text{for every $k$} \\
\implies \quad &\mathrlap{\lim_{n\to\infty}|f_{n}(x)-f(x)| = 0} \hphantom{\lim_{n\to\infty}\|f_{n}-f\|_{[0,1-1/k]} = 0}\, \quad \text{for every $x\in[0,1)$}.
}

\medskip

\noindent \ref{lem_convergence_b}$\implies$\ref{lem_convergence_c}:
Consider any differentiability point $t\in(0,1)$ of $f$.
By convexity, the assumed pointwise convergence $f_n\to f$ implies $\partial^+f_n(t)\to f'(t)$, e.g.\ see \cite[Theorem~B.12(7)]{friedli_velenik17}.

\medskip

\noindent\ref{lem_convergence_c}$\implies$\ref{lem_convergence_a}:
For every $x\in[0,1]$, we have
\eeq{ \label{2bcj4}
|f_n(x) - f(x)| 
\stackref{1_deriv_eq}{=} \Big|\int_0^x \partial^+f_n(t)\ \dd t - \int_0^x \partial^+f(t)\ \dd t\Big|
\le \int_0^x |\partial^+f_n(t) - \partial^+f(t)|\ \dd t.
}
By assumption, we have $\partial^+f_n(t)\to \partial^+f(t)$ for all $t\in(0,1)$ except possibly countably many discontinuity points of $\partial^+f$.
Furthermore, for every $\delta>0$ and $t\in[0,1-\delta]$, convexity implies
\eq{
\partial^+g(t) \le \frac{g(1)-g(t)}{1-t} \le \frac{1-0}{1-t} \le \frac{1}{\delta} \quad \text{for every $g\in\cF$.}
}
Therefore, dominated convergence gives
\eeq{ \label{2kb8cz}
\lim_{n\to\infty}\int_0^{1-\delta} |\partial^+f_n(t) - \partial^+f(t)|\ \dd t = 0 \quad \text{for every $\delta>0$.}
}
Now \eqref{2bcj4} and \eqref{2kb8cz} together imply $\|f_n-f\|_{[0,1-1/k]}\to0$ as $n\to\infty$, for every $k$.
Hence $d(f_n,f)\to0$.
\end{proof}

The following consequence of the two previous lemmas will be used much later, in the proof of Lemma~\ref{lem_measurable}.
It may appear as though we are asserting that $(\cF,d)$ is separable, but the statement is actually stronger since we demand pointwise convergence on $[0,1]$ rather than $[0,1)$.

\begin{proposition}[Dense subsets] \label{prop_separable}
Every closed set $\cK$ in the metric space $(\cF,d)$ contains a countable subset $\cD\subseteq\cK$ such that the following statement holds.
For every $f\in\cK$, there exists a sequence $(f_n)_{n\ge1}$ in $\cD$ such that $\lim_{n\to\infty}f_n(x) = f(x)$ for every $x\in[0,1]$.
\end{proposition}

\begin{proof}
We begin with the following observation, which relies on convexity.

\begin{claim}[Lower semicontinuity] \label{claim_lsc}
If $\lim_{m\to\infty}d(g_m,f)=0$, then $f(1)\le \liminf_{m\to\infty}g_m(1)$.
\end{claim}

\begin{proofclaim}
Since $d(g_m,f)\to0$, we have $\partial^+g_m(x)\to f'(x)$ for every differentiability point $x\in(0,1)$ of $f$, by Lemma~\ref{lem_convergence}.
This justifies the penulimate equality below:
\begin{align*}
\liminf_{m\to\infty}g_m(1)
&\stackrel{\parbox{\widthof{\footnotesize(Fatou)}}{\centering\footnotesize\eqref{1_deriv_eq}}}{=} \liminf_{m\to\infty} \int_0^1 \partial^+g_m(x)\ \dd x \\
&\stackrel{\mbox{\footnotesize(Fatou)}}{\ge} \int_0^1 \liminf_{m\to\infty}\partial^+g_m(x)\ \dd x \\
&\stackrel{\hphantom{\mbox{\footnotesize(Fatou)}}}{=} \int_0^1 f'(x)\ \dd x 
\stackref{1_deriv_eq}{=} f(1). \qedhere
\end{align*}
\end{proofclaim}

For each rational $r\in\Q\cap[0,1]$, define $\cK_r = \{f\in\cK:\, f(1)\le r\}$.
Since $\cK$ is closed with respect to the metric $d$, so too is $\cK_r$ thanks to Claim~\ref{claim_lsc}.
Given that the ambient space $(\cF,d)$ is compact (Lemma~\ref{lem_compact}), it follows that $\cK_r$ is compact and hence separable.
Choose a countable subset $\cD_r\subseteq\cK_r$ that is dense in $\cK_r$, and let $\cD = \bigcup_{r\in\Q\cap[0,1]}\cD_r$.
Clearly $\cD$ is countable and a subset of $\cK$, so we just need to verify the final sentence of the proposition.

To this end, consider any $f\in\cK$.
Choose rationals $(q_n)_{n\ge1}$ and $(r_n)_{n\ge1}$ in $[0,1]$ such that $q_n \le f(1) \le r_n$ for every $n$, and 
\eeq{ \label{rationals_converge}
\lim_{n\to\infty}q_n = \lim_{n\to\infty}r_n = f(1).
}
Since $f\in\cK_{r_n}$, there is a sequence $(g_{n,m})_{m\ge1}$ in $\cD_{r_n}$ such that $d(g_{n,m},f)\to0$ as $m\to\infty$.
Furthermore, by Claim~\ref{claim_lsc} we have
\eq{
q_n \le f(1) \le \liminf_{m\to\infty} g_{n,m}(1).
}
So choose $m_n$ large enough that $d(g_{n,m_n},f) \le n^{-1}$ and $q_n - n^{-1} \le g_{n,m_n}(1) \le r_n$.
By setting $f_n = g_{n,m_n}$, we have $d(f_n,f)\to0$ (meaning $f_n(x)\to f(x)$ for all $x\in[0,1)$ by Lemma~\ref{lem_convergence}), as well as $f_n(1)\to f(1)$ thanks to \eqref{rationals_converge}.
\end{proof}

\subsection{Associated space of measures} \label{subsec_space_measures}

In upcoming arguments, it will be helpful to think of $\partial^+f$ as the distribution function of a measure on $[0,1)$.
This leads to an identification scheme between $\cF$ and a space of measures, which is made precise below.

Let $\cM$ denote the set of locally finite positive Borel measures $\mu$ on $[0,1)$ such that
\begin{align} \label{meas_constraint}
\int_0^1 \mu[0,t]\ \dd t \le 1.
\end{align}
Every $\mu\in\cM$ is $\sigma$-finite and thus admits a unique decomposition $\mu = \mu^\ac + \mu^\sing$, where $\mu^\ac$ is absolutely continuous with respect to Lebesgue measure, and $\mu^\sing$ is mutually singular with Lebesgue measure.
We denote the Radon--Nikodym derivative of $\mu^\ac$ with respect to Lebesgue measure by
\eeq{ \label{RN_deriv}
D_\mu \coloneqq \frac{\dd\mu^\ac}{\dd x}.
}
Given $\mu\in\mathcal{M}$, define the function $f_\mu\colon[0,1]\to[0,1]$ by
\begin{align} \label{f_from_meas}
f_\mu(x) \coloneqq \int_0^x \mu[0,t]\ \dd t, \quad x\in[0,1].
\end{align}

\begin{lemma}[Identification scheme] \label{lem_identification_scheme}
The following statements hold.
\begin{enumerate}[label=\textup{(\alph*)}]

\item \label{scheme_bijection} The map $\mu \mapsto f_\mu$ defined in \eqref{f_from_meas} is a bijection $\mathcal{M}\to\mathcal{F}$.

\item \label{scheme_right_deriv} $\partial^+f_\mu(x) = \mu[0,x] \le (1-x)^{-1}$ for every $x\in[0,1)$.

\item \label{scheme_left_deriv} $\partial^-f_\mu(x) = \mu[0,x)$ for every $x\in(0,1)$.

\item \label{scheme_2nd_deriv} $f_\mu''(x) = D_\mu(x)$ for almost every $x\in(0,1)$.

\end{enumerate}
\end{lemma}

\begin{proof}
We first check that $f_\mu$ belongs to $\mathcal{F}$, provided $\mu$ belongs to $\mathcal{M}$.
From \eqref{meas_constraint}, it is trivial that $f_\mu(x)\in[0,1]$ for all $x\in[0,1]$, and $f_\mu(0)=0$.
To see that $f_\mu$ is convex, consider the following difference quotient for any $x<y$ in $[0,1]$:
\begin{align} \label{diff_quotient}
\frac{f_\mu(y) - f_\mu(x)}{y-x}
= \frac{1}{y-x}\int_x^y \mu[0,t]\ \dd t.
\end{align}
Since $t\mapsto\mu[0,t]$ is nondecreasing, the average on the right-hand side is nondecreasing in both $x$ and $y$, meaning that $f_\mu$ is convex.
We have thus verified that $f_\mu$ belongs to $\mathcal{F}$.

Sending $y\searrow x$ in \eqref{diff_quotient} yields the first equality below:
\eq{
\partial^+f_\mu(x) = \lim_{t\searrow x}\mu[0,t] = \mu[0,x] \le \frac{1}{1-x}\int_x^1 \mu[0,t]\ \dd t
\le \frac{1}{1-x}\int_0^1 \mu[0,t]\ \dd t \stackref{meas_constraint}{\le} \frac{1}{1-x}.
}
This proves part~\ref{scheme_right_deriv}.
Part~\ref{scheme_left_deriv} follows because
\eq{
\partial^-f_\mu(x)
= \lim_{t\nearrow x}\partial^+f_\mu(t)
= \lim_{t\nearrow x}\mu[0,t]
= \mu[0,x).
}

We next check that $\mu\mapsto f_\mu$ is a bijection.
For injectivity, suppose $\mu,\nu\in\mathcal{M}$ are such that $f_\mu(x) = f_\nu(x)$ for every $x\in[0,1]$.
It follows from part~\ref{scheme_right_deriv} that $\mu[0,t] = \nu[0,t]$ for every $t\in[0,1)$.
Hence $\mu = \nu$.
For surjectivity, consider any $f\in\cF$.
Since $f(0)=0$, we have
\begin{align} \label{f_via_deriv}
f(x) = \int_0^x \partial^+f(t)\ \dd t \quad \text{for all $x\in[0,1]$.}
\end{align}
Furthermore, since $\partial^+f$ is nonnegative, nondecreasing, and right-continuous, it is the distribution function of some measure. 
That is, there exists a Borel measure $\mu$ on $[0,1)$ satisfying
\begin{align} \label{meas_from_f}
\mu[0,t] = \partial^+f(t)<\infty \quad \text{for all $t\in[0,1)$}.
\end{align}
This measure belongs to $\mathcal{M}$ since
\begin{align*}
\int_0^1 \mu[0,t]\ \dd t = \int_0^1 \partial^+f(t)\ \dd t \stackref{f_via_deriv}{=} f(1) \le 1.
\end{align*}
The identities \eqref{f_via_deriv} and \eqref{meas_from_f} together show $f = f_\mu$.
This completes the proof of part~\ref{scheme_bijection}.

Finally, for part~\ref{scheme_2nd_deriv} recall that \eqref{second_deriv_def} holds for almost every $x\in(0,1)$.
Also $f_\mu'(x)$ exists for almost every $x\in(0,1)$.
If $x$ has both of these properties, then
\eq{
f_\mu''(x) 
= \frac{f_\mu''(x) - (-f_\mu''(x))}{2}
&= \lim_{r\searrow0}\frac{\partial^- f_\mu(x+r)-f_\mu'(x)}{2r} - \lim_{r\searrow0}\frac{\partial^+ f_\mu(x-r)-f_\mu'(x)}{2r} \\
&= \lim_{r\searrow0}\frac{\partial^- f_\mu(x+r)-\partial^+ f_\mu(x-r)}{2r} \\
&= \lim_{r\searrow0}\frac{\mu(x-r,x+r)}{2r}.
}
The Lebesgue differentiation theorem \cite[Theorem~5.8.8]{bogachev07} implies that for almost every $x\in(0,1)$, the limit on the last line is equal to $D_\mu(x)$.
\end{proof}

We can now rewrite Lemma~\ref{lem_convergence} as follows.

\begin{corollary}[Characterization of convergence, version 2] \label{cor_convergence}
For any $\mu\in\cM$ and any sequence $(\mu_n)_{n\ge1}$ in $\cM$, the following statements are equivalent:
\begin{enumerate}[label=\textup{(\alph*)}]

\item \label{cor_convergence_a}
$d(f_{\mu_n},f_\mu)\to0$ as $n\to\infty$.

\item \label{cor_convergence_b}
$f_{\mu_n}(x)\to f_\mu(x)$ as $n\to\infty$, for every $x\in[0,1)$.

\item \label{cor_convergence_c}
$\mu_n[0,t]\to\mu[0,t]$ as $n\to\infty$, for every continuity point $t\in(0,1)$ of $\mu$.

\end{enumerate}
\end{corollary}

\begin{proof}
By Lemma~\ref{lem_identification_scheme}\hyperref[scheme_right_deriv]{(b,c)}, the continuity points of $\mu$ are precisely the points of differentiability of $f_\mu$.
Therefore, the statement of the corollary is equivalent to Lemma~\ref{lem_convergence}.
\end{proof}

\subsection{Upper semicontinuity of objective function} \label{subsec_usc}

The goal of this section is to prove the following result, which will be essential for studying optimizers in Section~\ref{subsec_maximizers}.

\begin{proposition}[Upper semicontinuity] \label{prop_usc}
Assume \eqref{p_cont}.
Then the functional $J\colon\cF\to[0,\infty)$ given by \eqref{J_def1} is upper semicontinuous with respect to the metric $d$ from \eqref{metric_def}.
\end{proposition}

Because of Lemma~\ref{lem_identification_scheme}, the functional $J$ can be rewritten as follows.
Given $f$, let $\mu\in\cM$ be such that $f_\mu = f$. Then
\eq{
J(f) = \int_0^1\big[D_\mu(x)\cdot\fp(x,f(x))\big]^{1/3}\ \dd x.
}
This motivates the following lemma.

\begin{lemma}[Precursor to upper semicontinuity]\label{lem_usc}
Let $\mu_n, \mu \in \mathcal{M}$.
Fix $t\in[0,1)$.
If $\mu_n\big|_{[0,t]}$ converges weakly to $\mu\big|_{[0,t]}$ as $n\to\infty$, then for any continuous function $w\colon[0,t] \to [0, \infty)$, we have
\eeq{ \label{usc_eta}
\limsup_{n\to\infty}\int_0^{t} w(x)\cdot D_{\mu_n}(x)^{1/3}\ \dd x  \le  \int_0^{t} w(x)\cdot D_\mu(x)^{1/3}\ \dd x.
}
\end{lemma}

\begin{proof}
On the domain $[0,\infty)$, we can express the concave function $y\mapsto y^{1/3}$ as the infimum of its tangent lines:
\begin{align}
\label{eq_cube_root_var}
y^{1/3}=\inf_{a>0}\Big\{\tfrac13 a^{-2/3} y +\tfrac23 a^{1/3}\Big\}.
\end{align}
%For $y>0$ the minimum is attained at $a=y$, while for $y=0$ the infimum is $0$, obtained by taking $a\downarrow0$. 
Now fix a continuous function $w\colon [0,t] \to [0, \infty)$, which is necessarily bounded since $[0,t]$ is compact.
Consider any strictly positive continuous function $A\colon [0,t] \to (0, \infty)$, which is necessarily bounded away from $0$ since $[0,t]$ is compact.
We have
\begin{align*}
\limsup_{n\to\infty}\int_0^t w \cdot (D_{\mu_n})^{1/3}\ \dd x 
&\stackref{eq_cube_root_var}{\le}  \limsup_{n\to\infty}\int_0^t w\cdot \big( \tfrac13 A^{-2/3} D_{\mu_n} + \tfrac23 A^{1/3} \big)\ \dd x\\
&\stackrefpp{RN_deriv}{eq_cube_root_var}{\le}   \limsup_{n\to\infty}\int_{[0,t]} w\cdot \big( \tfrac13 A^{-2/3}\ \mu_n(\dd x) + \tfrac23 A^{1/3}\ \dd x \big) \\
&\stackrefp{eq_cube_root_var}{=} \int_{[0,t]} w \cdot \big( \tfrac13 A^{-2/3}\ \mu(\dd x) + \tfrac23 A^{1/3}\ \dd x \big),
\end{align*}
where the last equality uses the assumption of weak convergence $\mu_n\big|_{[0,t]}\to\mu\big|_{[0,t]}$.
To complete the proof of \eqref{usc_eta}, it suffices to show 
\begin{equation}\label{inf_cont_A}
  \inf_{ A} \int_{[0,t]} w \cdot \big( \tfrac13 A^{-2/3}\ \mu(\dd x) + \tfrac23 A^{1/3}\ \dd x \big) 
  \le \int_0^t w\cdot D_\mu^{1/3}\ \dd x,  
\end{equation}
where the infimum is taken over all strictly positive continuous functions $A\colon [0,t] \to (0, \infty)$. 
We decompose the integral on the left-hand side of \eqref{inf_cont_A}
as $\cI_1(A) + \cI_2(A)$, where
\begin{align*}
\cI_1(A) &\coloneqq \int_{[0,t]} w \cdot \tfrac13 A^{-2/3}\ \mu^\sing(\dd x), \\
\cI_2(A) &\coloneqq  \int_0^t w \cdot\big( \tfrac13 A^{-2/3} D_\mu  + \tfrac23 A^{1/3} \big)\ \dd x.
\end{align*}
To proceed, we approximate $D_\mu$ by a continuous function $\phi_\eps$.
Define $\phi\colon\R\to[0,\infty)$ by
\eeq{ \label{phi_Dmu_def}
\phi(x) \coloneqq \begin{cases}
D_\mu(x) &\text{if $x\in [0,t]$}\\
0 &\text{otherwise.}
\end{cases}
}
Note that $\phi$ is integrable since $\int_{-\infty}^\infty \phi\, \dd x = \mu^\ac[0,t]<\infty$.
Let $(\psi_\eps)_{\eps>0}$ be a family of smooth mollifiers on $\mathbb R$ that approximate the Dirac measure at $0$ as $\eps\searrow0$, such as $\psi_\eps(t) = \frac{1}{\sqrt{2\pi\eps}}\exp(\frac{-t^2}{2\eps})$.
Let $\phi_\eps \coloneqq \phi*\psi_\eps$.
Then $\phi_\eps$ is continuous and nonnegative, and as $\eps\searrow0$, we have $\phi_\eps\to\phi$ almost everywhere, and also $\int_{-\infty}^\infty |\phi_\eps - \phi|\,\dd x\to0$.
For each $m\in\Z_{\ge1}$, choose $\eps_m\in(0,m^{-1}]$ such that $\int_{-\infty}^\infty |\phi_{\eps_m} - \phi | \le m^{-1}$. 
Let $\vphi_m \coloneqq \phi_{\eps_m}.$

Since $\mu^\sing\big|_{[0,t]}$ is singular and finite, there exists a compact set $K_m\subset [0,t]$ of Lebesgue measure zero such that 
\eeq{ \label{nc4khp}
\mu^\sing([0,t]\setminus K_m) \le m^{-1}.
}
Next choose an open set $U_m \supset K_m$ such that
\eeq{ \label{U_m_properties}
\mathrm{Leb}(U_m) \le m^{-1} \quad \text{and} \quad \int_{U_m}\phi\ \dd x = \mu^\ac\big|_{[0,t]}(U_m)\le m^{-1}.
}
Combining the latter inequality with $\int_{-\infty}^\infty |\vphi_m  - \phi| \le m^{-1}$, we obtain
\begin{equation}\label{gm_over_um_int_bdd}
    \int_{U_m} \vphi_m\ \dd x \le 2m^{-1}.
\end{equation}
Let $\chi_m\colon [0,t]\to[0,1]$ be a  continuous bump function such that 
$\chi_m \equiv 1$ on $K_m$, and $\chi_m \equiv 0 $ on $[0,t]\setminus U_m$.
Define the following continuous function on $[0,t]$:
\[ A_m \coloneqq (\vphi_m \vee m^{-1} ) (1 -\chi_m)  + m \chi_m. \]
Note that
\eeq{ \label{A_m_bounds}
m^{-1} \le A_m \le \vphi_m + m.
}
Since $A_m\equiv m$ on $K_m$ and $A_m \ge m^{-1}$ otherwise, we have %$\mu^\sing([0,t]\setminus K_m)\le m^{-1}$, we have
\eeq{ \label{I_2_to_zero}
    \cI_1(A_m) 
    &\le \int_{K_m} w \cdot  \tfrac13 m^{-2/3}\ \mu^\sing(\dd x) 
    + \int_{[0,t]\setminus K_m} w\cdot \tfrac{1}{3} m^{2/3}\ \mu^\sing(\dd x) \\
    &\le  \tfrac{1}{3}\| w \|_\infty\Big(m^{-2/3}\mu^\sing(K_m) + m^{2/3}\mu^\sing([0,t]\setminus K_m)\Big)  \xrightarrow[m\to\infty]{\mbox{\footnotesize\eqref{nc4khp}}} 0.
}
Meanwhile, since $A_m = \vphi_m \vee m^{-1}$ on $[0,t]\setminus U_m$, we have
\eeq{ \label{I_1_control}
    \cI_2(A_m) &= \int_{[0,t]\setminus U_m}w \cdot\Big( \tfrac13 (\vphi_m \vee m^{-1} )^{-2/3} \phi  + \tfrac23 (\vphi_m \vee m^{-1} )^{1/3} \Big)\ \dd x  \\
    &\phantom{=}+ \int_{[0,t]\cap U_m}w \cdot\Big( \tfrac13 A_m^{-2/3} \phi  + \tfrac23 A_m^{1/3} \Big)\ \dd x.
}
The integral over $[0,t]\cap U_m$ is controlled as follows:
\begin{align*}
&\int_{[0,t]\cap U_m}w \cdot\Big( \tfrac13 A_m^{-2/3} \phi  + \tfrac23 A_m^{1/3} \Big)\ \dd x 
\stackref{A_m_bounds}{\le} \tfrac13\| w \|_\infty\int_{[0,t]\cap U_m} \Big(m^{2/3} \phi + 2(\vphi_m+m)^{1/3}\Big)\ \dd x \\
&\stackrefp{U_m_properties}{\le} \tfrac13\| w \|_\infty\Big[m^{2/3}\int_{U_m}\phi\ \dd x + 2\int_{U_m}(\vphi_m^{1/3} + m^{1/3})\ \dd x\Big] \\
&\stackref{U_m_properties}{\le} \tfrac13\| w \|_\infty\Big[m^{-1/3} + 2\Big(\int_{U_m} \vphi_m\ \dd x\Big)^{1/3} + 2m^{-2/3}\Big] 
\xrightarrow[m\to\infty]{\mbox{\footnotesize\eqref{gm_over_um_int_bdd}}}0.
%&\stackref{gm_over_um_int_bdd}{\le} \tfrac{1}{3}m^{-1/3} + \tfrac{2^{4/3}}{3}m^{-1} + \tfrac{2}{3}m^{-2/3}.
\end{align*}
Meanwhile, the integral over $[0,t]\setminus U_m$ in \eqref{I_1_control} is bounded from above by
\begin{align*}
    \int_0^t w \cdot \Big( \tfrac13 (\vphi_m \vee m^{-1} )^{-2/3} \vphi_m  + \tfrac23 (\vphi_m \vee m^{-1} )^{1/3} \Big)\ \dd x +  \tfrac13  m^{2/3} \| w \|_\infty  \int_0^t  |\vphi_m - \phi|\ \dd x.
\end{align*}
The second integral above vanishes as $m\to\infty$ since $\int_{-\infty}^\infty |\vphi_m - \phi|\, \dd x \le m^{-1}.$
Concerning the first integral, note that the almost everywhere convergence $\vphi_m\to\phi$ implies
\[ 
\lim_{m\to\infty} w \cdot \Big( \tfrac13 (\vphi_m \vee m^{-1} )^{-2/3} \vphi_m  + \tfrac23 (\vphi_m \vee m^{-1} )^{1/3} \Big) = w \cdot\phi^{1/3}  \quad \text{a.e.}
\]
Moreover, the prelimiting sequence is dominated as follows:
\begin{align*}
&w \cdot \Big( \tfrac13 (\vphi_m \vee m^{-1} )^{-2/3} \vphi_m  + \tfrac23 (\vphi_m \vee m^{-1} )^{1/3} \Big) \\
&\le
     \| w\|_\infty\Big(  \tfrac13 (\vphi_m \vee m^{-1} )^{-2/3} (\vphi_m \vee m^{-1})  + \tfrac23 (\vphi_m \vee m^{-1} )^{1/3} \Big) \\
     &=\| w\|_\infty  (\vphi_m \vee m^{-1})^{1/3}  \le \| w\|_\infty (\vphi_m^{1/3} +1).
\end{align*}
Note that $\sup_m \int_0^t  (\vphi_m^{1/3})^3\, \dd x = \sup_m \int_0^t  \vphi_m\, \dd x \le \int_0^t \phi\, \dd x + \sup_m\int_{-\infty}^\infty|\vphi_m-\phi|\,\dd x < \infty$. 
This $L^3$ bound implies uniform integrability of the prelimiting sequence, from which we conclude
\begin{align*}
\lim_{m\to\infty}\int_0^t w\cdot \Big(\tfrac13 (\vphi_m \vee m^{-1})^{-2/3} \vphi_m  + \tfrac23 (\vphi_m \vee m^{-1})^{1/3} \Big)\ \dd x = \int_0^t w\cdot \phi^{1/3}\ \dd x.
\end{align*}
Putting together all the observations from \eqref{I_1_control} onward, we obtain
\eq{
\limsup_{m\to\infty}\cI_2(A_m) 
\le \int_0^t w\cdot\phi^{1/3}\ \dd x
\stackref{phi_Dmu_def}{=} \int_0^t w\cdot D_\mu^{1/3}\ \dd x.
}
Together with \eqref{I_2_to_zero}, this implies \eqref{inf_cont_A}.
\end{proof}

\begin{proof}[Proof of Proposition~\ref{prop_usc}]
Consider any convergent sequence $f_n\to f$ in the metric space $(\cF,d)$.
Our goal is to show 
\eeq{ \label{usc_to_show}
\limsup_{n\to\infty} J(f_n) \le J(f).
}
For any $t\in[0,1)$, we have the following sequence of inequalities:
\eq{ %\label{2c85kb}
\int_t^1 \big[f_n''(x)\cdot \fp(x,f_n(x))\big]^{1/3}\ \dd x
&\stackrefpp{p_upper}{eq_unif_case_a}{\le} (2\fC)^{1/3}\int_t^1 f_n''(x)^{1/3}\ \dd x \\
&\stackref{eq_unif_case_a}{\le} \big[8\fC(1-t)\big(f_n(1)-f_n(t)-(1-t)\cdot \partial^+f_n(t)\big)\big]^{1/3} \\
&\stackrefp{eq_unif_case_a}{\le} \big[8\fC(1-t)\big]^{1/3}.
}
It follows that
\eq{
\limsup_{n\to\infty} J(f_n) 
\le \big[8\fC(1-t)\big]^{1/3} + \limsup_{n\to\infty}\int_0^t \big[f_n''(x)\cdot \fp(x,f_n(x))\big]^{1/3}\ \dd x.
}
Sending $t\nearrow1$ will yield \eqref{usc_to_show}, provided we prove
\eeq{ \label{usc_suffices}
\limsup_{n\to\infty}\int_0^t \big[f_n''(x)\cdot \fp(x,f_n(x))\big]^{1/3}\ \dd x \le J(f) \quad \text{for every $t\in(0,1)$}.
}
The rest of the proof is showing \eqref{usc_suffices}.

As in Lemma~\ref{lem_identification_scheme}, let $\mu_n,\mu\in\cM$ be such that $f_{\mu_n} = f_n$ and $f_\mu = f$.
By part~\ref{scheme_2nd_deriv} of that lemma, we have almost sure equalities $D_{\mu_n} = f_n''$ and $D_{\mu} = f''$.
The implication \ref{cor_convergence_a}$\implies$\ref{cor_convergence_c} in Corollary~\ref{cor_convergence} gives $\lim_{n\to\infty}\mu_n[0,t] = \mu[0,t]$ for every continuity point $t\in(0,1)$ of $\mu$.
It follows that $\mu_n\big|_{[0,t]}\to\mu\big|_{[0,t]}$ weakly for every continuity point $t\in(0,1)$ of $\mu$.
For such $t$ we may apply Lemma~\ref{lem_usc} with $w(x) = \fp(x,f(x))^{1/3}$, to obtain
\eeq{ \label{ebc7b}
\limsup_{n\to\infty} \int_0^t \big[f_n''(x)\cdot \fp(x,f(x))\big]^{1/3}\ \dd x
&\le \int_0^t \big[f''(x)\cdot \fp(x,f(x))\big]^{1/3}\ \dd x 
\le J(f).
}
This is almost \eqref{usc_suffices}, but we still need to replace $f$ with $f_n$ in the leftmost integral.
To this end, let $\eps>0$ be arbitrary.
By continuity of $\fp$ from \eqref{p_cont}, there exists $\delta>0$ such that
\eq{
\big|\fp(x,y)^{1/3}-\fp(x,y')^{1/3}\big| \le \eps(1-t)^{1/3}
\quad \text{whenever $|y-y'|\le\delta$.}
}
Since $d(f_n,f)\to0$ as $n\to\infty$, we have $\|f_n - f\|_{[0,t]}\le\delta$ for all large $n$.
This justifies the first inequality below:
\eq{
&\limsup_{n\to\infty}
\Big|\int_0^t \big[f_n''(x)\cdot \fp(x,f_n(x))\big]^{1/3}\ \dd x
- \int_0^t \big[f_n''(x)\cdot \fp(x,f(x))\big]^{1/3}\ \dd x\Big| \\
&\stackrel{\parbox{\widthof{\footnotesize(Lemma~\ref{lem_identification_scheme}\ref{scheme_2nd_deriv})}}{\hphantom{x}}}{\le}
\eps(1-t)^{1/3}\int_0^t f_n''(x)^{1/3}\ \dd x
\le \eps(1-t)^{1/3}\Big(\int_0^t f_n''(x)\Big)^{1/3}\ \dd x \\
&\stackrel{\parbox{\widthof{\footnotesize(Lemma~\ref{lem_identification_scheme}\ref{scheme_2nd_deriv})}}{\centering\footnotesize (Lemma~\ref{lem_identification_scheme}\ref{scheme_2nd_deriv})}}{=}
\eps(1-t)^{1/3}(\mu_n^\ac[0,t])^{1/3} %\\
\stackrel{\parbox{\widthof{\footnotesize(Lemma~\ref{lem_identification_scheme}\ref{scheme_2nd_deriv})}}{\centering\footnotesize (Lemma~\ref{lem_identification_scheme}\ref{scheme_right_deriv})}}{\le}
\eps(1-t)^{1/3} (1-t)^{-1/3} 
= \eps.
}
As $\eps$ is arbitrary, we conclude
\eeq{ \label{4bcjh}
\limsup_{n\to\infty}\int_0^t \big[f_n''(x)\cdot \fp(x,f_n(x))\big]^{1/3}\ \dd x
= \limsup_{n\to\infty}\int_0^t \big[f_n''(x)\cdot \fp(x,f(x))\big]^{1/3}\ \dd x.
}
Combining \eqref{4bcjh} and \eqref{ebc7b} yields 
\eeq{ \label{usc_cont}
\limsup_{n\to\infty}\int_0^t\big[f_n''(x)\cdot \fp(x,f_n(x))\big]^{1/3}\ \dd x \le J(f)
}
for every continuity point $t\in(0,1)$ of $\mu$.
Since continuity points are dense, and the left-hand side of \eqref{usc_cont} is monotone in $t$, \eqref{usc_suffices} follows.
\end{proof}

\subsection{Properties of the set of maximizers} \label{subsec_maximizers}
Recall the optimal value $J_\star$ from \eqref{Jstar_def}, and the set of maximizers $\argmax J$ from \eqref{argmax_def}.

\begin{proposition}[Optimizers exist] \label{prop_maximizers}
Assume \eqref{p_cont}.
Then the set $\argmax J$ is nonempty and compact in the metric space $(\cF,d)$.
\end{proposition}

\begin{proof}
Note that $J_\star$ must be finite by the following inequalities:
\eq{
J(f) = \int_0^1 \big[f''(x)\cdot \fp(x,f(x))\big]^{1/3}\ \dd x 
\stackref{p_upper}{\le} (2\fC)^{1/3}\int_0^1 f''(x)^{1/3}\ \dd x
\stackref{eq_unif_case_a}{\le} 2\fC^{1/3}.
}
Now consider any sequence $(f_n)_{n\ge1}$ in $\cF$ such that $J(f_n)\to J_\star$ as $n\to\infty$.
By compactness (Lemma~\ref{lem_compact}), we may pass to a subsequence and assume $d(f_n,f)\to0$ for some $f\in\cF$.
By upper semicontinuity (Proposition~\ref{prop_usc}), we must have $J(f) = J_\star$.
This verifies that $\argmax J$ is nonempty and closed.
Since the ambient space $(\cF,d)$ is compact, it follows that $\argmax J$ is compact.
\end{proof}

The next proposition gives several properties shared by every optimizer.
It is needed for several results that come after, but also interesting in its own right.

\begin{proposition}[Properties of optimizers] \label{prop_maximizer_properties}
Assume \eqref{p_upper} and \eqref{p_lower}.
Then every $f\in\argmax J$ has the following properties.
\begin{enumerate}[label=\textup{(\alph*)}]

\item \label{prop_maximizer_1_at_1} $f(1)=1$.

\item \label{prop_maximizer_no_linear} $f$ is strictly convex.

\item \label{prop_maximizer_abs_cont} $x\mapsto\partial^+f(x)$ is absolutely continuous on $[0,u)$ for every $u\in(0,1)$.

\end{enumerate}
\end{proposition}

In general we do not know how many optimizers there are, but the following result offers a sufficient condition for uniqueness.

\begin{proposition}[Uniqueness from concavity]\label{prop_unique_y_concave}
Assume \eqref{p_upper}, \eqref{p_lower}, and that
$y\mapsto \fp(x,y)$ is concave on $[0,x]$ for every $x\in[0,1]$.
Then $J$ is concave on $\cF$, and $\argmax J$ is a singleton.
\end{proposition}

In general, however, optimizers are not necessarily unique.
In the following, we say a function $\vphi\colon\cT\to\R$ is \textit{smooth} if it is the restriction of a smooth function defined on all of $\R^2$.

\begin{proposition}[Multiple optimizers] \label{prop_multiple_optimizers}
For any $k\in\Z_{\ge1}$, there exists a smooth probability density $\fp$ on $\cT$ satisfying \eqref{p_lower}, such that $\argmax J$ contains at least $k$ elements.
\end{proposition}

Next we state an important regularity among the optimizers.

\begin{proposition}[Uniform equicontinuity] \label{prop_equicontinuity}
Assume \eqref{p_cont} and \eqref{p_lower}.
\begin{enumerate}[label=\textup{(\alph*)}]

\item \label{prop_equicontinuity_a}
For every $\eps>0$, there exists $\eta>0$ such that $f(1-\eta) \ge 1-\eps$ for every $f\in\argmax J$.

\item \label{prop_equicontinuity_b} 
For every $\eps>0$, there exists $\eta>0$ such that the following implication holds for all $x,y\in[0,1]$ and $f\in\argmax J$:
\eeq{ \label{nka3c}
|x-y|\le\eta \quad \implies \quad |f(x)-f(y)| \le \eps.
}

\end{enumerate}
\end{proposition}

The following consequence of Proposition~\ref{prop_equicontinuity}  will be crucial in Section~\ref{sec_upper_bound}.

\begin{proposition}[Penalty for separation from optimizers] \label{prop_far_from_optimizer}
Assume \eqref{p_cont} and \eqref{p_lower}.
Then for every $\eps>0$, there exists $\theta>0$ such that the following implication holds for every $f\in\cF$:
\eeq{ \label{far_implication}
\inf_{h\in\argmax J}\|f-h\|_{[0,1]} \ge \eps \quad \implies \quad J(f) \le J_\star - \theta.
}
\end{proposition}

The argument for Proposition~\ref{prop_maximizer_properties} is the most technical, so we postpone it to the end of the section.
We now prove the other results assuming Proposition~\ref{prop_maximizer_properties}.

\begin{proof}[Proof of Proposition~\ref{prop_unique_y_concave}]
Let $\Phi(a,q)\coloneqq(aq)^{1/3}$ for $a\ge0$ and $q>0$, so that we can write the functional $J$ as
\eeq{ \label{J_rewrite}
J(f) = \int_0^1 \Phi\big(f''(x),\fp(x,f(x))\big)\ \dd x.
}
\begin{claim}[Strong concavity] \label{claim_monomial}
For every $\theta\in(0,1)$, $a_0,a_1\ge0$, and $q_0,q_1>0$, we have
\eeq{ \label{Phi_concave}
\Phi\big((1-\theta)a_0+\theta a_1,(1-\theta)q_0+\theta q_1\big)
\ge (1-\theta)\Phi(a_0,q_0)+\theta\Phi(a_1,q_1).
}
Furthermore, equality implies either $(a_0,q_0)=(a_1,q_1)$ or $a_0=a_1=0$.
In particular, equality always implies $a_0=a_1$.
\end{claim}

\begin{proofclaim}
Using the notation $\lambda_0\coloneqq 1-\theta$ and $\lambda_1 \coloneqq \theta$, we have
\eq{
\text{RHS of \eqref{Phi_concave}}
= \sum_{i=0}^1 \lambda_i(a_iq_i)^{1/3}
&\stackrel{\hphantom{\mbox{\footnotesize(H\"older)}}}{=} \sum_{i=0}^1 (\lambda_i a_i)^{1/3}(\lambda_i q_i)^{1/3}\lambda_i^{1/3} \\
&\stackrel{\mbox{\footnotesize(H\"older)}}{\le} \Big(\sum_{i=0}^1\lambda_i a_i\Big)^{1/3}
\Big(\sum_{i=0}^1\lambda_i q_i\Big)^{1/3}\Big(\sum_{i=0}^1\lambda_i\Big)^{1/3} \\
&\stackrel{\hphantom{\mbox{\footnotesize(H\"older)}}}{=} \text{LHS of \eqref{Phi_concave}}.
}
Furthermore, the step using H\"older is an equality if and only if there exist constants $\beta,\beta',\beta''$, not all zero, such that
\eeq{ \label{h2dcds}
\beta\lambda_ia_i = \beta'\lambda_iq_i = \beta''\lambda_i \quad \text{for each $i\in\{1,2\}$.}
}
If both $a_1$ and $a_2$ are strictly positive, then every quantity in \eqref{h2dcds} must be strictly positive, and we can solve for $(a_i,q_i) = (\beta''/\beta,\beta''/\beta')$.
If $a_1 = 0$, then \eqref{h2dcds} implies $\beta' = \beta'' = 0$ (since $\lambda_1q_1 > 0$), which forces $\beta \neq 0$, which in turn forces $a_2 = 0$ (since $\beta\lambda_2\ne 0$ while $\beta\lambda_2a_2=\beta'=0$).
By analogous reasoning, if $a_2 = 0$, then $a_1 = 0$.
\end{proofclaim}

Fix $\theta\in(0,1)$ and consider any $f_0,f_1\in\cF$.
Let $f_\theta \coloneqq (1-\theta)f_0 + \theta f_1$, which also belongs to $\cF$.
Define
\[
q_i(x)\coloneqq \fp(x,f_i(x)) \qquad \text{and} \qquad
q_\theta(x)\coloneqq(1-\theta)q_0(x)+\theta q_1(x).
\]
Since $y\mapsto \fp(x,y)$ is concave, we have
\eeq{ \label{ndvc09}
\fp(x,f_\theta(x))\ge q_\theta(x) \quad \text{for every $x\in[0,1]$.}
}
Since $q\mapsto\Phi(a,q)=(aq)^{1/3}$ is nondecreasing, it follows from \eqref{ndvc09} that
\eeq{ \label{pre_J_concave}
\Phi\big(f_\theta''(x),\fp(x,f_\theta(x))\big)
&\stackrefp{Phi_concave}{\ge} \Phi\big(f_\theta''(x),q_\theta(x)\big) \\
&\stackref{Phi_concave}{\ge} (1-\theta)\Phi\big(f_0''(x),q_0(x)\big)
+\theta\Phi\big(f_1''(x),q_1(x)\big)
}
for every $x\in(0,1)$ at which $f_\theta'',f_0,f_1$ are defined.
Integrating over $x$ and recalling \eqref{J_rewrite} yields
\eeq{ \label{J_concave}
J(f_\theta)\ge (1-\theta)J(f_0)+\theta J(f_1),
}
so $J$ is concave.

Now suppose $f_0,f_1\in\argmax J$.
Then \eqref{J_concave} is an equality, which implies \eqref{pre_J_concave} is an equality for almost every $x\in(0,1)$.  
By the final sentence of Claim~\ref{claim_monomial}, we conclude that
\eeq{ \label{wjb6cn}
f_0''(x)=f_1''(x)
\quad\text{for almost every $x\in(0,1)$}.
}
Now consider any $t\in(0,1)$.
By Proposition~\ref{prop_maximizer_properties}\ref{prop_maximizer_abs_cont}, the function
\[
H(x) \coloneqq \partial^+f_1(x)-\partial^+f_0(x)
\]
is absolutely continuous on $[0,t]$, and by \eqref{wjb6cn} its derivative satisfies
\[
H'(x)=0 \quad \text{for almost every $x\in(0,1)$.}
\]
Therefore, $H$ is constant on $[0,t]$. 
Since $t\in(0,1)$ is arbitrary, there is a constant $c$ such that
\[
\partial^+f_1(x)-\partial^+f_0(x)=c
\quad\text{for all $x\in[0,1)$}.
\]
Since $f_0(0)=f_1(0)=0$, we have
\eeq{ \label{7gj9bp}
f_1(x)-f_0(x)
\stackref{1_deriv_eq}{=}
\int_0^x\big(\partial^+f_1(u)-\partial^+f_0(u)\big)\,\dd u
=
cx
\quad\text{for every $x\in[0,1)$}.
}
Finally, $f_0(1)=f_1(1)=1$ by Proposition~\ref{prop_maximizer_properties}\ref{prop_maximizer_1_at_1}, so
\[
0=f_1(1)-f_0(1)
=\lim_{x\nearrow1}[f_1(x)-f_0(x)]
\stackref{7gj9bp}{=}\lim_{x\nearrow1}cx=c.
\]
Now \eqref{7gj9bp} shows $f_0=f_1$, so $\argmax J$ is a singleton.
\end{proof}

\begin{proof}[Proof of Proposition~\ref{prop_multiple_optimizers}]
\begin{figure}[t]

\tikzset{every picture/.style={line width=0.75pt}} %set default line width to 0.75pt        

\begin{tikzpicture}[x=0.75pt,y=0.75pt,yscale=-1,xscale=1]
%uncomment if require: \path (0,262); %set diagram left start at 0, and has height of 262

%Shape: Right Triangle [id:dp5049256193891675] 
\draw   (400,20) -- (201,220) -- (400,220) -- cycle ;
%Shape: Square [id:dp624527214260693] 
\draw  [fill={rgb, 255:red, 184; green, 181; blue, 181 }  ,fill opacity=0.3 ][dash pattern={on 0.75pt off 0.75pt}] (300.5,120) -- (320,120) -- (320,139.5) -- (300.5,139.5) -- cycle ;
%Shape: Square [id:dp315743498444888] 
\draw  [fill={rgb, 255:red, 184; green, 181; blue, 181 }  ,fill opacity=0.3 ][dash pattern={on 0.75pt off 0.75pt}] (320.5,140) -- (340,140) -- (340,159.5) -- (320.5,159.5) -- cycle ;
%Shape: Square [id:dp7727112392373703] 
\draw  [fill={rgb, 255:red, 184; green, 181; blue, 181 }  ,fill opacity=0.3 ][dash pattern={on 0.75pt off 0.75pt}] (340.5,160) -- (360,160) -- (360,179.5) -- (340.5,179.5) -- cycle ;
%Shape: Square [id:dp31712855112723237] 
\draw  [fill={rgb, 255:red, 184; green, 181; blue, 181 }  ,fill opacity=0.3 ][dash pattern={on 0.75pt off 0.75pt}] (360.5,180) -- (380,180) -- (380,199.5) -- (360.5,199.5) -- cycle ;
%Shape: Square [id:dp14626083247795085] 
\draw  [fill={rgb, 255:red, 184; green, 181; blue, 181 }  ,fill opacity=0.3 ][dash pattern={on 0.75pt off 0.75pt}] (380.5,200.5) -- (400,200.5) -- (400,220) -- (380.5,220) -- cycle ;
%Shape: Circle [id:dp7638594272563132] 
\draw  [fill={rgb, 255:red, 0; green, 0; blue, 0 }  ,fill opacity=1 ] (297,120) .. controls (297,118.07) and (298.57,116.5) .. (300.5,116.5) .. controls (302.43,116.5) and (304,118.07) .. (304,120) .. controls (304,121.93) and (302.43,123.5) .. (300.5,123.5) .. controls (298.57,123.5) and (297,121.93) .. (297,120) -- cycle ;
%Shape: Circle [id:dp8260125430355363] 
\draw  [fill={rgb, 255:red, 0; green, 0; blue, 0 }  ,fill opacity=1 ] (396.5,220) .. controls (396.5,218.07) and (398.07,216.5) .. (400,216.5) .. controls (401.93,216.5) and (403.5,218.07) .. (403.5,220) .. controls (403.5,221.93) and (401.93,223.5) .. (400,223.5) .. controls (398.07,223.5) and (396.5,221.93) .. (396.5,220) -- cycle ;
%Straight Lines [id:da723028871647449] 
\draw [color={rgb, 255:red, 155; green, 155; blue, 155 }  ,draw opacity=1 ][fill={rgb, 255:red, 155; green, 155; blue, 155 }  ,fill opacity=1 ]   (350,130) -- (323,129.77) ;
\draw [shift={(320,129.75)}, rotate = 0.48] [fill={rgb, 255:red, 155; green, 155; blue, 155 }  ,fill opacity=1 ][line width=0.08]  [draw opacity=0] (8.93,-4.29) -- (0,0) -- (8.93,4.29) -- cycle    ;
%Straight Lines [id:da6386885880712486] 
\draw [color={rgb, 255:red, 155; green, 155; blue, 155 }  ,draw opacity=1 ][fill={rgb, 255:red, 155; green, 155; blue, 155 }  ,fill opacity=1 ]   (370,151) -- (343,150.77) ;
\draw [shift={(340,150.75)}, rotate = 0.48] [fill={rgb, 255:red, 155; green, 155; blue, 155 }  ,fill opacity=1 ][line width=0.08]  [draw opacity=0] (8.93,-4.29) -- (0,0) -- (8.93,4.29) -- cycle    ;
%Straight Lines [id:da44028036724307484] 
\draw [color={rgb, 255:red, 155; green, 155; blue, 155 }  ,draw opacity=1 ][fill={rgb, 255:red, 155; green, 155; blue, 155 }  ,fill opacity=1 ]   (430,212) -- (403,211.77) ;
\draw [shift={(400,211.75)}, rotate = 0.48] [fill={rgb, 255:red, 155; green, 155; blue, 155 }  ,fill opacity=1 ][line width=0.08]  [draw opacity=0] (8.93,-4.29) -- (0,0) -- (8.93,4.29) -- cycle    ;

% Text Node
\draw (260,90) node [anchor=north west][inner sep=0.75pt]    {$\left(\tfrac{1}{2} ,\tfrac{1}{2}\right)$};
% Text Node
\draw (228,156.4) node [anchor=north west][inner sep=0.75pt]    {$\mathcal{T}$};
% Text Node
\draw (380.5,229.4) node [anchor=north west][inner sep=0.75pt]    {$( 1,0)$};
% Text Node
\draw (353,120.4) node [anchor=north west][inner sep=0.75pt]  [color={rgb, 255:red, 155; green, 155; blue, 155 }  ,opacity=1 ]  {$\mathcal{Q}_{1}$};
% Text Node
\draw (373,141.4) node [anchor=north west][inner sep=0.75pt]  [color={rgb, 255:red, 155; green, 155; blue, 155 }  ,opacity=1 ]  {$\mathcal{Q}_{2}$};
% Text Node
\draw (433,202.4) node [anchor=north west][inner sep=0.75pt]  [color={rgb, 255:red, 155; green, 155; blue, 155 }  ,opacity=1 ]  {$\mathcal{Q}_{k}$};
% Text Node
%\draw (401,160) node [anchor=north west][inner sep=0.75pt]  [color={rgb, 255:red, 155; green, 155; blue, 155 }  ,opacity=1 ]  {$\ddots $};

\end{tikzpicture}
\caption{A possible choice of the open sets $\cQ_1,\ldots,\cQ_k$, in the case $k=5$.
Each $\cQ_i$ supports a smooth function $\vphi_i$ that is used to place additional mass inside $\cQ_i$, as in \eqref{unnorm_q_def}.  By appropriately tuning the mass sizes via parameters $\lambda_1,\ldots,\lambda_k$, as in \eqref{eq_choose_lambdas}, we find a density function \eqref{normalized_p_def} whose $J$-functional has at least $k$ maximizers, one passing through each $\cQ_i$.}
\label{fig_squares}
\end{figure}
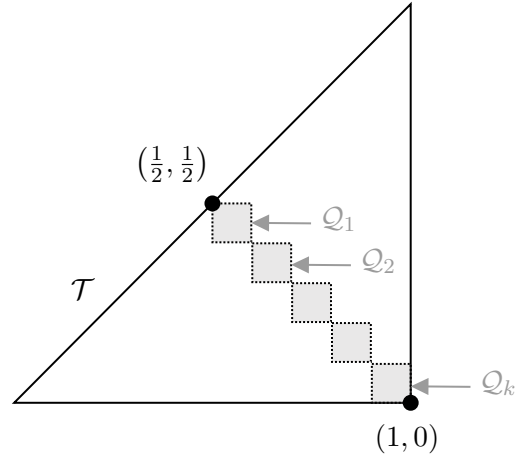
Let $\cQ_1,\ldots,\cQ_k\subset\cT$ be nonempty disjoint open sets such that the following implication holds for all $i<j$ in $\{1,\ldots,k\}$: 
\eeq{ \label{3ibnl1}
(x_1,y_1)\in \cQ_i \text{ and } (x_2,y_2)\in \cQ_j
\quad \implies \quad
x_1<x_2 \text{ and } y_1>y_2.
}
For example, one can take the following open squares (depicted in Figure~\ref{fig_squares}):
\eq{
\cQ_i=\Big(\frac{1}{2}+\frac{i-1}{2k},\frac{1}{2}+\frac{i}{2k}\Big)\times\Big(\frac{1}{2}-\frac{i}{2k},\frac{1}{2}-\frac{i-1}{2k}\Big),
\quad i\in\{1,\ldots,k\}.
}
For each $i$, choose a smooth function $\varphi_i\colon\cT\to\R$ that satisfies $0\le \varphi_i\le 1$, is not identically zero, and whose support is contained in $\cQ_i$.
For $f\in\cF$, define
\[
K_0(f) \coloneqq \int_0^1 f''(x)^{1/3}\ \dd x
\quad \text{and} \quad
K_i(f) \coloneqq \int_0^1 f''(x)^{1/3}\cdot \varphi_i(x,f(x))\ \dd x,
\quad i\in\{1,\ldots,k\}.
\]
Since every $f\in\cF$ is nondecreasing, \eqref{3ibnl1} implies that the graph of $f$ intersects at most one of $\cQ_1,\ldots,\cQ_k$.
Consequently,
\begin{equation}\label{eq_K_incompat}
\text{$K_i(f)>0$ and $i\in\{1,\ldots,k\}$} \quad \Longrightarrow \quad K_j(f)=0\text{ for all $j\in\{1,\ldots,k\}\setminus\{i\}$}.
\end{equation}
%Set $B \coloneqq \sup_{f\in\cF}K_0(f) < \infty.$ 
Next define, for $i\in\{1,\ldots,k\}$ and $\lambda\ge0$, the following optimal value:
\eeq{ \label{Wi_def}
W_i(\lambda)  \coloneqq
\sup_{f\in\cF} \big( K_0(f)+\lambda K_i(f)\big) = \int_0^1\big[f''(x)\cdot \fq_{i,\lambda}(x,f(x))\big]^{1/3} \,\dd x,
}
where $\fq_{i,\lambda}(x,y)\coloneqq(1+\lambda \varphi_i(x,y))^3$. 
Since $\fq_{i,\lambda}$ is continuous, Proposition~\ref{prop_maximizers} guarantees (after we normalize $\fq_{i,\lambda}$ to be a probability density) that the supremum is attained.

We make three observations about the map $\lambda\mapsto W_i(\lambda)$.
First, we know its value at $\lambda=0$:
\eq{
W_i(0) = \sup_{f\in\cF} K_0(f) \stackref{eq_unif_case_a}{=} 4^{1/3}.
}
Second, the map $\lambda\mapsto W_i(\lambda)$ is continuous.
To see this, observe that $0\le K_i(f)\le K_0(f)\le 4^{1/3}$ for every $f\in\cF$, so we have
\[
|W_i(\lambda)-W_i(\lambda')|\le 4^{1/3}|\lambda-\lambda'| \quad \text{for every $\lambda,\lambda'\ge0$}.
\]
Third and finally, we claim $W_i(\lambda)\to\infty$ as $\lambda\to\infty$. 
To verify this, choose a point $(x_0,y_0)\in \cQ_i$ such that $\varphi_i(x_0,y_0)>0$. Since $0<y_0<x_0<1$, the quadratic function
\[
f(x)\coloneqq cx^2+\Big(\frac{y_0}{x_0}-cx_0\Big)x
\]
belongs to $\cF$ for all sufficiently small $c>0$, and satisfies $f(x_0)=y_0$. Since $f''(x)=2c>0$, it follows from continuity of $\vphi_i$ that $K_i(f)>0$. Therefore, as $\lambda \to \infty$ we have
\[
W_i(\lambda)\ge K_0(f)+\lambda K_i(f)\to\infty.
\]

Now choose any number $L>4^{1/3}$. 
Because of the three observations from the previous paragraph, there exist $\lambda_1,\ldots,\lambda_k>0$ such that
\begin{equation}\label{eq_choose_lambdas}
W_i(\lambda_i) = L \quad \text{for every $i\in\{1,\ldots,k\}$}.
\end{equation}
Consider the following (unnormalized) smooth strictly positive function:
\eeq{ \label{unnorm_q_def}
\fq(x,y) \coloneqq
\Big(1+\sum_{i=1}^k\lambda_i \varphi_i(x,y)\Big)^3, \quad (x, y) \in \cT.
}
Denote the corresponding functional by
\[
J_{\fq}(f)
\coloneqq \int_0^1 \bigl[f''(x)\cdot \fq(x,f(x))\bigr]^{1/3}\,\dd x
= K_0(f)+\sum_{i=1}^k\lambda_iK_i(f), \quad f\in\cF.
\]
Using \eqref{eq_K_incompat}, we get the following global upper bound: 
\eeq{ \label{eq_global_upper_bound}
J_{\fq}(f) \le K_0(f) + \max_{i\in\{1,\ldots,k\}}\lambda_i K_i(f)
\stackref{Wi_def}{\le} \max_{i\in\{1,\ldots,k\}} W_i(\lambda_i)
\stackref{eq_choose_lambdas}{=} L
\quad \text{for every $f\in\cF$.}
}
For each $i$, let $f_i\in\cF$ be a maximizer for $W_i(\lambda_i)$ from \eqref{Wi_def}.
In light of \eqref{eq_choose_lambdas}, this means
\eq{ 
K_0(f_i) + \lambda_i K_i(f_i) = L.
}
Since $L>4^{1/3}\ge K_0(f_i)$ we must have $K_i(f_i)>0$.
Thus \eqref{eq_K_incompat} forces $K_j(f_i)=0$ for all $j\in\{1,\ldots,k\}\setminus i$.
Consequently,
\[
J_{\fq}(f_i)
=
K_0(f_i)+\lambda_iK_i(f_i)
=
L.
\]
Together with \eqref{eq_global_upper_bound}, this shows that $f_1,\ldots,f_k$ are maximizers of $J_{\fq}$. They are distinct, since $K_i(f_i)>0$ but $K_j(f_i)=0$ for $j\in\{1,\ldots,k\}\setminus\{i\}$. 
To complete the proof, we normalize $\fq$ by setting
\eeq{ \label{normalized_p_def}
\fp(x,y) \coloneqq \frac{\fq(x,y)}{\int_{\cT}\fq(x,y)\,\dd x\,\dd y}.
}
Then $\fp$ is smooth and strictly positive on $\cT$. 
Since the normalization does not change the set of maximizers, $f_1,\ldots,f_k$ remain distinct maximizers of $J_{\fp}$.
\end{proof}

\begin{proof}[Proof of Proposition~\ref{prop_equicontinuity}]
\noindent Part~\ref{prop_equicontinuity_a}: Let $\eps>0$ be arbitrary.
Suppose toward a contradiction that for every $\eta>0$, there exists $f\in\argmax J$ such that $f(1-\eta) < 1-\eps$.
Choose any sequence $\eta_n\searrow0$, and take $f_n\in\argmax J$ such that $f_n(1-\eta_n) < 1-\eps$.
By compactness (Lemma~\ref{lem_compact}), we can pass to a subsequence and assume $(f_n)_{n\ge1}$ converges to some $f\in\cF$ with respect to the metric $d$.
By Lemma~\ref{lem_convergence}, this means $f_n(x)\to f(x)$ for every $x\in[0,1)$.
On one hand, continuity of $f$ implies the existence of $\eta>0$ such that $f(1) \le f(1-\eta)+\eps/2$, from which it follows that
\eeq{ \label{3yoh9c}
f(1) \le \lim_{n\to\infty} f_n(1-\eta) + \frac{\eps}{2} \le \limsup_{n\to\infty} f_n(1-\eta_n) + \frac{\eps}{2} \le 1 - \frac{\eps}{2}.
}
On the other hand, closedness of $\argmax J$ (from Proposition~\ref{prop_maximizers}) implies $f\in\argmax J$, hence $f(1) = 1$ by Proposition~\ref{prop_maximizer_properties}\ref{prop_maximizer_1_at_1}.
This contradicts \eqref{3yoh9c}.

\medskip

\noindent Part~\ref{prop_equicontinuity_b}: We prove \eqref{nka3c} with the $\eta$ from part~\ref{prop_equicontinuity_a}.
Consider any $f\in\argmax J$ and any $x,y\in[0,1]$ such that $0<y-x\le\eta$.
If $x\ge1-\eta$, then we have the desired result:
\eq{
f(y) - f(x) \le f(1) - f(1-\eta) \stackrel{\mbox{\footnotesize(part~\ref{prop_equicontinuity_a})}}{\le} \eps.
}
If instead $x\le 1-\eta$, then convexity gives
\eq{
\frac{f(y)-f(x)}{y-x} \le \frac{f(1)-f(1-\eta)}{1-(1-\eta)} \stackrel{\mbox{\footnotesize(part~\ref{prop_equicontinuity_a})}}{\le} \frac{\eps}{\eta}.
}
Multiplying by $y-x$ yields the desired result:
\[
f(y) - f(x) \le \eps\cdot\frac{y-x}{\eta}\le \eps. \qedhere
\]
\end{proof}

\begin{proof}[Proof of Proposition~\ref{prop_far_from_optimizer}]
We first claim a weaker result: For every $\delta>0$, there exists $\theta>0$ such that following implication holds for all $f\in\cF$:
\eeq{ \label{29gicb}
\inf_{h\in\argmax J} d(f,h) \ge \delta \quad \implies \quad J(f) \le J_\star - \theta.
}
Indeed, this follows from compactness (Lemma~\ref{lem_compact}) and upper semicontinuity of $J$ (Proposition~\ref{prop_usc}), similar to the proof of Proposition~\ref{prop_maximizers}.
The remainder of the argument is to show that \eqref{29gicb} is sufficient.

Let $\eps>0$ be given.
By Proposition~\ref{prop_equicontinuity}\ref{prop_equicontinuity_a}, there exists $\eta>0$ such that $h(1-\eta)\ge1-\eps/2$ for every $h\in\argmax J$.
By definition of the metric $d$ in \eqref{metric_def}, there exists $\delta>0$ such that the following implication holds for all $f,g\in\cF$:
\eeq{ \label{2bk3xss}
\|f-g\|_{[0,1-\eta]} \ge \eps/2 \quad \implies \quad d(f,g)\ge\delta.
}
By the previous paragraph, there exists $\theta>0$ such that \eqref{29gicb} is true.
Combining \eqref{2bk3xss} and \eqref{29gicb} yields the following implication:
\eeq{ \label{tuc7n}
\inf_{h\in\argmax J}\|f-h\|_{[0,1-\eta]} \ge \eps/2 \quad \implies \quad J(f) \le J_\star - \theta.
}
Now suppose the hypothesis of \eqref{far_implication} holds, meaning that for every $h\in\argmax J$, there is $x\in[0,1]$ such that $|f(x)-h(x)|\ge\eps$.
We consider two cases:
\begin{itemize}
\item If $x\le1-\eta$, then we trivially have $\|f-h\|_{[0,1-\eta]}\ge\eps$.
\item If $x\ge 1-\eta$, then $h(x) \ge h(1-\eta) \ge 1-\eps/2$, which forces $f(x)\le1-\eps$. Hence $f(1-\eta)\le1-\eps$, so $|f(1-\eta)-h(1-\eta)|\ge\eps/2$.
\end{itemize}
In either case, $\|f-h\|_{[0,1-\eta]}\ge\eps/2$.
Therefore, \eqref{tuc7n} yields the desired conclusion.
\end{proof}

The proof of Proposition~\ref{prop_maximizer_properties}, specifically part~\ref{prop_maximizer_abs_cont}, requires the following measure-theoretic lemma.

\begin{lemma}[Middle-third localization of singular mass]\label{lem:middle_third_singular_mass}
Fix $z>0$.
Let $\nu_1$ and $\nu_2$ be positive Borel measures on $(0,z)$ such that
$\nu_1\perp\nu_2$ and $0<\nu_2((0,z))<\infty$.
For every $\beta>3$ and $\eps>0$, there exists a compact interval $I=[a,b]\subset (0,z)$ whose middle third $K=\bigl[a+\tfrac13(b-a),\, b-\tfrac13(b-a)\bigr]$ satisfies
\[
\nu_2(I) \le \beta \nu_2(K), \qquad
\nu_1(I)\le \eps\nu_2(K), \qquad \text{and} \qquad
\nu_2(K)>0.
\]
\end{lemma}

\begin{proof}
Since $\nu_1 \perp \nu_2$, the following limit holds for $\nu_2$-almost every $x\in(0,z)$:
\eeq{ \label{eq:lebes_diff}
 \lim_{r\searrow 0} \frac{\nu_1[x-r,x+r]}{\nu_2[x-r,x+r]} = 0.   
}
Since $\nu_2$ is a finite Borel measure, it is Radon. 
Therefore, by \cite[Lemma~1]{tolsa11} with dimension $d=1$ and scale factor $\alpha=3$, the following statement holds for $\nu_2$-almost every $x\in(0,z)$.
There exists a sequence of closed intervals $[x - r_n, x+ r_n]$ with $\lim_{n\to\infty}r_n= 0$, that are $(3, \beta)$-doubling, meaning 
\begin{equation}\label{eq:doubling}
\nu_2[x - 3r_n, x+ 3r_n]\le \beta \nu_2[x - r_n, x+ r_n].
\end{equation}
Now choose $x\in (0,z)\cap \operatorname{supp}(\nu_2)$ such that both \eqref{eq:lebes_diff} and \eqref{eq:doubling} hold.
Since $r_n\to 0$, we may assume $n$ is sufficiently large that $[x - 3r_n, x+ 3r_n] \subset (0,z)$, and
\begin{equation}\label{eq:lebes_diff2}
    \nu_1[x - 3r_n, x+ 3r_n]\le \frac{\varepsilon}{\beta} \nu_2[x - 3r_n, x+ 3r_n].
\end{equation}
Set $I=[x - 3r_n, x+ 3r_n]$ and $K = [x - r_n, x+ r_n]$ so that \eqref{eq:doubling} says $\nu_2(I)\le\beta\nu_2(K)$, while \eqref{eq:lebes_diff2} says $\nu_1(I) \le (\eps/\beta)\nu_2(I)$.
Combining the two inequalities yields $\nu_1(I)\le\eps\nu_2(K)$.
Finally, since $x\in \operatorname{supp}\nu_2$ and $K$ is centered at $x$, we have $\nu_2(K)>0$.
\end{proof}

\begin{proof}[Proof of Proposition~\ref{prop_maximizer_properties}]
Throughout the proof, let $\mu\in\cM$ be such that $f_\mu = f$, as in Lemma~\ref{lem_identification_scheme}.

\medskip

\noindent Part~\ref{prop_maximizer_1_at_1}:
Suppose $f(1)<1$.
Our goal is to find some $g\in\cF$ such that $J(f) < J(g)$, so that $f$ is not a maximizer.
Let $\eta \coloneqq 1-f(1) > 0$.
Given $\eps\in(0,1)$, we add to $\mu$ an appropriately scaled uniform measure on $[1-\eps,1)$, thereby obtaining a perturbed measure $\mu_\eps$ given by
\eq{
\mu_\eps(B) \coloneqq \mu(B) + \frac{2\eta}{\eps^2}\mathrm{Leb}(B\cap[1-\eps,1)) \quad \text{for Borel $B\subseteq[0,1)$.}
}
Note that
\eq{
\int_0^1\mu_\eps[0,t]\ \dd t 
= \int_0^1\mu[0,t]\ \dd t
+ \frac{2\eta}{\eps^2}\int_{1-\eps}^1 \big(t-(1-\eps)\big)\ \dd t
\stackref{f_from_meas}{=} f(1) + \eta = 1,
}
so $\mu_\eps$ belongs to $\cM$.
Since the perturbation only affects the interval $[1-\eps,1]$, we have
\eq{
J(f_{\mu_\eps}) - J(f)
&\stackref{eq_unif_case_a}{=} \int_{1-\eps}^1 \big[D_{\mu_\eps}(x)\cdot \fp(x,f_{\mu_\eps}(x))\big]^{1/3}\ \dd x
- \int_{1-\eps}^1 \big[f''(x)\cdot \fp(x,f(x))\big]^{1/3}\ \dd x \\
&\stackrefpp{p_lower}{eq_unif_case_a}{\ge} (2\fc)^{1/3}\int_{1-\eps}^1 \Big(\frac{2\eta}{\eps^2}\Big)^{1/3}\ \dd x
- (2\fC)^{1/3}\int_{1-\eps}^1 f''(x)^{1/3}\ \dd x \\
&\stackref{eq_unif_case_a}{\ge} (2\fc)^{1/3}\int_{1-\eps}^1 \Big(\frac{2\eta}{\eps^2}\Big)^{1/3}\ \dd x
- (8\fC\eps)^{1/3}\big(f(1)-f(1-\eps)-\eps\partial^+f(1-\eps)\big)^{1/3}.
}
The integral on the final line is equal to $(2\eta\eps)^{1/3}$, while the second term is $o(\eps^{1/3})$ as $\eps\searrow0$ because $f$ is continuous at $1$.
Therefore, $J(f_{\mu_\eps}) - J(f) > 0$ for all sufficiently small $\eps$.

\medskip

\noindent Part~\ref{prop_maximizer_no_linear}:
Assume $f\in\argmax J$.
Suppose toward a contradiction that $\partial^+ f$ is constant on $[a,b)$ for some $a<b$, or equivalently
\eeq{ \label{line_seg}
f(b) - f(a) = (b-a)\partial^+f(a) = (b-a)\partial^-f(b).
}
We assume the interval $[a,b)$ is maximal in the sense that for every $\eps>0$, $\partial^+ f$ is not constant on $[a-\eps,b)$ nor on $[a,b+\eps)$.
We cannot have $[a,b) = [0,1)$ since the function $f(x)=x$ is not a maximizer of $J$.
Therefore, we must have $a>0$ or $b<1$.

First consider the case $a>0$.
In what follows, we always assume $a-\eps \ge 0$.
On one hand, we have
\eeq{ \label{kbix0}
\int_{a-\eps}^{b} \big[f''(x)\cdot \fp(x,f(x))\big]^{1/3}\ \dd x
&\stackrefp{eq_unif_case_b}{=} \int_{a-\eps}^{a} \big[f''(x)\cdot \fp(x,f(x))\big]^{1/3}\ \dd x \\
&\stackrefpp{p_upper}{eq_unif_case_b}{\le} (2\fC)^{1/3}\int_{a-\eps}^{a} f''(x)^{1/3}\ \dd x \\
&\stackref{eq_unif_case_b}{\le} (8\fC\eps )^{1/3}\big[\eps\partial^-f(a)-(f(a)-f(a-\eps))\big]^{1/3}.
}
On the other hand, consider the function $g_\eps\colon[a-\eps,b]\to[f(a-\eps),f(b)]$ that achieves the supremum in \eqref{unif_case_c_sup}, subject to the boundary data
\eq{
g_\eps(a-\eps) = f(a-\eps), \quad
g_\eps(b) = f(b), \quad
\partial^+g_\eps(a-\eps) = \partial^+f(a-\eps), \quad
\partial^-g_\eps(b) = \partial^-f(b).
}
The value of the supremum being achieved is
\eeq{ \label{dcu309}
\int_{a-\eps}^b g_\eps''(x)\ \dd x
&= \Big(\frac{4\Delta_1\Delta_2}{\partial^-f(a)-\partial^+f(a-\eps)}\Big)^{1/3},
}
where
\eq{
\Delta_1 
&\stackrefpp{eq_unif_case_a}{line_seg}{=} f(b) - f(a-\eps) - (b-a+\eps)\partial^+f(a-\eps) \\
&\stackref{line_seg}{=} (b-a)\big(\partial^+f(a)-\partial^+f(a-\eps)\big) + f(a) - f(a-\eps) - \eps\partial^+f(a-\eps) \\
&\stackrefp{line_seg}{\ge} (b-a)\big(\partial^-f(a)-\partial^+f(a-\eps)\big), \\
\text{and} \qquad \Delta_2 
&\stackrefpp{eq_unif_case_b}{line_seg}{=} (b-a+\eps)\partial^-f(b)-(f(b)-f(a-\eps)) \\
&\stackref{line_seg}{=} f(b)-f(a) + \eps\partial^+f(a) - (f(b)-f(a-\eps)) \\
&\stackrefp{line_seg}{\ge} \eps\partial^-f(a) - (f(a)-f(a-\eps)).
}
Using these inequalities in \eqref{dcu309} and also invoking \eqref{p_lower}, we find
\eeq{ \label{kbix4}
\int_{a-\eps}^b \big[g_\eps''(x)\cdot \fp(x,g_\eps(x))\big]^{1/3}\ \dd x 
\ge \big(8\fc(b-a)\big)^{1/3}\big[\eps\partial^-f(a) - (f(a)-f(a-\eps))\big]^{1/3}.
}
Finally, define $f_\eps\colon[0,1]\to[0,1]$ by replacing the portion of $f$ on $[a-\eps,b]$ with $g_\eps$, as follows:
\eq{
f_\eps(x) \coloneqq \begin{cases}
g_\eps(x) &\text{if $x\in[a-\eps,b]$} \\
f(x) &\text{otherwise.}
\end{cases}
}
Since $g_\eps$ and $f$ have the same boundary data at $a-\eps$ and $b$, this new function $f_\eps$ is convex and continuous, so $f_\eps\in\cF$.
Comparing \eqref{kbix0} and \eqref{kbix4}, we see that $J(f) < J(f_\eps)$ for all sufficiently small $\eps$.
This contradicts the hypothesis $f\in\argmax J$.

The case $b<1$ is analogous, with the replacement taking place on the interval $[a,b+\eps]$ instead of $[a-\eps,b]$.

\medskip

\noindent Part~\ref{prop_maximizer_abs_cont}:
Let $f\in\argmax J$.
We wish to show that $x\mapsto\partial^+f(x)$ is absolutely continuous on every compact subinterval of $[0,1)$.
By Lemma~\ref{lem_identification_scheme}\ref{scheme_right_deriv}, it suffices to show $\mu\big|_{(0,1)}$ is absolutely continuous with respect to Lebesgue measure.
We write the usual decomposition $\mu=\mu^\ac+\mu^\sing$, and suppose toward a contradiction that $\mu^\sing(0,1)>0$.
Choose $z<1$ sufficiently close to $1$ that $\mu^\sing(0,z) > 0$.
Note that $\mu^\sing(0,z) \le \mu(0,z] \le (1-z)^{-1}<\infty$ by Lemma~\ref{lem_identification_scheme}\ref{scheme_right_deriv}.

Fix $\beta>3$, and choose $\eps>0$ so small that
\begin{equation} \label{eq:eps_choice_ac}
(2\fC\eps)^{1/3}
<
\Big(\frac{8\fc}{9(\beta+\eps)}\Big)^{1/3}.
\end{equation}
By Lemma~\ref{lem:middle_third_singular_mass} with $\nu_1=\mu^\ac\big|_{(0,z)}$ and $\nu_2=\mu^\sing\big|_{(0,z)}$,
there exists a compact interval $I=[a,b]\subset(0,z)$ whose middle third $K=\bigl[a+\tfrac13(b-a),\, b-\tfrac13(b-a)\bigr]$ satisfies
\begin{equation} \label{eq:alpha_small_ac}
\mu^\sing(I)\le \beta \mu^\sing(K), \qquad
\mu^\ac(I)\le \eps \mu^\sing(K), \qquad 
\text{and} \qquad
\mu^\sing(K) > 0.
\end{equation}
Denote the boundary data of $f$ on $I$ by
\[
f_a=f(a), \qquad f_b=f(b), \qquad s_a=\partial^+f(a), \qquad s_b=\partial^-f(b).
\]
Recall from Lemma~\ref{lem_identification_scheme} that $\partial^+f(x) = \mu[0,x]$ and $\partial^-f(x) = \mu[0,x)$ for every $x\in(0,1)$.
In particular,
\begin{equation} \label{eq:slope_gap_ub_ac}
s_b-s_a 
= \mu[0,b) - \mu[0,a]
= \mu(a,b)
\le \mu(I)
= \mu^\ac(I)+\mu^\sing(I)
\stackref{eq:alpha_small_ac}{\le} (\eps+\beta)\mu^\sing(K).
\end{equation}
The quantities $\Delta_1$ and $\Delta_2$ from Proposition~\ref{prop_unif_case} can be expressed as
\eq{
\Delta_1
&\stackrefpp{eq_unif_case_a}{1_deriv_eq}{=} f_b - f_a - s_a(b-a) \\
&\stackref{1_deriv_eq}{=} \int_a^b\bigl(\mu[0,t]-\mu[0,a]\bigr)\ \dd t \\
&\stackrefp{1_deriv_eq}{=} \int_a^b\int_{(a,b]}\one\{a<x\le t\}\ \mu(\dd x)\,\dd t
 = \int_{(a,b]} (b-x)\ \mu(\dd x), \\
\text{and} \qquad \Delta_2
&\stackrefpp{eq_unif_case_b}{1_deriv_eq}{=}  s_b(b-a) - (f_b - f_a) \\
&\stackref{1_deriv_eq}{=} \int_a^b\bigl(\mu[0,b)-\mu[0,t]\bigr)\ \dd t \\
&\stackrefp{1_deriv_eq}{=} \int_a^b\int_{[a,b)} \one\{t< x<b\}\ \mu(\dd x)\, \dd t
 = \int_{[a,b)} (x-a)\ \mu(\dd x).
}
Since every $x\in K$ satisfies
\[
b-x\ge \frac{b-a}{3}
\qquad \text{and} \qquad
x-a\ge \frac{b-a}{3},
\]
we deduce from these expressions that
\begin{equation} \label{eq:Delta_lb_ac}
\Delta_1 \ge \frac{b-a}{3}\mu^\sing(K)
\qquad \text{and} \qquad
\Delta_2 \ge \frac{b-a}{3}\mu^\sing(K).
\end{equation}
In particular, we have $\Delta_1>0$ and $\Delta_2>0$ thanks to the third inequality in \eqref{eq:alpha_small_ac}.
This observation implies
\eq{
s_a<\frac{f_b-f_a}{b-a}<s_b,
}
which allows us to invoke Proposition~\ref{prop_unif_case} below.

Let $g\colon [a,b]\to\R$ be the unique maximizer from
Proposition~\ref{prop_unif_case}\ref{prop_unif_case_c} with boundary data
\eq{
g(a)=f_a,\quad g(b)=f_b,\quad \partial^+g(a)= s_a,\quad \partial^-g(b)= s_b.
}
Define $\hat f\colon[0,1]\to[0,1]$ by
\[
\hat f(x)\coloneqq
\begin{cases}
g(x) & \text{if } x\in [a,b]\\
f(x) & \text{if } x\notin [a,b].
\end{cases}
\]
Since $g$ and $f$ have the same boundary data at $a$ and $b$, this new function $\hat f$ is convex and continuous, so $\hat f\in\cF$.
We claim $J(\hat f)>J(f)$, contradicting the maximality of $f$.
To prove this claim, it is enough to compare integrals over $[a,b]$, since $\hat f$ and $f$ agree outside this interval.

To this end, we observe that our choice of $g$ yields
\eeq{ \label{3xh4c}
\int_a^b g''(x)^{1/3}\ \dd x
\stackref{unif_case_c_sup}{=} \Big(\frac{4\Delta_1\Delta_2}{s_b-s_a}\Big)^{1/3} 
&\stackref{eq:slope_gap_ub_ac,eq:Delta_lb_ac}{\ge} \Big(\frac{4}{9(\eps+\beta)}\Big)^{1/3}(b-a)^{2/3}\mu^\sing(K)^{1/3}.
}
Using \eqref{p_lower} in addition to \eqref{3xh4c}, we obtain
\begin{equation} \label{eq:new_local_lb_ac}
\int_a^b \bigl[g''(x)\cdot\fp(x,g(x))\bigr]^{1/3}\ \dd x
\ge\Big(\frac{8\fc}{9(\eps+\beta)}\Big)^{1/3}(b-a)^{2/3}\mu^\sing(K)^{1/3}.
\end{equation}
On the other hand, $f''=D_{\mu} = D_{\mu^\ac}$ almost everywhere by Lemma~\ref{lem_identification_scheme}\ref{scheme_2nd_deriv}, so we have
\eeq{ \label{yx9nk}
\int_a^b f''(x)\ \dd x = \mu^\ac(I) \stackref{eq:alpha_small_ac}{\le} \eps\mu^\sing(K).
}
Now Hölder's inequality leads to
\begin{align}
\int_a^b f''(x)^{1/3}\ \dd x
&\le
(b-a)^{2/3}\Big(\int_a^b f''(x)\ \dd x\Big)^{1/3}
\stackref{yx9nk}{\le} \eps^{1/3}(b-a)^{2/3}\mu^\sing(K)^{1/3}. \label{eq:f_upper_ac}
\end{align}
Using \eqref{p_upper} in addition to \eqref{eq:f_upper_ac}, we obtain
\begin{equation} \label{eq:old_local_ub_ac}
\int_a^b \bigl[f''(x)\cdot\fp(x,f(x))\bigr]^{1/3}\ \dd x
\le
(2\fC\eps)^{1/3}(b-a)^{2/3}\mu^\sing(K)^{1/3}.
\end{equation}
By comparing \eqref{eq:new_local_lb_ac} and
\eqref{eq:old_local_ub_ac}, and recalling our choice of $\eps$ from \eqref{eq:eps_choice_ac}, we see that
\[
\int_a^b \bigl[g''(x)\cdot\fp(x,g(x))\bigr]^{1/3}\ \dd x
>
\int_a^b \bigl[f''(x)\cdot\fp(x,f(x))\bigr]^{1/3}\ \dd x.
\]
Hence $J(\hat f)>J(f)$ as claimed.
\end{proof}

\subsection{Approximation by smooth functions} \label{subsec_approximation}

In Proposition~\ref{prop_maximizer_properties}\ref{prop_maximizer_abs_cont}, we showed that every maximizer of $J$ has an absolutely continuous first derivative.
In general we do not know if maximizers have higher-order derivatives, but the following lemma---which will be used in Section~\ref{sec_lower_bound}---implies that maximizers can be approximated by smooth functions to arbitrary precision.
We say that a function $f\colon[0,1]\to\R$ is \textit{smooth} if it is the restriction of a smooth function defined on all of $\R$.

\begin{lemma}[Approximation by smooth functions] \label{lem_smooth_f}
Assume \eqref{p_cont}.
Then for every $f\in\cF$, there exists a sequence $(f_n)_{n\ge1}$ in $\cF$ with the following properties:
\begin{enumerate}[label=\textup{(\roman*)}]

\item \label{lem_smooth_f_i}
$d(f_n,f)\to0$ as $n\to\infty$.

\item \label{lem_smooth_f_ii}
For every $n$, $f_n$ is smooth with $f_n''(x)>0$ for all $x\in(0,1)$.

\item \label{lem_smooth_f_iii}
$J(f_n) \to J(f)$ as $n\to\infty$.

\end{enumerate}
\end{lemma}

\begin{proof}
Consider the unique measure $\mu\in\cM$ such that $f_\mu = f$, as guaranteed by Lemma~\ref{lem_identification_scheme}\ref{scheme_bijection}.
If $\mu$ is the zero measure, then $f \equiv 0$, and it suffices to take $f_n(x) = n^{-1}x^2$.
So henceforth we assume $\mu$ is not the zero measure.

We first restrict $\mu$ as follows.
For $\delta>0$, define $\mu_\delta$ by $\mu_\delta(B) \coloneqq \mu(B\cap[0,1-\delta]) \le \delta^{-1}$, where the inequality comes from Lemma~\ref{lem_identification_scheme}\ref{scheme_right_deriv}.
Now $\mu_\delta$ is a finite measure (as opposed to just $\sigma$-finite), and we assume $\delta$ is sufficiently small that $\mu_\delta$ is not the zero measure.
We next smoothen $\mu_\delta$ as follows.
Let $\psi\colon\R\to[0,\infty)$ be a smooth probability density function supported on $[0,1]$ and everywhere positive on $(0,1)$, such as
\eq{
\psi(x) = \begin{cases} 
\frac{1}{Z}\exp\big(\frac{-1}{x(1-x)}\big) &\text{if $x\in(0,1)$} \\
0 &\text{if $x\in\R\setminus(0,1)$},
\end{cases}
}
where $Z$ is chosen such that $\int_0^1\psi(x)\,\dd x = 1$.
For $\eps>0$, let $\kappa_\eps$ denote the probability measure whose density function is $\eps^{-1}\psi(x/\eps)$.
Since $\kappa_\eps$ has a smooth density and is supported on $[0,\eps]$, convolution with $\kappa_\eps$ results in the following properties:
\begin{subequations} \label{d6b43}
\begin{gather}
\text{$\mu_\delta * \kappa_\eps$ has a smooth density function;} \label{wik4bx} \\
\text{$(\mu_\delta * \kappa_\eps)[0,t] \nearrow \mu_\delta[0,t]$ as $\eps\searrow0$, for every continuity point $t\in[0,1)$ of $\mu_\delta$.} \label{2bniqe}
\end{gather}
We assume $\eps<\delta$ so that $\mu_\delta*\kappa_\eps$ is supported on $[0,1-\delta+\eps]\subset[0,1)$.
Since $\mu_\delta$ is not the zero measure, we have $\mu_\delta\neq\mu_\delta*\kappa_\eps$, so the following difference is strictly positive:
\eq{
h_{\delta,\eps} \coloneqq \int_0^1 \mu_\delta[0,t]\ \dd t - \int_0^1 (\mu_\delta * \kappa_\eps)[0,t]\ \dd t > 0.
}
On the other hand, \eqref{2bniqe} implies
\eeq{ \label{wdd39x}
\lim_{\eps\searrow0} h_{\delta,\eps} = 0.
}
\end{subequations}
We now perturb the convolution by adding a multiple of $\kappa_1$:
\eeq{ \label{mu_delta_eps_def}
\mu_{\delta,\eps} \coloneqq \mu_\delta * \kappa_\eps + h_{\delta,\eps}\cdot \Big(\int_0^1 \kappa_1[0,t]\ \dd t\Big)^{-1}\kappa_1.
}
The definition of $h_{\delta,\eps}$ ensures
\eq{
\int_0^1\mu_{\delta,\eps}[0,t]\ \dd t
= \int_0^1\mu_\delta[0,t]\ \dd t 
\le\int_0^1\mu[0,t]\ \dd t \le 1,
}
so $\mu_{\delta,\eps}$ belongs to $\cM$.
Furthermore, the density function of $\kappa_1$ is $\psi$, which is bounded and strictly positive on $(0,1)$.
This fact together with \eqref{d6b43} implies that
\begin{subequations}
\begin{gather}
\text{$\mu_{\delta,\eps}$ has a smooth density function that is strictly positive on $(0,1)$, and} \label{dqb7x} \\
\text{$\lim_{\eps\searrow0}\mu_{\delta,\eps}[0,t] = \mu_\delta[0,t]$ for every continuity point $t\in[0,1)$ of $\mu_\delta$.} \label{wib7rh}
\end{gather}
\end{subequations}

We are now ready to define the desired sequence $(f_n)_{n\ge1}$.
Clearly $\mu_\delta[0,t]\to\mu[0,t]$ for every $t\in[0,1)$ as $\delta\searrow0$, so Corollary~\ref{cor_convergence} implies $d(f_{\mu_\delta},f_\mu)\to 0$.
Let $\delta_n\in(0,n^{-1}]$ be sufficiently small that $d(f_{\mu_{\delta_n}},f_\mu)\le n^{-1}$.
Having selected $\delta_n$, we invoke \eqref{wib7rh} (again followed by Corollary~\ref{cor_convergence}) to obtain $d(f_{\mu_{\delta_n,\eps}},f_{\mu_{\delta_n}})\to0$ as $\eps\searrow0$.
So choose $\eps_n\in(0,\delta_n)$ sufficiently small that 
$d(f_{\mu_{\delta_n,\eps_n}},f_{\mu_{\delta_n}})\le n^{-1}$.
Set $f_n = f_{\mu_{\delta_n,\eps_n}}$.
Since $d$ satisfies the triangle inequality, we have $d(f_n,f)\le 2n^{-1}$, so property~\ref{lem_smooth_f_i} holds.
Hence $f_n\to f$ pointwise on $[0,1)$, by Lemma~\ref{lem_convergence}.
By continuity of $\fp$ from \eqref{p_cont}, it follows that
\eeq{ \label{ni4oo9}
\lim_{n\to\infty} \fp(x,f_n(x)) = \fp(x,f(x)) \quad \text{for every $x\in[0,1)$.}
}
Lemma~\ref{lem_identification_scheme}\ref{scheme_2nd_deriv} implies that $f_n''$ is equal to the density function of $\mu_{\delta_n,\eps_n}$, so property~\ref{lem_smooth_f_ii} is a consequence of \eqref{dqb7x}.
Furthermore, for every $x\in(0,1)$ we have
\eeq{ \label{nm3xf}
f_n''(x) 
= D_{\mu_{\delta_n,\eps_n}}(x)
\stackref{mu_delta_eps_def}{\ge} D_{\mu_{\delta_n}*\kappa_{\eps_n}}(x)
&\stackrel{\hphantom{\mbox{\footnotesize(Lemma~\ref{lem_identification_scheme}\ref{scheme_2nd_deriv})}}}{=}\int_{[0,1-\delta_n]} \eps_n^{-1}\psi\Big(\frac{x-y}{\eps_n}\Big)\ \mu(\dd y) \\
&\stackrel{\hphantom{\mbox{\footnotesize(Lemma~\ref{lem_identification_scheme}\ref{scheme_2nd_deriv})}}}{\ge}\int_{[0,1-\delta_n]} \eps_n^{-1}\psi\Big(\frac{x-y}{\eps_n}\Big)\ \mu^\ac(\dd y) \\
&\stackrel{\mbox{\footnotesize (Lemma~\ref{lem_identification_scheme}\ref{scheme_2nd_deriv})}}{=} \int_0^{1-\delta_n} \eps_n^{-1}\psi\Big(\frac{x-y}{\eps_n}\Big)\cdot f''(y)\ \dd y.
}
The integral on the final line is estimated as follows.
Observe that $y\mapsto \eps^{-1}\psi((x-y)/\eps)$ is supported on $[x-\eps,x]$ and integrates to $1$.
So whenever $[x-\eps,x]\subseteq[0,1-\delta]$, we have
\eeq{ \label{3bon9y}
\Big|\int_0^{1-\delta} \eps^{-1}\psi\Big(\frac{x-y}{\eps}\Big)\cdot f''(y)\ \dd y
- f''(x)\Big|
&= \Big|\int_{x-\eps}^x\eps^{-1}\psi\Big(\frac{x-y}{\eps}\Big)\cdot[f''(y)-f''(x)]\ \dd y\Big| \\
&\le \|\psi\|_\infty\cdot\eps^{-1}\int_{x-\eps}^x |f''(y)-f''(x)|\ \dd y.
\raisetag{1.5\baselineskip}
}
Note that $f''\big|_{[0,x]}$ is integrable for every $x\in[0,1)$, since Lemma~\ref{lem_identification_scheme} gives $\int_0^{x}f''(t)\, \dd t = \mu^\ac[0,x] \le \mu[0,x] \le (1-x)^{-1}$.
Therefore, we may use the Lebesgue differentiation theorem to say that as $\eps\searrow0$, the final line of \eqref{3bon9y} tends to zero for almost every $x\in(0,1)$.
It now follows from \eqref{nm3xf} that
\eq{
\liminf_{n\to\infty} f_n''(x) \ge f''(x) \quad \text{for almost every $x\in(0,1)$.} 
}
This observation, together with Fatou's lemma and \eqref{ni4oo9}, yields
\eq{
\liminf_{n\to\infty} \int_0^1 \big[f_n''(x)\cdot \fp(x,f_n(x))\big]^{1/3}\ \dd x
\ge \int_0^1 \big[f''(x)\cdot \fp(x,f(x))\big]^{1/3}\ \dd x.
}
We have thus shown $\liminf J(f_n) \ge J(f)$, while Proposition~\ref{prop_usc} gives $\limsup J(f_n) \le J(f)$.
Hence property~\ref{lem_smooth_f_iii} holds, thereby completing the proof.
\end{proof}

\section{Two-sided inputs} \label{sec_two_sided_inputs}
In this section we collect several estimates that will be used in both the lower bound arguments (Section~\ref{sec_lower_bound}) and the upper bound arguments (Section~\ref{sec_upper_bound}).

\subsection{Tangency triangles} \label{subsec_tangency_triangles}
We begin with a crucial definition.

\begin{definition}[Tangency triangles] \label{def_triangle}
Let $f\colon[a,b]\to\R$ be continuous and convex.
Let us write $\nA = (a,f(a))$ and $\nB = (b,f(b))$.
\begin{itemize}

\item (Typical case) 
If $\partial^+f(a) < \partial^-f(b)$ and at least one of these two derivatives is finite, then we denote by $\Tri_f(a,b)$ the triangle bounded by the following three lines: the tangent line of $f$ at $a$ with slope $\partial^+f(a)$, the tangent line of $f$ at $b$ with slope $\partial^-f(b)$, and the secant line of $f$ between $a$ and $b$.
The vertices of $\Tri_f(a,b)$ are thus $\nA$, $\nB$, and the unique intersection point of the two tangent lines.
See Figure~\ref{fig_tangency_triangle} (left).

\item (Degenerate case 1: flat)
If $\partial^+f(a)=\partial^-f(b)<\infty$, then we let $\Tri_f(a,b)$ denote the line segment between $\nA$ and $\nB$.

\item (Degenerate case 2: doubly infinite)
If $\partial^+f(a)=-\infty$ and $\partial^-f(b)=\infty$, then we let $\Tri_f(a,b)$ denote the infinite region below the secant line, defined in \eqref{below_secant}.
\qedrem

\end{itemize}
\end{definition}

Following this definition, we note two basic consequences of convexity.
First, the graph of $f$ remains inside the tangency triangle: 
\eeq{ \label{graph_in_triangle}
\{(x,f(x)):\,x\in[a,b]\} \subseteq \Tri_f(a,b).
}
Second, tangency triangles respect inclusion (see Figure~\ref{fig_tangency_triangle}):
\eeq{ \label{triangle_inclusion}
\Tri_f(a,x) \subseteq \Tri_f(a,b) \quad \text{and} \quad \Tri_f(x,b) \subseteq \Tri_f(x,b) \quad \text{for every $x\in(a,b)$.}
}
These facts will be useful in several places.

In the following lemma and beyond, we use the notation
\eq{
\Slope\big((x,y),(w,z)\big) \coloneqq \begin{cases} 
\frac{z-y}{w-x} &\text{if $x\neq w$} \\
+\infty &\text{if $x=w$ and $y<z$} \\
-\infty &\text{if $x=w$ and $y>z$.}
\end{cases}
}

\begin{lemma}[Convex position in terms of slopes] \label{lem_position_slopes}
Let $f\colon[a,b]\to\R$ be continuous and convex.
Let us write $\nA = (a,f(a))$ and $\nB = (b,f(b))$.
Given a set
\eq{
\cC = \{(x_1,y_1),\ldots,(x_n,y_n)\}\subset\Tri_f(a,b)\setminus\{\nA,\nB\} \quad \text{such that} \quad
a \le x_1 \le \cdots \le x_n \le b,
}
let us write $\nP_i = (x_i,y_i)$.
Then the following two statements are equivalent:
\begin{enumerate}[label=\textup{(\alph*)}]

\item \label{lem_position_slopes_a}
$\{\nA,\nB\}\cup\cC$ is in convex position.

\item \label{lem_position_slopes_b}
$\Slope(\nA,\nP_1)<\Slope(\nP_1,\nP_2)<\cdots<\Slope(\nP_{n-1},\nP_n)<\Slope(\nP_n,\nB)$.

\end{enumerate}
\end{lemma}

\begin{proof}
For convenience, let us write 
\eq{
\nP_0 \coloneqq (x_0,y_0) = (a,f(a)) \quad \text{and} \quad 
\nP_{n+1} \coloneqq (x_{n+1},y_{n+1}) = (b,f(b)).
}

\medskip

\noindent \ref{lem_position_slopes_b}$\implies$\ref{lem_position_slopes_a}:
Suppose \ref{lem_position_slopes_b} is true.
Denote the convex hull of $\{\nA,\nB\}\cup\cC$ by $\cH$.
We wish to show that every element of $\{\nA,\nB\}\cup\cC$ is an extreme point of $\cH$.

First we show $\nA$ and $\nB$ are extreme points.
Consider the region $\cR$ consisting of all points on or below the secant line between $\nA$ and $\nB$:
\eeq{ \label{below_secant}
\cR \coloneqq \Big\{(x,y)\in[a,b]\times\R:\, y \le \frac{b-x}{b-a}f(a) + \frac{x-a}{b-a}f(b)\Big\}.
}
Note that $\cR$ is convex, the points $\nA$ and $\nB$ are extreme points of $\cR$, and $\cH\subset\Tri_f(a,b)\subset\cR$.
Therefore, $\nA$ and $\nB$ are extreme points of $\cH$.

Now we show $\nP_1,\ldots,\nP_n$ are extreme points.
Fix $i\in\{1,\ldots,n\}$.
Our assumption \ref{lem_position_slopes_b} allows us to choose some number $s\in\R$ satisfying
\eq{
\Slope(\nP_{i-1},\nP_i) < s < \Slope(\nP_i,\nP_{i+1}).
}
Let $\vphi\colon\R\to\R$ be the affine function with slope $s$ and $\vphi(x_i)=y_i$.
By choice of $s$ and assumption~\ref{lem_position_slopes_b}, we have $\vphi(x_j)<y_j$ for all $j\in\{0,\ldots,n+1\}\setminus\{i\}$.
Let $\eps>0$ be such that $\vphi(x_j)+\eps<y_j$ for all $j\in\{0,\ldots,n+1\}\setminus\{i\}$, and then consider the region $\cQ$ consisting of all points on or above the graph of $\vphi+\eps$:
\eq{
\cQ \coloneqq \{(x,y)\in\R^2:\, \vphi(x)+\eps\le y\}.
}
Note that $\cQ$ is convex, $\cQ$ contains every $\nP_j$ such that $j\neq i$, and $\nP_i\notin\cQ$.
Therefore, the convex hull of $\{\nP_0,\ldots,\nP_{n+1}\}\setminus\{\nP_i\}$ does not contain $\nP_i$.
It follows that $\nP_i$ must be an extreme point of $\cH$, as desired.

\medskip

\noindent \ref{lem_position_slopes_a}$\implies$\ref{lem_position_slopes_b}:
Suppose \ref{lem_position_slopes_b} is false.
That is, there is $i\in\{1,\ldots,n\}$ such that
\eq{
\Slope(\nP_{i-1},\nP_i) \ge \Slope(\nP_i,\nP_{i+1}).
}
This implies the equality below:
\eq{
\Slope(\nP_{i-1},\nP_{i+1}) \le \max\{\Slope(\nP_{i-1},\nP_i),\Slope(\nP_i,\nP_{i+1})\}
= \Slope(\nP_{i-1},\nP_i).
}
Therefore, $\nP_i$ lies on or above the line segment connecting $\nP_{i-1}$ and $\nP_{i+1}$.
On the other hand, since $\nP_{i-1},\nP_i,\nP_{i+1}\in\Tri_f(a,b)$, all three points $\nP_{i-1},\nP_{i},\nP_{i+1}$ lie on or below the line segment connecting $\nA$ and $\nB$.
The two previous sentences together imply that $\nP_i$ lies in the convex hull of $\{\nA,\nB,\nP_{i-1},\nP_{i+1}\}$.
Hence \ref{lem_position_slopes_a} is false.
\end{proof}

\begin{figure}[t]

\tikzset{every picture/.style={line width=0.75pt}} %set default line width to 0.75pt        

\begin{tikzpicture}[x=0.75pt,y=0.75pt,yscale=-1,xscale=1,scale=0.9]

%Shape: Polygon [id:ds9506955071414543] 
\draw  [color={rgb, 255:red, 74; green, 144; blue, 226 }  ,draw opacity=1 ][fill={rgb, 255:red, 74; green, 144; blue, 226 }  ,fill opacity=0.1 ] (50,220) -- (290,50) -- (228,202) -- cycle ;
%Shape: Polygon [id:ds2725995063341654] 
\draw  [color={rgb, 255:red, 74; green, 144; blue, 226 }  ,draw opacity=1 ][fill={rgb, 255:red, 74; green, 144; blue, 226 }  ,fill opacity=0.3 ] (50,220) -- (186,154.5) -- (100,215) -- cycle ;
%Shape: Polygon [id:ds5591067661839962] 
\draw  [color={rgb, 255:red, 74; green, 144; blue, 226 }  ,draw opacity=1 ][fill={rgb, 255:red, 74; green, 144; blue, 226 }  ,fill opacity=0.3 ] (186,154.5) -- (290,50) -- (273,92) -- cycle ;
%Curve Lines [id:da08051833442644107] 
\draw [line width=0.75]    (50,220) .. controls (126,213) and (271,95) .. (290,50) ;
%Shape: Circle [id:dp4329340480616283] 
\draw  [fill={rgb, 255:red, 0; green, 0; blue, 0 }  ,fill opacity=1 ] (182.5,153.5) .. controls (182.5,151.57) and (184.07,150) .. (186,150) .. controls (187.93,150) and (189.5,151.57) .. (189.5,153.5) .. controls (189.5,155.43) and (187.93,157) .. (186,157) .. controls (184.07,157) and (182.5,155.43) .. (182.5,153.5) -- cycle ;
%Shape: Circle [id:dp8690310960267837] 
\draw  [fill={rgb, 255:red, 0; green, 0; blue, 0 }  ,fill opacity=1 ] (46.5,220) .. controls (46.5,218.07) and (48.07,216.5) .. (50,216.5) .. controls (51.93,216.5) and (53.5,218.07) .. (53.5,220) .. controls (53.5,221.93) and (51.93,223.5) .. (50,223.5) .. controls (48.07,223.5) and (46.5,221.93) .. (46.5,220) -- cycle ;
%Shape: Circle [id:dp539780717718849] 
\draw  [fill={rgb, 255:red, 0; green, 0; blue, 0 }  ,fill opacity=1 ] (286.5,50) .. controls (286.5,48.07) and (288.07,46.5) .. (290,46.5) .. controls (291.93,46.5) and (293.5,48.07) .. (293.5,50) .. controls (293.5,51.93) and (291.93,53.5) .. (290,53.5) .. controls (288.07,53.5) and (286.5,51.93) .. (286.5,50) -- cycle ;
%Shape: Polygon [id:ds5075511268968711] 
\draw  [color={rgb, 255:red, 74; green, 144; blue, 226 }  ,draw opacity=1 ][fill={rgb, 255:red, 74; green, 144; blue, 226 }  ,fill opacity=0.3 ] (350,220) -- (486,154.5) -- (400,215) -- cycle ;
%Shape: Polygon [id:ds22433134114243913] 
\draw  [color={rgb, 255:red, 74; green, 144; blue, 226 }  ,draw opacity=1 ][fill={rgb, 255:red, 74; green, 144; blue, 226 }  ,fill opacity=0.3 ] (486,154.5) -- (590,50) -- (573,92) -- cycle ;
%Curve Lines [id:da14360266540535105] 
\draw [line width=0.75]    (350,220) .. controls (426,213) and (571,95) .. (590,50) ;
%Shape: Circle [id:dp29418103398737117] 
\draw  [fill={rgb, 255:red, 0; green, 0; blue, 0 }  ,fill opacity=1 ] (482.5,153.5) .. controls (482.5,151.57) and (484.07,150) .. (486,150) .. controls (487.93,150) and (489.5,151.57) .. (489.5,153.5) .. controls (489.5,155.43) and (487.93,157) .. (486,157) .. controls (484.07,157) and (482.5,155.43) .. (482.5,153.5) -- cycle ;
%Shape: Circle [id:dp8249039840316097] 
\draw  [fill={rgb, 255:red, 0; green, 0; blue, 0 }  ,fill opacity=1 ] (346.5,220) .. controls (346.5,218.07) and (348.07,216.5) .. (350,216.5) .. controls (351.93,216.5) and (353.5,218.07) .. (353.5,220) .. controls (353.5,221.93) and (351.93,223.5) .. (350,223.5) .. controls (348.07,223.5) and (346.5,221.93) .. (346.5,220) -- cycle ;
%Shape: Circle [id:dp9227722158513907] 
\draw  [fill={rgb, 255:red, 0; green, 0; blue, 0 }  ,fill opacity=1 ] (586.5,50) .. controls (586.5,48.07) and (588.07,46.5) .. (590,46.5) .. controls (591.93,46.5) and (593.5,48.07) .. (593.5,50) .. controls (593.5,51.93) and (591.93,53.5) .. (590,53.5) .. controls (588.07,53.5) and (586.5,51.93) .. (586.5,50) -- cycle ;

% Text Node
\draw (269,22.4) node [anchor=north west][inner sep=0.75pt]    {$\nB=( b,f( b))$};
% Text Node
\draw (25,235.4) node [anchor=north west][inner sep=0.75pt]    {$\nA=( a,f( a))$};
% Text Node
\draw (103,108.4) node [anchor=north west][inner sep=0.75pt]  [color={rgb, 255:red, 74; green, 144; blue, 226 }  ,opacity=1 ]  {$\Tri_{f}( a,b)$};
% Text Node
\draw (569,22.4) node [anchor=north west][inner sep=0.75pt]    {$\nB=( b,f( b))$};
% Text Node
\draw (325,235.4) node [anchor=north west][inner sep=0.75pt]    {$\nA=( a,f( a))$};
% Text Node
\draw (367,154.4) node [anchor=north west][inner sep=0.75pt]  [color={rgb, 255:red, 74; green, 144; blue, 226 }  ,opacity=1 ]  {$\Tri_{f}( a,x)$};
% Text Node
\draw (491.5,156.9) node [anchor=north west][inner sep=0.75pt]    {$\nX=( x,f( x))$};
% Text Node
\draw (471,75.4) node [anchor=north west][inner sep=0.75pt]  [color={rgb, 255:red, 74; green, 144; blue, 226 }  ,opacity=1 ]  {$\Tri_{f}( x,b)$};

\end{tikzpicture}
\caption{The tangency triangle $\Tri_f(a,b)$ is the lightly shaded region on the left. This region contains the smaller tangency triangles $\Tri_f(a,x)$ and $\Tri_f(x,b)$ shown on the right.
In both, the graph of $f$ is the solid black curve.}
\label{fig_tangency_triangle}
\end{figure}
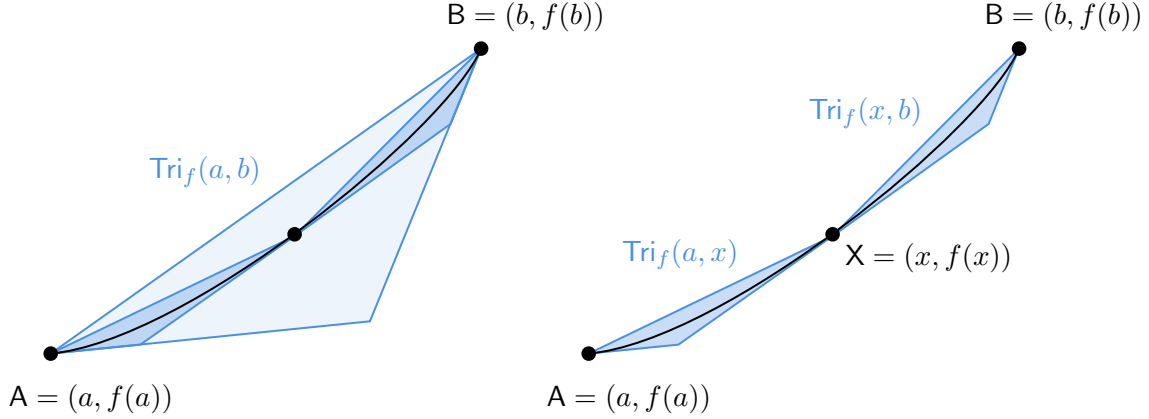

The setting of the next lemma is illustrated in Figure~\ref{fig_tangency_triangle}.

\begin{lemma}[Concatenation] \label{lem_concatenation}
Let $f\colon[a,b]\to\R$ be continuous and convex, and $x\in(a,b)$.
Let us write $\nA = (a,f(a))$, $\nB = (b,f(b))$, and $\nX = (x,f(x))$.
Suppose $\cC_1\subset\Interior(\Tri_f(a,x))$ and $\cC_2\subset\Interior(\Tri_f(x,b))$ are finite sets such that $\{\nA,\nX\}\cup\cC_1$ is in convex position, and $\{\nX,\nB\}\cup\cC_2$ is in convex position.
Then $\cC_1\cup\cC_2\subset\Interior(\Tri_f(a,b))$, and $\{\nA,\nB\}\cup\cC_1\cup\cC_2$ is in convex position.
\end{lemma}

\begin{proof}
From \eqref{triangle_inclusion} we know
\eq{
\Interior(\Tri_f(a,x)) \subseteq \Interior(\Tri_f(a,b)) \quad \text{and} \quad
\Interior(\Tri_f(x,b)) \subseteq \Interior(\Tri_f(a,b)).
}
This ensures $\cC_1\cup\cC_2\subset\Interior(\Tri_f(a,b))$.

Now we argue that $\{\nA,\nB\}\cup\cC_1\cup\cC_2$ is in convex position.
There is nothing to show if $\cC_1$ and $\cC_2$ are both empty.
So we assume $\cC_1$ is nonempty; the argument is similar if $\cC_2$ is nonempty.
Label the elements of $\cC_1$ in increasing order of their horizontal coordinates, as $\nP_1,\ldots,\nP_m$.
Similarly, label the elements of $\cC_2$ in increasing order of their horizontal coordinates, as $\nQ_1,\ldots,\nQ_n$.
Since $\{\nA,\nX\}\cup\cC_1$ is in convex position, Lemma~\ref{lem_position_slopes} implies
\eeq{ \label{in1fv}
\Slope(\nA,\nP_1) < \Slope(\nP_1,\nP_2) < \cdots < \Slope(\nP_{m-1},\nP_m) < \Slope(\nP_m,\nX).
}
Furthermore, since $\nP_m\in\Interior(\Tri_f(a,x))$, we have
\eq{
\Slope(\nP_m,\nX) < \partial^-f(x).
}
If $n\ge1$ (i.e.\ $\cC_2$ is nonempty), then analogous reasoning gives
\eeq{ \label{in2fv}
\partial^+f(x) < \Slope(\nX,\nQ_1) < \Slope(\nQ_1,\nQ_2) < \cdots < \Slope(\nQ_{n-1},\nQ_n) < \Slope(\nQ_n,\nB).
}
In this case, we can concatenate \eqref{in1fv} and \eqref{in2fv} to obtain
\eeq{ \label{in3fv}
&\Slope(\nA,\nP_1) < \Slope(\nP_1,\nP_2) < \cdots < \Slope(\nP_{m-1},\nP_m) < \Slope(\nP_m,\nX) \\
<\ &\Slope(\nX,\nQ_1) < \Slope(\nQ_1,\nQ_2) < \cdots < \Slope(\nQ_{n-1},\nQ_n) < \Slope(\nQ_n,\nB).
}
If $n=0$, then instead of \eqref{in2fv} we simply observe that
\eq{
\partial^+f(x) \le \Slope(\nX,\nB),
}
so \eqref{in3fv} can be replaced by
\eq{
&\Slope(\nA,\nP_1) < \Slope(\nP_1,\nP_2) < \cdots < \Slope(\nP_{m-1},\nP_m) < \Slope(\nP_m,\nX)
< \Slope(\nX,\nB).
}
In either case, use Lemma~\ref{lem_position_slopes} to conclude that $\{\nA,\nB\}\cup(\cC_1\cup\cC_2\cup\{\nX\})$ is in convex position.
In particular, $\{\nA,\nB\}\cup\cC_1\cup\cC_2$ is in convex position.
\end{proof}

\begin{lemma}[Area of tangency triangles] \label{lem_area}
Let $f\colon[a,b]\to\R$ be continuous and convex.
Assume $\partial^+f(a)=s_a<s_b=\partial^-f(b)$, where at least one of $s_a,s_b$ is finite.
\begin{enumerate}[label=\textup{(\alph*)}]

\item \label{lem_area_a}
Using the quantities $\Delta_1$ and $\Delta_2$ defined in \eqref{eq_unif_case_a} and \eqref{eq_unif_case_b}, we have
\eeq{ \label{triangle_area}
\Area\big(\Tri_f(a,b)\big) = \frac{1}{2}\begin{cases}
\Delta_1(b-a) &\text{if $-\infty<s_a<s_b=\infty$} \\
\Delta_2(b-a) &\text{if $-\infty=s_a<s_b<\infty$} \\
\Delta_1\Delta_2/(s_b-s_a) &\text{if $-\infty<s_a<s_b<\infty$.}
\end{cases}
}

\item \label{lem_area_b}
Assume further that $\partial^+f$ is absolutely continuous on $[a,u]$ for every $u\in(a,b)$, and $\partial^-f(b)<\infty$.
If $0 < c \le f''(u) \le C$ for almost every $u\in(a,b)$, then
\eeq{ \label{area_bound}
\frac{c^2}{C} \le \frac{\Area\big(\Tri_f(a,b)\big)}{\tfrac{1}{8}(b-a)^3} \le \frac{C^2}{c}\wedge 2C.
}

\end{enumerate}
\end{lemma}

\begin{proof}
The area of $\Tri_f(a,b)$ is given by the determinant formula
\eeq{ \label{3fg8c}
\Area\big(\Tri_f(a,b)\big)
= \frac{1}{2}\left|\det\begin{bmatrix}
a & b & x \\
f(a) & f(b) & y \\
1 & 1 & 1 \\
\end{bmatrix}\right|,
}
where $(x,y)$ is the unique intersection point of the tangent line at $a$ (with slope $s_a$) and the tangent line at $b$ (with slope $s_b$).
That is,
\eeq{ \label{y_def}
x \begin{cases}
= a &\text{if $s_a=-\infty$} \\
= b &\text{if $s_b=\infty$} \\
\in (a,b) &\text{if $-\infty<s_a<s_b<\infty$}
\end{cases}
\quad \text{and} \quad
y = \begin{cases}
f(a) + s_a(x-a) &\text{if $s_a>-\infty$} \\
f(b) + s_b(x-b) &\text{if $s_b<\infty$.}
\end{cases}
}
Since $(1,1,1)$ is the third row of the matrix in \eqref{3fg8c}, we may subtract $(a,a,a)$ from the first row without changing the determinant. 
Assuming $s_a>-\infty$, this results in
\eeq{ \label{2kcd0}
\Area\big(\Tri_f(a,b)\big)
&\stackrefp{y_def}{=} \frac{1}{2}\left|\det\begin{bmatrix}
0 & b-a & x-a \\
f(a) & f(b) & y \\
1 & 1 & 1 \\
\end{bmatrix}\right| \\
&\stackrefp{y_def}{=} \tfrac{1}{2}\big|(x-a)\cdot(f(b)-f(a))-(b-a)\cdot(y - f(a))\big| \\
&\stackref{y_def}{=} \tfrac{1}{2}(x-a)\big|f(b)-f(a)-s_a(b-a)\big|
\stackref{eq_unif_case_a}{=} \tfrac{1}{2}(x-a)\Delta_1.
}
If $s_b=\infty$, then $x=b$, and we obtain the first case in \eqref{triangle_area}.
The second case (where $-\infty=s_a<s_b<\infty$) is proved in the same way, except we subtract $(b,b,b)$ instead of $(a,a,a)$.
In the third case $(-\infty<s_a<s_b<\infty)$, we solve for $x$ in \eqref{y_def}, and then subtract $a$, to obtain
\eeq{ \label{c_def}
x - a &= \frac{s_b(b-a)-(f(b)-f(a))}{s_b-s_a}
\stackref{eq_unif_case_b}{=} \frac{\Delta_2}{s_b-s_a}.
}
Insert this identity into the final expression of \eqref{2kcd0} to obtain the desired formula.
This completes the proof of part~\ref{lem_area_a}.

For part~\ref{lem_area_b}, the absolute continuity assumption yields the following three equalities (respectively justified by \eqref{29vc5}--\eqref{2c7h5}, \eqref{29vc6}--\eqref{2c7h6}, and \eqref{2_deriv_eq}):
\begin{align}
\Delta_1 &= \int_a^b f''(u)(u-a)\ \dd u \in \big[c\tfrac{(b-a)^2}{2},C\tfrac{(b-a)^2}{2}\big], \label{wixj1} \\
\Delta_2 &= \int_a^b f''(u)(b-u)\ \dd u \in \big[c\tfrac{(b-a)^2}{2},C\tfrac{(b-a)^2}{2}\big], \label{wixj2} \\
s_b - s_a &= \int_a^b f''(u)\ \dd u \in \big[c(b-a),C(b-a)\big]. \label{wixj3}
\end{align}
Using the estimates \eqref{wixj2} and \eqref{wixj3} in \eqref{c_def}, we obtain
\eeq{ \label{3bxxqk}
\frac{c}{2C}(b-a) \le x-a \le \Big(\frac{C}{2c}\wedge 1\Big)(b-a).
}
Using \eqref{3bxxqk} and \eqref{wixj1} in \eqref{2kcd0}, we obtain \eqref{area_bound}.
\end{proof}

\begin{corollary}[Characterization of uniform optimizers] \label{cor_unif_case_characterization}
Let $f$ be the unique optimizer from Proposition~\ref{prop_unif_case}.
For any $u < v$ in $[a,b]$, we have
\eeq{ \label{integral_to_triangle}
\int_u^v f''(x)^{1/3}\ \dd x = 2\cdot\Area\big(\Tri_f(u,v)\big)^{1/3}.
}
\end{corollary}

\begin{proof}
Fix any $u<v$ in $[a,b]$.
Let $g\colon[u,v]\to\R$ be the optimizer from Proposition~\ref{prop_unif_case} with boundary data 
\eq{
g(u) = f(u), \quad
g(v) = f(v), \quad
\partial^+g(u) = \partial^+f(u), \quad
\partial^-g(v) = \partial^-f(v).
}
By comparing the three cases of \eqref{triangle_area} with \eqref{eq_unif_case_a}, \eqref{eq_unif_case_b}, and \eqref{unif_case_c_sup} respectively, we see that
\eeq{ \label{jg4xp0}
\int_u^v g''(x)^{1/3}\ \dd x = 2\cdot \Area\big(\Tri_g(u,v)\big)^{1/3}.
}
In addition, we must have $f\big|_{[u,v]}=g$, since otherwise $f\big|_{[u,v]}$ could be replaced by $g$ to contradict the optimality of $f$.
Hence \eqref{jg4xp0} implies \eqref{integral_to_triangle}.
\end{proof}

\subsection{Convex chains in general triangles} \label{sec_general_triangles}
In this section we work on a general triangle instead of $\cT$.
By Remark~\ref{rem_convex_position}, there would be no loss of generality in restricting attention to $\cT$, but keeping the notation general will make later arguments more transparent.
This is because the results we prove here---gathered in a single proposition below---will be applied to small triangles within $\cT$ rather than to $\cT$ itself.
To distinguish the two settings, we use $\wh\fp$ instead of $\fp$.
But we will continue to write $\sL_n$ and $\cS_n$, for simplicity.

\begin{proposition}[Convex chains in general triangles] \label{prop_general}
Let $\triangle$ be a nondegenerate triangle with two distinguished vertices $\nA,\nB\in\R^2$.
Let $\wh\fp$ be a probability density on $\triangle$ such that
\eeq{ \label{ratio_assumption}
\frac{\inf_{(x,y)\in\triangle}\wh\fp(x,y)}{1/\Area(\triangle)} \ge r\in [0,1]
\qquad \text{and} \qquad
\frac{\sup_{(x,y)\in\triangle}\wh\fp(x,y)}{1/\Area(\triangle)} \le R\in[1,\infty).
}
Let $\cS_n$ %= \{\nP_1,\ldots,\nP_n\}$ 
be a set of $n$ independent samples from $\wh\fp$, and define
\eq{
\sL_n \coloneqq \max\{\#\cC:\, \cC\subseteq\cS_n \text{ and $\{\nA,\nB\}\cup\cC$ is in convex position}\}.
}
Then the following statements hold.
(If $r = 0$, then by convention $R/r = \infty$.)
\begin{enumerate}[label=\textup{(\alph*)}]

\item \label{prop_general_compare} \textup{(Comparison to uniform case)} 
Let $\alpha$ be the constant from Theorem~\ref{thm_uniform_known}. 
For every $\delta>0$, there exists $\fn = \fn(R/r,\delta)$ large enough that
\eeq{ \label{mean_upper_lower}
(1-\delta)\Big(\frac{r}{R}\Big)^{1/3} \le \frac{\E\sL_n}{\alpha n^{1/3}} \le (1+\delta)\Big(\frac{R}{r}\Big)^{1/3} \quad \text{for all $n\ge\fn$.}
}

\item \label{prop_general_likelihood} \textup{(Likelihood of a full convex chain)}
For every $n\ge1$, we have
\eeq{ \label{1o0nc}
r^n\frac{2^n}{n!(n+1)!} \le \P(\{\nA,\nB\}\cup\cS_n \text{ is in convex position})
\le R^n\frac{2^n}{n!(n+1)!}.
}

\item \label{prop_general_tail} \textup{(Deep tail event)}
For every $n\ge1$ and $t\ge0$, we have
\eeq{ \label{tail_deep}
\P\big(\sL_n \ge (4R\e^3 n)^{1/3} + t\big) < 2^{-(4R\e^3 n)^{1/3}-t}.
}

\item \label{prop_general_concentration} \textup{(Concentration)}
For every $t_0>0$ and $\beta\in(\tfrac16,\tfrac13]$, there exist $C=C(t_0,\beta)$ large enough and $\fa = \fa(R,t_0)>0$ small enough that every $n\ge1$, $t\ge t_0$, we have
\eeq{ \label{concentration_ineq}
\mathbb P\Bigl(\big|\sL_n-\E \sL_n\big| \ge tn^\beta\Bigr)
\le
C \exp\bigl(-\fa tn^{2\beta-\frac13} \bigr).
}

\item \label{prop_general_convergence} \textup{(Convergence after centering)}
For every $\gamma>\tfrac16$, we have
\eeq{ \label{3hno2x}
\frac{\sL_n - \E\sL_n}{n^\gamma} \xrightarrow[n\to\infty]{} 0 \quad \text{almost surely and in $L^p$ for every $p\in[1,\infty)$.}
}

\end{enumerate}
\end{proposition}

\begin{remark}[A better comparison]
The bounds in \eqref{mean_upper_lower} are not optimal.
For instance, if $\wh\fp$ is continuous, then Theorem~\ref{thm_var_formula} implies
\eeq{ \label{better_comparison}
(1+o(1))r^{1/3} \le \frac{\E\sL_n}{\alpha n^{1/3}} \le (1+o(1))R^{1/3}.
}
Nevertheless, we use Proposition~\ref{prop_general}\ref{prop_general_compare} on our way to proving Theorem~\ref{thm_var_formula}.
In fact, our applications of Proposition~\ref{prop_general}\ref{prop_general_compare} will be to scenarios in which $R\approx r\approx 1$, in which case \eqref{mean_upper_lower} and \eqref{better_comparison} are essentially equivalent.
\qedrem
\end{remark}

\begin{proof}[Proof of Proposition~\ref{prop_general}\ref{prop_general_compare} and \ref{prop_general_likelihood}]
Denote the uniform density function on $\triangle$ by $\fu \equiv 1/\Area(\triangle)$.
Let $\cU_n$ denote a set of $n$ uniform samples from $\triangle$, and define
\eq{
\sL_n^\unif \coloneqq \max\{\#\cC:\, \cC\subseteq\cU_n \text{ and $\{\nA,\nB\}\cup\cC$ is in convex position}\}.
}

\medskip

\noindent Part~\ref{prop_general_compare}:
If $r = 0$, then \eqref{mean_upper_lower} is trivial, so we assume $r>0$.
The strategy is to couple the sample sets $(\cS_n)_{n\ge1}$ and $(\cU_n)_{n\ge1}$ in such a way that the ratio $\sL_n/\sL_n^\unif$ is well behaved.
By Theorem~\ref{thm_uniform_known} (and Remark~\ref{rem_convex_position} to convert from $\cT$ to $\triangle$), there is $n_1 = n_1(\delta)$ large enough that
\eeq{ \label{n0_choice}
\alpha\sqrt{1-\delta}\, n^{1/3} \le \E \sL_n^\unif \le \alpha\sqrt{1+\delta}\, n^{1/3} \quad \text{for all $n\ge n_1$.}
}
Let $(X_i,Y_i)_{i\ge1}$ be independent uniform samples from $\fu$, so that we may assume
\eq{
\cU_N = \{(X_1,Y_1),\ldots,(X_N,Y_N)\} \quad \text{for every $N\ge1$.}
}
Since $\wh \fp(x,y)/\fu(x,y) \le R$ by \eqref{ratio_assumption}, we can perform the following rejection sampling:

\begin{enumerate}[label=\textup{\arabic*.}]

\item Let $(U_i)_{i\ge1}$ be i.i.d.\ uniform $[0,1]$-valued random variables that are independent of $(X_i,Y_i)_{i\ge1}$.
If
\eeq{ \label{accept_condition}
U_i \le \frac{1}{R}\cdot \frac{\wh \fp(X_i,Y_i)}{\fu(X_i,Y_i)},
}
then set $\chi_i = 1$ (accept).
Otherwise set $\chi_i = 0$ (reject).

\item Denote the index of the $n$-th acceptance by $\tau_n = \inf\{N:\,\chi_1+\cdots+\chi_N = n\}$.
It is a standard fact (e.g.\ \cite[Corollary~2.17]{robert_casella04}) that the accepted samples $(X_{\tau_i},Y_{\tau_i})_{i\ge1}$ are independent and distributed according to $\wh\fp$.
So we may assume
\eq{
\cS_n = \big\{(X_{\tau_1},Y_{\tau_1}),\ldots,(X_{\tau_n},Y_{\tau_n})\big\} \quad \text{for every $n\ge1$.}
}

\end{enumerate}
If $\tau_n \le N$, then $\cS_n \subseteq \cU_N$, hence
\begin{subequations} \label{df7b3c}
\eeq{ \label{2mgv7}
\E(\sL_n\one\{\tau_n\le N\})
\le \E \sL_N^\unif.
}
On the other hand, since $\sL_n\le n$, we trivially have
\eeq{ \label{3cb4c}
\E(\sL_n\one\{\tau_n> N\})
\le n\P(\tau_n> N).
}
\end{subequations}
We now seek control on the right-hand sides of \eqref{2mgv7} and \eqref{3cb4c}, with the choice
\eeq{ \label{N_from_n}
N = \bigg\lceil\frac{R}{r}(1+\delta)n\bigg\rceil.
}
On one hand, if $n\ge n_1$, then also $N \ge n_1$, so \eqref{n0_choice} implies
\begin{subequations}
\eeq{ \label{n1_choice}
\E\sL_N^\unif 
\le \alpha\sqrt{1+\delta}N^{1/3} 
\stackref{N_from_n}{\le} \alpha\sqrt{1+\delta}\Big(\frac{R}{r}(1+\delta)n+1\Big)^{1/3}
\quad \text{for all $n\ge n_1$.}
}
On the other hand, since $\wh \fp(x,y)/\fu(x,y) \ge r$ by \eqref{ratio_assumption}, the acceptance condition \eqref{accept_condition} holds if $U_i \le r/R$.
Consequently, $\chi_1+\cdots+\chi_N$ stochastically dominates the binomial distribution with $N$ trials and success probability $r/R$.
This domination, together with Hoeffding's inequality, justifies the first inequality below:
\eq{
\P(\tau_n > N)
= \P\big(\chi_1+\cdots+\chi_N < n\big)
&\le \exp\Big(\frac{-2(\frac{r}{R}N-n)^2}{N}\Big)
\stackref{N_from_n}{\le} \exp\Big(\frac{-2(\delta n)^2}{\frac{R}{r}(1+\delta)n+1}\Big).
}
So let $n_2 = n_2(R/r,\delta)$ be large enough that
\eq{
\exp\Big(\frac{-2(\delta n)^2}{\frac{R}{r}(1+\delta)n+1}\Big) \le 1/n \quad \text{for all $n\ge n_2$.}
}
The two previous displays together give
\eeq{ \label{n2_choice}
n\P(\tau_n > N) \le 1 \quad \text{for all $n \ge n_2$.}
}
\end{subequations}
Using \eqref{n1_choice} and \eqref{n2_choice} in \eqref{2mgv7} and \eqref{3cb4c} respectively, we arrive at
\eq{
\E\sL_n \le \alpha\sqrt{1+\delta}\Big(\frac{R}{r}(1+\delta)n+1\Big)^{1/3} + 1 \quad \text{for all $n\ge\max\{n_1,n_2\}$.}
}
Finally, let $n_3 = n_3(R/r,\delta)$ be large enough that
\eq{
\alpha\sqrt{1+\delta}\Big(\frac{R}{r}(1+\delta)n+1\Big)^{1/3} + 1 
\le \alpha(1+\delta)\Big(\frac{R}{r}\Big)^{1/3}n^{1/3} \quad \text{for all $n\ge n_3$}.
}
The two previous displays together give
\eq{
\E\sL_n \le \alpha(1+\delta)\Big(\frac{R}{r}\Big)^{1/3}n^{1/3} \quad \text{for all $n\ge \max\{n_1,n_2,n_3\}$.}
}
This establishes the upper bound in \eqref{mean_upper_lower} with $\fn = \max\{n_1,n_2,n_3\}$.

The lower bound is proved by exchanging the roles of $\fu$ and $\wh \fp$.
That is, the samples $(X_i,Y_i)_{i\ge1}$ are drawn according to $\fp$ instead of $\fu$, and the acceptance condition is
\eq{
U_i \le r\cdot\frac{\fu(X_i,Y_i)}{\wh \fp(X_i,Y_i)}
}
instead of \eqref{accept_condition}.
Denoting the raw samples (from $\wh \fp$) by $\cS_N = \{(X_1,Y_1),\ldots,(X_N,Y_N)\}$ and the accepted samples (from $\fu$) by $\cU_n  = \{(X_{\tau_1},Y_{\tau_1}),\ldots,(X_{\tau_n},Y_{\tau_n})\}$, we replace \eqref{df7b3c} with
\eeq{ \label{h8vylp}
\E\sL_N \ge \E(\sL_n^\unif\one\{\tau_n\le N\})
\ge \E\sL_n^\unif - n\P(\tau_n>N).
}
Assuming $\delta<1$ (otherwise the lower bound in \eqref{mean_upper_lower} is trivial), we may choose
\eq{
n = \bigg\lfloor\frac{r}{R}(1-\delta)N\bigg\rfloor.
}
For $N$ sufficiently large, we bound $\E\sL_n^\unif$ from below using \eqref{n0_choice}, bound $n\P(\tau_n > N)$ from above exactly as before, and use these two estimates in \eqref{h8vylp}.

\medskip

\noindent Part~\ref{prop_general_likelihood}:
For uniform samples, it is known (e.g.\ see \cite[Theorem~1]{barany_rote_steiger_zhang00} or \cite[display~(5.3)]{barany99}) that
\eeq{ \label{0o0nc}
\P(\{\nA,\nB\}\cup\cU_n\text{ is in convex position}) = \frac{2^n}{n!(n+1)!}.
}
Let $\nP_1,\ldots,\nP_n$ denote $n$ independent samples from the density $\wh\fp$.
By our assumption \eqref{ratio_assumption}, the density function of $(\nP_1,\ldots,\nP_n)$ with respect to product uniform measure on $\triangle$ is bounded from above by $R^n$ and from below by $r^n$.
Therefore, \eqref{1o0nc} follows from \eqref{0o0nc}.
\end{proof}

\begin{remark}[Another route] \label{rem_alternative}
Within Section~\ref{sec_upper_bound}, we give a short alternative proof of \eqref{0o0nc} in passing.
Specifically, let $\cU_n^\square$ denote $n$ uniform samples from the unit square $[0,1]^2$.
Then start at the paragraph containing \eqref{exp_rep}, and read until \eqref{qbd7v}.
\qedrem
\end{remark}

\begin{proof}[Proof of Proposition~\ref{prop_general}\ref{prop_general_tail}]
The following argument adapts the strategy of \cite[Section~4]{ambrus_barany09}.
By part~\ref{prop_general_likelihood} (specifically the upper bound in \eqref{1o0nc}) and a union bound over subsets of size $k$, we have
\eq{
\P(\sL_n \ge k) 
\le {n \choose k}\frac{2^k}{k!(k+1)!}R^k
\le \frac{n^k}{k!}\cdot\frac{2^k}{(k!)^2}R^k
= \frac{(2R n)^k}{(k!)^3}.
}
Next use the standard estimate $k! \ge (\frac{k}{\e})^k\sqrt{2\pi k}$ to obtain
\eq{
\P(\sL_n \ge k) \le \frac{(2R\e^3 n)^k}{k^{3k}(2\pi k)^{3/2}}
< \Big(\frac{2R\e^3 n}{k^3}\Big)^{k}.
}
Choose $k = \ceil{(4R\e^3 n)^{1/3}+t}$ to obtain \eqref{tail_deep}.
\end{proof}

To obtain the concentration inequality in part~\ref{prop_general_concentration}, we will invoke the following consequence of Talagrand's convex distance inequality \cite[Section~6]{talagrand96}.

\begin{lemma}[Talagrand's concentration inequality, \textup{\cite[Theorem~2.29]{janson_luczak_rucinski00}}] \label{lem_talagrand}
Let $Z_1,\ldots,Z_n$ be independent random variables taking values in $\cZ_1,\ldots,\cZ_n$, respectively.
Let $f\colon\cZ_1\times\cdots\times\cZ_n\to\R$ be a measurable map satisfying the following two assumptions:
\begin{enumerate}[label=\textup{(\roman*)}]
\item \label{talagrand_lipschitz} (Lipschitz with respect to Hamming distance) 
For every $j\in\{1,\ldots,n\}$, we have
\eq{ 
|f(z_1,\ldots,z_n) - f(z_1',\ldots,z_n')| \le 1 \quad \text{whenever $z_i = z_i'$ for every $i\ne j$.}
}
\item \label{talagrand_certifiable} (Lower certifiable) 
There exists a map $\psi\colon\R\to[0,\infty)$ with the following property.
If $f(z_1,\ldots,z_n)\ge r$, then there exists $I\subseteq\{1,\ldots,n\}$ such that $\#I\le\psi(r)$ and
\eq{
f(z_1',\dots,z_n') \ge r \quad \text{whenever $z_i = z_i'$ for every $i\in I$.}
}
\end{enumerate}
Let $Z = f(Z_1,\ldots,Z_n)$.
For every $r\in\R$ and $s\ge0$, we have
\eq{
\P(Z\le r-s)\P(Z\ge r) \le \exp\Big(\frac{-s^2}{4\psi(r)}\Big).
}
In particular, for any median $m$ of $Z$, we have
\eq{
\P(Z\le m-s) \le 2\exp\Big(\frac{-s^2}{4\psi(m)}\Big) \quad \text{and} \quad
\P(Z\ge m+s) \le 2\exp\Big(\frac{-s^2}{4\psi(m+s)}\Big).
}
\end{lemma}

\begin{proof}[Proof of Proposition~\ref{prop_general}\ref{prop_general_concentration}]
We claim that the random variable $Z = \sL_n$ fits the setting of Lemma~\ref{lem_talagrand} with the following choices:
\begin{itemize}
\item $Z_1,\dots,Z_n$ are i.i.d.\ samples from $\triangle$, with density function $\wh\fp$.
\item $f$ is the map $\triangle^n\to\Z_{\ge1}$ that returns the length of the longest convex chain:
\eq{
f(z_1,\dots,z_n) = \max\{\#I:\,I\subseteq\{1,\ldots,n\} \text{ and $\{\nA,\nB\}\cup\{z_i:\,i\in I\}$ is in convex position}\}.
}
\item $\psi(r) = r$.
\end{itemize}
Indeed, the Lipschitz condition~\ref{talagrand_lipschitz} is trivial, and the certifiability condition~\ref{talagrand_certifiable} holds by letting $I$ be any maximizer.
Therefore, the conclusion of Lemma~\ref{lem_talagrand} gives
\eeq{ \label{ybn84c}
\P\bigl(|\sL_n-m_n|\ge s\bigr)
\le
4\exp\Bigl(\frac{-s^2}{4(m_n+s)}\Bigr) \quad \text{for every $s\ge0$},
}
where $m_n$ is a median of $\sL_n$.
From \eqref{tail_deep} with $t=0$, we infer that
\eeq{ \label{median_upper_all}
m_n < \fm n^{1/3} \quad \text{where} \quad \fm = (4R\e^3)^{1/3}.
}
Ultimately we wish to replace the median with the mean, so we note that
\eq{
|\E\sL_n - m_n|
\le \E|\sL_n - m_n|
&\stackrefp{ybn84c}{=} \int_0^\infty \P\big(|\sL_n - m_n|>s\big)\ \dd s \\
&\stackref{ybn84c}{\le}4\int_0^\infty \exp\Big(\frac{-s^2}{4(m_n+s)}\Big)\ \dd s \\
&\stackrefp{ybn84c}{\le}4\int_0^{\infty}\exp\Big(\frac{-s^2}{8m_n}\Big)\ \dd s + 4\int_{m_n}^\infty \exp\Big(\frac{-s^2}{8s}\Big)\ \dd s \\
&\stackrefp{ybn84c}{=} \sqrt{32\pi m_n} + 32\exp\Big(\frac{-m_n}{8}\Big).
}
Since $\sqrt{32\pi m_n} \ge \sqrt{32\pi} \ge 10$ while $32\exp(-m_n/8) \le 32\exp(-1/8) \le 30$, we in fact have 
\eeq{ \label{mean_median_gap}
|\E\sL_n - m_n|
\le 4\sqrt{32\pi m_n}
\stackref{median_upper_all}{\le} 4\sqrt{32\pi}\cdot\sqrt{\fm}n^{1/6}
\le 41\sqrt{\fm}n^{1/6}.
}
Fix $t_0>0$ and $\beta\in(\tfrac16,\tfrac13]$.
Since $\beta>\frac16$, we can consider the smallest positive integer $n_0=n_0(t_0,\beta)$ such that
\eq{
\frac{t_0}{2}(n_0+1)^\beta \ge 41\sqrt{\fm} (n_0+1)^{1/6}.
}
Since $n\mapsto n^{\beta-1/6}$ is nondecreasing, it follows that
\eeq{ \label{bo24ch}
\text{for every $n>n_0$, we have} \quad
\frac{t_0}{2} n^\beta 
\ge 41\sqrt{\fm} n^{1/6}
\stackref{mean_median_gap}{\ge} |\E\sL_n - m_n|.
}
We now consider three cases:
\begin{enumerate}[label=\textup{\arabic*.}]
\item Suppose $n>n_0$ and $t\ge t_0$.
We then have
\eq{
\mathbb P\Bigl(\bigl|\sL_n-\mathbb E \sL_n\bigr| \ge tn^\beta\Bigr)
&\stackref{bo24ch}{\le}
\mathbb P\Bigl(\bigl|\sL_n-m_n\bigr| \ge (t-\tfrac{ t_0}{2})n^\beta\Bigr) \\
&\stackref{ybn84c}{\le}
4\exp\Bigl(
\frac{-(t-\tfrac{ t_0}{2})^2n^{2\beta}}{4(m_n+ (t-\tfrac{ t_0}{2})n^\beta)}
\Bigr) \\
&\stackref{median_upper_all}{\le} 4\exp\Bigl(
\frac{-(t/2)^2n^{2\beta}}{4(\fm n^{1/3}+ t n^\beta)}
\Bigr) \\
&\stackrefp{median_upper_all}{\le} 4\exp\Big(\frac{-tn^{2\beta-1/3}}{16(\fm/t_0+1)}\Big),
}
where the last inequality uses the assumptions $t\ge t_0$ and $\beta \le \tfrac13$.

\item Suppose $n\le n_0$ and $t \le n_0^{1-\beta}$.
By minimality of $n_0$ we have
\eeq{ \label{win8xh}
\frac{t_0}{2}n_0^\beta < 41\sqrt{\fm} n_0^{1/6}.
}
Therefore,
\eq{
\P\Bigl(\bigl|\sL_n-\mathbb E \sL_n\bigr| \ge tn^\beta\Bigr)
\stackrefp{win8xh}{\le} 1
&\stackrefp{win8xh}{\le} \exp\Big(\frac{tn_0^{2\beta-1/3}}{16(\fm/t_0+1)}\Big)\exp\Big(\frac{-tn^{2\beta-1/3}}{16(\fm/t_0+1)}\Big) \\
&\stackref{win8xh}{\le} \exp\Big(\frac{t(41\sqrt{\fm}\cdot 2/t_0)^2}{16(\fm/t_0+1)}\Big)\exp\Big(\frac{-tn^{2\beta-1/3}}{16(\fm/t_0+1)}\Big) \\
&\stackrefp{win8xh}{\le} \exp\Big(\frac{n_0^{1-\beta}\cdot 41^2}{4 t_0}\Big)\exp\Big(\frac{-tn^{2\beta-1/3}}{16(\fm/t_0+1)}\Big).
}

\item Suppose $n\le n_0$ and $t > n_0^{1-\beta}$.
Then $tn^\beta > n_0^{1-\beta}n^\beta \ge n$, hence
\eq{
\P\Bigl(\bigl|\sL_n-\mathbb E \sL_n\bigr| \ge tn^\beta\Bigr) = 0.
}

\end{enumerate}
To cover all three cases, we set $C = \max\{4,\exp(n_0^{1-\beta}\cdot 41^2/(4 t_0)\}$ and $\fa = \frac{1}{16(\fm/t_0+1)}$, so that \eqref{concentration_ineq} holds for every $n\ge1$ and $t\ge t_0$.
\end{proof}

\begin{proof}[Proof of Proposition~\ref{prop_general}\ref{prop_general_convergence}]
Since $\gamma>\tfrac16$, we can choose some $\beta\in(\tfrac16,\gamma\wedge\tfrac13)$.
In the steps below that invoke the concentration inequality \eqref{concentration_ineq}, we are using part~\ref{prop_general_concentration} with $t_0=1$.

For every $n\ge1$, we have
\eq{
\P\Big(\frac{|\sL_n - \E\sL_n|}{n^\gamma} \ge n^{-(\gamma-\beta)}\Big)
= \P\big(|\sL_n - \E\sL_n| \ge n^\beta\big)
\stackref{concentration_ineq}{\le} C \exp\bigl(-\fa n^{2\beta-\frac13} \bigr).
}
By choice of $\beta$, we have $n^{-(\gamma-\beta)}\to0$ as $n\to\infty$, and also $\sum_{n\ge1}\exp(-\fa n^{2\beta-1/3})<\infty$.
Therefore, the first Borel--Cantelli lemma implies the almost sure convergence claimed in \eqref{3hno2x}.
For $L^p$ convergence, observe that
\eq{
\E\Big(\frac{|\sL_n - \E\sL_n|^p}{n^{\gamma p}}\Big)
&\stackrefp{concentration_ineq}{=} \frac{1}{n^{\gamma p}}\int_0^\infty pu^{p-1}\P\Big(|\sL_n - \E\sL_n| \ge u\Big)\ \dd u \\
&\stackrefp{concentration_ineq}{=} \frac{1}{n^{p(\gamma-\beta)}}\int_0^\infty pt^{p-1}\P\Big(|\sL_n - \E\sL_n| \ge tn^\beta\Big)\ \dd t \\
&\stackrefp{concentration_ineq}{\le} \frac{1}{n^{p(\gamma-\beta)}}\Big[1+\int_1^\infty pt^{p-1}\P\Big(|\sL_n - \E\sL_n| \ge tn^\beta\Big)\ \dd t\Big] \\
&\stackref{concentration_ineq}{\le} \frac{1}{n^{p(\gamma-\beta)}}\Big[1+\int_1^\infty pt^{p-1}\cdot C \exp\bigl(-\fa tn^{2\beta-\frac13} \bigr)\ \dd t\Big].
}
Since $\gamma-\beta>0$ and $2\beta-\tfrac13 > 0$, the final line tends to $0$ as $n\to\infty$.
\end{proof}

\begin{remark}[Concentration is very general] \label{rem_concentration}
Only two facts about $\sL_n$ were used in the proof of Proposition~\ref{prop_general}\ref{prop_general_concentration}.
First, $\sL_n$ fits the setting of Lemma~\ref{lem_talagrand} with certificate function $\psi(r)=r$, so \eqref{ybn84c} holds.
Second, $\sL_n$ has a median which is $O(n^{1/3})$, specifically \eqref{median_upper_all} holds.
The rest of the proof used only these two starting points.
Furthermore, the proof of Proposition~\ref{prop_general}\ref{prop_general_convergence} used only the result from part~\ref{prop_general_concentration}.
\qedrem
\end{remark}

\section{Lower bounds} \label{sec_lower_bound}

The goal of this section is to prove the following proposition, which concerns convex chains that are close to the graph of a given function $f$.
Define the following random variable, which measures the longest convex chain that remains $\eps$-close to $f$:
\eeq{ \label{Lnf_def}
\sL_n(f,\eps) \coloneqq \max\Big\{\#\cC:\, \text{$\cC\subseteq\cS_n$, $\cC$ is a convex chain, $\displaystyle\max_{(x,y)\in\cC} |f(x)-y|<\eps$}\Big\}.
}
We also define the following enlargement of $J(f)$:
\eeq{ \label{Jfeps_def}
J(f,\eps) \coloneqq \sup\{J(g):\,\text{$g\in\cF$ and $\|g-f\|_{[0,1]}<\eps$}\}.
}

\begin{proposition}[Lower bounds] \label{prop_lower_bound}
Assume \eqref{p_cont} and \eqref{p_lower}.
Then the following statements hold for every $f\in\cF$ and $\eps>0$.
\begin{enumerate}[label=\textup{(\alph*)}]

\item \label{prop_lower_bound_typical} \textup{(Typical convex chain near $f$)}
Let $\alpha$ be the constant from Theorem~\ref{thm_uniform_known}. Then
\eq{
\liminf_{n\to\infty} \frac{\sL_n(f,\eps)}{n^{1/3}} \ge \frac{\alpha}{2}J(f,\eps) \quad \text{almost surely.}
}

\item \label{prop_lower_bound_full} \textup{(Full convex chain near $f$)}
We have
\eq{
\liminf_{n\to\infty}\frac{1}{n}\log\Big[\frac{(3n)!}{n!}\P\big(\sL_n(f,\eps)=n\big)\Big] \ge \log\Big(\frac{27 J(f,\eps)^3}{4}\Big).
}

\end{enumerate}
\end{proposition}

In the proof of part~\ref{prop_lower_bound_full}, we will use the following lemma about multinomial distributions.
Roughly speaking, it says that every value within a fixed distance of the mode occurs with a probability that is comparable to that of the mode.
It is a simple application of the local central limit theorem, and we omit the proof.

\begin{lemma}[Multinomial lower bound] \label{lem_multinomial}
Fix $\eps_1,\eps_2>0$ and an integer $K\ge2$. Let $(M_1, M_2, \ldots, M_K)$ have the multinomial distribution with $m$ trials and probability parameters $(q_1, q_2, \ldots, q_K)$, where $q_k\ge\eps_1$ for every $k$. 
Assume $m_1,\dots,m_K$ are nonnegative integers that sum to $m$, and
\eq{
|m_k - mq_k| \le 1/\eps_2 \quad \text{for every $k$.}
}
Then there is a constant $c = c(K,\eps_1,\eps_2) > 0$ not depending on $m,m_1,\dots,m_K$, such that 
\eq{
\P(M_1=m_1,\ldots,M_K=m_K) \ge \frac{c}{m^{(K-1)/2}}.
}
\end{lemma}

\begin{proof}[Proof of Proposition~\ref{prop_lower_bound}]
Fix $h\in\cF$ and $\eps>0$.
We must show
\eeq{ \label{10npl4}
\liminf_{n\to\infty} \frac{\sL_n(h,\eps)}{n^{1/3}} \ge \frac{\alpha}{2}J(h,\eps) \quad \text{almost surely},
}
as well as
\eeq{ \label{t5c6b}
\liminf_{n\to\infty}\frac{1}{n}\log\Big[\frac{(3n)!}{n!}\P\big(\sL_n(h,\eps)=n\big)\Big] \ge \log\Big(\frac{27 J(h,\eps)^3}{4}\Big).
}
Let $\delta\in\big(0,\tfrac13J(h,\eps)\big)$ be arbitrary.
By the definition of $J(h,\eps)$ in \eqref{Jfeps_def}, there exists $g\in\cF$ such that
\begin{align}
\label{2ie9nc}
\|g-h\|_{[0,1]}<\eps
\qquad \text{and} \qquad
 J(g) &\ge J(h,\eps) - \delta.
\end{align}
Recall the constant $\fC$ from \eqref{p_upper}.
Having fixed $\delta$, we choose $\eta>0$ small enough that
\eeq{ \label{377bx3}
4(\fC\eta)^{1/3} \le \delta.
}
By \eqref{2ie9nc} and Lemma~\ref{lem_smooth_f}, there exists $f\in\cF$ such that the following three statements hold:
\begin{subequations} \label{wknvh6}
\begin{gather}
\|f - g\|_{[0,1-\eta]} + \|g-h\|_{[0,1]} < \eps, \label{qjb7w} \\
\text{$f$ is smooth with $f''(x)>0$ for all $x\in(0,1)$,} \label{kgjb7} \\
J(f)\ge J(h,\eps)-2\delta. \label{r3xby6}
\end{gather}
\end{subequations}
Our choices of $\eta$ and $f$ ensure the following:
\eeq{ \label{small_J_loss}
&J(h,\eps) - \int_\eta^{1-\eta}\big[f''(x)\cdot\fp(x,f(x))\big]^{1/3}\ \dd x \\
&\stackref{r3xby6}{\le} 2\delta + J(f) - \int_\eta^{1-\eta}\big[f''(x)\cdot\fp(x,f(x))\big]^{1/3}\ \dd x \\
&\stackrefp{r3xby6}{=} 2\delta + \int_0^\eta\big[f''(x)\cdot\fp(x,f(x))\big]^{1/3}\ \dd x+
\int_{1-\eta}^1 \big[f''(x)\cdot\fp(x,f(x))\big]^{1/3}\ \dd x \\
&\stackrefpp{p_upper}{r3xby6}{\le} 2\delta+(2\fC)^{1/3}\Big(\int_0^\eta f''(x)^{1/3}\ \dd x + \int_{1-\eta}^1f''(x)^{1/3}\ \dd x\Big) \\
&\stackrefpp{eq_unif_case_a}{r3xby6}{\le} 2\delta + (2\fC)^{1/3}\big((4\eta)^{1/3}+(4\eta)^{1/3}\big)
\stackref{377bx3}{\le} 3\delta.
}

Now divide the interval $[\eta,1-\eta]$ into subintervals $[x_1,x_2], [x_2,x_3], \ldots, [x_{L},x_{L+1}]$ of equal length, meaning
\eq{
x_\ell \coloneqq \eta + \frac{\ell-1}{L}(1-2\eta), \quad \ell\in\{1,\ldots,L+1\}.
}
We choose $L$ large enough to satisfy certain conditions detailed below.
For $\ell\in\{1,\ldots,L\}$, consider the tangency triangle $\triangle_\ell \coloneqq \Tri_f(x_\ell,x_{\ell+1})$ from Definition~\ref{def_triangle}. 
This triangle is nondegenerate because of the strict convexity from \eqref{kgjb7}.
It has two vertices on the graph $f$, namely
\eq{
\nA_\ell \coloneqq (x_\ell,f(x_\ell))
\qquad \text{and} \qquad
\nB_\ell \coloneqq \nA_{\ell+1}=(x_{\ell+1},f(x_{\ell+1})).
}
Define the quantities
\eq{
p_\ell \coloneqq \int_{\triangle_\ell}\fp(x,y)\ \dd x\, \dd y,
\qquad
r_\ell \coloneqq \frac{\inf_{(x,y)\in\triangle_\ell}\fp(x,y)}{p_\ell/\Area(\triangle_\ell)},
\qquad
R_\ell \coloneqq \frac{\sup_{(x,y)\in\triangle_\ell}\fp(x,y)}{p_\ell/\Area(\triangle_\ell)},
}
and note that $r_\ell\ge\fc/\fC$ and $R_\ell\le\fC/\fc$ by \eqref{p_upper} and \eqref{p_lower}.
Since $f$ is nondecreasing, the triangle $\triangle_\ell$ lies within the rectangle with corners $\nA_\ell$ and $\nB_\ell$ (see Figure~\ref{fig_tangency_triangle} for a visualization): 
\eeq{ \label{28ch5}
\triangle_\ell \subset [x_\ell,x_{\ell+1}] \times [f(x_\ell),f(x_{\ell+1})].
}
Since $x_{\ell+1}-x_\ell = \frac{1-2\eta}{L}$ and $f$ is continuous, the diameter of the rectangle on the right-hand side vanishes as $L\to\infty$, uniformly in $\ell$.
Therefore, the diameter of $\triangle_\ell$ vanishes as $L\to\infty$, uniformly in $\ell$.
Since $\fp$ is continuous by assumption \eqref{p_cont}, it follows that
\eeq{ \label{unk3c}
\lim_{L\to\infty}\sup_{\ell\in\{1,\ldots,L\}}\sup_{(x,y)\in\triangle_\ell}\Big|\frac{p_\ell}{\Area(\triangle_\ell)}-\fp(x,y)\Big| = 0.
}
Since $\fp$ is bounded away from $0$ because of \eqref{p_lower}, \eqref{unk3c} implies
\eq{
\lim_{L\to\infty}\sup_{\ell\in\{1,\ldots,L\}}\sup_{(x,y)\in\triangle_\ell}\Big|\frac{p_\ell}{\Area(\triangle_\ell)\fp(x,y)} - 1\Big| = 0.
}
Because of this limit, we may choose $L$ large enough that the following inequalities hold for every $\ell\in\{1,\ldots,L\}$:
\begin{gather}
p_\ell \ge (1-\delta)\Area(\triangle_\ell)\cdot\fp(x_\ell,f(x_\ell)), \label{283jbm} \\
r_\ell \ge 1 - \delta \quad \text{and} \quad R_\ell \le (1-\delta)^{-1}. \label{r_ell_small}
\end{gather}
In addition, in light of \eqref{qjb7w}, we may use continuity of $f$ to assume $L$ is large enough that for every $\ell$, we have
\eeq{ \label{cihoo1}
f(x_{\ell+1}) - f(x_\ell) < \eps - \|f - h\|_{[0,1-\eta]}.
}
The graph of $f\big|_{[x_\ell,x_{\ell+1}]}$ remains inside $\triangle_\ell$ (see \eqref{graph_in_triangle}), which justifies the first implication below:
\eeq{ \label{b7n0o1c}
(x,y)\in\triangle_\ell \quad 
&\stackrefp{cihoo1}{\implies} \quad (x,y)\in\triangle_\ell \text{ and } (x,f(x))\in\triangle_\ell \\
&\stackrefpp{28ch5}{cihoo1}{\implies} \quad |f(x)-y| \le f(x_{\ell+1}) - f(x_\ell) \\
&\stackref{cihoo1}{\implies} \quad |f(x)-y| < \eps - \|f - h\|_{[0,1-\eta]} \\
&\stackrefp{cihoo1}{\implies} \quad |h(x)-y| < \eps.
}
In addition, since $f''$ is continuous and bounded away from $0$ on $[\eta,1-\eta]$ thanks to \eqref{kgjb7}, we may assume $L$ is large enough that
\eq{
(1-\delta)f''(x_\ell)\le f''(u)\le (1-\delta)^{-1}f''(x_\ell) \quad \text{for every $u\in[x_\ell,x_{\ell+1}]$.}
}
Lemma~\ref{lem_area}\ref{lem_area_b} then gives the following inequality:
\eeq{ \label{2ub7x7}
\Area(\triangle_\ell) 
\ge (1-\delta)^3f''(x_\ell)\cdot\tfrac18(x_{\ell+1}-x_\ell)^3.
}
Combining \eqref{283jbm} and \eqref{2ub7x7} yields
\eeq{ \label{30n0x}
p_\ell 
&\ge (1-\delta)^4\cdot \tfrac18f''(x_\ell)\cdot\fp(x_\ell,f(x_\ell))\cdot(x_{\ell+1}-x_\ell)^3.
}
Finally, we assume $L$ is large enough that the following Riemann sum approximation holds:
\eeq{ \label{28b11e}
\sum_{\ell=1}^{L} \big[f''(x_\ell)\cdot\fp(x_\ell,f(x_\ell))\big]^{1/3}(x_{\ell+1}-x_\ell)
\ge (1-\delta)\int_\eta^{1-\eta} \big[f''(x)\cdot\fp(x,f(x))\big]^{1/3}\ \dd x.
}
Here we have used the continuity of $x\mapsto f''(x)\cdot\fp(x,f(x))$, which is guaranteed by \eqref{kgjb7} and \eqref{p_cont}.
We now have
\eeq{ \label{qvb3cs}
\sum_{\ell=1}^{L} p_\ell^{1/3} 
&\stackref{30n0x,28b11e}{\ge} \frac{(1-\delta)^{7/3}}{2}\int_\eta^{1-\eta}\big[f''(x)\cdot\fp(x,f(x))\big]^{1/3}\ \dd x \\
&\stackrefpp{small_J_loss}{30n0x,28b11e}{\ge}\frac{(1-\delta)^{7/3}}{2}\Big(J(h,\eps)-3\delta\Big).
}

Consider the length of the longest convex chain inside the triangle $\triangle_\ell$:
\eq{
\sL_{n,\ell}
\coloneqq \max\big\{\#\cC:\, \text{$\cC\subseteq \cS_n\cap\Interior(\triangle_\ell)$ and $\cC\cup\{\nA_\ell,\nB_\ell\}$ is in convex position}\big\}
}

\begin{claim}[Concatenation] \label{claim_concatenate}
$\sL_n(h,\eps) \ge \sum_{\ell=1}^{L}\sL_{n,\ell}$.
\end{claim}

\begin{proofclaim}
Let $\cC_1,\ldots,\cC_{L}$ be maximizing sets for $\sL_{n,1},\ldots,\sL_{n,L}$, respectively.
These sets are disjoint since $\Interior(\triangle_1),\ldots,\Interior(\triangle_{L})$ are disjoint.
Denote the disjoint union by $\cC_\mathsf{all} = \biguplus_{\ell=1}^{L} \cC_\ell$.
By $L-1$ many applications of Lemma~\ref{lem_concatenation}, the set $\{\nA_1,\nB_L\}\cup\cC_\mathsf{all}$ is in convex position.
By two more applications of Lemma~\ref{lem_concatenation}---each time concatenating with an empty set---we obtain that $\{(0,f(0)),(1,f(1))\}\cup\cC_{\mathsf{all}}$ is in convex position.
By Lemma~\ref{lem_position_slopes}, this means that if we label the elements of $\cC_\mathsf{all}$ in increasing order of their horizontal coordinates, as $\nP_1,\ldots,\nP_m$, then
\eeq{ \label{3odncp}
\Slope\big((0,f(0)),\nP_1\big)<\Slope(\nP_1,\nP_2)<\cdots<\Slope(\nP_{m-1},\nP_m)<\Slope\big(\nP_m,(1,f(1))\big).
}
Since $f\in\cF$, we have $f(0)=0$ and $f(1)\le 1$, so \eqref{3odncp} implies
\eq{
\Slope\big((0,0),\nP_1\big)<\Slope(\nP_1,\nP_2)<\cdots<\Slope(\nP_{m-1},\nP_m)<\Slope\big(\nP_m,(1,1)\big).
}
Therefore, $\cC_\mathsf{all}$ is a convex chain.
Since $\cC_\ell\subset\triangle_\ell$, \eqref{b7n0o1c} ensures $|h(x)-y|<\eps$ for all $(x,y)\in\cC_\mathsf{all}$.
This makes $\cC_\mathsf{all}$ a candidate for $\sL_n(h,\eps)$ as defined in \eqref{Lnf_def}, hence
\[
\sL_n(h,\eps) \ge \#\cC_\mathsf{all} = \sum_{\ell=1}^{L}\#\cC_\ell = \sum_{\ell=1}^{L}\sL_{n,\ell}.
\qedhere
\]
\end{proofclaim}

We are now ready to prove \eqref{10npl4}.
Denote the number of sample points that land in $\Interior(\triangle_\ell)$ by
\eeq{ \label{N_nell_def}
N_{n,\ell} \coloneqq \#\big(\cS_n\cap\Interior(\triangle_\ell)\big)\sim\mathsf{Binomial}(n,p_\ell).
}
By Hoeffding's inequality, we have
\eeq{ \label{3inplk}
\P\big(N_{n,\ell} \ge np_\ell - n^{2/3}\big)
\ge 1 - \exp(-2n^{1/3}).
}
Conditional on $N_{n,\ell}$, the random set $\cS_n\cap\triangle_\ell$ has the same law as $N_{n,\ell}$ independent samples from $\fp$ conditioned to be inside $\triangle_\ell$. %, and these sets are independent across $\ell$.
Therefore, Proposition~\ref{prop_general}\ref{prop_general_compare} (with $\wh\fp$ equal to the conditional density on $\triangle_\ell$) gives
\eq{
\E\givenk{\sL_{n,\ell}}{N_{n,\ell}} \ge (1-\delta)\Big(\frac{r_\ell}{R_\ell}\Big)^{1/3}\alpha N_{n,\ell}^{1/3}\one\{N_{n,\ell}\ge\fn\},
}
where $\fn$ depends only on $\fC/\fc$ and $\delta$.
Since $(1-\delta)(r_\ell/R_\ell)^{1/3}\le 1$, we can remove the indicator on the right-hand side by instead writing
\eq{
\E\givenk{\sL_{n,\ell}}{N_{n,\ell}} + \alpha\fn^{1/3} \ge 
(1-\delta)\Big(\frac{r_\ell}{R_\ell}\Big)^{1/3}\alpha N_{n,\ell}^{1/3}.
}
Meanwhile, Proposition~\ref{prop_general}\ref{prop_general_concentration} (with $t_0 =1$, $\beta=\tfrac14$, $t=n^{1/13}$) gives
\eq{
\P\givenk[\Big]{\sL_{n,\ell} \ge \E\givenk{\sL_{n,\ell}}{N_{n,\ell}} - n^{1/13}N_{n,\ell}^{1/4}}{N_{n,\ell}}
\ge 1 - C\exp(-\fa n^{1/13}),
}
where $C$ is a universal constant, and $\fa$ depends only on $\fC/\fc$.
Since $N_{n,\ell}\le n$, it follows from the two previous displays that
\eq{
\P\Big(\sL_{n,\ell} \ge (1-\delta)\Big(\frac{r_\ell}{R_\ell}\Big)^{1/3}\alpha N_{n,\ell}^{1/3} - \alpha\fn^{1/3} - n^{1/13}\cdot n^{1/4}\Big)
\ge 1 - C\exp(-\fa n^{1/13}).
}
By a union bound with \eqref{3inplk}, it follows that
\eq{
&\P\Big(\sL_{n,\ell} \ge (1-\delta)\Big(\frac{r_\ell}{R_\ell}\Big)^{1/3}\alpha\big(np_\ell-n^{2/3}\big)^{1/3} - \alpha\fn^{1/3} - n^{17/52}\Big) \\
&\ge 1 - C\exp(-\fa n^{1/13}) - \exp(-2n^{1/3}).
}
Using the first Borel--Cantelli lemma, we conclude that with probability one, we have
\eeq{  \label{2ih4d}
\sL_{n,\ell} \ge (1-\delta)\Big(\frac{r_\ell}{R_\ell}\Big)^{1/3}\alpha\big(np_\ell-n^{2/3}\big)^{1/3} - \alpha\fn^{1/3} - n^{17/52}
}
for all large $n$ and every $\ell$.
Now sum over $\ell$ and invoke Claim~\ref{claim_concatenate} to obtain the following almost sure statement:
\eq{
\liminf_{n\to\infty} \frac{\sL_n(h,\eps)}{n^{1/3}}
&\stackrefpp{2ih4d}{r_ell_small,qvb3cs}{\ge} (1-\delta)\alpha\sum_{\ell=1}^{L}\Big(\frac{r_\ell}{R_\ell}\Big)^{1/3}p_\ell^{1/3} \\
&\stackref{r_ell_small,qvb3cs}{\ge} \frac{(1-\delta)^{4}\alpha}{2}\Big(J(h,\eps)-3\delta\Big).
}
As $\delta>0$ is arbitrary, we can send $\delta\searrow0$ we obtain \eqref{10npl4}.

Our final task is to prove \eqref{t5c6b}.
Observe that the random vector $(N_{n,1},\ldots,N_{n,L},n-\sum_{\ell=1}^{L}N_{n,\ell})$ from \eqref{N_nell_def} has the multinomial distribution with $n$ trials and probability parameters $(p_1,\ldots,p_{L},1-\sum_{\ell=1}^{L}p_\ell)$.
Let us fix nonnegative integers $n_1,\ldots,n_{L}$ such that 
\begin{subequations} \label{nj4bk1}
\begin{gather}
\sum_{\ell=1}^{L} n_\ell = n, \label{3n9c4g} \\
\text{and} \quad |n_\ell - 3\bar q_\ell n| \le 1 \quad \text{for every $\ell$,} \quad \text{where} \quad
\bar q_\ell \coloneqq \frac{p_\ell^{1/3}}{3\sum_{k=1}^{L}p_k^{1/3}} \stackref{p_lower}{>}0. \label{q_ell_def}
\end{gather}
\end{subequations}
Because of \eqref{3n9c4g}, we have
\eeq{ \label{i8mb4p}
\P\bigg(\bigcap_{\ell=1}^{L}\{N_{n,\ell} = n_\ell\}\bigg)
= {n \choose n_1,\ldots,n_{L}}p_1^{n_1}\cdots p_{L}^{n_{L}}
&= n!\prod_{\ell=1}^{L}\frac{p_\ell^{n_\ell}}{n_\ell!}.
}
Denote the intersection event in \eqref{i8mb4p} by $E_n$.
Conditional on $E_n$, the random set $\cS_n\cap\triangle_\ell$ has the same law as $n_\ell$ independent samples from $\fp$ conditioned to be inside $\triangle_\ell$, and these sets are independent across $\ell$.
Therefore, Proposition~\ref{prop_general}\ref{prop_general_likelihood} (with $\wh\fp$ equal to the conditional density on $\triangle_\ell$) implies the first inequality below:
\eeq{ \label{fe8vx4}
\P\givenp[\Big]{\bigcap_{\ell=1}^{L}\{\sL_{n,\ell} = n_\ell\}}{E_n}
&\stackrefpp{1o0nc}{r_ell_small,3n9c4g}{\ge} \prod_{\ell=1}^{L} r_\ell^{n_\ell}\frac{2^{n_\ell}}{n_\ell!(n_\ell+1)!} \\
&\stackref{r_ell_small,3n9c4g}{\ge} (1-\delta)^{n}\frac{2^n}{(n+1)^{L}}\prod_{\ell=1}^{L}\frac{1}{(n_\ell!)^2}.
}
Denote the event in \eqref{fe8vx4} by
\[
F_n \coloneqq \bigcap_{\ell=1}^{L}\{\sL_{n,\ell}=n_\ell\}.
\]
On $F_n$, Claim~\ref{claim_concatenate} gives $\sL_n(h,\eps) \ge \sum_{\ell=1}^{L} n_\ell = n$, and the reverse inequality $\sL_n(h,\eps)\le n$ is trivial since there are $n$ sample points. 
Hence
\eeq{ \label{184lb0}
\P\big(\sL_n(h,\eps)=n\big) \ge \P(F_n).
}
Combining \eqref{i8mb4p} and \eqref{fe8vx4}, and then invoking the definition of $\bar q_\ell$ from \eqref{q_ell_def}, yields
\eq{
\P(F_n) 
&\ge (1-\delta)^{n}\frac{2^nn!}{(n+1)^{L}}\prod_{\ell=1}^{L} \frac{p_\ell^{n_\ell}}{(n_\ell!)^3} \\
&= (1-\delta)^{n}\frac{2^nn!}{(n+1)^{L}(3n)!}\Big(3\sum_{\ell=1}^{L}p_\ell^{1/3}\Big)^{3n}{3n\choose n_1,n_1,n_1,\ldots,n_{L},n_{L},n_{L}}\prod_{\ell=1}^{L} (\bar q_\ell)^{3n_\ell}.
}
Because of \eqref{nj4bk1}, we can apply the multinomial estimate from Lemma~\ref{lem_multinomial} with $\eps_1 \coloneqq \min_\ell \bar q_\ell$, $\eps_2 \coloneqq 1$, $m \coloneqq 3n$ trials, $K \coloneqq 3L$ types of outcomes, and parameters
\eq{
(q_1,q_2,q_3,\ldots,q_{3L-2},q_{3L-1},q_{3L}) &\coloneqq (\bar q_1, \bar q_1, \bar q_1, \ldots, \bar q_{L},\bar q_{L}, \bar q_{L}), \\
(m_1,m_2,m_3,\ldots,m_{3L-2},m_{3L-1},m_{3L}) &\coloneqq (n_1, n_1, n_1, \ldots, n_{L},n_{L}, n_{L}).
}
This application of Lemma~\ref{lem_multinomial} results in the first inequality below:
\eq{
\P(F_n) 
&\stackrefp{qvb3cs}{\ge} (1-\delta)^{n}\frac{2^nn!}{(n+1)^{L}(3n)!}\Big(3\sum_{\ell=1}^{L}p_\ell^{1/3}\Big)^{3n}\cdot\frac{c}{(3n)^{(3L-1)/2}} \\
&\stackref{qvb3cs}{\ge}(1-\delta)^{8n}\frac{(27/4)^nn!}{(3n)!}\Big(J(h,\eps)-3\delta\Big)^{3n}\cdot\frac{c}{(n+1)^L(3n)^{(3L-1)/2}},
}
where $c$ depends on $L$ and the value of $\eps_1$, but not on $n$.
It follows that
\eq{
\liminf_{n\to\infty}\frac{1}{n}\log\Big[\frac{(3n)!}{n!}\P(F_n)\Big] \ge 8\log(1-\delta) + \log(27/4) + 3\log\Big(J(h,\eps)-3\delta\Big).
}
As $\delta>0$ is arbitrary, we can send $\delta\searrow0$ to conclude
\eq{
\liminf_{n\to\infty}\frac{1}{n}\log\Big[\frac{(3n)!}{n!}\P(F_n)\Big] \ge \log(27/4) + 3\log J(h,\eps).
}
In light of \eqref{184lb0}, this proves \eqref{t5c6b}.
\end{proof}

\section{Upper bounds} \label{sec_upper_bound}

The goal of this section is to prove the following proposition, which complements the lower bounds from Proposition~\ref{prop_lower_bound}.
Parts~\ref{prop_upper_bound_far_typical} and \ref{prop_upper_bound_far_full} below concern convex chains that are ``far from'' every optimizer of $J$.
Recall from \eqref{L_neps_def} the random quantity $\sL_n^\far(\eps)$, which is measurable by application of Lemma~\ref{lem_measurable} stated further below.

\begin{proposition}[Upper bounds] \label{prop_upper_bound}
Assume \eqref{p_cont} and \eqref{p_lower}.
\begin{enumerate}[label=\textup{(\alph*)}]

\item \label{prop_upper_bound_typical} \textup{(Typical convex chain)}
Let $\alpha$ be the constant from Theorem~\ref{thm_uniform_known}. Then
\eeq{ \label{upper_almost_sure}
\limsup_{n\to\infty} \frac{\sL_n}{n^{1/3}} \le \frac{\alpha}{2}J_\star \quad \text{almost surely}.
}

\item \label{prop_upper_bound_far_typical} \textup{(Typical convex chains that are far from maximizers)}
For every $\eps>0$, there exists $\theta>0$ such that
\eeq{ \label{upper_far_almost_sure}
\limsup_{n\to\infty}\frac{\sL_n^\far(\eps)}{n^{1/3}} \le \frac{\alpha}{2}(J_\star-\theta) \quad \text{almost surely}.
}

\item \label{prop_upper_bound_full} \textup{(Full convex chain)}
We have
\eeq{ \label{upper_probability}
\limsup_{n\to\infty}\frac{1}{n}\log\Big[\frac{(3n)!}{n!}\P(\sL_n=n)\Big] \le \log\Big(\frac{27 J_\star^3}{4}\Big).
}

\item \label{prop_upper_bound_far_full} \textup{(Full convex chains that are far from maximizers)}
For every $\eps>0$, we have
\eeq{ \label{prop_upper_bound_eps}
\limsup_{n\to\infty}\frac{1}{n}\log\Big[\frac{(3n)!}{n!}\P\big(\sL_n^\far(\eps)=n\big)\Big] < \log\Big(\frac{27 J_\star^3}{4}\Big).
}
\end{enumerate}
\end{proposition}

Note that \eqref{upper_almost_sure} and \eqref{upper_probability} are actually equalities thanks to Proposition~\ref{prop_lower_bound}, but in this section we are just arguing the upper bounds.
The proof of Proposition~\ref{prop_upper_bound} will be presented in Section~\ref{subsec_upper_proof} after various lemmas are given in Section~\ref{subsec_upper_lemmas}.
Even without these lemmas, the proof is long, so we set the stage with an outline.

\subsection{Proof outline for Proposition~\ref{prop_upper_bound}}\label{subsec_upper_outline}
The argument will be presented in five steps.

\medskip

\noindent \textit{Step 1}. Convex chains lying very near the boundary lines $y=0$ and $x=1$ raise certain technical challenges.
Therefore, in this step we simply rule out this scenario by invoking several lemmas.
This allows us to assume that all relevant convex chains are $\rho$-regular (Definition~\ref{def_regular}), and that the sample set is $(\eta,\delta)$-good (Definition~\ref{def_good}).
Being $\rho$-regular means the convex chain does not enter the square $[1-\rho,1]\times[0,\rho]$, the key point being that the chain cannot hug the axes for too long.
Meanwhile, being $(\eta,\delta)$-good means that at most $\delta n$ many points lie in each of the strips $[0,1]\times[0,\eta]$ and $[1-\eta,1]\times[0,1]$.
So that these points effectively contribute nothing in Step 3, we send $\delta\searrow0$ as $n\to\infty$, gradually enough that $\eta$ can be kept as large as $n^{-c_1}$.
The constant $c_1>0$ is chosen small enough (see \eqref{parameter_selection}) to accommodate certain parameter choices in Step 2.

\medskip

\noindent \textit{Step 2}. Given a $\rho$-regular convex chain $\cC$ with $\#\cC\le n$, we will identify a function $\tilde f = \tilde f_{\cC,n}$ such that $\cC$ is contained in the tangency triangles of $\tilde f$, and these triangles have vanishing diameter as $n\to\infty$.
Therefore, the density function $\fp$ is nearly constant on each triangle.
Crucially, $\tilde f$ will belong to a finite collection $\wt \cF_n$ whose size grows sufficiently slowly in $n$, specifically $\#\wt\cF_n \le \e^{n^{c_2}}$ for some sufficiently small $c_2>0$; see \eqref{cardinality_upper}.
Unfortunately, $\tilde f$ will not be convex, but rather piecewise convex.
Because of the latter, it is still possible to define $J(\tilde f)$, and we will construct $\tilde f$ in such a way that $J(\tilde f)$ is nearly as large as possible under the constraint that $\tilde f$ is close to $\cC$.

\medskip 

\noindent \textit{Step 3}.
Using analysis similar to the proof of Proposition~\ref{prop_lower_bound}, we (essentially) show
\eq{
\limsup_{n\to\infty}\sup_{\cC}\Big[\frac{\#\cC}{\alpha n^{1/3}} - \frac{J(\tilde f_{\cC,n})}{2}\Big] \le 0,
}
where the supremum is over $\rho$-regular convex chains $\cC$ with $\#\cC\le n$.
This statement is captured by Claims~\ref{claim_deconcatenation} and \ref{claim_as_upper}, and goes toward proving parts~\ref{prop_upper_bound_typical} and \ref{prop_upper_bound_far_typical}.
For parts~\ref{prop_upper_bound_full} and \ref{prop_upper_bound_far_full}, we (essentially) show
\eq{
\lim_{n\to\infty}\sup_{\tilde f\in\wt\cF_n}\bigg[\frac{1}{n}\log\Big[\frac{(3n)!}{n!}\P\big(\text{$\cS_n$ is a convex chain and $\tilde f_{\cS_n,n} = \tilde f$})\Big]-\log\Big(\frac{27 J(\tilde f)^3}{4}\Big)\bigg] \le 0.
}
For the more precise version, see Claim~\ref{claim_prob_upper}.

\medskip

\noindent \textit{Step 4}. To account for the fact that $\tilde f$ is not convex, we identify a nearby function that \textit{is} convex.
Namely, we will show that for every $\tilde f\in\wt\cF_n$, there exists $g\in\cF$ such that $\|\tilde f-g\|_{[0,1]} = o(1)$ and $|J(\tilde f) - J(g)| = o(1)$, where $o(1)\to0$ as $n\to\infty$, uniformly in $\tilde f$.

\medskip

\noindent \textit{Step 5.}
Parts~\ref{prop_upper_bound_typical} and \ref{prop_upper_bound_full} of Proposition~\ref{prop_upper_bound} are straightforward consequences of Steps~3 and 4, by taking a union bound over $\tilde f\in\cF_n$.
Parts~\ref{prop_upper_bound_far_typical} and \ref{prop_upper_bound_far_full} require the following additional ingredient, which is spelled out more carefully in Claim~\ref{claim_distance_consequence}.
Suppose $\cC$ is $\eps$-far from every $h\in\argmax J$, and $\tilde f = \tilde f_{\cC,n}$.
Since $\|\tilde f - g\|_{[0,1]} = o(1)$ (by Step 4), and $\tilde f(x) = y + o(1)$ for every $(x,y)\in\cC$ (by Step 2), it follows that $\|g-h\|_{[0,1]}\ge\eps-o(1)$ for every $h\in\argmax J$.
Then Proposition~\ref{prop_far_from_optimizer} implies $J(g) \le J_\star - \theta$ for some $\theta>0$ depending only on $\eps$.
Since $J(\tilde f) = J(g) + o(1)$ (also by Step 4), the estimates from Step 3 lead to the desired results.

\subsection{Lemmas needed for the upper bounds} \label{subsec_upper_lemmas}

Our first lemma settles any concerns about measurability.
It is needed for parts~\ref{prop_upper_bound_far_typical} and \ref{prop_upper_bound_far_full} of Proposition~\ref{prop_upper_bound} to make sense, but it will not be explicitly invoked in any later proofs.

\begin{lemma}[Distance from optimizers is measurable] \label{lem_measurable}
Assume \eqref{p_cont}.
For every $k\ge1$, the map
\eq{
\cT^k\mapsto[0,1] \quad \text{given by} \quad (x_i,y_i)_{i=1}^k\mapsto \inf_{f\in\argmax J}\max_{i\in\{1,\ldots,k\}} |f(x_i)-y_i|
}
is measurable (with respect to the usual Borel $\sigma$-algebras).
\end{lemma}

\begin{proof}
Under the metric $d$ from \eqref{metric_def}, the set $\argmax J$ is closed (Proposition~\ref{prop_maximizers}).
Therefore, Proposition~\ref{prop_separable} guarantees a countable subset $\cD\subseteq\argmax J$ such that for every $f\in\argmax J$, there exists a sequence $(f_n)_{n\ge1}$ in $\cD$ satisfying
\eeq{ \label{3hon0c}
\lim_{n\to\infty} f_n(x) = f(x) \quad \text{for every $x\in[0,1]$.}
}
To verify the lemma, it suffices to show
\eeq{ \label{jn4ngy}
\inf_{f\in\argmax J}\max_{i\in\{1,\ldots,k\}} |f(x_i)-y_i|
= \inf_{f\in\cD}\max_{i\in\{1,\ldots,k\}} |f(x_i)-y_i|,
}
since the right-hand side of \eqref{jn4ngy} is clearly measurable.
The $\le$ direction of \eqref{jn4ngy} is trivial, so we just need to argue the $\ge$ direction.
To this end, consider any $(x_i,y_i)_{i=1}^k\in\cT^k$ and any $\eps>0$.
Choose $f\in\argmax J$ such that
\eeq{ \label{9nldx}
\text{LHS of \eqref{jn4ngy}} \ge \max_{i\in\{1,\ldots,k\}}|f(x_i) - y_i| - \eps.
}
Given $f$, choose a sequence $(f_n)_{n\ge1}$ in $\cD$ that satisfies \eqref{3hon0c}.
We now have
\eq{
\text{LHS of \eqref{jn4ngy}}
&\stackref{9nldx}{\ge} \max_{i\in\{1,\ldots,k\}}|f(x_i) - y_i| - \eps \\
&\stackref{3hon0c}{=} \lim_{n\to\infty}\max_{i\in\{1,\ldots,k\}}|f_n(x_i) - y_i| - \eps 
\ge (\text{RHS of \eqref{jn4ngy}}) - \eps.
}
Sending $\eps\searrow0$ completes the proof.
\end{proof}

The next two lemmas concern the following definition, which quantifies the property of not having too many points close to the boundary lines $x=1$ and $y=0$.

\begin{definition} \label{def_good}
Given $\eta>0$ and $\delta>0$, we say the sample set $\cS_n$ is \textit{$(\eta,\delta)$-good} if
\eeq{ \label{good_def}
\#\big(\cS_n\cap([1-\eta,1]\times[0,1])\big)\le\delta n \qquad \text{and} \qquad
\#\big(\cS_n\cap([0,1]\times[0,\eta])\big)\le\delta n.
}
Otherwise we say $\cS_n$ is \textit{$(\eta,\delta)$-bad}.
\qedrem
\end{definition}

By choosing $\eta$ sufficiently small, we can ensure $(\eta,\delta)$-badness is an exponentially unlikely event, even conditional on the (super-exponentially unlikely) event that $\cS_n$ is a convex chain.
The following lemma goes slightly further by allowing $a$ to vary, but we will only use the case $a = 1-\eta$ from \eqref{gx64h9_a}, and $a=0$ from \eqref{gx64h9_b}, which explains the ``in particular'' portion of the lemma.

\begin{lemma}[Points cannot be too dense] \label{lem_not_too_dense}
Assume \eqref{p_upper} and \eqref{p_lower}.
For every $\epsilon>0$ and $\delta>0$, there exists $\eta = \eta(\epsilon,\delta,\fC/\fc)>0$ small enough that
\begin{subequations} \label{gx64h9}
\begin{align} 
\label{gx64h9_a}
\sup_{a\in[0,1]}\P\Big(\parbox{0.36\textwidth}{\centering $\cS_n$ is a convex chain such that \\ $\#\big(\cS_n\cap([a,a+\eta]\times[0,1])\big)\ge \delta n$}\Big)
&\le \epsilon^n\cdot \P(\text{$\cS_n$ is a convex chain}) \\
\label{gx64h9_b}
\text{and} \quad
\sup_{a\in[0,1]}\P\Big(\parbox{0.36\textwidth}{\centering $\cS_n$ is a convex chain such that \\ $\#\big(\cS_n\cap([0,1]\times[a,a+\eta])\big)\ge \delta n$}\Big)
&\le \epsilon^n\cdot \P(\text{$\cS_n$ is a convex chain})
\end{align}
\end{subequations}
for all sufficiently large $n$.
In particular, for all large $n$ we have
\eq{
\P(\text{$\cS_n$ is an $(\eta,\delta)$-bad convex chain})
\le 2\epsilon^n\cdot\P(\text{$\cS_n$ is a convex chain}).
}
\end{lemma}

\begin{proof}
The two statements \eqref{gx64h9_a} and \eqref{gx64h9_b} are in symmetry with each other, by reflection over the line $x+y=1$.
More specifically, \eqref{gx64h9_b} is equivalent to \eqref{gx64h9_a} in the model where $\fp$ is replaced by $\wt\fp(x,y) = \fp(1-y,1-x)$.
So we just prove \eqref{gx64h9_a}.

First we reduce to the uniform case.
Denote the $n$ independent samples from $\fp$ by 
\eq{
\cS_n = \{(X_1,Y_1),\ldots,(X_n,Y_n)\}.
}
By \eqref{p_upper} and \eqref{p_lower}, the probability density function of $\big((X_1,Y_1),\ldots,(X_n,Y_n)\big)$ with respect to the product uniform measure on $\cT$ is bounded from above by $\fC^n$, and from below by $\fc^{n}$.
So let $\cU_n$ denote a collection of $n$ independent uniform samples from $\cT$.
For any event of the form $\{\text{$\cS_n$ has property $\sP$}\}$ that is measurable with respect to $\big((X_1,Y_1),\ldots,(X_n,Y_n)\big)$, we have the following upper and lower bounds:
\eq{
\fc^{n}\cdot\P(\text{$\cU_n$ has property $\sP$})
\le \P(\text{$\cS_n$ has property $\sP$}) 
\le \fC^n\cdot\P(\text{$\cU_n$ has property $\sP$}).
}
Consequently, it suffices to prove the lemma for $\cU_n$ in place of $\cS_n$ (and $\epsilon\cdot\fc/\fC$ in place of $\epsilon$).
%The first inequality can be used to lower bound the right-hand sides of \eqref{gx64h9}, while the second can be used to upper bound the left-hand sides of \eqref{gx64h9}. Then adjust parameters.

Next we make a further reduction.
Let $\cU_n^\square$ denote a set of $n$ independent uniform samples from the unit square $[0,1]^2$.
We will say that $\cU_n^\square$ is a convex chain if $\cU_n^\square\subset\cT$ and $\cU_n^\square\cup\{(0,0),(1,1)\}$ is in convex position.
If we condition on the event $\cU_n^\square\subset\cT$, then $\cU_n^\square$ has the same law as $\cU_n$, which justifies the second equality below:
\eq{
&\P(\{\text{$\cU_n^\square$ is a convex chain}\}\cap\{\text{$\cU_n^\square$ has property $\sP$}\}) \\
&=\P(\{\cU_n^\square\subset\cT\}\cap\{\text{$\cU_n^\square\cup\{(0,0),(1,1)\}$ is in convex position}\}\cap\{\text{$\cU_n^\square$ has property $\sP$}\}) \\
&=\P(\text{$\cU_n^\square\subset\cT$})\cdot\P(\{\text{$\cU_n$ is a convex chain}\}\cap\{\text{$\cU_n$ has property $\sP$}\}) \\
&= 2^{-n}\cdot\P(\{\text{$\cU_n$ is a convex chain}\}\cap\{\text{$\cU_n$ has property $\sP$}\}).
}
Therefore, it suffices to prove the lemma for $\cU_n^\square$ instead of $\cU_n$.

To this end, we realize $\cU_n^\square$ as follows.
Let $\xi_1,\ldots,\xi_{n+1},\zeta_1,\ldots,\zeta_{n+1}$ be $\mathrm{Exp}(1)$ random variables, and let $\pi$ be a uniformly random permutation on $n$ elements; assume all these quantities are independent.
Now define
\eeq{ \label{exp_rep}
U_i \coloneqq \frac{\xi_1+\cdots+\xi_i}{\xi_1+\cdots+\xi_{n+1}}, \qquad
V_i \coloneqq \frac{\zeta_1+\cdots+\zeta_i}{\zeta_1+\cdots+\zeta_{n+1}}.
}
The vectors $(U_1,\dots,U_n)$ and $(V_1,\dots,V_n)$ are independent realizations of the order statistics for $n$ independent uniform samples from $[0,1]$.
Therefore, the (unordered) set $\big\{(U_i,V_{\pi(i)}):\,i\in\{1,\ldots,n\}\big\}$ constitutes $n$ independent uniform samples from $[0,1]^2$, so we may assume
\eeq{ \label{unif_rep}
\cU_n^\square = \big\{(U_i,V_{\pi(i)}):\,i\in\{1,\ldots,n\}\big\}.
}

A necessary condition for $\cU_n^\square$ to be a convex chain is that it is a monotone chain. 
Since $U_1<U_2<\cdots<U_n$, this monotonicity means $V_{\pi(1)}< V_{\pi(2)}<\cdots<V_{\pi(n)}$, which forces $\pi$ to be the identity.
Therefore,
\eeq{ \label{aj3cb}
&\{\text{$\cU_n^\square$ is a convex chain}\} \\
&\stackrel{\hphantom{\mbox{\footnotesize (Lemma~\ref{lem_position_slopes})}}}{=} \{\pi=\id\}\cap\{\text{$\cU_n^\square\subset\cT$ and $\cU_n^\square\cup\{(0,0),(1,1)\}$ is in convex position}\} \\
&\stackrel{\mbox{\footnotesize (Lemma~\ref{lem_position_slopes})}}{=} \{\pi=\id\}\cap\Big\{\frac{V_1}{U_1}<\frac{V_2-V_1}{U_2-U_1}<\cdots<\frac{V_n-V_{n-1}}{U_n-U_{n-1}}<\frac{1-V_n}{1-U_n}\Big\} \\
&\stackrel{\parbox{\widthof{\footnotesize (Lemma~\ref{lem_position_slopes})}}{\centering\footnotesize\eqref{exp_rep}}}{=} \{\pi=\id\}\cap\Big\{\frac{\zeta_1}{\xi_1}<\frac{\zeta_2}{\xi_2}<\cdots<\frac{\zeta_n}{\xi_n}<\frac{\zeta_{n+1}}{\xi_{n+1}}\Big\}.
}
By independence of $\pi$ and exchangeability of $(\xi_i,\zeta_i)_{i=1}^{n+1}$, it follows that
\eeq{ \label{qbd7v}
\P(\text{$\cU_n^\square$ is a convex chain}) = \frac{1}{n!(n+1)!}.
}

In the sequel, $\llbrack a,b\rrbrack$ denotes the integer interval $\{a,a+1,\ldots,b\}$.

\begin{claim}[Controlling positive correlation] \label{claim_correlation}
Let $\cP_1,\dots,\cP_m$ be a partition of $\llbrack1,n+1\rrbrack$.
Let $\Sigma$ be the set of permutations $\sigma$ of $\llbrack1,n+1\rrbrack$ such that $\sigma(\cP_i)=\cP_i$ for every $i$.
Assume $\Phi\colon(\R^2)^{n+1}\to\R$ is measurable and invariant under the action of any $\sigma\in\Sigma$, meaning
\eeq{ \label{nbk4xd}
\Phi\big((x_i,y_i)_{i=1}^{n+1}\big) = \Phi\big((x_{\sigma(i)},y_{\sigma(i)})_{i=1}^{n+1}\big)  
\quad \text{$\forall$ $\sigma\in\Sigma$ and $(x_1,y_1),\ldots,(x_{n+1},y_{n+1})\in\R^2$.}
}
Then for any Borel set $B\subseteq\R$, we have
\eeq{ \label{bi4xh}
&\P\Big(\{\textup{$\cU_n^\square$ is a convex chain}\}\cap\big\{\Phi\big((\xi_i,\zeta_i)_{i=1}^{n+1}\big)\in B\big\}\Big) \\
&\le\P(\textup{$\cU_n^\square$ is a convex chain})\cdot\P\Big(\Phi\big((\xi_i,\zeta_i)_{i=1}^{n+1}\big)\in B\Big)\cdot{n+1 \choose \#\cP_1,\ldots,\#\cP_m}.
}
\end{claim}

\begin{proofclaim}
By independence of $\pi$, we have
\eeq{ \label{nk3pml}
&\P\Big(\Phi\big((\xi_i,\zeta_i)_{i=1}^{n+1}\big)\in B\Big)
= n!\cdot\P\Big(\big\{\Phi\big((\xi_i,\zeta_i)_{i=1}^{n+1}\big)\in B\big\}\cap\{\pi=\id\}\Big).
}
By definition of $\Sigma$, we have $\#\Sigma = (\#\cP_1)!\cdots(\#\cP_m)!$, hence
\eeq{ \label{Sigma_size}
{n+1 \choose \#\cP_1,\ldots,\#\cP_m} = \frac{(n+1)!}{\#\Sigma}.
}
Now multiply \eqref{qbd7v}, \eqref{nk3pml}, and \eqref{Sigma_size}, to obtain
\eeq{ \label{2jkm0s}
\text{RHS of \eqref{bi4xh}}
&= \P\Big(\big\{\Phi\big((\xi_i,\zeta_i)_{i=1}^{n+1}\big)\in B\big\}\cap\{\pi=\id\}\Big)\cdot(\#\Sigma)^{-1}.
}
We next partition the event on the right-hand side, according to the ordering of the ratios $\{\zeta_i/\xi_i:\,i\in\llbrack1,n+1\rrbrack\}$ from smallest to largest:
\eq{
&\P\Big(\big\{\Phi\big((\xi_i,\zeta_i)_{i=1}^{n+1}\big)\in B\big\}\cap\{\pi=\id\}\Big) \\
&\stackrefp{nbk4xd}{\ge} \sum_{\sigma\in\Sigma}\P\bigg(\big\{\Phi\big((\xi_i,\zeta_i)_{i=1}^{n+1}\big)\in B\big\}\cap\{\pi=\id\}\cap\Big\{\frac{\zeta_{\sigma(1)}}{\xi_{\sigma(1)}} < \cdots < \frac{\zeta_{\sigma(n+1)}}{\xi_{\sigma(n+1)}}\Big\}\bigg) \\
&\stackref{nbk4xd}{=}\sum_{\sigma\in\Sigma}\P\bigg(\big\{\Phi\big((\xi_{\sigma(i)},\zeta_{\sigma(i)})_{i=1}^{n+1}\big)\in B\big\}\cap\{\pi=\id\}\cap\Big\{\frac{\zeta_{\sigma(1)}}{\xi_{\sigma(1)}} < \cdots < \frac{\zeta_{\sigma(n+1)}}{\xi_{\sigma(n+1)}}\Big\}\bigg).
}
By exchangeability of $(\xi_i,\zeta_i)_{i=1}^{n+1}$ and independence of $\pi$, the probability on the final line is the same for every $\sigma\in\Sigma$.
Hence
\eeq{ \label{381m2x}
&\P\Big(\big\{\Phi\big((\xi_i,\zeta_i)_{i=1}^{n+1}\big)\in B\big\}\cap\{\pi=\id\}\Big) \\
&\stackrefp{qbd7v}{\ge}\#\Sigma\cdot\P\bigg(\big\{\Phi\big((\xi_i,\zeta_i)_{i=1}^{n+1}\big)\in B\big\}\cap\{\pi=\id\}\cap\Big\{\frac{\zeta_{1}}{\xi_{1}} < \cdots < \frac{\zeta_{n+1}}{\xi_{n+1}}\Big\}\bigg) \\
&\stackrefpp{aj3cb}{qbd7v}{=}\#\Sigma\cdot\P\Big(\big\{\Phi\big((\xi_i,\zeta_i)_{i=1}^{n+1}\big)\in B\big\}\cap\{\text{$\cU_n^\square$ is a convex chain}\}\Big).
}
Now insert \eqref{381m2x} into \eqref{2jkm0s} to obtain \eqref{bi4xh}.
\end{proofclaim}

We are now ready to prove \eqref{gx64h9_a} with $\cU_n^\square$ in place of $\cS_n$.
That is, we will find $\eta = \eta(\epsilon,\delta)>0$ such that
\eeq{ \label{20vjj3}
\sup_{a\in[0,1]}\P\Big(\parbox{0.36\textwidth}{\centering $\cU_n^\square$ is a convex chain such that \\ $\#\big(\cU_n^\square\cap([a,a+\eta]\times[0,1])\big)\ge \delta n$}\Big)
\le \epsilon^n\cdot \P(\text{$\cU_n^\square$ is a convex chain}).
}
The case $\epsilon\ge1$ is trivial, so assume $\epsilon\in(0,1)$. 
The case $\delta>1$ is also trivial, so assume $\delta \in (0,1].$
We assume $\eta$ is sufficiently small to satisfy certain conditions given later in the argument.

Recall the representation \eqref{unif_rep} of $\cU_n^\square$.
Define the random indices
\eq{ 
I_1 &\coloneqq \inf\big\{i\in\{1,\ldots,n\}:\, U_i\ge a\big\} \wedge (n+1), \\
I_2 &\coloneqq \inf\big\{i\in\{1,\ldots,n\}:\, U_i > a+\eta\big\} \wedge (n+1).
}
It follows from these definitions that for every $i\in\{1,\ldots,n\}$, we have
\eeq{ \label{oh8c4}
U_i\in[a,a+\eta] \quad \text{if and only if} \quad I_1 \le i \le I_2-1.
}
So we wish to show $I_2-I_1$ is very unlikely to exceed $\delta n$. 

Let $k_{\delta}$ be the largest integer such that $k_\delta\floor{\tfrac{\delta}{3}n}\le n$.
%In particular, $k_\delta\ge2$.
We assume $n$ is large enough that 
\eeq{ \label{2ion9}
k_\delta \le \ceil{3/\delta}.
}
Using the notation $i_k = k\floor{\tfrac{\delta}{3} n}$, 
define the following integer intervals:
\eq{
\cI_k \coloneqq \llbrack 1+i_k,i_{k+1}\rrbrack, \quad k\in\{0,\ldots,k_\delta\}.
}
Note that $\cI_0,\ldots,\cI_{k_\delta}$ each contain $\floor{\tfrac{\delta}{3}n}$ elements, and collectively cover the interval $\llbrack 1,n\rrbrack$.
Therefore, any subinterval of $\llbrack1,n\rrbrack$ with at least $\delta n$ elements must intersect at least three consecutive $\cI_k$'s, and thus contains $\llbrack i_k,i_{k+1}\rrbrack$ for some $k\in\{1,\ldots,k_\delta-1\}$.
This justifies the first containment below:
\eeq{ \label{vk5cn}
\big\{I_2 - I_1 \ge \delta n\big\}
&\stackrefp{oh8c4}{\subseteq} \bigcup_{k=1}^{k_\delta-1}\Big\{\llbrack I_1,I_2-1\rrbrack\supseteq\llbrack i_k,i_{k+1}\rrbrack\Big\} \\
&\stackref{oh8c4}{\subseteq} \bigcup_{k=1}^{k_\delta-1} \{\text{$U_i\in[a,a+\eta]$ for every $i\in\llbrack i_k,i_{k+1}\rrbrack$}\} \\
&\stackrefp{oh8c4}{\subseteq} \bigcup_{k=1}^{k_\delta-1} \{U_{i_{k+1}}-U_{i_k}\le\eta\} 
\stackref{exp_rep}{=} \bigcup_{k=1}^{k_\delta-1} \Big\{\frac{\xi_{1+i_k}+\cdots+\xi_{i_{k+1}}}{\xi_1+\cdots+\xi_{n+1}} \le \eta\Big\}.
}
Denote the event in the final line by
\eeq{ \label{wj8c4x}
A_k \coloneqq \Big\{\frac{\xi_{1+i_k}+\cdots+\xi_{i_{k+1}}}{\xi_1+\cdots+\xi_{n+1}} \le \eta\Big\}.
}
Now apply Claim~\ref{claim_correlation} with $\cP_1 \coloneqq \llbrack1,i_k\rrbrack$, $\cP_2 \coloneqq \llbrack1+i_k,i_{k+1}\rrbrack$, $\cP_3 \coloneqq \llbrack1+i_{k+1},n+1\rrbrack$, and
\eq{
\Phi\big((x_i,y_i)_{i=1}^{n+1}\big) \coloneqq \frac{x_{1+i_k}+\cdots+x_{i_{k+1}}}{x_1+\cdots+x_{n+1}}, \qquad B \coloneqq [0,\eta],
}
to obtain
\eeq{ \label{3bm7cy}
\P\big(\{\text{$\cU_n^\square$ is a convex chain}\}\cap A_k\big)
\le \P(\textup{$\cU_n^\square$ is a convex chain})\cdot\P(A_k)\cdot 3^{n+1}.
}
Putting various observations together, we arrive at
\eeq{ \label{yfb99}
\P\Big(\parbox{0.37\textwidth}{\centering $\cU_n^\square$ is a convex chain such that \\ $\#\big(\cU_n^\square\cap([a,a+\eta]\times[0,1])\big)\ge \delta n$}\Big)
&\stackrefpp{oh8c4}{3bm7cy}{=} \P\big(\{\text{$\cU_n^\square$ is a convex chain}\}\cap\{I_2-I_1\ge\delta n\}\big) \\
&\stackrefpp{vk5cn}{3bm7cy}{\le} \sum_{k=1}^{k_\delta-1} \P\big(\{\text{$\cU_n^\square$ is a convex chain}\}\cap A_k\big) \\
&\stackref{3bm7cy}{\le}\P(\textup{$\cU_n^\square$ is a convex chain})\cdot 3^{n+1}\sum_{k=1}^{k_\delta-1}\P(A_k).
}
This result prompts us to seek an upper bound on $\P(A_k)$.
The next paragraph will derive a bound that is uniform in $k$.

From the definition of $A_k$ in \eqref{wj8c4x} and a simple union bound, we have the following for any $L\ge 0$:
\eeq{ \label{bon83x}
\P(A_k)
\le \P\Big(\xi_1+\cdots+\xi_{n+1}>L(n+1)\Big)
+ \P\Big(\xi_{1+i_k}+\cdots+\xi_{i_{k+1}} \le \eta L(n+1)\Big).
}
We next use Chernoff's method to control each term on the right-hand side.
First,
\eq{
\P\Big(\xi_1+\cdots+\xi_{n+1}> L(n+1)\Big)
&= \P\Big(\e^{(\xi_1+\cdots+\xi_{n+1})/2}> \e^{L(n+1)/2}\Big) \\
&\le \frac{(\E\e^{\xi_1/2})^{n+1}}{\e^{L(n+1)/2}}
= \Big(\frac{2}{\e^{L/2}}\Big)^{n+1}.
}
We assume $L = L(\epsilon,\delta)$ is sufficiently large that $2/\e^{L/2} \le (\delta/18)\wedge(\epsilon/3)$, so that the previous display implies
\eeq{ \label{d6v3jy}
\P\Big(\xi_1+\cdots+\xi_{n+1}> L(n+1)\Big) \le (\delta/18)\cdot(\epsilon/3)^{n} \quad \text{for all $n\ge1$}.
}
Now we turn our attention to the second term on the right-hand side of \eqref{bon83x}. 
For any $\theta>0$ we have
\eeq{ \label{ev7ep}
\P\Big(\xi_{1+i_k}+\cdots+\xi_{i_{k+1}} \le \eta L(n+1)\Big)
&= \P\Big(\e^{-\theta(\xi_{1+i_k}+\cdots+\xi_{i_{k+1}})} \ge \e^{-\theta\eta L(n+1)}\Big) \\
&\le \frac{(\E \e^{-\theta\xi_1})^{i_{k+1}-i_k}}{\e^{-\theta\eta L(n+1)}}.
}
Since $\P(\xi_1>0)=1$, we have $\E \e^{-\theta\xi_1}\searrow0$ as $\theta\nearrow\infty$. 
This limit, togther with the fact $i_{k+1}-i_k = \floor{\tfrac{\delta}{3}n}$, allows us to choose $\theta = \theta(\epsilon,\delta)$ so large that
\eq{
(\E \e^{-\theta\xi_1})^{i_{k+1}-i_k} \le (\delta/18)\cdot(\epsilon/3)^{2n} \quad \text{for all sufficiently large $n$}.
}
Now choose $\eta = \eta(\theta,L,\epsilon) > 0$ so small that
\eq{
\e^{-\theta\eta L(n+1)} \ge (\epsilon/3)^n \quad \text{for all $n\ge1$.}
}
Applying the two previous displays to the final line of \eqref{ev7ep}, we obtain
\eeq{ \label{cu4bq}
\P\Big(\xi_{1+i_k}+\cdots+\xi_{i_{k+1}} \le \eta L(n+1)\Big)
\le (\delta/18)\cdot(\epsilon/3)^{n} \quad \text{for all sufficiently large $n$}.
}
Using \eqref{d6v3jy} and \eqref{cu4bq} in \eqref{bon83x}, we arrive at
\eeq{ \label{2ib8w}
\P(A_k) \le (\delta/9)\cdot(\epsilon/3)^{n} \quad \text{for all sufficiently large $n$.}
}
Finally, insert \eqref{2ib8w} into \eqref{yfb99} to obtain the following for all sufficiently large $n$:
\eq{
&\P\Big(\parbox{0.37\textwidth}{\centering $\cU_n^\square$ is a convex chain such that \\ $\#\big(\cU_n^\square\cap([a,a+\eta]\times[0,1])\big)\ge \delta n$}\Big) \\
&\stackrefp{2ion9}{\le} \P(\text{$\cU_n^\square$ is a convex chain})\cdot 3^{n+1}\cdot(k_\delta-1)\cdot (\delta/9)\cdot (\epsilon/3)^n \\
&\stackref{2ion9}{\le} \epsilon^n\cdot \P(\text{$\cU_n^\square$ is a convex chain}).
}
This completes the proof of \eqref{20vjj3}.
\end{proof}

Ultimately we will want the sample set to be $(\eta,\delta)$-good with $\delta=\delta_n\to0$ as $n\to\infty$.
In order for this to be possible, we must also take $\eta = \eta_n \to0$, which in turn threatens the usefulness of $(\eta,\delta)$-goodness in the first place.
Fortunately, the following lemma shows that even very slow convergence $\eta_n\to0$ can be accommodated.

\begin{lemma}[Full convex chains are unlikely to be bad] \label{lem_bad_unlikely}
Assume \eqref{p_upper} and \eqref{p_lower}.
Let $(\eta_n)_{n\ge1}$ be any sequence of positive numbers such that $\lim_{n\to\infty}\eta_n=0$.
For every $\epsilon>0$, there exists a sequence $(\delta_n)_{n\ge1}$ of positive numbers such that $\lim_{n\to\infty}\delta_n=0$ and
\eeq{ \label{vo48xv}
\P(\text{$\cS_n$ is an $(\eta_n,\delta_n)$-bad convex chain})
\le\epsilon^n\cdot\P(\text{$\cS_n$ is a convex chain})
}
for all sufficiently large $n$.
\end{lemma}

\begin{proof}
Fix $\epsilon>0$.
Lemma~\ref{lem_not_too_dense} guarantees that for any $\delta>0$, there exists $\bar\eta = \bar\eta(\delta)>0$ and $\bar n = \bar n(\delta)<\infty$ such that
\eeq{ \label{4eb8c}
\P(\text{$\cS_n$ is an $(\bar\eta(\delta),\delta)$-bad convex chain})
\le \epsilon^n\cdot\P(\text{$\cS_n$ is a convex chain}) \quad \text{$\forall$ $n\ge \bar n(\delta)$}.
}
Define a sequence of integers $0 = n_0 < n_1 < n_2 < \cdots$ inductively as follows.
Given $n_{k-1}$, let
\eq{
n_k \coloneqq \sup\big\{n:\, \eta_n > \bar\eta(k^{-1}) \text{ or } n < \bar n(k^{-1})\big\} \vee (n_{k-1}+1).
}
By assumption $\lim_{n\to\infty}\eta_n = 0$, so we must have $n_k<\infty$.
Whenever $n > n_k$, we have $\eta_n \le \bar\eta(k^{-1})$ and $n\ge\bar n(k^{-1})$, hence
\eq{
\P(\text{$\cS_n$ is an $(\eta_n,k^{-1})$-bad convex chain})
&\stackrefp{4eb8c}{\le} \P(\text{$\cS_n$ is an $(\bar\eta(k^{-1}),k^{-1})$-bad convex chain}) \\
&\stackref{4eb8c}{\le} \epsilon^n\cdot\P(\text{$\cS_n$ is a convex chain}).
}
Therefore, it suffices to take $\delta_n = k^{-1}$ for all $n\in\{n_k+1,\ldots,n_{k+1}\}$.
\end{proof}

The final set of lemmas are in service of the following definition, which quantifies staying away from the corner $(1,0)$.
See Figure~\ref{fig_regular} for an illustration.

\begin{definition} \label{def_regular}
Given a convex chain $\cC$, let $f_\cC\colon[0,1]\to[0,1]$ be the maximal convex function starting at $f_\cC(0)=0$ and satisfying $f_\cC(x)=y$ for every $(x,y)\in\cC$.
That is, $f_\cC$ is piecewise linear with corners at the elements of $\cC$, and $f_\cC(1)=1$ unless $(1,y)\in\cC$ for some $y\in[0,1)$.

Given $\rho\in(0,\frac12)$, we say $\cC$ is \textit{$\rho$-regular} if $f_\cC(1-\rho)\ge\rho$ and $f_\cC(1)=1$.
Otherwise we say $\cC$ is \textit{$\rho$-irregular}. 
Note that if $\cC$ is $\rho$-regular, then it is $\rho'$-regular for every $\rho'\in(0,\rho]$.
\qedrem
\end{definition}

\begin{figure}[h]

\tikzset{every picture/.style={line width=0.75pt}} %set default line width to 0.75pt        

\begin{tikzpicture}[x=0.75pt,y=0.75pt,yscale=-1,xscale=1]
%uncomment if require: \path (0,250); %set diagram left start at 0, and has height of 250

%Straight Lines [id:da5599118529408579] 
\draw [color={rgb, 255:red, 74; green, 144; blue, 226 }  ,draw opacity=1 ][line width=1.5]    (361,220) -- (436.5,212.5) ;
%Straight Lines [id:da23072407956192054] 
\draw [color={rgb, 255:red, 74; green, 144; blue, 226 }  ,draw opacity=1 ][line width=1.5]    (436.5,212.5) -- (488.5,201.5) ;
%Straight Lines [id:da11087768581603141] 
\draw [color={rgb, 255:red, 74; green, 144; blue, 226 }  ,draw opacity=1 ][line width=1.5]    (548.5,115.5) -- (560,20) ;
%Shape: Circle [id:dp7338520132094327] 
\draw  [fill={rgb, 255:red, 0; green, 0; blue, 0 }  ,fill opacity=1 ] (433,212.5) .. controls (433,210.57) and (434.57,209) .. (436.5,209) .. controls (438.43,209) and (440,210.57) .. (440,212.5) .. controls (440,214.43) and (438.43,216) .. (436.5,216) .. controls (434.57,216) and (433,214.43) .. (433,212.5) -- cycle ;
%Shape: Square [id:dp6796523115922546] 
\draw  [fill={rgb, 255:red, 184; green, 181; blue, 181 }  ,fill opacity=0.3 ][dash pattern={on 0.84pt off 2.51pt}] (510,170) -- (560,170) -- (560,220) -- (510,220) -- cycle ;
%Straight Lines [id:da23484065848688496] 
\draw [color={rgb, 255:red, 74; green, 144; blue, 226 }  ,draw opacity=1 ][line width=1.5]    (538,158.5) -- (548.5,115.5) ;
%Straight Lines [id:da7561321651032296] 
\draw [color={rgb, 255:red, 74; green, 144; blue, 226 }  ,draw opacity=1 ][line width=1.5]    (488.5,201.5) -- (538,158.5) ;
%Shape: Circle [id:dp5601922208272282] 
\draw  [fill={rgb, 255:red, 0; green, 0; blue, 0 }  ,fill opacity=1 ] (485,201.5) .. controls (485,199.57) and (486.57,198) .. (488.5,198) .. controls (490.43,198) and (492,199.57) .. (492,201.5) .. controls (492,203.43) and (490.43,205) .. (488.5,205) .. controls (486.57,205) and (485,203.43) .. (485,201.5) -- cycle ;
%Shape: Circle [id:dp0308703957418488] 
\draw  [fill={rgb, 255:red, 0; green, 0; blue, 0 }  ,fill opacity=1 ] (534.5,158.5) .. controls (534.5,156.57) and (536.07,155) .. (538,155) .. controls (539.93,155) and (541.5,156.57) .. (541.5,158.5) .. controls (541.5,160.43) and (539.93,162) .. (538,162) .. controls (536.07,162) and (534.5,160.43) .. (534.5,158.5) -- cycle ;
%Shape: Circle [id:dp8974148643884926] 
\draw  [fill={rgb, 255:red, 0; green, 0; blue, 0 }  ,fill opacity=1 ] (545,115.5) .. controls (545,113.57) and (546.57,112) .. (548.5,112) .. controls (550.43,112) and (552,113.57) .. (552,115.5) .. controls (552,117.43) and (550.43,119) .. (548.5,119) .. controls (546.57,119) and (545,117.43) .. (545,115.5) -- cycle ;
%Shape: Right Triangle [id:dp18322537994013344] 
\draw   (560,20) -- (361,220) -- (560,220) -- cycle ;
%Straight Lines [id:da6808692544491816] 
\draw [color={rgb, 255:red, 74; green, 144; blue, 226 }  ,draw opacity=1 ][line width=1.5]    (91,220) -- (151.5,206.5) ;
%Straight Lines [id:da30812310790793185] 
\draw [color={rgb, 255:red, 74; green, 144; blue, 226 }  ,draw opacity=1 ][line width=1.5]    (205.5,188.5) -- (249.5,154.5) ;
%Straight Lines [id:da3426191539654023] 
\draw [color={rgb, 255:red, 74; green, 144; blue, 226 }  ,draw opacity=1 ][line width=1.5]    (249.5,154.5) -- (278.5,100.5) ;
%Straight Lines [id:da7407874580559863] 
\draw [color={rgb, 255:red, 74; green, 144; blue, 226 }  ,draw opacity=1 ][line width=1.5]    (278.5,100.5) -- (290,20) ;
%Shape: Right Triangle [id:dp4291793069867975] 
\draw   (290,20) -- (91,220) -- (290,220) -- cycle ;
%Shape: Circle [id:dp705814255091614] 
\draw  [fill={rgb, 255:red, 0; green, 0; blue, 0 }  ,fill opacity=1 ] (246,154.5) .. controls (246,152.57) and (247.57,151) .. (249.5,151) .. controls (251.43,151) and (253,152.57) .. (253,154.5) .. controls (253,156.43) and (251.43,158) .. (249.5,158) .. controls (247.57,158) and (246,156.43) .. (246,154.5) -- cycle ;
%Shape: Circle [id:dp10504945556144385] 
\draw  [fill={rgb, 255:red, 0; green, 0; blue, 0 }  ,fill opacity=1 ] (275,100.5) .. controls (275,98.57) and (276.57,97) .. (278.5,97) .. controls (280.43,97) and (282,98.57) .. (282,100.5) .. controls (282,102.43) and (280.43,104) .. (278.5,104) .. controls (276.57,104) and (275,102.43) .. (275,100.5) -- cycle ;
%Straight Lines [id:da5072642732982903] 
\draw [color={rgb, 255:red, 74; green, 144; blue, 226 }  ,draw opacity=1 ][line width=1.5]    (151.5,206.5) -- (205.5,188.5) ;
%Shape: Circle [id:dp5345638850243576] 
\draw  [fill={rgb, 255:red, 0; green, 0; blue, 0 }  ,fill opacity=1 ] (148,206.5) .. controls (148,204.57) and (149.57,203) .. (151.5,203) .. controls (153.43,203) and (155,204.57) .. (155,206.5) .. controls (155,208.43) and (153.43,210) .. (151.5,210) .. controls (149.57,210) and (148,208.43) .. (148,206.5) -- cycle ;
%Shape: Circle [id:dp33766111006145705] 
\draw  [fill={rgb, 255:red, 0; green, 0; blue, 0 }  ,fill opacity=1 ] (202,188.5) .. controls (202,186.57) and (203.57,185) .. (205.5,185) .. controls (207.43,185) and (209,186.57) .. (209,188.5) .. controls (209,190.43) and (207.43,192) .. (205.5,192) .. controls (203.57,192) and (202,190.43) .. (202,188.5) -- cycle ;
%Shape: Square [id:dp643850277947071] 
\draw  [fill={rgb, 255:red, 184; green, 181; blue, 181 }  ,fill opacity=0.3 ][dash pattern={on 0.84pt off 2.51pt}] (240,170) -- (290,170) -- (290,220) -- (240,220) -- cycle ;

% Text Node
\draw (566,165) node [anchor=north west][inner sep=0.75pt]    {$\rho $};
% Text Node
\draw (492,223.4) node [anchor=north west][inner sep=0.75pt]    {$1-\rho $};
% Text Node
\draw (296,165) node [anchor=north west][inner sep=0.75pt]    {$\rho $};
% Text Node
\draw (222,223.4) node [anchor=north west][inner sep=0.75pt]    {$1-\rho $};
% Text Node
\draw (242,114.4) node [anchor=north west][inner sep=0.75pt]  [color={rgb, 255:red, 74; green, 144; blue, 226 }  ,opacity=1 ]  {$f_{\mathcal{C}}$};
% Text Node
\draw (520,127) node [anchor=north west][inner sep=0.75pt]  [color={rgb, 255:red, 74; green, 144; blue, 226 }  ,opacity=1 ]  {$f_{\mathcal{C}}$};

\end{tikzpicture}

\caption{\textit{Left}: a $\rho$-regular convex chain.
\textit{Right}: a $\rho$-irregular convex chain.}
\label{fig_regular}
\end{figure}
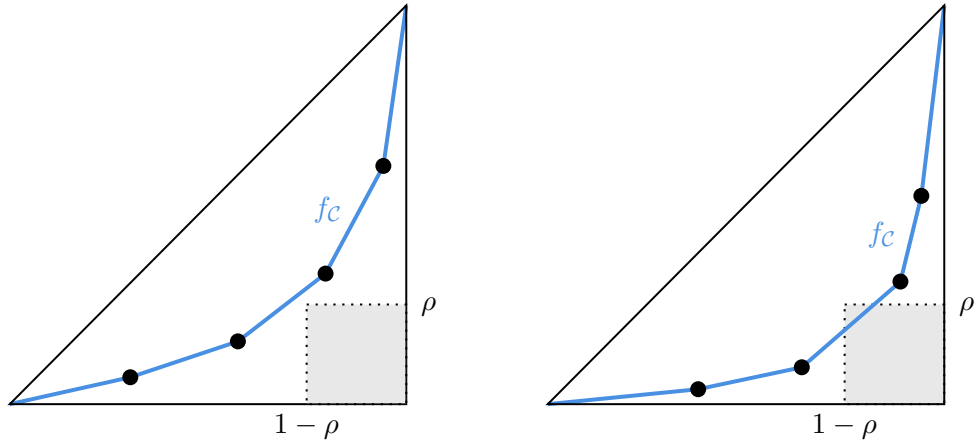

\begin{lemma}[Full convex chains are unlikely to be irregular] \label{lem_irreg_unlikely}
Assume \eqref{p_upper} and \eqref{p_lower}.
For every $\epsilon>0$, there exists $\rho = \rho(\epsilon,\fC/\fc)\in(0,\tfrac12)$ small enough that
\eeq{ \label{ob8hp}
\P(\text{$\cS_n$ is a $\rho$-irregular convex chain})
\le \epsilon^n\cdot\P(\text{$\cS_n$ is a convex chain})
}
for all sufficiently large $n$.
\end{lemma}

\begin{proof}
Fix any $\delta\in(0,\frac12)$. 
Lemma~\ref{lem_not_too_dense} allows us to choose $\eta = \eta(\epsilon,\delta,\fC/\fc)\in(0,\tfrac12)$ such that
\eeq{ \label{n5c42c}
\P(\text{$\cS_n$ is an $(\eta,\delta)$-bad convex chain})
\le \epsilon^n\cdot\P(\text{$\cS_n$ is a convex chain})
}
for all sufficiently large $n$.
Now suppose $\cS_n$ is an $(\eta,\delta)$-good convex chain.
Let $f = f_{\cS_n}$ as in Definition~\ref{def_regular}.
Since $\delta<\frac12$, \eqref{good_def} implies that the two strips $[1-\eta,1]\times[0,1]$ and $[0,1]\times[0,\eta]$ do not cover all of $\cS_n$.
Therefore, $\cS_n$ contains an element $(x,y)\in[0,1-\eta)\times(\eta,1]$, hence
\eq{
f(1-\eta) \ge f(x) = y > \eta.
}
So $\cS_n$ is $\eta$-regular.
We have thus argued (by contrapositive) that
\eq{
\{\text{$\cS_n$ is an $\eta$-irregular convex chain}\} \subseteq \{\text{$\cS_n$ is an $(\eta,\delta)$-bad convex chain}\}.
}
Consequently, \eqref{ob8hp} follows from \eqref{n5c42c} if we choose $\rho \le \eta$.
\end{proof}

\begin{lemma}[Convex positioning within parallelogram] \label{lem_lozenge}
Let $\lozenge\subset\R^2$ be a parallelogram.
Let $\wh\fp$ be a probability density function on $\lozenge$ such that
\eeq{ \label{ratio_assumption_loz}
\frac{\sup_{(x,y)\in\lozenge}\wh\fp(x,y)}{1/\Area(\lozenge)} \le R\in[1,\infty).
}
Let $\cS_n^\lozenge$
be a set of $n$ independent samples from $\wh\fp$, and define
\eq{
\sL_n^\lozenge \coloneqq \max\{\#\cC:\, \cC\subseteq\cS_n^\lozenge \text{ and $\cC$ is in convex position}\}.
}
Then for every $n\ge 1$ and $t\ge0$, we have the tail bound
\eeq{ \label{lozenge_tail}
\P\big(\sL_n^\lozenge \ge (32R\e^3 n)^{1/3}+t\big)
\le
2^{-(32R\e^3 n)^{1/3}-t}.
}
\end{lemma}

\begin{proof}
Let us write $\cS_n^\lozenge = \{\nP_1,\ldots,\nP_n\}$. 
By \eqref{ratio_assumption_loz}, the probability density function of $(\nP_1,\ldots,\nP_n)$ with respect to product uniform measure on $\lozenge$ is bounded from above by $R^n$.
If $\cU_n^\lozenge$ is a set of $n$ independent uniform samples from $\lozenge$, then \cite[Theorem~1]{valtr95} gives
\eq{
\P(\text{$\cU_n^\lozenge$ is in convex position}) = \bigg(\frac{{2n-2 \choose n-1}}{n!}\bigg)^2.
}
Hence
\eq{
\P(\text{$\cS_n^\lozenge$ is in convex position}) \le R^n\bigg(\frac{{2n-2 \choose n-1}}{n!}\bigg)^2
= R^n\bigg(\frac{\frac{(2n-2)!}{(n-1)!(n-1)!}}{n!}\bigg)^2
\le R^n\Big(\frac{(2n)!}{(n!)^3}\Big)^2.
}
Taking a union bound over subsets of size $k$, we have
\eq{
\P(\sL_n^\lozenge \ge k)
\le {n \choose k}R^k\Big(\frac{(2k)!}{(k!)^3}\Big)^2
\le \frac{n^k}{k!}R^k\Big(\frac{(2k)!}{(k!)^3}\Big)^2
= (Rn)^k\frac{\big((2k)!\big)^2}{(k!)^7}.
}
Next use the standard estimate $(\frac{k}{\e})^k\sqrt{\e k} \le k! \le (\frac{k}{\e})^k\sqrt{\e^2 k}$ to obtain
\[
\P(\sL_n^\lozenge \ge k)
\le
(Rn)^k\cdot\frac{(\frac{2k}{\e})^{4k}\e^2 2k}{(\frac{k}{\e})^{7k}\e^{7/2}k^{7/2}}
=
\Big(\frac{16R\e^3 n}{k^3}\Big)^k\cdot\frac{2}{\e^{3/2}k^{5/2}}
\le
\Big(\frac{16R\e^3 n}{k^3}\Big)^k.
\]
Now choose $k = \ceil{(32 R\e^3 n)^{1/3} + t}$ to obtain \eqref{lozenge_tail}.
\end{proof}

\begin{lemma}[Irregular convex chains are not competitive] \label{lem_irreg_nonoptimal}
Assume \eqref{p_upper}.
Given $\rho\in(0,\tfrac12)$, define the random variable
\eq{
\sL_n^\irr \coloneqq \max\big\{\#\cC:\, \text{$\cC\subseteq\cS_n$ and $\cC$ is a $\rho$-irregular convex chain}\big\}.
}
For every $\epsilon>0$, there exists $\rho = \rho(\epsilon,\fC)\in(0,\tfrac12)$ small enough that
\eeq{ \label{irreg_epsilon}
\limsup_{n\to\infty}\frac{\sL_n^\irr}{n^{1/3}} \le \epsilon \quad \text{almost surely.}
}
\end{lemma}

\begin{proof}
Our strategy is to show that $\rho$-irregular convex chains are confined to a region whose area is $O(\rho)$.
This will mean that only $O(\rho n)$ sample points are available for forming a $\rho$-irregular convex chain, which in turn means that the longest such chain will be $O((\rho n)^{1/3})$.

More precisely, given $\epsilon>0$, choose $\rho\in(0,\tfrac12)$ small enough that
\eq{
(64\fC\rho\e^3)^{1/3} \le \tfrac{\epsilon}{4}.
}
In particular, we may assume $n$ is large enough that
\eeq{ \label{uebbx}
(64\fC\rho\e^3 n)^{1/3} + n^{1/4} \le \tfrac{\epsilon}{3}n^{1/3}.
}
Using this value of $\rho$ henceforth, let $\cC_n$ be a maximizing set for $\sL_n^\irr$.
Let us assume $\cC_n\subset\Interior(\cT)$, which occurs with probability one.
Let $f \coloneqq f_{\cC_n}$ as in Definition~\ref{def_regular}.
Since $\cC_n$ is $\rho$-irregular, we have 
\eeq{ \label{irreg_consequence}
f(1-\rho)<\rho.
}
Now partition $\cC_n$ into three subsets (see Figure~\ref{fig_irregular} for an illustration):
\eq{
\cC_{n,1} &\coloneqq \cC_n \cap \big([0,1-\rho)\times[0,1]\big), \\
\cC_{n,2} &\coloneqq \cC_n \cap \big([0,1]\times(\rho,1]\big), \\
\cC_{n,3} &\coloneqq \cC_n \cap \big([1-\rho,1]\times[0,\rho]\big).
}
Let $\cT_1$ denote the triangle with vertices $(0,0)$, $(1-\rho,\rho)$, and $(1-\rho,0)$, and let $\cT_2$ denote the triangle with vertices $(1,1)$, $(1-\rho,\rho)$, and $(1,\rho)$.
By convexity of $f$, it follows from \eqref{irreg_consequence} that
\eeq{ \label{w8b3x}
\cC_{n,1} \subset \cT_1
\quad \text{and} \quad
\cC_{n,2} \subset \cT_2.
}
For convenience when discussing $\cC_{n,3}$, we denote the relevant square by $\cQ \coloneqq [1-\rho,1]\times[0,\rho]$.

The following claim is depicted in Figure~\ref{fig_irregular}.

\begin{claim}[Convex positioning for each subset] \label{claim_rho}
The following statements hold:
\begin{enumerate}[label=\textup{(\roman*)}]

\item \label{claim_rho_1}
$\cC_{n,1}\cup\{(0,0),(1-\rho,\rho)\}$ is in convex position.

\item \label{claim_rho_2}
$\cC_{n,2}\cup\{(1,1),(1-\rho,\rho)\}$ is in convex position.

\item \label{claim_rho_3}
$\cC_{n,3}$ is in convex position.

\end{enumerate}
\end{claim}

\begin{proofclaim}
Statement~\ref{claim_rho_3} is trivial, since $\cC_{n,3}\subseteq\cC_n$ and $\cC_n$ is a convex chain.
The proof of \ref{claim_rho_2} is the same as the proof of \ref{claim_rho_1} after applying the reflection $(x,y)\mapsto(1-y,1-x)$.
So we just prove \ref{claim_rho_1}.

Recall from \eqref{w8b3x} that $\cC_{n,1}\subset\cT_1$.
If $\#\cC_{n,1} \in \{0,1\}$, then it is trivially that case that $\cC_{n,1}\cup\{(0,0),(1-\rho,\rho)\}$ is in convex position.
So assume $\#\cC_{n,1} = m \ge 2$, and label the elements of $\cC_{n,1}$ in increasing order of their horizontal coordinate, as $\nP_1,\ldots,\nP_m$.
Since $\cC_{n,1}\subseteq\cC_n$ and $\cC_n$ is a convex chain, we know from Lemma~\ref{lem_position_slopes} that
\eq{
\Slope\big((0,0),\nP_1)<\Slope(\nP_1,\nP_2)<\cdots<\Slope(\nP_{m-1},\nP_m).
}
By the same lemma, we will have \ref{claim_rho_1} upon showing that
\eeq{ \label{2knnp0}
\Slope(\nP_{m-1},\nP_m) < \Slope\big(\nP_m,(1-\rho,\rho)\big).
}
To this end, we write $\nP_m = (x,f(x))$ and observe that
\begin{align*}
\Slope(\nP_{m-1},\nP_m) = \partial^-f(x)
\le \partial^+f(x)
&\stackrefp{irreg_consequence}{\le} \frac{f(1-\rho)-f(x)}{1-\rho-x} \\
&\stackref{irreg_consequence}{<} \frac{\rho-f(x)}{1-\rho-x} 
= \Slope\big(\nP_m,(1-\rho,\rho)\big).
\end{align*}
Indeed, \eqref{2knnp0} is true, so we are done.
\end{proofclaim}

Consider the random variables
\eq{
\sL_{n,1} &\coloneqq \max\big\{\#S:\, \text{$S\subseteq\cS_n\cap\cT_1$ and $S\cup\{(0,0),(1-\rho,\rho)\}$ is in convex position}\big\}, \\
\sL_{n,2} &\coloneqq \max\big\{\#S:\, \text{$S\subseteq\cS_n\cap\cT_2$ and $S\cup\{(1,1),(1-\rho,\rho)\}$ is in convex position}\big\}, \\
\sL_{n,3} &\coloneqq \max\big\{\#S:\, \text{$S\subseteq\cS_n\cap\cQ$ and $S$ is in convex position}\big\}.
}
Claim~\ref{claim_rho} and \eqref{w8b3x} together show
\eeq{ \label{C_by_L}
\#\cC_{n,\ell} \le \sL_{n,\ell} \quad \text{for each $\ell\in\{1,2,3\}$,}
}
so now our strategy is to bound $\sL_{n,1}$ from above.

\begin{claim}[Optimal values are small] \label{claim_optimal_small}
With probability one, we have $\sL_{n,\ell}\le\frac{\epsilon}{3}n^{1/3}$ for each $\ell\in\{1,2,3\}$ and all large $n$.
\end{claim}

\begin{figure}[t]
\tikzset{every picture/.style={line width=0.75pt}} %set default line width to 0.75pt        
\begin{tikzpicture}[x=0.75pt,y=0.75pt,yscale=-1,xscale=1]
%uncomment if require: \path (0,262); %set diagram left start at 0, and has height of 262

%Shape: Square [id:dp7209600006074993] 
\draw  [fill={rgb, 255:red, 184; green, 181; blue, 181 }  ,fill opacity=0.3 ][dash pattern={on 0.84pt off 2.51pt}] (240,180) -- (290,180) -- (290,230) -- (240,230) -- cycle ;
%Straight Lines [id:da3951271614330314] 
\draw [color={rgb, 255:red, 74; green, 144; blue, 226 }  ,draw opacity=1 ][line width=1.5]    (91,230) -- (166.5,222.5) ;
%Straight Lines [id:da20848432242642168] 
\draw [color={rgb, 255:red, 74; green, 144; blue, 226 }  ,draw opacity=1 ][line width=1.5]    (166.5,222.5) -- (218.5,211.5) ;
%Straight Lines [id:da5194751864693146] 
\draw [color={rgb, 255:red, 74; green, 144; blue, 226 }  ,draw opacity=1 ][line width=1.5]    (278.5,125.5) -- (290,30) ;
%Shape: Circle [id:dp18633152540379738] 
\draw  [fill={rgb, 255:red, 0; green, 0; blue, 0 }  ,fill opacity=1 ] (163,222.5) .. controls (163,220.57) and (164.57,219) .. (166.5,219) .. controls (168.43,219) and (170,220.57) .. (170,222.5) .. controls (170,224.43) and (168.43,226) .. (166.5,226) .. controls (164.57,226) and (163,224.43) .. (163,222.5) -- cycle ;
%Straight Lines [id:da836361373607615] 
\draw [color={rgb, 255:red, 74; green, 144; blue, 226 }  ,draw opacity=1 ][line width=1.5]    (268,168.5) -- (278.5,125.5) ;
%Straight Lines [id:da6325762256601438] 
\draw [color={rgb, 255:red, 74; green, 144; blue, 226 }  ,draw opacity=1 ][line width=1.5]    (218.5,211.5) -- (248.5,201.5) ;
%Shape: Circle [id:dp24585632276348746] 
\draw  [fill={rgb, 255:red, 0; green, 0; blue, 0 }  ,fill opacity=1 ] (215,211.5) .. controls (215,209.57) and (216.57,208) .. (218.5,208) .. controls (220.43,208) and (222,209.57) .. (222,211.5) .. controls (222,213.43) and (220.43,215) .. (218.5,215) .. controls (216.57,215) and (215,213.43) .. (215,211.5) -- cycle ;
%Shape: Circle [id:dp8069494144519829] 
\draw  [fill={rgb, 255:red, 0; green, 0; blue, 0 }  ,fill opacity=1 ] (275,125.5) .. controls (275,123.57) and (276.57,122) .. (278.5,122) .. controls (280.43,122) and (282,123.57) .. (282,125.5) .. controls (282,127.43) and (280.43,129) .. (278.5,129) .. controls (276.57,129) and (275,127.43) .. (275,125.5) -- cycle ;
%Shape: Right Triangle [id:dp5018677631567552] 
\draw   (290,30) -- (91,230) -- (290,230) -- cycle ;
%Straight Lines [id:da36955181707363205] 
\draw [color={rgb, 255:red, 74; green, 144; blue, 226 }  ,draw opacity=1 ][line width=1.5]    (248.5,201.5) -- (260.5,188.5) ;
%Straight Lines [id:da16899218716439834] 
\draw [color={rgb, 255:red, 74; green, 144; blue, 226 }  ,draw opacity=1 ][line width=1.5]    (260.5,188.5) -- (268,168.5) ;
%Shape: Circle [id:dp43050451597451367] 
\draw  [fill={rgb, 255:red, 0; green, 0; blue, 0 }  ,fill opacity=1 ] (245,201.5) .. controls (245,199.57) and (246.57,198) .. (248.5,198) .. controls (250.43,198) and (252,199.57) .. (252,201.5) .. controls (252,203.43) and (250.43,205) .. (248.5,205) .. controls (246.57,205) and (245,203.43) .. (245,201.5) -- cycle ;
%Shape: Circle [id:dp01620587857446598] 
\draw  [fill={rgb, 255:red, 0; green, 0; blue, 0 }  ,fill opacity=1 ] (257,188.5) .. controls (257,186.57) and (258.57,185) .. (260.5,185) .. controls (262.43,185) and (264,186.57) .. (264,188.5) .. controls (264,190.43) and (262.43,192) .. (260.5,192) .. controls (258.57,192) and (257,190.43) .. (257,188.5) -- cycle ;
%Shape: Circle [id:dp7901094948101214] 
\draw  [fill={rgb, 255:red, 0; green, 0; blue, 0 }  ,fill opacity=1 ] (264.5,168.5) .. controls (264.5,166.57) and (266.07,165) .. (268,165) .. controls (269.93,165) and (271.5,166.57) .. (271.5,168.5) .. controls (271.5,170.43) and (269.93,172) .. (268,172) .. controls (266.07,172) and (264.5,170.43) .. (264.5,168.5) -- cycle ;
%Shape: Square [id:dp5027171340949559] 
\draw  [fill={rgb, 255:red, 184; green, 181; blue, 181 }  ,fill opacity=0.3 ][dash pattern={on 0.84pt off 2.51pt}] (510,180) -- (560,180) -- (560,230) -- (510,230) -- cycle ;
%Shape: Right Triangle [id:dp8766291774730809] 
\draw  [fill={rgb, 255:red, 184; green, 181; blue, 181 }  ,fill opacity=0.3 ] (510,180) -- (361,230) -- (510,230) -- cycle ;
%Shape: Right Triangle [id:dp25788228062867147] 
\draw  [fill={rgb, 255:red, 184; green, 181; blue, 181 }  ,fill opacity=0.3 ] (560,30) -- (510,180) -- (560,180) -- cycle ;
%Straight Lines [id:da09059366881985464] 
\draw [color={rgb, 255:red, 74; green, 144; blue, 226 }  ,draw opacity=1 ][line width=1.5]    (361,230) -- (436.5,222.5) ;
%Straight Lines [id:da32780212294381306] 
\draw [color={rgb, 255:red, 74; green, 144; blue, 226 }  ,draw opacity=1 ][line width=1.5]    (436.5,222.5) -- (488.5,211.5) ;
%Straight Lines [id:da060494641434161056] 
\draw [color={rgb, 255:red, 74; green, 144; blue, 226 }  ,draw opacity=1 ][line width=1.5]    (548.5,125.5) -- (560,30) ;
%Shape: Circle [id:dp04820141995757632] 
\draw  [fill={rgb, 255:red, 0; green, 0; blue, 0 }  ,fill opacity=1 ] (433,222.5) .. controls (433,220.57) and (434.57,219) .. (436.5,219) .. controls (438.43,219) and (440,220.57) .. (440,222.5) .. controls (440,224.43) and (438.43,226) .. (436.5,226) .. controls (434.57,226) and (433,224.43) .. (433,222.5) -- cycle ;
%Straight Lines [id:da6417568220352599] 
\draw [color={rgb, 255:red, 74; green, 144; blue, 226 }  ,draw opacity=1 ][line width=1.5]    (538,168.5) -- (548.5,125.5) ;
%Straight Lines [id:da516602141832653] 
\draw [color={rgb, 255:red, 74; green, 144; blue, 226 }  ,draw opacity=1 ][line width=1.5]    (488.5,211.5) -- (510,180) ;
%Shape: Circle [id:dp7854888545705256] 
\draw  [fill={rgb, 255:red, 0; green, 0; blue, 0 }  ,fill opacity=1 ] (485,211.5) .. controls (485,209.57) and (486.57,208) .. (488.5,208) .. controls (490.43,208) and (492,209.57) .. (492,211.5) .. controls (492,213.43) and (490.43,215) .. (488.5,215) .. controls (486.57,215) and (485,213.43) .. (485,211.5) -- cycle ;
%Shape: Circle [id:dp6379488787988992] 
\draw  [fill={rgb, 255:red, 0; green, 0; blue, 0 }  ,fill opacity=1 ] (545,125.5) .. controls (545,123.57) and (546.57,122) .. (548.5,122) .. controls (550.43,122) and (552,123.57) .. (552,125.5) .. controls (552,127.43) and (550.43,129) .. (548.5,129) .. controls (546.57,129) and (545,127.43) .. (545,125.5) -- cycle ;
%Shape: Right Triangle [id:dp5049256193891675] 
\draw   (560,30) -- (361,230) -- (560,230) -- cycle ;
%Straight Lines [id:da7906103234041774] 
\draw [color={rgb, 255:red, 74; green, 144; blue, 226 }  ,draw opacity=1 ][line width=1.5]    (518.5,201.5) -- (530.5,188.5) ;
%Straight Lines [id:da3100878026991666] 
\draw [color={rgb, 255:red, 74; green, 144; blue, 226 }  ,draw opacity=1 ][line width=1.5]    (510,180) -- (538,168.5) ;
%Shape: Circle [id:dp5218516581741669] 
\draw  [fill={rgb, 255:red, 0; green, 0; blue, 0 }  ,fill opacity=1 ] (515,201.5) .. controls (515,199.57) and (516.57,198) .. (518.5,198) .. controls (520.43,198) and (522,199.57) .. (522,201.5) .. controls (522,203.43) and (520.43,205) .. (518.5,205) .. controls (516.57,205) and (515,203.43) .. (515,201.5) -- cycle ;
%Shape: Circle [id:dp11604751250838652] 
\draw  [fill={rgb, 255:red, 0; green, 0; blue, 0 }  ,fill opacity=1 ] (527,188.5) .. controls (527,186.57) and (528.57,185) .. (530.5,185) .. controls (532.43,185) and (534,186.57) .. (534,188.5) .. controls (534,190.43) and (532.43,192) .. (530.5,192) .. controls (528.57,192) and (527,190.43) .. (527,188.5) -- cycle ;
%Shape: Circle [id:dp029970101805855354] 
\draw  [fill={rgb, 255:red, 0; green, 0; blue, 0 }  ,fill opacity=1 ] (534.5,168.5) .. controls (534.5,166.57) and (536.07,165) .. (538,165) .. controls (539.93,165) and (541.5,166.57) .. (541.5,168.5) .. controls (541.5,170.43) and (539.93,172) .. (538,172) .. controls (536.07,172) and (534.5,170.43) .. (534.5,168.5) -- cycle ;
%Shape: Square [id:dp772930161384697] 
\draw  [color={rgb, 255:red, 208; green, 2; blue, 27 }  ,draw opacity=1 ][fill={rgb, 255:red, 208; green, 2; blue, 27 }  ,fill opacity=1 ] (503.75,173.75) -- (516.25,173.75) -- (516.25,186.25) -- (503.75,186.25) -- cycle ;

% Text Node
\draw (296,175) node [anchor=north west][inner sep=0.75pt]    {$\rho $};
% Text Node
\draw (222,233.4) node [anchor=north west][inner sep=0.75pt]    {$1-\rho $};
% Text Node
\draw (566,175) node [anchor=north west][inner sep=0.75pt]    {$\rho $};
% Text Node
\draw (492,233.4) node [anchor=north west][inner sep=0.75pt]    {$1-\rho $};
% Text Node
\draw (257,134.4) node [anchor=north west][inner sep=0.75pt]  [color={rgb, 255:red, 74; green, 144; blue, 226 }  ,opacity=1 ]  {$f$};
% Text Node
\draw (439.5,177.4) node [anchor=north west][inner sep=0.75pt]    {$\mathcal{T}_{1}$};
% Text Node
\draw (505.5,107.4) node [anchor=north west][inner sep=0.75pt]    {$\mathcal{T}_{2}$};
% Text Node
\draw (542,209.4) node [anchor=north west][inner sep=0.75pt]    {$\mathcal{Q}$};
\end{tikzpicture}
\caption{Argument for Lemma~\ref{lem_irreg_nonoptimal}.
\textit{Left}: The original convex chain $\cC_n$, which is assumed to be $\rho$-irregular.
\textit{Right}: The same chain partitioned into three subsets.
The subset in $\cT_1$ is in convex position with $(0,0)$ and the new vertex $(1-\rho,\rho)$, shown as a red square.
Similarly, the subset in $\cT_2$ is in convex position with $(1,1)$ and the new vertex.
Finally, the subset in the square $\cQ$ is trivially in convex position.}
\label{fig_irregular}
\end{figure}
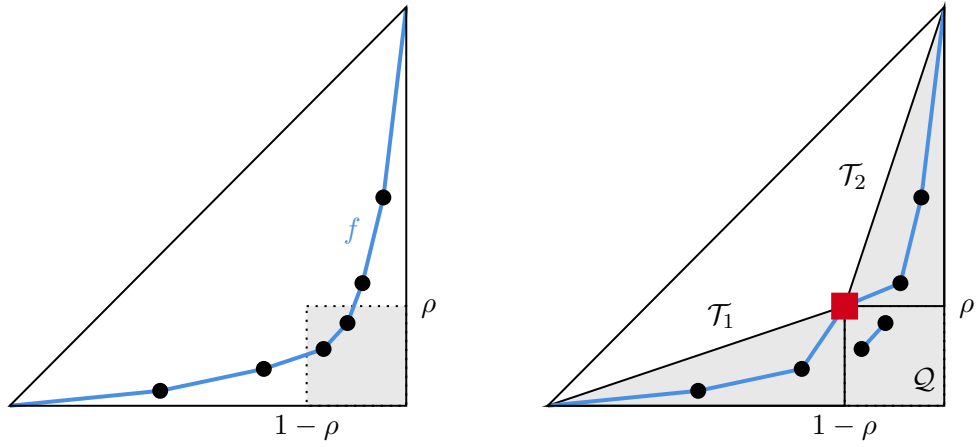

\begin{proofclaim}
Consider the number of sample points landing in each region:
\eq{
N_{n,1} \coloneqq \#(\cS_n\cap\cT_1), \quad
N_{n,2} \coloneqq \#(\cS_n\cap\cT_2), \quad
N_{n,3} \coloneqq \#(\cS_n\cap\cQ), \quad
}
Note that $N_{n,\ell}$ is binomially distributed with $n$ trials and success probability $p_\ell$, where
\eq{ 
p_1 &\coloneqq \int_{\cT_1}\fp(x,y)\ \dd x\,\dd y,
\qquad
p_2 \coloneqq \int_{\cT_2}\fp(x,y)\ \dd x\,\dd y,
\qquad
p_3 \coloneqq \int_{\cQ}\fp(x,y)\ \dd x\,\dd y.
}
By Hoeffding's inequality, we have
\eeq{ \label{38ndv1}
\P(N_{n,\ell} \ge 2np_\ell)
= \P(N_{n,\ell}-\E N_{n,\ell} \ge np_\ell)
\le \exp(-2p_\ell^2n).
}
If $p_1 = 0$, then $\sL_{n,1}=N_{n,1} =0$ almost surely, in which case the claim trivially holds for $\ell=1$.
So assume $p_1>0$.
Conditional on $N_{n,1}$, the random set $\cS_n\cap\cT_1$ has the same law as $N_{n,1}$ independent samples from $\fp$ conditioned to be inside $\cT_1$.
Furthermore, if $\wh\fp$ is the conditional density on $\cT_1$, then
\eq{
\sup_{(x,y)\in\cT_1}\frac{\wh\fp(x,y)}{1/\Area(\cT_1)}
= \sup_{(x,y)\in\cT_1}\frac{\fp(x,y)}{p_1}\cdot\frac{\rho(1-\rho)}{2}
\stackref{p_upper}{\le} \frac{\fC\rho(1-\rho)}{p_1} \le \frac{\fC\rho}{p_1}.
}
Therefore, Proposition~\ref{prop_general}\ref{prop_general_tail} gives
\eeq{ \label{ibm35x}
\P\givenk[\Big]{\sL_{n,1}\ge \Big(4\frac{\fC\rho}{p_1}\e^3 N_{n,1}\Big)^{1/3} + n^{1/4}}{N_{n,1}} \le 2^{-n^{1/4}}.
}
We now have
\eq{
\P\Big(\sL_{n,1} \ge \tfrac{\epsilon}{3}n^{1/3}\Big)
&\stackrefpp{uebbx}{ibm35x,38ndv1}{\le} \P\Big(\sL_{n,1} \ge (8\fC\rho\e^3 n)^{1/3} + n^{1/4}\Big) \\
&\stackrefp{ibm35x,38ndv1}{\le} \P\Big(\sL_{n,1} \ge \Big(4\frac{\fC\rho}{p_1}\e^3 N_{n,1}\Big)^{1/3} + n^{1/4}\Big) + \P(N_{n,1} \ge 2np_1) \\
&\stackref{ibm35x,38ndv1}{\le}
2^{-n^{1/4}}+\exp(-2p_1^2n).
}
By the first Borel--Cantelli lemma, we conclude that $\sL_{n,1}<\tfrac{\epsilon}{3}n^{1/3}$ for all large $n$.
The argument for $\sL_{n,2}$ is exactly the same.

We finally consider $\sL_{n,3}$.
If $p_3 = 0$, then $\sL_{n,3}=N_{n,3}=0$ almost surely, so the claim trivially holds.
So assume $p_3>0$.
Conditional on $N_{n,3}$, the random set $\cS_n\cap\cQ$ has the same law as $N_{n,3}$ independent samples from $\fp$ conditioned to be inside $\cQ$.
Furthermore, if $\wh\fp$ is the conditional density on $\cQ$, then
\eq{
\sup_{(x,y)\in\cQ}\frac{\wh\fp(x,y)}{1/\Area(\cQ)}
= \sup_{(x,y)\in\cT_1}\frac{\fp(x,y)}{p_3}\cdot\rho^2
\stackref{p_upper}{\le} \frac{2\fC\rho^2}{p_3} \le \frac{\fC\rho}{p_3}.
}
Therefore, Lemma~\ref{lem_lozenge} (with $\lozenge=\cQ$ and $t=n^{1/4}$) yields
\eeq{ \label{ibm35y}
\P\givenk[\Big]{\sL_{n,3} \ge \Big(32\frac{\fC\rho}{p_3}\e^3 N_{n,3}\Big)^{1/3}+n^{1/4}}{N_{n,3}} \le 2^{-n^{1/4}}.
}
We now have
\eq{
\P\Big(\sL_{n,3} \ge \tfrac{\epsilon}{3}n^{1/3}\Big)
&\stackrefpp{uebbx}{ibm35y,38ndv1}{\le} \P\Big(\sL_{n,3} \ge (64\fC\rho\e^3 n)^{1/3} + n^{1/4}\Big) \\
&\stackrefp{ibm35y,38ndv1}{\le} \P\Big(\sL_{n,3} \ge \Big(32\frac{\fC\rho}{p_3}\e^3 N_{n,3}\Big)^{1/3} + n^{1/4}\Big) + \P(N_{n,3} \ge 2np_3) \\
&\stackref{ibm35y,38ndv1}{\le}
2^{-n^{1/4}}+\exp(-2p_3^2n).
}
Using Borel--Cantelli once more, we deduce that $\sL_{n,3}<\tfrac{\epsilon}{3}n^{1/3}$ for all large $n$.
\end{proofclaim}

In summary, with probability one we have the following for all large $n$:
\[
\sL_n^\irr = \#\cC_n
= \#\cC_{n,1} + \#\cC_{n,2} + \#\cC_{n,3}
\stackref{C_by_L}{\le}
\sL_{n,1} + \sL_{n,2} + \sL_{n,3}
\stackrel{\mbox{\footnotesize(Claim~\ref{claim_optimal_small})}}{\le}
\epsilon n^{1/3}. \qedhere
\]
\end{proof}

\subsection{{Proof of Proposition~\ref{prop_upper_bound}}} \label{subsec_upper_proof}
Throughout the proof, we will just write ``$\cC$ is convex'' instead of ``$\cC$ is a convex chain.''

\medskip

\noindent \textbf{Step 0: parameter selection.}
Set
\eeq{ \label{parameter_selection}
L_n \coloneqq \floor{n^{1/157}} \qquad \text{and} \qquad \eta_n \coloneqq n^{-1/158}.
}
Recall the constants $\fC$ and $\fc$ from \eqref{p_upper} and \eqref{p_lower}.
Let $\epsilon = \epsilon(\fp) >0$ be small enough that
\begin{align}
\log(54\fC/\fc) + \log \epsilon &< \log(27 J_\star^3/4),  \label{10noll} \\
\text{and} \quad  \epsilon/\alpha &< J_\star/2. \label{11noll}
\end{align}
Then choose $\delta_n>0$ as in Lemma~\ref{lem_bad_unlikely}, take $\rho_1 = \rho_1(\epsilon,\fC/\fc)\in(0,\frac12)$ as in Lemma~\ref{lem_irreg_unlikely}, and take $\rho_2 = \rho_2(\epsilon,\fC/\fc)\in(0,\frac12)$ as in Lemma~\ref{lem_irreg_nonoptimal}.
Set $\rho = \min\{\rho_1,\rho_2\}$ so that both \eqref{ob8hp} and \eqref{irreg_epsilon} are in effect.
Henceforth we simply write $L,\eta,\delta$ for $L_n,\eta_n,\delta_n$.
Since $n^{-1}\ll L^{-1} \ll \eta \ll 1$, we may assume $n$ is large enough that
\begin{subequations} \label{parameter_dominance}
\begin{align}
(\tfrac{16}{\rho})n^{-1} &\le L^{-1}, \label{parameter_dominance1} \\
4L^{-1} &\le \eta, \label{parameter_dominance2} \\
4\eta &\le \rho. \label{parameter_dominance3}
\end{align}
\end{subequations}
One particular consequence is
\eeq{ \label{particular_dominance}
\max\{\tfrac{4}{\rho}L,4\eta^{-1}\}
\stackref{parameter_dominance2}{\le} \max\{\tfrac{4}{\rho}L,L\} 
= \tfrac{4}{\rho}L 
\stackref{parameter_dominance1}{\le} \tfrac{1}{4}n.
}
We also assume $n$ is large enough that $L\ge2$.

\medskip

\noindent \textbf{Step 1a: for parts~\ref{prop_upper_bound_typical} and \ref{prop_upper_bound_far_typical}, reduce to regular realizations.}
Having fixed $\rho$, we define the following ``regular'' versions of $\sL_n$ and $\sL_n^\far(\eps)$:
\eeq{ \label{regular_lengths}
\sL_n^\reg &\coloneqq \max\big\{\#\cC:\, \text{$\cC\subseteq\cS_n$, $\cC$ is convex and $\rho$-regular}\big\}, \\
\sL_n^{\reg,\far}(\eps) &\coloneqq \max\left\{\#\cC:\, \parbox{0.4\textwidth}{$\cC\subseteq\cS_n$, $\cC$ is convex and $\rho$-regular, and $\displaystyle\inf_{h\in\argmax J}\max_{(x,y)\in\cC}|h(x)-y|\ge\eps$}\right\}.
}
\begin{claim}[Reduction to regular realizations] \label{claim_suffice_typical}
To prove Propositions~\ref{prop_upper_bound}\ref{prop_upper_bound_typical} and \ref{prop_upper_bound}\ref{prop_upper_bound_far_typical}, respectively, it suffices to show
\begin{subequations} \label{upper_reg}
\begin{align}
\limsup_{n\to\infty}\frac{\sL_n^\reg}{\alpha n^{1/3}} &\le \mathrlap{\frac{J_\star}{2}}\hphantom{\frac{J_\star-\theta}{2}} \quad \text{almost surely,} \label{upper_reg1} \\
\limsup_{n\to\infty}\frac{\sL_n^{\reg,\far}(\eps)}{\alpha n^{1/3}} &\le \frac{J_\star-\theta}{2} \quad \text{almost surely, for some $\theta>0$.} \label{upper_reg2}
\end{align}
\end{subequations}
\end{claim}

\begin{proofclaim}
By Definition~\ref{def_regular}, every convex chain is either $\rho$-regular or $\rho$-irregular.
The irregular chains satisfy
\eq{
\limsup_{n\to\infty}\frac{\sL_n^\irr}{\alpha n^{1/3}} 
\stackref{irreg_epsilon}{\le} \frac{\epsilon}{\alpha} 
\stackref{11noll}{<} %\frac{\bar J(\phi,\eps)}{2} 
\frac{J_\star}{2}.
}
Therefore, \eqref{upper_almost_sure} follows from \eqref{upper_reg1}, and \eqref{upper_far_almost_sure} follows from \eqref{upper_reg2}. 
\end{proofclaim}

\medskip

\noindent \textbf{Step 1b: for parts~\ref{prop_upper_bound_full} and \ref{prop_upper_bound_far_full}, reduce to good regular realizations.}
Let $(\Omega,\mathfrak{S},\P)$ denote the probability space on which the sample set $\cS_n = \{(X_1,Y_1),\ldots,(X_n,Y_n)\}$ is defined.
Define the event
\eeq{ \label{H_eps_def}
\Omega_\eps &\coloneqq \Big\{\inf_{h\in\argmax J} \max_{i\in\{1,\ldots,n\}} |h(X_i)-Y_i|\ge\eps\Big\}.
}

\begin{claim}[Reduction to good regular realizations] \label{claim_suffice_full}
To prove Propositions~\ref{prop_upper_bound}\ref{prop_upper_bound_full} and \ref{prop_upper_bound}\ref{prop_upper_bound_far_full}, respectively, it suffices to show
\begin{subequations} \label{29bx4r}
\begin{align}
\label{29bx4r_a}
\limsup_{n\to\infty}\frac{1}{n}\log\Big[\frac{(3n)!}{n!} \P(\text{$\cS_n$ is convex, $\rho$-regular, $(\eta,\delta)$-good})\Big] 
&\le \log\Big(\frac{27 J_\star^3}{4}\Big), \\
\label{29bx4r_b}
\limsup_{n\to\infty}\frac{1}{n}\log\Big[\frac{(3n)!}{n!} \P\big(\{\text{$\cS_n$ is convex, $\rho$-regular, $(\eta,\delta)$-good}\}\cap\Omega_\eps\big)\Big]
&< \log\Big(\frac{27 J_\star^3}{4}\Big).
\end{align}
\end{subequations}
\end{claim}

\begin{proofclaim}
Using the notation of \eqref{H_eps_def}, Propositions~\ref{prop_upper_bound}\ref{prop_upper_bound_full} and \ref{prop_upper_bound}\ref{prop_upper_bound_far_full}
assert that
\eeq{ \label{28bx4r}
\limsup_{n\to\infty}\frac{1}{n}\log\Big[\frac{(3n)!}{n!}\P(\{\text{$\cS_n$ is convex}\}\cap A)\Big] 
\begin{cases} 
\le \log(27 J_\star^3/4) &\text{if $A = \Omega$} \\
< \log(27 J_\star^3/4) &\text{if $A = \Omega_\eps$}.
\end{cases}
}
So we will prove \eqref{28bx4r} assuming \eqref{29bx4r}.
By the upper bound in Proposition~\ref{prop_general}\ref{prop_general_likelihood}, we have
\eeq{ \label{2ob99n}
\frac{1}{n}\log\Big[\frac{(3n)!}{n!}\P(\text{$\cS_n$ is convex})\Big]
&\le \frac{1}{n}\log\Big[\frac{(3n)!}{n!}\cdot(\fC/\fc)^n\cdot\frac{2^n}{n!(n+1)!}\Big] \\
&= \log(54\fC/\fc) + o(1),
}
where the equality comes from Stirling's formula.
For all large $n$, a trivial union bound gives
\eeq{ \label{ekb7x}
\P(\{\text{$\cS_n$ is convex}\}\cap A)
&\stackrefp{ob8hp,vo48xv}{\le} \P(\{\text{$\cS_n$ is convex, $\rho$-regular, $(\eta,\delta)$-good}\}\cap A) \\
&\phantom{\stackrefp{ob8hp,vo48xv}{\le}}+ \P(\text{$\cS_n$ is convex and $\rho$-irregular}) \\
&\phantom{\stackrefp{ob8hp,vo48xv}{\le}}+ \P(\text{$\cS_n$ is convex and $(\eta,\delta)$-bad}) \\
&\stackref{ob8hp,vo48xv}{\le}  \P(\{\text{$\cS_n$ is convex, $\rho$-regular, $(\eta,\delta)$-good}\}\cap A) \\
&\phantom{\stackrefp{ob8hp,vo48xv}{\le}} + 2\epsilon^n\cdot\P(\text{$\cS_n$ is convex}).
}
There are two cases to consider, depending on which term in the final line is largest.
If the first term is largest, then we simply multiply it by two to obtain
\eq{
&\frac{1}{n}\log\Big[\frac{(3n)!}{n!}\P(\{\text{$\cS_n$ is convex}\}\cap A)\Big] \\
&\stackrefp{29bx4r}{\le} \frac{1}{n}\log\Big[\frac{(3n)!}{n!}\cdot 2 \P(\{\text{$\cS_n$ is convex, $\rho$-regular, $(\eta,\delta)$-good}\}\cap A)\Big] \\ 
&\hspace{-1.9ex}\begin{cases} 
\stackref{29bx4r}{\le} \log(27 J_\star^3/4) + o(1) &\text{if $A = \Omega$} \\
\stackref{29bx4r}{<} \log(27 J_\star^3/4) + o(1) &\text{if $A = \Omega_\eps$}.
\end{cases}
}
Otherwise the second term $2\epsilon^n\cdot\P(\text{$\cS_n$ is convex})$ is largest, in which case \eqref{ekb7x} implies
\eq{
\frac{1}{n}\log\Big[\frac{(3n)!}{n!}\P(\{\text{$\cS_n$ is convex}\}\cap A)\Big]
&\stackrefp{2ob99n}{\le} \frac{1}{n}\log\Big[\frac{(3n)!}{n!}\cdot 4\epsilon^n\cdot\P(\text{$\cS_n$ is convex})\Big] \\
&\stackref{2ob99n}{\le} \log(54\fC/\fc) + \log\epsilon + o(1) \\
&\stackref{10noll}{<}\log(27 J_\star^3/4) + o(1).
}
Allowing either case, we obtain \eqref{28bx4r}.
\end{proofclaim}

\medskip

\noindent \textbf{Step 2: identify function from data set.}
In this step, we fix $n$ and a $\rho$-regular convex chain $\cC\subset\Interior(\cT)$ with $\#\cC\le n$.
We will construct a function $\tilde f = \tilde f_{\cC,n} \colon[0,1]\to[0,1]$ that is completely determined by $\cC$ and $n$.
Ultimately we only care about the case when $\cC$ is a maximizer for one of the lengths in \eqref{regular_lengths}, but optimality is not needed here.
In fact, there is no randomness at all in this step.

Let $f$ be the piecewise linear function $f_\cC$ from Definition~\ref{def_regular}.
Since $\cC$ is $\rho$-regular, we have
\eeq{ \label{regular_consequence}
f(1-\rho) \ge \rho.
}
Since every element of $\cC$ lies on the graph of $f$, it follows from \eqref{graph_in_triangle} that
\eeq{ \label{v28bi9}
\cC \cap ([a,b]\times[0,1]) \subset \Tri_f(a,b) \quad \text{for every $a<b$ in $[0,1]$}.
}
Roughly speaking, the goal of Step 2 is to construct a function $\tilde f$ such that $J(\tilde f)$ is as large as possible subject to the constraint that $\tilde f \approx f$ in the sense that \eqref{v28bi9} still holds for $\tilde f$ (not necessarily for every interval $[a,b]$, but rather a set of intervals whose union covers $[0,1]$).
The challenging part is that we must limit ourselves to a small number of choices for $\tilde f$ (see \eqref{cardinality_upper}), so that union bounds in Steps 3 and 5 are not too loose.

\medskip

\noindent \textbf{Step 2a: horizontal discretization.}
Since $f(0)=0$ and $f(1)=1$, the intermediate value theorem gives some $u\in(0,1)$ such that 
\eeq{ \label{u_def}
f(u) = \eta/4.
}
Since $f(1-\rho)\ge\rho>\eta/4$ by \eqref{regular_consequence} and \eqref{parameter_dominance3}, and $f$ is nondecreasing, we have $u < 1-\rho$.
Define $a_1 = n^{-1}\ceil{n u}$, and note that
\eeq{ \label{a1_upper}
a_1 
\le u + n^{-1} 
< 1 - \rho + n^{-1} 
\stackref{parameter_dominance1}{\le} 1 - \rho + L^{-1} 
\stackref{parameter_dominance2}{\le} 1 - \rho + \eta
\stackref{parameter_dominance3}{\le} 1 - 3\rho/4.
}
Now divide $[0,1]$ into $L+1$ subintervals $[x_0,x_1],[x_1,x_2],\ldots,[x_{L},x_{L+1}]$, as follows.
Set $x_0=0$ and $x_{L+1}=1$.
To decide the remaining values $x_1,\ldots,x_{L}$, we consider the following equally spaced reference points:
\eeq{ \label{a_ell_def}
a_\ell \coloneqq a_1 + \frac{\ell-1}{L-1}(1-\eta-a_1), \quad \ell\in\{2,\ldots,L\}.
}
Here we have used the assumption $L\ge2$ to ensure $L-1\ne0$.
Note that
\eeq{ \label{a_ell_spacing_lower}
a_{\ell+1}-a_\ell = \frac{1-\eta-a_1}{L-1} 
\stackref{a1_upper}{\ge} \frac{3\rho/4-\eta}{L} 
\stackref{parameter_dominance3}{\ge} \frac{\rho/2}{L}
\quad \text{for every $\ell\in\{1,\ldots,L-1\}$}.
}
We also have a trivial upper bound:
\eeq{ \label{a_ell_spacing_upper}
a_{\ell+1} - a_{\ell} \le (L-1)^{-1} \le 2L^{-1} \quad \text{for every $\ell\in\{1,\ldots,L-1\}$}.
}
For each $\ell\in\{1,\ldots,L\}$, partition the interval $[a_\ell,a_\ell+\tfrac{\rho}{4}L^{-1}]$ into $n+1$ disjoint subintervals of equal length.
Since $\#\cC\le n$, at least one of these $n+1$ subintervals has empty intersection with the set
\eq{
\cX \coloneqq \big\{x:\, \text{$(x,y)\in\cC$ for some $y\in[0,1]$}\big\}.
}
Choose such a subinterval (in any manner), and denote its midpoint by $x_\ell$ so that
\eeq{ \label{28buf}
\mathrm{dist}(x_\ell,\cX) 
\ge \frac{\rho}{8(n+1)}L^{-1}
\ge \frac{\rho}{16 n}L^{-1}
\stackref{parameter_dominance1}{\ge} n^{-2}
\quad \text{for every $\ell\in\{1,\ldots,L\}$.}
}
Furthermore, since $x_\ell\notin\cX$ and $f$ is locally affine (hence differentiable) at every $x\in(0,1)\setminus\cX$, we know
\eeq{ \label{f_diff_at_x_ell}
\text{$f$ is differentiable at $x_\ell$ for every $\ell\in\{1,\ldots,L\}$.}
}
Since $a_1=n^{-1}\ceil{nu}$ can take at most $n$ different values, and $x_1,\ldots,x_L$ were each chosen from among $n+1$ possibilities (given the value of $a_1$), we have that
\eeq{ \label{x_ell_count}
\text{the number of possible outcomes for the sequence $(x_\ell)_{\ell=0}^{L+1}$ is at most $n(n+1)^{L}$.}
}
We complete Step 2a by recording estimates on the length of the intervals $[x_\ell,x_{\ell+1}]$.
For $\ell\in\{1,\ldots,L\}$, we have $x_\ell\in(a_\ell,a_\ell+\frac{\rho}{4}L^{-1})$.
In the particular, the $\ell=0$ interval $[0,x_1]$ obeys the following bounds:
\begin{align}
\label{x1_lower}
x_1 &> a_1 = n^{-1}\ceil{n u} \ge n^{-1}, \\
\label{x1_upper}
\text{and} \quad x_1 &< a_1 + \tfrac{\rho}{4}L^{-1} \stackref{a1_upper}{\le} 1 - 3\rho/4 + \tfrac{\rho}{4} = 1-\rho/2.
\end{align}
Meanwhile, the interior intervals $[x_\ell,x_{\ell+1}]$ satisfy
\begin{align}
\label{increment_lower}
x_{\ell+1}-x_\ell 
&> a_{\ell+1}-(a_\ell+\tfrac{\rho}{4}L^{-1})
\stackref{a_ell_spacing_lower}{\ge}\tfrac{\rho}{4}L^{-1} \quad \text{for all $\ell\in\{1,\ldots,L-1\}$}, \\
\label{increment_upper}
\text{and} \quad
x_{\ell+1}-x_\ell
&<\mathrlap{(a_{\ell+1}+\tfrac{\rho}{4}L^{-1}) - a_\ell
\stackref{a_ell_spacing_upper}{\le} 3L^{-1}}\hphantom{a_{\ell+1}-(a_\ell+\tfrac{\rho}{4}L^{-1})
\stackref{a_ell_spacing_lower}{\ge}\tfrac{\rho}{4}L^{-1}} \quad \text{for all $\ell\in\{1,\ldots,L-1\}$}.
\end{align}
Finally, the last interval $[x_L,1]$ has length $\asymp \eta$, since
\begin{align}
\label{last_lower}
&x_L > a_L \stackref{a_ell_def}{=} 1-\eta, \\
\label{last_upper}
\text{and} \quad
&x_L 
< a_L + \tfrac{\rho}{4}L^{-1} 
\stackref{a_ell_def}{=} 1-\eta + \tfrac{\rho}{4}L^{-1} 
\le 1-\eta + L^{-1}
\stackref{parameter_dominance2}{\le} 1-\eta/2.
\end{align}

\medskip

\noindent \textbf{Step 2b: vertical coordinates and slopes.}
See Figure~\ref{fig_step2b} for an illustration of the following definitions.
First define the following discretized vertical coordinate:
\eeq{ \label{y_ell_def}
y_\ell \coloneqq n^{-3}\ceil{n^3 f(x_\ell)} \wedge x_\ell.
}
In particular, $y_0=0$ and $y_{L+1}=1$.
For each $\ell\in\{1,\ldots,L\}$ we have $f(x_\ell)\in(0,1)$, hence $y_\ell\in\{n^{-3},2n^{-3},\ldots,(n^3-1)n^{-3},x_\ell\}$.
Therefore,
\eeq{ \label{y_ell_count}
\parbox{0.8\textwidth}{\centering the number of possible outcomes for the sequence $(y_\ell)_{\ell=0}^{L+1}$ is at most $n^{3L}$ (given the values of $x_1,\ldots,x_L$).}
}
Furthermore, the sequence $(y_\ell)_{\ell=0}^{L+1}$ is nondecreasing, and 
\eeq{ \label{y_ell_consequence}
f(x_\ell)\le y_\ell\le (f(x_\ell)+n^{-3})\wedge x_\ell
\quad \text{for every $\ell\in\{0,\ldots,L+1\}$.}
}
Also note that
\eeq{ \label{y1_upper}
y_1 
\stackref{y_ell_consequence}{\le} f(x_1) + n^{-3}
&\stackrefpp{u_def}{x1_upper,a1_upper}{=} \eta/4 + \frac{f(x_1) - f(u)}{x_1-u}\cdot(x_1-u) + n^{-3} \\
&\stackrefp{x1_upper,a1_upper}{\le} \eta/4 + \frac{f(1)-f(x_1)}{1-x_1}\cdot(x_1-u) + n^{-3} \\
&\stackrefp{x1_upper,a1_upper}{\le} \eta/4 + \frac{1}{1-x_1}\cdot[(x_1-a_1) + (a_1 - u)] + n^{-3} \\
&\stackref{x1_upper,a1_upper}{\le} \eta/4 + (2/\rho)\cdot[\tfrac{\rho}{4}L^{-1} + n^{-1}] + n^{-3} \\
&\stackrefpp{parameter_dominance1}{x1_upper,a1_upper}{\le} \eta/4 + [\tfrac12L^{-1} + L^{-1}] + L^{-1}
\stackref{parameter_dominance2}{\le} \eta.
}
Because of \eqref{f_diff_at_x_ell}, we can define the following discretized slopes (to be used when approaching $x_\ell$ from the left or right):
\eeq{ \label{s_ell_def}
s_\ell^- &\coloneqq \begin{cases}
n^{-1}(\ceil{n f'(x_\ell)}+1) &\text{if $\ell\in\{1,\ldots,L\}$} \\
\infty &\text{if $\ell=L+1$},
\end{cases} \\
s_\ell^+ &\coloneqq \begin{cases}
0 &\text{if $\ell=0$} \\
n^{-1}(\floor{n f'(x_\ell)}-1) &\text{if $\ell\in\{1,\ldots,L\}$}.
\end{cases}
}
Note the following immediate consequences of these definitions:
\begin{align}
s_\ell^- &\ge f'(x_\ell)+n^{-1} \quad \text{for every $\ell\in\{1,\ldots,L\}$,}\label{s_ell_minus_consequence} \\
s_\ell^+ &\le f'(x_\ell)-n^{-1} \quad \text{for every $\ell\in\{1,\ldots,L\}$,} \label{s_ell_plus_consequence} \\
s_\ell^- - s_\ell^+ &\le \mathrlap{3n^{-1}}\hphantom{f'(x_\ell)+n^{-1}}
\quad \text{for every $\ell\in\{1,\ldots,L\}$.} \label{s_ell_jump}
\end{align}
For every $\ell\in\{1,\ldots,L\}$, we have
\eeq{ \label{ho95c}
f'(x_\ell) \ge f'(x_1) \ge \frac{f(x_1)-f(0)}{x_1-0} \ge f(x_1) \ge f(u) \stackref{u_def}{=} \eta/4 \stackref{parameter_dominance2}{\ge} L^{-1}
\stackref{parameter_dominance1}{>} n^{-1}.
}
Considering the definition \eqref{s_ell_def} of $s_\ell^-$ and $s_\ell^+$, \eqref{ho95c} implies
\eeq{ \label{s_ell_lower}
s_\ell^- \ge 3n^{-1} \quad \text{and} \quad s_\ell^+ \ge 0 \quad \text{for every $\ell\in\{1,\ldots,L\}$}.
}
On the other hand, for every $\ell\in\{1,\ldots,L\}$ we have
\eeq{ \label{6rg7jc}
f'(x_\ell) \le f'(x_L) \le \frac{f(1) - f(x_L)}{1-x_L} \le \frac{1}{1-x_L} \stackref{last_upper}{\le} 2\eta^{-1}.
}
Therefore, $(s_\ell^-,s_\ell^+)$ is equal to $(K n^{-1},k n^{-1})$ for some integer $K\in\{3,\ldots,\ceil{2\eta^{-1}n}+1\}$ and some $k\in\{K-3,K-2\}$.
We conclude that
\eeq{ \label{s_ell_count}
\parbox{0.5\textwidth}{\centering the number of possible outcomes for the sequence $(s_{\ell+1}^-,s_\ell^+)_{\ell=0}^{L}$ is at most $(4\eta^{-1} n)^{L}$.}
}

\begin{figure}[t]

\tikzset{every picture/.style={line width=0.75pt}} %set default line width to 0.75pt        

\begin{tikzpicture}[x=0.75pt,y=0.75pt,yscale=-1,xscale=1]
%uncomment if require: \path (0,236); %set diagram left start at 0, and has height of 236

%Straight Lines [id:da7117754098717279] 
\draw [color={rgb, 255:red, 74; green, 144; blue, 226 }  ,draw opacity=1 ][line width=1.5]    (330.5,177.5) -- (365.5,154.5) ;
%Straight Lines [id:da329687852412648] 
\draw [color={rgb, 255:red, 74; green, 144; blue, 226 }  ,draw opacity=1 ][line width=1.5]    (220,210) -- (330.5,177.5) ;
%Straight Lines [id:da9079643072377225] 
\draw [color={rgb, 255:red, 74; green, 144; blue, 226 }  ,draw opacity=1 ][line width=1.5]    (419.5,99.5) -- (460,40) ;
%Straight Lines [id:da5367717713231114] 
\draw [color={rgb, 255:red, 74; green, 144; blue, 226 }  ,draw opacity=1 ][line width=1.5]    (398.5,122.5) -- (419.5,99.5) ;
%Straight Lines [id:da4867394135187899] 
\draw [color={rgb, 255:red, 74; green, 144; blue, 226 }  ,draw opacity=1 ][line width=1.5]    (365.5,154.5) -- (398.5,122.5) ;
%Shape: Polygon [id:ds23055359346385795] 
\draw  [color={rgb, 255:red, 74; green, 144; blue, 226 }  ,draw opacity=1 ][fill={rgb, 255:red, 74; green, 144; blue, 226 }  ,fill opacity=0.1 ][dash pattern={on 1.5pt off 1.5pt}] (220,210) -- (371.82,102.46) -- (460,40) -- (374.5,164) -- cycle ;
%Shape: Square [id:dp9935078945992281] 
\draw  [color={rgb, 255:red, 208; green, 2; blue, 27 }  ,draw opacity=1 ][fill={rgb, 255:red, 208; green, 2; blue, 27 }  ,fill opacity=1 ] (216.75,196.75) -- (223.25,196.75) -- (223.25,203.25) -- (216.75,203.25) -- cycle ;
%Shape: Square [id:dp24002980172429378] 
\draw  [color={rgb, 255:red, 208; green, 2; blue, 27 }  ,draw opacity=1 ][fill={rgb, 255:red, 208; green, 2; blue, 27 }  ,fill opacity=1 ] (456.75,26.75) -- (463.25,26.75) -- (463.25,33.25) -- (456.75,33.25) -- cycle ;
%Shape: Circle [id:dp4045661464065464] 
\draw  [fill={rgb, 255:red, 0; green, 0; blue, 0 }  ,fill opacity=1 ] (327,177.5) .. controls (327,175.57) and (328.57,174) .. (330.5,174) .. controls (332.43,174) and (334,175.57) .. (334,177.5) .. controls (334,179.43) and (332.43,181) .. (330.5,181) .. controls (328.57,181) and (327,179.43) .. (327,177.5) -- cycle ;
%Shape: Circle [id:dp38607427759060686] 
\draw  [fill={rgb, 255:red, 0; green, 0; blue, 0 }  ,fill opacity=1 ] (362,154.5) .. controls (362,152.57) and (363.57,151) .. (365.5,151) .. controls (367.43,151) and (369,152.57) .. (369,154.5) .. controls (369,156.43) and (367.43,158) .. (365.5,158) .. controls (363.57,158) and (362,156.43) .. (362,154.5) -- cycle ;
%Shape: Circle [id:dp8432551989495121] 
\draw  [fill={rgb, 255:red, 0; green, 0; blue, 0 }  ,fill opacity=1 ] (395,122.5) .. controls (395,120.57) and (396.57,119) .. (398.5,119) .. controls (400.43,119) and (402,120.57) .. (402,122.5) .. controls (402,124.43) and (400.43,126) .. (398.5,126) .. controls (396.57,126) and (395,124.43) .. (395,122.5) -- cycle ;
%Shape: Circle [id:dp8840797968540758] 
\draw  [fill={rgb, 255:red, 0; green, 0; blue, 0 }  ,fill opacity=1 ] (416,99.5) .. controls (416,97.57) and (417.57,96) .. (419.5,96) .. controls (421.43,96) and (423,97.57) .. (423,99.5) .. controls (423,101.43) and (421.43,103) .. (419.5,103) .. controls (417.57,103) and (416,101.43) .. (416,99.5) -- cycle ;
%Straight Lines [id:da9555770219860351] 
\draw [color={rgb, 255:red, 208; green, 2; blue, 27 }  ,draw opacity=1 ]   (220,200) -- (415.5,184.5) ;
%Shape: Circle [id:dp947486115533901] 
\draw  [fill={rgb, 255:red, 255; green, 255; blue, 255 }  ,fill opacity=1 ] (217,210) .. controls (217,208.34) and (218.34,207) .. (220,207) .. controls (221.66,207) and (223,208.34) .. (223,210) .. controls (223,211.66) and (221.66,213) .. (220,213) .. controls (218.34,213) and (217,211.66) .. (217,210) -- cycle ;
%Shape: Circle [id:dp29575278884959155] 
\draw  [fill={rgb, 255:red, 255; green, 255; blue, 255 }  ,fill opacity=1 ] (457,40) .. controls (457,38.34) and (458.34,37) .. (460,37) .. controls (461.66,37) and (463,38.34) .. (463,40) .. controls (463,41.66) and (461.66,43) .. (460,43) .. controls (458.34,43) and (457,41.66) .. (457,40) -- cycle ;
%Straight Lines [id:da9397953416857351] 
\draw [color={rgb, 255:red, 208; green, 2; blue, 27 }  ,draw opacity=1 ]   (415.5,184.5) -- (460,30) ;

% Text Node
\draw (369,120) node [anchor=north west][inner sep=0.75pt]  [color={rgb, 255:red, 74; green, 144; blue, 226 }  ,opacity=1 ]  {$f$};
% Text Node
\draw (169,177.4) node [anchor=north west][inner sep=0.75pt]  [color={rgb, 255:red, 208; green, 2; blue, 27 }  ,opacity=1 ]  {$( x_{\ell } ,y_{\ell })$};
% Text Node
\draw (262,95) node [anchor=north west][inner sep=0.75pt]  [color={rgb, 255:red, 74; green, 144; blue, 226 }  ,opacity=1 ]  {$\Tri_{f}( x_{\ell } ,x_{\ell +1})$};
% Text Node
\draw (380,7.4) node [anchor=north west][inner sep=0.75pt]  [color={rgb, 255:red, 208; green, 2; blue, 27 }  ,opacity=1 ]  {$( x_{\ell +1} ,y_{\ell +1})$};
% Text Node
\draw (222,213.4) node [anchor=north west][inner sep=0.75pt]    {$( x_{\ell } ,f( x_{\ell }))$};
% Text Node
\draw (462,43.4) node [anchor=north west][inner sep=0.75pt]    {$( x_{\ell +1} ,f( x_{\ell +1}))$};
% Text Node
\draw (345,190) node [anchor=north west][inner sep=0.75pt]  [color={rgb, 255:red, 208; green, 2; blue, 27 }  ,opacity=1 ] [align=left] {slope $\displaystyle s_{\ell }^{+}$};
% Text Node
\draw (428,149) node [anchor=north west][inner sep=0.75pt]  [color={rgb, 255:red, 208; green, 2; blue, 27 }  ,opacity=1 ] [align=left] {slope $\displaystyle s_{\ell +1}^{-}$};

\end{tikzpicture}
\caption{
Illustration of Step 2b. 
The piecewise linear function $f$ is shown in solid blue.  The shaded region (outlined in dashed blue) is the tangency triangle of $f$ between $x_\ell$ and $x_{\ell+1}$. 
The vertical coordinates $y_\ell$ and $y_{\ell+1}$ are chosen slightly above $f(x_\ell)$ and $f(x_{\ell+1})$, respectively, according to \eqref{y_ell_def}.
The slope $s_\ell^+$ is chosen slightly smaller than $f'(x_{\ell})$, while $s_{\ell+1}^-$ is chosen slightly larger than $f'(x_{\ell+1})$, according to \eqref{s_ell_def}.
Claim~\ref{claim_boundary_good} shows that these slopes flank that of the line segment between $(x_\ell,y_\ell)$ and $(x_{\ell+1},y_{\ell+1})$ (not shown).
}
\label{fig_step2b}

\end{figure}
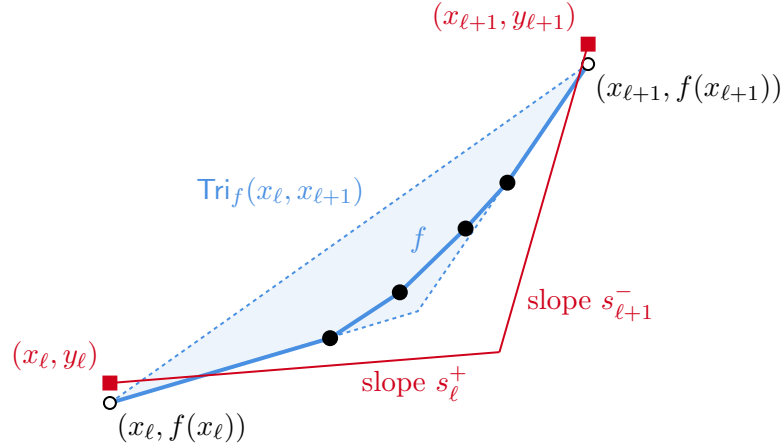

Finally, we record the following fact which will be needed in Step 2c.

\begin{claim}[Slopes make sense] \label{claim_boundary_good}
For every $\ell\in\{0,\ldots,L\}$, we have
\eeq{ \label{cx64mg}
s_\ell^+ < \frac{y_{\ell+1}-y_\ell}{x_{\ell+1}-x_\ell} < s_{\ell+1}^-.
}
\end{claim}

\begin{proofclaim}
For $\ell=0$, we have $s_0^+ = 0$ while $y_1 - y_0 = y_1 > 0$, so the first inequality in \eqref{cx64mg} holds.
For $\ell\in\{1,\ldots,L\}$, we argue as follows:
\eq{
s_\ell^+ 
\stackref{s_ell_plus_consequence}{\le} f'(x_\ell) - n^{-1}
&\stackrefp{last_upper,increment_lower}{\le} \frac{f(x_{\ell+1})-f(x_\ell)}{x_{\ell+1}-x_\ell} - n^{-1} \\
&\stackrefpp{y_ell_consequence}{last_upper,increment_lower}{\le} \frac{y_{\ell+1} - (y_\ell-n^{-3})}{x_{\ell+1}-x_\ell} - n^{-1} \\
&\stackref{increment_lower,last_upper}{<} \frac{y_{\ell+1} - y_\ell}{x_{\ell+1}-x_\ell} + (\tfrac{4}{\rho}L\vee 2\eta^{-1})n^{-3} - n^{-1} \\
&\stackrefpp{particular_dominance}{last_upper,increment_lower}{\le}\frac{y_{\ell+1} - y_\ell}{x_{\ell+1}-x_\ell} + (\tfrac14n) n^{-3} - n^{-1}
<\frac{y_{\ell+1} - y_\ell}{x_{\ell+1}-x_\ell}.
}
The second inequality in \eqref{cx64mg} is verified similarly: it is trivial for $\ell=L$ since $s_{L+1}^-=\infty$, while for $\ell\in\{0,\ldots,L-1\}$ we have
\begin{align*}
s_{\ell+1}^-
\stackref{s_ell_minus_consequence}{\ge} f'(x_{\ell+1}) + n^{-1}
&\stackrefp{x1_lower,increment_lower}{\ge} \frac{f(x_{\ell+1})-f(x_\ell)}{x_{\ell+1}-x_\ell} + n^{-1} \\
&\stackrefpp{y_ell_consequence}{x1_lower,increment_lower}{\ge} \frac{y_{\ell+1}-n^{-3} - y_\ell}{x_{\ell+1}-x_\ell} + n^{-1} \\
&\stackref{x1_lower,increment_lower}{>} \frac{y_{\ell+1} - y_\ell}{x_{\ell+1}-x_\ell} - (n \vee \tfrac{4}{\rho}L)n^{-3} + n^{-1} \\
&\stackrefpp{parameter_dominance1}{x1_lower,increment_lower}{=}\frac{y_{\ell+1} - y_\ell}{x_{\ell+1}-x_\ell} - n\cdot n^{-3} + n^{-1}
> \frac{y_{\ell+1} - y_\ell}{x_{\ell+1}-x_\ell}. 
\end{align*}
We have now proved both inequlities in \eqref{cx64mg}.
\end{proofclaim}

\medskip

\noindent \textbf{Step 2c: the induced function.}
We finally define the desired function $\tilde f = \tilde f_{\cC,n}\colon[0,1]\to[0,1]$, using a piecewise construction.
On the interval $[x_\ell,x_{\ell+1}]$, let $\tilde f$ be the optimizer from Proposition~\ref{prop_unif_case} with boundary data
\eeq{ \label{tilde_f_def}
\tilde f(x_\ell) = y_\ell, \quad
\tilde f(x_{\ell+1}) = y_{\ell+1}, \quad
\partial^+\tilde f(x_\ell) = s_\ell^+, \quad
\partial^-\tilde f(x_{\ell+1}) = s_{\ell+1}^-.
}
More specifically, we use the formula \eqref{17xxih} when $\ell\in\{0,\ldots,L-1\}$ (since $s_{\ell+1}^-<\infty$), and we use the formula \eqref{29cwq} when $\ell=L$ (since $s_{L+1}^-=\infty$).
These formulas are well-defined for the boundary data \eqref{tilde_f_def} because the required hypothesis \eqref{eq:feasible} was verified in Claim~\ref{claim_boundary_good}.
We are using $s_\ell^+\ge 0$ from \eqref{s_ell_lower} to ensure $\tilde f(x)\ge 0$.
Because $y_\ell\le x_\ell$ and $y_{\ell+1}\le x_{\ell+1}$ by \eqref{y_ell_consequence}, we also have $\tilde f(x) \le x$ by convexity.
Therefore, $\tilde f(x)$ belongs to $[0,x]$, meaning $(x,f(x))$ belongs to the triangle $\cT$.
This, together with the fact that $\tilde f''(x)$ exists and is nonnegative at every $x\in(0,1)\setminus\{x_1,\ldots,x_L\}$, allows us to define $J(\tilde f)$ as in \eqref{J_def1}.
Finally, note that Corollary~\ref{cor_unif_case_characterization} gives
\eeq{ \label{tilde_f_area}
\int_{x_\ell}^{x_{\ell+1}}\tilde f''(x)^{1/3}\ \dd x = 2\cdot\Area\big(\Tri_{\tilde f}(x_\ell,x_{\ell+1})\big)^{1/3}.
}

Let $\wt\cF_n$ denote the set of possible outcomes of $\tilde f$; that is,
\eq{
\wt\cF_n \coloneqq \{\tilde f_{\cC,n}:\, \text{$\cC$ is a $\rho$-regular convex chain with $\#\cC\le n$}\}.
}
Putting together \eqref{x_ell_count}, \eqref{y_ell_count}, and \eqref{s_ell_count}, we obtain the first inequality below:
\eeq{ \label{cardinality_upper}
\#\wt\cF_n 
&\stackrefp{particular_dominance}{\le} n (n+1)^{L} \cdot n^{3L} \cdot (4\eta^{-1}n)^{L} \\
&\stackref{particular_dominance}\le n(2n)^L\cdot n^{3L} \cdot (\tfrac14n^2)^L
\le n^{7L}.
}
In particular, because $L = o(n/\log n)$ as is obvious from \eqref{parameter_selection}, we have
\eeq{ \label{cardinality_subexp}
\lim_{n\to\infty}\frac{\log \#\wt\cF_n}{n} = 0.
}

\begin{remark}[Small convexity defects of $\tilde f$] \label{rem_jumps}
By positivity of the second derivatives in \eqref{17xxih} and \eqref{29cwq}, the function $\tilde f$ is strictly convex on $[x_\ell,x_{\ell+1}]$, but not convex on $[0,1]$ because $s_\ell^->s_\ell^+$ for every $\ell\in\{1,\ldots,L\}$.
On the other hand, \eqref{s_ell_jump} gives $s_\ell^--s_\ell^+\le 3n^{-1}$, so these issues become negligible as $n\to\infty$.
This slight failure of convexity will be accounted for in Step 4.
\qedrem
\end{remark}

\begin{remark}[Extension of $J$ is still bounded] \label{rem_extension_bound}
It is shown in Step 4 that
\eeq{ \label{extension_bound}
\lim_{n\to\infty}\sup_{\tilde f\in\wt\cF_n}J(\tilde f)
\le \sup_{g\in\cF} J(g) = J_\star.
}
Specifically, see the second limit in \eqref{step_4_goal}.
Because of \eqref{extension_bound}, we can treat $J(\tilde f)$ as a uniformly bounded quantity, which will be useful for the estimates in Step 3.
\qedrem
\end{remark}

The remaining conclusions of Step 2 are in the next two claims. 
Denote the tangency triangles of $\tilde f$, and associated vertices, by
\eeq{ \label{tilde_triangle}
\triangle_\ell \coloneqq \Tri_{\tilde f}(x_\ell,x_{\ell+1}),
\quad
\nA_\ell \coloneqq (x_\ell,\tilde f(x_\ell)),
\quad
\nB_\ell \coloneqq \nA_{\ell+1} = (x_{\ell+1},\tilde f(x_{\ell+1})).
}
The subset of $\cC$ corresponding to the horizontal interval $[x_\ell,x_{\ell+1}]$ will be denoted by
\eeq{ \label{Cell_def}
\cC_\ell \coloneqq \cC \cap \big([x_\ell,x_{\ell+1}]\times[0,1]\big).
}

\begin{claim}[Capturing the convex chain] \label{claim_capture} \phantom{x}  \noindent
\begin{enumerate}[label=\textup{(\alph*)}]

\item \label{claim_capture_tri}
$\cC_\ell \subset \triangle_\ell\setminus\{\nA_\ell,\nB_\ell\}$, for every $\ell\in\{0,\ldots,L\}$.

\item \label{claim_capture_convex}
$\{\nA_\ell,\nB_\ell\}\cup\cC_\ell$ is in convex position, for every $\ell\in\{0,\ldots,L\}$.

\item \label{claim_capture_pt}
If $(x,y)\in\cC$ and $x \ge x_L$, then $y \ge \tilde f(x_L) - n^{-3}$.

\end{enumerate}
\end{claim}

\begin{claim}[Properties of the tangency triangles] \phantom{x} \label{claim_triangles}
\begin{enumerate}[label=\textup{(\alph*)}]

\item \label{claim_triangles_valid}
For every $\ell\in\{0,\ldots,L\}$, we have
\eeq{ \label{tri_in_tri}
\triangle_\ell\subset\cT\cap\big([x_\ell,x_{\ell+1}]\times[\tilde f(x_\ell),\tilde f(x_{\ell+1})]\big).
}

\item \label{claim_triangles_small}
There exists a sequence $(c_n)_{n\ge1}$ not depending on $\cC$ such that $\lim_{n\to\infty}c_n = 0$ and
\eeq{ \label{tri_in_small_rec}
\triangle_\ell\subset[x_\ell,x_\ell+c_n]\times[\tilde f(x_\ell),\tilde f(x_\ell)+c_n] \quad \text{for every $\ell\in\{1,\ldots,L-1\}$}.
}

\end{enumerate}
\end{claim}

\begin{proofclaim}[Proof of Claim~\ref{claim_capture}]
Recall that every element of $\cC$ lies on the graph of $f$, specifically at a corner.
First we prove part~\ref{claim_capture_pt}. 
If $(x,y)\in\cC$ and $x \ge x_L$, then
\eq{
y = f(x) \ge f(x_L)
\stackref{y_ell_consequence}{\ge} y_L - n^{-3}
\stackref{tilde_f_def}{=} \tilde f(x_L) - n^{-3}.
}
Next we prove part~\ref{claim_capture_convex}.
Label the elements of $\cC_\ell$ in increasing order of their horizontal coordinates, as $\nP_1,\ldots,\nP_{m}$.
Since $\nP_1,\ldots,\nP_m$ are all corners in the graph of $f$, we have
\eq{ %\label{3on9xh}
\Slope\big((x_\ell,f(x_\ell)),\nP_1\big) < \Slope(\nP_1,\nP_2) < \cdots &< \Slope(\nP_{m-1},\nP_m) 
< \Slope\big(\nP_m,(x_{\ell+1},f(x_{\ell+1}))\big).
}
The first and last inequalities remain true if we replace $f(x_\ell)$ and $f(x_{\ell+1})$ with $\tilde f(x_\ell)$ and $\tilde f(x_{\ell+1})$ respectively, since \eqref{y_ell_consequence} gives $f(x_\ell) \le y_\ell = \tilde f(x_\ell)$ and $f(x_{\ell+1})\le y_{\ell+1} = \tilde f(x_{\ell+1})$.
Therefore, $\cC_\ell\cup\{(x_\ell,\tilde f(x_\ell)),(x_{\ell+1},\tilde f(x_{\ell+1}))\}$ is in convex position by Lemma~\ref{lem_position_slopes}.

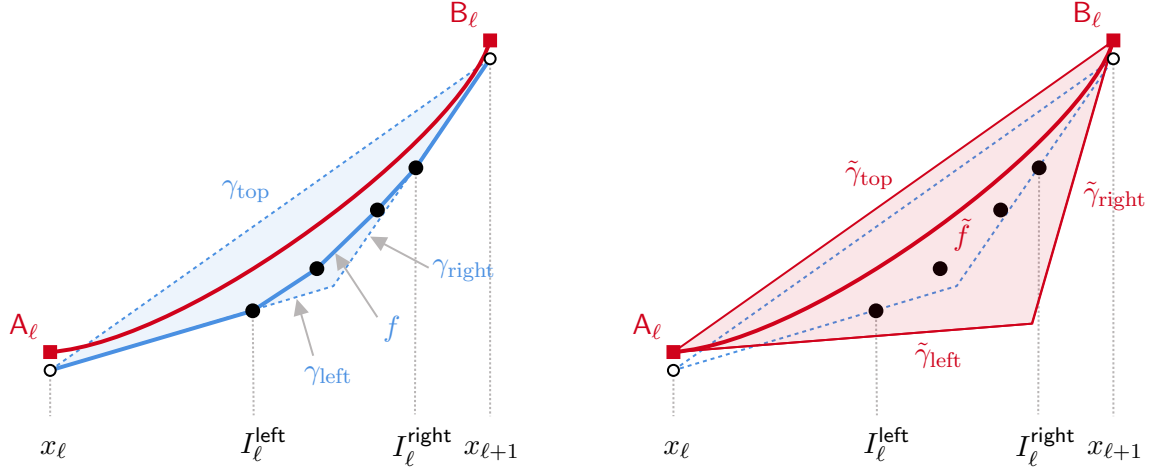
\begin{figure}

\tikzset{every picture/.style={line width=0.75pt}} %set default line width to 0.75pt        

\begin{tikzpicture}[x=0.75pt,y=0.75pt,yscale=-0.92,xscale=0.92]
%uncomment if require: \path (0,267); %set diagram left start at 0, and has height of 267

%Shape: Polygon [id:ds739952034564482] 
\draw  [color={rgb, 255:red, 74; green, 144; blue, 226 }  ,draw opacity=1 ][dash pattern={on 1.5pt off 1.5pt}] (380,210) -- (620,40) -- (534.5,164) -- cycle ;
%Straight Lines [id:da14957986456043604] 
\draw [color={rgb, 255:red, 74; green, 144; blue, 226 }  ,draw opacity=1 ][line width=1.5]    (150.5,177.5) -- (185.5,154.5) ;
%Straight Lines [id:da3547603248346124] 
\draw [color={rgb, 255:red, 74; green, 144; blue, 226 }  ,draw opacity=1 ][line width=1.5]    (40,210) -- (150.5,177.5) ;
%Straight Lines [id:da9297144685408857] 
\draw [color={rgb, 255:red, 74; green, 144; blue, 226 }  ,draw opacity=1 ][line width=1.5]    (239.5,99.5) -- (280,40) ;
%Straight Lines [id:da3624172044694799] 
\draw [color={rgb, 255:red, 74; green, 144; blue, 226 }  ,draw opacity=1 ][line width=1.5]    (218.5,122.5) -- (239.5,99.5) ;
%Straight Lines [id:da5729301506789737] 
\draw [color={rgb, 255:red, 74; green, 144; blue, 226 }  ,draw opacity=1 ][line width=1.5]    (185.5,154.5) -- (218.5,122.5) ;
%Straight Lines [id:da3809182005369419] 
\draw [color={rgb, 255:red, 184; green, 181; blue, 181 }  ,draw opacity=1 ][dash pattern={on 0.84pt off 0.84pt}]   (380,210) -- (380,235) ;
%Straight Lines [id:da6873061762711743] 
\draw [color={rgb, 255:red, 184; green, 181; blue, 181 }  ,draw opacity=1 ][dash pattern={on 0.84pt off 0.84pt}]   (620,40) -- (620,235) ;
%Straight Lines [id:da37797117877016373] 
\draw [color={rgb, 255:red, 184; green, 181; blue, 181 }  ,draw opacity=1 ][dash pattern={on 0.84pt off 0.84pt}]   (490.5,177.5) -- (491,234.5) ;
%Straight Lines [id:da3806667690580239] 
\draw [color={rgb, 255:red, 184; green, 181; blue, 181 }  ,draw opacity=1 ][dash pattern={on 0.84pt off 0.84pt}]   (579.5,99.5) -- (579,234) ;
%Straight Lines [id:da6720665070079442] 
\draw [color={rgb, 255:red, 184; green, 181; blue, 181 }  ,draw opacity=1 ][dash pattern={on 0.84pt off 0.84pt}]   (40,210) -- (40,235) ;
%Straight Lines [id:da5079291405976082] 
\draw [color={rgb, 255:red, 184; green, 181; blue, 181 }  ,draw opacity=1 ][dash pattern={on 0.84pt off 0.84pt}]   (280,40) -- (280,235) ;
%Straight Lines [id:da1040184925795018] 
\draw [color={rgb, 255:red, 184; green, 181; blue, 181 }  ,draw opacity=1 ][dash pattern={on 0.84pt off 0.84pt}]   (150.5,177.5) -- (151,234.5) ;
%Straight Lines [id:da18766363560799715] 
\draw [color={rgb, 255:red, 184; green, 181; blue, 181 }  ,draw opacity=1 ][dash pattern={on 0.84pt off 0.84pt}]   (239.5,99.5) -- (239,234) ;
%Shape: Polygon [id:ds3323657581110937] 
\draw  [color={rgb, 255:red, 74; green, 144; blue, 226 }  ,draw opacity=1 ][fill={rgb, 255:red, 74; green, 144; blue, 226 }  ,fill opacity=0.1 ][dash pattern={on 1.5pt off 1.5pt}] (40,210) -- (280,40) -- (194.5,164) -- cycle ;
%Curve Lines [id:da35926669942472245] 
\draw [color={rgb, 255:red, 208; green, 2; blue, 27 }  ,draw opacity=1 ][line width=1.5]    (40,200) .. controls (118,194) and (265,82) .. (280,30) ;
%Shape: Square [id:dp3184820001890556] 
\draw  [color={rgb, 255:red, 208; green, 2; blue, 27 }  ,draw opacity=1 ][fill={rgb, 255:red, 208; green, 2; blue, 27 }  ,fill opacity=1 ] (36.75,196.75) -- (43.25,196.75) -- (43.25,203.25) -- (36.75,203.25) -- cycle ;
%Shape: Square [id:dp07702834841438977] 
\draw  [color={rgb, 255:red, 208; green, 2; blue, 27 }  ,draw opacity=1 ][fill={rgb, 255:red, 208; green, 2; blue, 27 }  ,fill opacity=1 ] (276.75,26.75) -- (283.25,26.75) -- (283.25,33.25) -- (276.75,33.25) -- cycle ;
%Shape: Circle [id:dp9236588693362585] 
\draw  [fill={rgb, 255:red, 0; green, 0; blue, 0 }  ,fill opacity=1 ] (147,177.5) .. controls (147,175.57) and (148.57,174) .. (150.5,174) .. controls (152.43,174) and (154,175.57) .. (154,177.5) .. controls (154,179.43) and (152.43,181) .. (150.5,181) .. controls (148.57,181) and (147,179.43) .. (147,177.5) -- cycle ;
%Shape: Circle [id:dp521588561305091] 
\draw  [fill={rgb, 255:red, 0; green, 0; blue, 0 }  ,fill opacity=1 ] (182,154.5) .. controls (182,152.57) and (183.57,151) .. (185.5,151) .. controls (187.43,151) and (189,152.57) .. (189,154.5) .. controls (189,156.43) and (187.43,158) .. (185.5,158) .. controls (183.57,158) and (182,156.43) .. (182,154.5) -- cycle ;
%Shape: Circle [id:dp35246118043328123] 
\draw  [fill={rgb, 255:red, 0; green, 0; blue, 0 }  ,fill opacity=1 ] (215,122.5) .. controls (215,120.57) and (216.57,119) .. (218.5,119) .. controls (220.43,119) and (222,120.57) .. (222,122.5) .. controls (222,124.43) and (220.43,126) .. (218.5,126) .. controls (216.57,126) and (215,124.43) .. (215,122.5) -- cycle ;
%Shape: Circle [id:dp9180379981108868] 
\draw  [fill={rgb, 255:red, 0; green, 0; blue, 0 }  ,fill opacity=1 ] (236,99.5) .. controls (236,97.57) and (237.57,96) .. (239.5,96) .. controls (241.43,96) and (243,97.57) .. (243,99.5) .. controls (243,101.43) and (241.43,103) .. (239.5,103) .. controls (237.57,103) and (236,101.43) .. (236,99.5) -- cycle ;
%Shape: Circle [id:dp025042177486653494] 
\draw  [fill={rgb, 255:red, 255; green, 255; blue, 255 }  ,fill opacity=1 ] (37,210) .. controls (37,208.34) and (38.34,207) .. (40,207) .. controls (41.66,207) and (43,208.34) .. (43,210) .. controls (43,211.66) and (41.66,213) .. (40,213) .. controls (38.34,213) and (37,211.66) .. (37,210) -- cycle ;
%Shape: Circle [id:dp739931831566322] 
\draw  [fill={rgb, 255:red, 255; green, 255; blue, 255 }  ,fill opacity=1 ] (277,40) .. controls (277,38.34) and (278.34,37) .. (280,37) .. controls (281.66,37) and (283,38.34) .. (283,40) .. controls (283,41.66) and (281.66,43) .. (280,43) .. controls (278.34,43) and (277,41.66) .. (277,40) -- cycle ;
%Curve Lines [id:da901692668375046] 
\draw [color={rgb, 255:red, 208; green, 2; blue, 27 }  ,draw opacity=1 ][line width=1.5]    (380,200) .. controls (458,194) and (605,82) .. (620,30) ;
%Shape: Square [id:dp6728201229549492] 
\draw  [color={rgb, 255:red, 208; green, 2; blue, 27 }  ,draw opacity=1 ][fill={rgb, 255:red, 208; green, 2; blue, 27 }  ,fill opacity=1 ] (376.75,196.75) -- (383.25,196.75) -- (383.25,203.25) -- (376.75,203.25) -- cycle ;
%Shape: Square [id:dp07215264948831723] 
\draw  [color={rgb, 255:red, 208; green, 2; blue, 27 }  ,draw opacity=1 ][fill={rgb, 255:red, 208; green, 2; blue, 27 }  ,fill opacity=1 ] (616.75,26.75) -- (623.25,26.75) -- (623.25,33.25) -- (616.75,33.25) -- cycle ;
%Shape: Circle [id:dp3701497346758722] 
\draw  [fill={rgb, 255:red, 0; green, 0; blue, 0 }  ,fill opacity=1 ] (487,177.5) .. controls (487,175.57) and (488.57,174) .. (490.5,174) .. controls (492.43,174) and (494,175.57) .. (494,177.5) .. controls (494,179.43) and (492.43,181) .. (490.5,181) .. controls (488.57,181) and (487,179.43) .. (487,177.5) -- cycle ;
%Shape: Circle [id:dp2164863177344577] 
\draw  [fill={rgb, 255:red, 0; green, 0; blue, 0 }  ,fill opacity=1 ] (522,154.5) .. controls (522,152.57) and (523.57,151) .. (525.5,151) .. controls (527.43,151) and (529,152.57) .. (529,154.5) .. controls (529,156.43) and (527.43,158) .. (525.5,158) .. controls (523.57,158) and (522,156.43) .. (522,154.5) -- cycle ;
%Shape: Circle [id:dp5082048659389655] 
\draw  [fill={rgb, 255:red, 0; green, 0; blue, 0 }  ,fill opacity=1 ] (555,122.5) .. controls (555,120.57) and (556.57,119) .. (558.5,119) .. controls (560.43,119) and (562,120.57) .. (562,122.5) .. controls (562,124.43) and (560.43,126) .. (558.5,126) .. controls (556.57,126) and (555,124.43) .. (555,122.5) -- cycle ;
%Shape: Circle [id:dp6973179094016229] 
\draw  [fill={rgb, 255:red, 0; green, 0; blue, 0 }  ,fill opacity=1 ] (576,99.5) .. controls (576,97.57) and (577.57,96) .. (579.5,96) .. controls (581.43,96) and (583,97.57) .. (583,99.5) .. controls (583,101.43) and (581.43,103) .. (579.5,103) .. controls (577.57,103) and (576,101.43) .. (576,99.5) -- cycle ;
%Straight Lines [id:da6266566142562249] 
\draw [color={rgb, 255:red, 208; green, 2; blue, 27 }  ,draw opacity=1 ]   (380,200) -- (575.5,184.5) ;
%Straight Lines [id:da7630218171628408] 
\draw [color={rgb, 255:red, 208; green, 2; blue, 27 }  ,draw opacity=1 ]   (575.5,184.5) -- (620,30) ;
%Shape: Polygon [id:ds5869701557586918] 
\draw  [color={rgb, 255:red, 208; green, 2; blue, 27 }  ,draw opacity=1 ][fill={rgb, 255:red, 208; green, 2; blue, 27 }  ,fill opacity=0.1 ] (380,200) -- (620.25,30) -- (575.75,184.5) -- (536.68,187.59) -- cycle ;
%Shape: Circle [id:dp5932997626208711] 
\draw  [fill={rgb, 255:red, 255; green, 255; blue, 255 }  ,fill opacity=1 ] (617,40) .. controls (617,38.34) and (618.34,37) .. (620,37) .. controls (621.66,37) and (623,38.34) .. (623,40) .. controls (623,41.66) and (621.66,43) .. (620,43) .. controls (618.34,43) and (617,41.66) .. (617,40) -- cycle ;
%Shape: Circle [id:dp05953547181040075] 
\draw  [fill={rgb, 255:red, 255; green, 255; blue, 255 }  ,fill opacity=1 ] (377,210) .. controls (377,208.34) and (378.34,207) .. (380,207) .. controls (381.66,207) and (383,208.34) .. (383,210) .. controls (383,211.66) and (381.66,213) .. (380,213) .. controls (378.34,213) and (377,211.66) .. (377,210) -- cycle ;
%Straight Lines [id:da9529227718594218] 
\draw [color={rgb, 255:red, 184; green, 181; blue, 181 }  ,draw opacity=1 ]   (183.5,202) -- (175.34,173.88) ;
\draw [shift={(174.5,171)}, rotate = 73.81] [fill={rgb, 255:red, 184; green, 181; blue, 181 }  ,fill opacity=1 ][line width=0.08]  [draw opacity=0] (8.93,-4.29) -- (0,0) -- (8.93,4.29) -- cycle    ;
%Straight Lines [id:da6635392441480985] 
\draw [color={rgb, 255:red, 184; green, 181; blue, 181 }  ,draw opacity=1 ]   (246.5,154) -- (217.15,138.41) ;
\draw [shift={(214.5,137)}, rotate = 27.98] [fill={rgb, 255:red, 184; green, 181; blue, 181 }  ,fill opacity=1 ][line width=0.08]  [draw opacity=0] (8.93,-4.29) -- (0,0) -- (8.93,4.29) -- cycle    ;
%Straight Lines [id:da9269456621830368] 
\draw [color={rgb, 255:red, 184; green, 181; blue, 181 }  ,draw opacity=1 ]   (218.5,181) -- (197.15,148.51) ;
\draw [shift={(195.5,146)}, rotate = 56.69] [fill={rgb, 255:red, 184; green, 181; blue, 181 }  ,fill opacity=1 ][line width=0.08]  [draw opacity=0] (8.93,-4.29) -- (0,0) -- (8.93,4.29) -- cycle    ;

% Text Node
\draw (220,178.4) node [anchor=north west][inner sep=0.75pt]  [color={rgb, 255:red, 74; green, 144; blue, 226 }  ,opacity=1 ]  {$f$};
% Text Node
\draw (16,177.4) node [anchor=north west][inner sep=0.75pt]  [color={rgb, 255:red, 208; green, 2; blue, 27 }  ,opacity=1 ]  {$\mathsf{A}_{\ell }$};
% Text Node
\draw (256,7.4) node [anchor=north west][inner sep=0.75pt]  [color={rgb, 255:red, 208; green, 2; blue, 27 }  ,opacity=1 ]  {$\mathsf{B}_{\ell }$};
% Text Node
\draw (33,247) node [anchor=north west][inner sep=0.75pt]    {$x_{\ell }$};
% Text Node
\draw (264,247) node [anchor=north west][inner sep=0.75pt]    {$x_{\ell +1}$};
% Text Node
\draw (142,239.4) node [anchor=north west][inner sep=0.75pt]    {$I_{\ell }^{\mathsf{left}}$};
% Text Node
\draw (225,239.4) node [anchor=north west][inner sep=0.75pt]    {$I_{\ell }^{\mathsf{right}}$};
% Text Node
\draw (531,125.4) node [anchor=north west][inner sep=0.75pt]  [color={rgb, 255:red, 208; green, 2; blue, 27 }  ,opacity=1 ]  {$\tilde f$};
% Text Node
\draw (356,177.4) node [anchor=north west][inner sep=0.75pt]  [color={rgb, 255:red, 208; green, 2; blue, 27 }  ,opacity=1 ]  {$\mathsf{A}_{\ell }$};
% Text Node
\draw (596,7.4) node [anchor=north west][inner sep=0.75pt]  [color={rgb, 255:red, 208; green, 2; blue, 27 }  ,opacity=1 ]  {$\mathsf{B}_{\ell }$};
% Text Node
\draw (373,247) node [anchor=north west][inner sep=0.75pt]    {$x_{\ell }$};
% Text Node
\draw (604,247) node [anchor=north west][inner sep=0.75pt]    {$x_{\ell +1}$};
% Text Node
\draw (482,239.4) node [anchor=north west][inner sep=0.75pt]    {$I_{\ell }^{\mathsf{left}}$};
% Text Node
\draw (565,239.4) node [anchor=north west][inner sep=0.75pt]    {$I_{\ell }^{\mathsf{right}}$};
% Text Node
\draw (177,206) node [anchor=north west][inner sep=0.75pt]  [color={rgb, 255:red, 74; green, 144; blue, 226 }  ,opacity=1 ]  {$\gamma _{\mathrm{left}}$};
% Text Node
\draw (247,148) node [anchor=north west][inner sep=0.75pt]  [color={rgb, 255:red, 74; green, 144; blue, 226 }  ,opacity=1 ]  {$\gamma _{\mathrm{right}}$};
% Text Node
\draw (132,105) node [anchor=north west][inner sep=0.75pt]  [color={rgb, 255:red, 74; green, 144; blue, 226 }  ,opacity=1 ]  {$\gamma _{\mathrm{top}}$};
% Text Node
\draw (472,94) node [anchor=north west][inner sep=0.75pt]  [color={rgb, 255:red, 208; green, 2; blue, 27 }  ,opacity=1 ]  {$\tilde\gamma _{\mathrm{top}}$};
% Text Node
\draw (602,105.4) node [anchor=north west][inner sep=0.75pt]  [color={rgb, 255:red, 208; green, 2; blue, 27 }  ,opacity=1 ]  {$\tilde\gamma _{\mathrm{right}}$};
% Text Node
\draw (510,193.4) node [anchor=north west][inner sep=0.75pt]  [color={rgb, 255:red, 208; green, 2; blue, 27 }  ,opacity=1 ]  {$\tilde\gamma _{\mathrm{left}}$};

\end{tikzpicture}
\caption{Illustration of Claim~\ref{claim_capture}\ref{claim_capture_tri}. The elements of $\cC_\ell$ are shown as solid black circles. 
$\Tri_f(x_\ell,x_{\ell+1})$ is the shaded region on the left (outlined in dashed blue), while $\triangle_\ell = \Tri_{\tilde f}(x_\ell,x_{\ell+1})$ is the shaded region on the right (outlined in solid red).
The containment $\cC_\ell\subset\triangle_\ell\setminus\{\nA_\ell,\nB_\ell\}$ is implied by the inequalities \eqref{top_ineq}, \eqref{left_ineq}, and \eqref{right_ineq}.
Although $\tilde\gamma_{\mathsf{left}}$ exceeds $\gamma_{\mathsf{left}}$ at $x_\ell$, the inequality is reversed for $x\ge I_\ell^{\mathsf{left}}$.
Similarly, although $\tilde\gamma_{\mathsf{right}}$ exceeds $\gamma_{\mathsf{right}}$ at $x_{\ell+1}$, the inequality is reversed for $x\le I_\ell^{\mathsf{right}}$.}
\label{fig_claim_capture}
\end{figure}

Finally, we prove part~\ref{claim_capture_tri}.
The argument involves quite a bit of notation, but Figure~\ref{fig_claim_capture} provides a useful aid. 
If $\cC_\ell$ is empty, then the statement is trivial.
So let us assume $\cC_\ell$ is nonempty; let $I_\ell^{\mathsf{left}}$ and $I_\ell^{\mathsf{right}}$ denote the horizontal coordinates of its leftmost and rightmost elements, respectively.
By \eqref{28buf}, we have
\eeq{ \label{w9bknp}
I_\ell^{\mathsf{left}} &\ge \mathrlap{x_\ell + n^{-2}}\phantom{x_{\ell+1}-n^{-2}} \quad \text{for every $\ell\in\{1,\ldots,L\}$,}  \\
I_\ell^{\mathsf{right}}&\le x_{\ell+1}-n^{-2} \quad \text{for every $\ell\in\{0,\ldots,L-1\}$.}
}
We wish to show $\cC_\ell \subset \triangle_\ell$.
Since $\cC_\ell\subset[I_\ell^{\mathsf{left}},I_\ell^{\mathsf{right}}]\times[0,1]$, and also $\cC_\ell\subset\Tri_f(x_{\ell},x_{\ell+1})$ by \eqref{v28bi9}, it suffices to show
\eeq{ \label{3nl2x}
\big([I_\ell^{\mathsf{left}},I_\ell^{\mathsf{right}}]\times[0,1]\big)\cap\Tri_{f}(x_\ell,x_{\ell+1}) \subset \triangle_\ell.
}
To this end, denote the three boundary functions of $\Tri_f(x_\ell,x_{\ell+1})$ by
\eq{
\gamma_{\mathsf{top}}(x) &\coloneqq \Big(\frac{x_{\ell+1}-x}{x_{\ell+1}-x_\ell}\Big)f(x_\ell)+\Big(\frac{x-x_\ell}{x_{\ell+1}-x_\ell}\Big)f(x_{\ell+1}), \\
\gamma_{\mathsf{left}}(x) &\coloneqq f(x_\ell) + f'(x_\ell)\cdot(x-x_\ell), \\
\gamma_{\mathsf{right}}(x) &\coloneqq f(x_{\ell+1})  + f'(x_{\ell+1})\cdot(x-x_{\ell+1}).
}
Figure~\ref{fig_claim_capture} explains the notation; to be precise, we mean the vertical sections of $\Tri_f(x_\ell,x_{\ell+1})$ are given by
\eeq{ \label{2jbk8f}
\{y\in\R:\, (x,y)\in\Tri_f(x_\ell,x_{\ell+1})\} = [\gamma_{\mathsf{left}}(x)\vee \gamma_{\mathsf{right}}(x),\gamma_{\mathsf{top}}(x)] \quad \text{$\forall$ $x\in[x_\ell,x_{\ell+1}]$}.
}
Similarly, the three boundary functions of $\triangle_\ell = \Tri_{\tilde f}(x_\ell,x_{\ell+1})$ are
\eq{
\tilde\gamma_{\mathsf{top}}(x) &\coloneqq \Big(\frac{x_{\ell+1}-x}{x_{\ell+1}-x_\ell}\Big)y_\ell+\Big(\frac{x-x_\ell}{x_{\ell+1}-x_\ell}\Big)y_{\ell+1}, \\
\tilde\gamma_{\mathsf{left}}(x) &\coloneqq y_\ell + s_\ell^+\cdot(x-x_\ell), \\
\tilde\gamma_{\mathsf{right}}(x) &\coloneqq \begin{cases}
y_{\ell+1}  + s_{\ell+1}^-\cdot(x-x_{\ell+1}) &\text{if $\ell<L$} \\
-\infty &\text{if $\ell = L$,}
\end{cases}
}
by which we mean the vertical sections of $\triangle_\ell$ are given by
\eeq{ \label{3jbk8f}
\{y\in\R:\,(x,y)\in\triangle_\ell\} = [\tilde\gamma_{\mathsf{left}}(x)\vee \tilde\gamma_{\mathsf{right}}(x),\tilde\gamma_{\mathsf{top}}(x)] \quad \text{$\forall$ $x\in[x_\ell,x_{\ell+1}]$}.
}
Comparing \eqref{2jbk8f} and \eqref{3jbk8f}, we see that \eqref{3nl2x} will be verified once we establish the following inequalities:
\begin{subequations} \label{tlr}
\begin{align}
\gamma_{\mathsf{top}}(x) &\le \mathrlap{\tilde\gamma_{\mathsf{top}}(x)}\hphantom{\tilde\gamma_{\mathsf{right}}(x)} \quad \text{for all $x\in[I_\ell^{\mathsf{left}},I_\ell^{\mathsf{right}}]$,}  \label{top_ineq} \\
\gamma_{\mathsf{left}}(x) &\ge \mathrlap{\tilde\gamma_{\mathsf{left}}(x)}\hphantom{\tilde\gamma_{\mathsf{right}}(x)} \quad \text{for all $x\in[I_\ell^{\mathsf{left}},I_\ell^{\mathsf{right}}]$,}  \label{left_ineq} \\
\gamma_{\mathsf{right}}(x) &\ge \tilde\gamma_{\mathsf{right}}(x) \quad \text{for all $x\in[I_\ell^{\mathsf{left}},I_\ell^{\mathsf{right}}]$.} \label{right_ineq}
\end{align}
\end{subequations}
The rest of the proof is establishing \eqref{tlr}.

To prove \eqref{top_ineq}, note that $f(x_\ell)\le y_\ell$ and $f(x_{\ell+1})\le y_{\ell+1}$ by \eqref{y_ell_consequence}, so $\gamma_{\mathsf{top}}(x) \le \tilde\gamma_{\mathsf{top}}(x)$ for all $x\in[x_\ell,x_{\ell+1}]$.

To prove \eqref{left_ineq}, we consider two cases.
If $\ell=0$, then $\tilde\gamma_{\mathsf{left}}\equiv0$, so \eqref{left_ineq} is trivial.
Otherwise $\ell\in\{1,\ldots,L\}$, and we have
\eeq{ \label{3hi8p}
\gamma_{\mathsf{left}}(I_\ell^{\mathsf{left}})
&\stackrefp{y_ell_consequence,s_ell_plus_consequence}{=} f(x_\ell) + f'(x_\ell)\cdot(I_\ell^{\mathsf{left}} - x_\ell) \\
&\stackref{y_ell_consequence,s_ell_plus_consequence}{\ge} y_\ell - n^{-3} + (s_\ell^+ + n^{-1})\cdot (I_\ell^{\mathsf{left}} - x_\ell) \\
&\stackrefpp{w9bknp}{y_ell_consequence,s_ell_plus_consequence}{\ge} y_\ell + s_\ell^+\cdot(I_\ell^{\mathsf{left}} - x_\ell)
= \tilde\gamma_{\mathsf{left}}(I_\ell^{\mathsf{left}}).
}
Since $\gamma_{\mathsf{left}}-\tilde\gamma_{\mathsf{left}}$ is affine with slope $f'(x_\ell)-s_\ell^+\ge n^{-1}$ by \eqref{s_ell_plus_consequence}, it is increasing. Therefore, \eqref{3hi8p} implies $\gamma_{\mathsf{left}}(x) \ge \tilde\gamma_{\mathsf{left}}(x)$ for every $x\ge I_\ell^{\mathsf{left}}$, which proves \eqref{left_ineq}.

The proof of \eqref{right_ineq} is analogous to the previous paragraph, in the following way.
The trivial case is now $\ell=L$ (since $\tilde\gamma_{\mathsf{right}}\equiv-\infty$ in this case) instead of $\ell=0$, and for $\ell\in\{0,\ldots,L-1\}$ we simply make the replacements
\eq{
\big(\mathsf{left},x_\ell,y_\ell,s_\ell^+ + n^{-1},\eqref{s_ell_plus_consequence}\big)
\to
\big(\mathsf{right},x_{\ell+1},y_{\ell+1},s_{\ell+1}^- - n^{-1},\eqref{s_ell_minus_consequence}\big)
}
to obtain $\gamma_{\mathsf{right}}(x) \ge \tilde\gamma_{\mathsf{right}}(x)$ for all $x\le I_\ell^{\mathsf{right}}$.
\end{proofclaim}

\begin{proofclaim}[Proof of Claim~\ref{claim_triangles}]
We continue using the notation from the proof of Claim~\ref{claim_capture}.
In view of \eqref{3jbk8f}, part~\ref{claim_triangles_valid} requires us to show
\begin{align}
\tilde\gamma_{\mathsf{top}}(x) &\le y_{\ell+1}\wedge x \quad \text{for all $x\in[x_\ell,x_{\ell+1}]$}, \label{jr4cv_1} \\
\text{and} \quad \tilde\gamma_{\mathsf{left}}(x)\vee \tilde\gamma_{\mathsf{right}}(x) &\ge  \mathrlap{y_\ell}\hphantom{x\vee y_{\ell+1}} \quad \text{for all $x\in[x_\ell,x_{\ell+1}]$}. \label{jr4cv_2}
\end{align}
To verify \eqref{jr4cv_1}, note that $\tilde\gamma_{\mathsf{top}}(x)$ is a convex combination of $y_\ell$ and $y_{\ell+1}$.
Since $y_\ell \le y_{\ell+1}$, we get $\tilde\gamma_{\mathsf{top}}(x) \le y_{\ell+1}$.
Since $y_\ell\le x_\ell$ and $y_{\ell+1}\le x_{\ell+1}$ by \eqref{y_ell_consequence}, we also have $\tilde\gamma_{\mathsf{top}}(x) \le x$.
Meanwhile, \eqref{jr4cv_2} holds because $s_\ell^+\ge0$ by \eqref{s_ell_lower}, hence $\tilde\gamma_{\mathsf{left}}(x)\ge y_\ell$ for all $x\ge x_\ell$.

For part~\ref{claim_triangles_small}, we must restrict to $\ell\in\{1,\ldots,L-1\}$ so that we can use \eqref{increment_upper} and \eqref{6rg7jc} below.
By part~\ref{claim_triangles_valid} we have
\eq{
\triangle_\ell
&\subset[x_\ell,x_{\ell+1}]\times[y_\ell,y_{\ell+1}].
}
Therefore, the desired containment \eqref{tri_in_small_rec} will follow if we can choose $c_n$ such that
\eeq{ \label{d6bkkh}
x_{\ell+1} \le x_\ell + c_n \quad \text{and} \quad
y_{\ell+1} \le y_\ell + c_n.
}
To this end, observe that
\eq{
x_{\ell+1} \stackref{increment_upper}{\le} x_\ell + 3L^{-1}
}
while
\eq{
y_{\ell+1} \stackref{y_ell_consequence}{\le} f(x_{\ell+1})+n^{-3}
&\stackrefp{y_ell_consequence,6rg7jc,increment_upper}{\le} f(x_\ell) + f'(x_{\ell+1})\cdot(x_{\ell+1}-x_\ell) +n^{-3} \\
&\stackref{y_ell_consequence,6rg7jc,increment_upper}{\le} y_\ell + 2\eta^{-1}\cdot 3L^{-1} + n^{-3}.
}
Therefore, both inequalities in \eqref{d6bkkh} are satisfied with $c_n = 6\eta^{-1}L^{-1} + n^{-3}$, which tends to $0$ thanks to our parameter choices in \eqref{parameter_selection}.
\end{proofclaim}

\medskip

\noindent \textbf{Step 3a: parameters induced by $\tilde f$.}
In Step 2 we were given a convex chain $\cC$ such that $\#\cC\le n$, and we constructed a function $\tilde f = \tilde f_{\cC,n}$.
In Step 3 we are given some $\tilde f\in\wt\cF_n$ (equivalently, we are given the sequence $(x_\ell,y_\ell,s_\ell^-,s_{\ell}^+)_{\ell=1}^L$), and we study random convex chains that are captured by the tangency triangles of $\tilde f$.
To prepare for the main probabilistic estimates in Steps 3b and 3c, here we define and analyze certain deterministic quantities associated with $\tilde f$.
Throughout the rest of the proof, $o(1)$ denotes a quantity that tends to $0$ as $n\to\infty$, uniformly in $\tilde f\in\wt\cF_n$ and $\ell\in\{0,\ldots,L\}$.

Using the notation from \eqref{tilde_triangle}, we define
\eq{
p_\ell \coloneqq \int_{\triangle_\ell}\fp(x,y)\ \dd x\, \dd y,
\qquad
r_\ell \coloneqq \frac{\inf_{(x,y)\in\triangle_\ell}\fp(x,y)}{p_\ell/\Area(\triangle_\ell)},
\qquad
R_\ell \coloneqq \frac{\sup_{(x,y)\in\triangle_\ell}\fp(x,y)}{p_\ell/\Area(\triangle_\ell)}.
}
These quantities make sense because of the containment $\triangle_\ell\subset\cT$ from \eqref{tri_in_tri}.
Note that
\eeq{ \label{p_boundary}
p_0 &\stackref{p_upper}{\le} 2\fC\cdot\Area(\triangle_0) \stackref{tri_in_tri}{\le} 2\fC\cdot \tilde f(x_1) = 2\fC\cdot y_1 \stackref{y1_upper}{\le} 2\fC\eta, \\
\text{and}\quad
p_L &\stackref{p_upper}{\le} 2\fC\cdot\Area(\triangle_L) \stackref{tri_in_tri}{\le} 2\fC\cdot(1-x_L) \stackref{last_lower}{\le} 2\fC\eta.
}
Since $\eta\to0$ as $n\to\infty$ by \eqref{parameter_selection}, we conclude from \eqref{p_boundary} that
\eeq{ \label{p_boundary_zero}
p_0 = o(1) \quad \text{and} \quad
p_L = o(1).
}
In addition, because of \eqref{p_upper} and \eqref{p_lower}, we trivially have
\eeq{ \label{Rr_trivial}
\frac{\fc}{\fC} \le r_\ell \le 1 \le R_\ell \le \frac{\fC}{\fc} \quad \text{for all $\ell\in\{0,\ldots,L\}$.}
}
For $\ell\in\{1,\ldots,L-1\}$, we can do much better than \eqref{p_boundary_zero} and \eqref{Rr_trivial}, as follows.
Since $\fp$ is continuous by assumption \eqref{p_cont}, and $c_n\to0$, the containment \eqref{tri_in_small_rec} implies
\eeq{ \label{39n0nd}
\lim_{n\to\infty}\sup_{\tilde f\in\wt\cF_n}\sup_{\ell\in\{1,\ldots,L-1\}}\sup_{(x,y)\in\triangle_\ell}\Big|\frac{p_\ell}{\Area(\triangle_\ell)}-\fp(x,y)\Big| = 0.
}
Since $\fp$ is bounded away from $0$ because of \eqref{p_lower}, \eqref{39n0nd} implies
\eeq{ \label{prechoice}
\lim_{n\to\infty}\sup_{\tilde f\in\wt\cF_n}\sup_{\ell\in\{1,\ldots,L-1\}}\sup_{(x,y)\in\triangle_\ell}\Big|\frac{p_\ell}{\Area(\triangle_\ell)\fp(x,y)}-1\Big| = 0.
}
If we choose $\fp(x,y)$ as small as possible, then \eqref{prechoice} implies
\begin{align}
\label{pbulk}
\text{for $\ell\in\{1,\ldots,L-1\}$,} \quad
p_\ell &\le (1+o(1))\Area(\triangle_\ell)\inf_{(x,y)\in\triangle_\ell}\fp(x,y), \\
\label{rbulk}
\text{or equivalently} \quad 
r_\ell &\ge (1+o(1))^{-1}.
\end{align}
If we choose $\fp(x,y)$ as large as possible, then \eqref{prechoice} implies
\eeq{ \label{R_bulk} 
R_\ell \le 1 + o(1) \quad \text{for $\ell\in\{1,\ldots,L-1\}$.}
}
Putting together various inequalities, we have
\eeq{ \label{psum_upper}
\sum_{\ell=0}^L p_\ell^{1/3}
&\stackref{p_boundary_zero,pbulk}{\le} o(1) + (1+o(1))\sum_{\ell=1}^{L-1}\Big(\Area(\triangle_\ell)\inf_{(x,y)\in\triangle_\ell}\fp(x,y)\Big)^{1/3} \\
&\stackref{tilde_f_area}{=} o(1) + \frac{(1+o(1))}{2}\sum_{\ell=1}^{L-1}\Big(\int_{x_\ell}^{x_{\ell+1}}\tilde f''(x)^{1/3}\ \dd x\Big)\Big(\inf_{(x,y)\in\triangle_\ell}\fp(x,y)\Big)^{1/3} \\
&\stackrefpp{graph_in_triangle}{tilde_f_area}{\le} o(1) + \frac{(1+o(1))}{2}\sum_{\ell=1}^{L-1}\int_{x_\ell}^{x_{\ell+1}}\big[\tilde f''(x)\cdot\fp(x,\tilde f(x))\big]^{1/3}\ \dd x \\
&\stackrefpp{J_def1}{tilde_f_area}{\le} o(1) + (1+o(1))\frac{J(\tilde f)}{2}
\stackrel{\mbox{\footnotesize(Remark~\ref{rem_extension_bound})}}{=} \frac{J(\tilde f)}{2} + o(1).
}

\medskip

\noindent\textbf{Step 3b: typical convex chains near $\tilde f$.}
For each $\tilde f\in\wt\cF_n$ and $\ell\in\{0,\ldots,L\}$, define the following random variable, which records the length of the longest convex chain in $\triangle_\ell$:
\eeq{ \label{sublcc_tilde_f}
\sL_{n,\ell}^{\tilde f} \coloneqq \max\Big\{ \#S:\, S\subseteq\cS_n\cap\triangle_\ell, \  \text{$\{\nA_\ell,\nB_\ell\}\cup S$ is in convex position}\Big\}.
}

\begin{claim}[De-concatenation] \label{claim_deconcatenation}
For any $\rho$-regular convex chain $\cC\subseteq\cS_n$, if $\tilde f = \tilde f_{\cC,n}$ is the function from Step 2, then
\eeq{ \label{3onxx}
\#\cC \le \sum_{\ell=0}^L \sL_{n,\ell}^{\tilde f}.
}
\end{claim}

\begin{proofclaim}
Using the notation $\cC_\ell = \cC\cap([x_\ell,x_{\ell+1}]\times[0,1])$ from \eqref{Cell_def}, we have $\cC = \biguplus_{\ell=0}^L \cC_\ell$.
Hence $\#\cC=\sum_{\ell=0}^L \#\cC_\ell$.
Claim~\ref{claim_capture} shows that $\cC_\ell$ is a candidate set for \eqref{sublcc_tilde_f}, hence $\#\cC_\ell \le \sL_{n,\ell}^{\tilde f}$.
The two previous sentences together yield \eqref{3onxx}.
\end{proofclaim}

\begin{claim}[Almost sure upper bound] \label{claim_as_upper}
For any $\lambda>0$, the following holds almost surely.
For all large $n$ and every $\tilde f\in\wt\cF_n$, we have
\eeq{ \label{eq_as_upper}
\sum_{\ell=0}^L\frac{\sL_{n,\ell}^{\tilde f}}{\alpha n^{1/3}}
\le (1+\lambda)\frac{J(\tilde f)}{2} + o(1),
}
where $o(1)\to0$ as $n\to\infty$, uniformly in $\tilde f\in\wt\cF_n$.
\end{claim}

\begin{proofclaim}
Fix $\lambda>0$.
Denote the number of sample points that land in $\triangle_\ell$ by
\eq{
N_{n,\ell} \coloneqq \#(\cS_n\cap\triangle_\ell)\sim\mathsf{Binomial}(n,p_\ell).
}
By Hoeffding's inequality, we have
\eeq{ \label{nk3cb}
\P\big(N_{n,\ell} \le np_\ell + n^{2/3}\big)
\ge 1 - \exp(-2n^{1/3}).
}
Conditional on $N_{n,\ell}$, the random set $\cS_n\cap\triangle_\ell$ has the same law as $N_{n,\ell}$ independent samples from $\fp$ conditioned to be inside $\triangle_\ell$. %, and these sets are independent across $\ell$.
Therefore, Proposition~\ref{prop_general}\ref{prop_general_compare} (with $\wh\fp$ equal to the conditional density on $\triangle_\ell$) gives
\eq{
\E\givenk[\big]{\sL_{n,\ell}^{\tilde f}}{N_{n,\ell}}\cdot \one\{N_{n,\ell}\ge\fn\} \le (1+\lambda)\Big(\frac{R_\ell}{r_\ell}\Big)^{1/3}\alpha N_{n,\ell}^{1/3},
}
where $\fn$ depends only on $\fC/\fc$ and $\lambda$. 
Since $\sL_{n,\ell}^{\tilde f} \le N_{n,\ell}$, we can remove the indicator on the left-hand side by instead writing
\eq{
\E\givenk[\big]{\sL_{n,\ell}^{\tilde f}}{N_{n,\ell}}
\le (1+\lambda)\Big(\frac{R_\ell}{r_\ell}\Big)^{1/3}\alpha N_{n,\ell}^{1/3} + \fn.
}
Meanwhile, Proposition~\ref{prop_general}\ref{prop_general_concentration} (with $t_0=1$, $\beta=\tfrac14$, $t=n^{1/13}$) gives
\eq{
\P\givenk[\Big]{\sL_{n,\ell}^{\tilde f} \le \E\givenk[\big]{\sL_{n,\ell}^{\tilde f}}{N_{n,\ell}} + n^{1/13}N_{n,\ell}^{1/4}}{N_{n,\ell}}
\ge 1 - C\exp\big(-\fa n^{1/13}\big),
}
where $C$ is a universal constant, and $\fa$ depends only on $\fC/\fc$.
Since $N_{n,\ell} \le n$, it follows from the two previous displays that
\eq{
\P\Big(\sL_{n,\ell}^{\tilde f} \le 
(1+\lambda)\Big(\frac{R_\ell}{r_\ell}\Big)^{1/3}\alpha N_{n,\ell}^{1/3} + \fn 
+ n^{1/13}\cdot n^{1/4}\Big)
\ge 1 - C\exp(-\fa n^{1/13}).
}
By a union bound with \eqref{nk3cb}, we have
\eq{
&\P\bigg(\bigcap_{\tilde f\in\wt\cF_n}\bigcap_{\ell=0}^L\Big\{\sL_{n,\ell}^{\tilde f} \le 
(1+\lambda)\Big(\frac{R_\ell}{r_\ell}\Big)^{1/3}\alpha\big(np_\ell + n^{2/3})^{1/3} + \fn 
+ n^{17/52}\Big\}\bigg) \\
&\stackrefp{cardinality_upper}{\ge} 1 - (\#\wt\cF_n)(L+1)\Big(\exp(-2n^{1/3})+C\exp(-\fa n^{1/13})\Big) \\
&\stackref{cardinality_upper}{\ge} 1 - n^{7L}(L+1)\Big(\exp(-2n^{1/3})+C\exp(-\fa n^{1/13})\Big) \\
&\stackref{parameter_selection}{\ge} 1- n^{7n^{1/157}}(n^{1/157}+1)\Big(\exp(-2n^{1/3})+C\exp(-\fa n^{1/13})\Big).
}
Since $1/157 < 1/13$, we can use the first Borel--Cantelli lemma to conclude that with probability one, we have
\eeq{  \label{2ih5d}
\sL_{n,\ell}^{\tilde f} 
&\le (1+\lambda)\Big(\frac{R_\ell}{r_\ell}\Big)^{1/3}\alpha\Big(np_\ell+n^{2/3}\Big)^{1/3} + \fn + n^{17/52} \\
&\le (1+\lambda)\Big(\frac{R_\ell}{r_\ell}\Big)^{1/3}\alpha\Big(n^{1/3}p_\ell^{1/3}+n^{2/9}\Big) + 2n^{17/52}
}
for every $\tilde f\in\wt\cF_n$, $\ell\in\{0,\ldots,L\}$, and all large $n$.
Now divide by $\alpha n^{1/3}$ and sum over $\ell$ to obtain
\begin{align*}
&\sum_{\ell=0}^L\frac{\sL_{n,\ell}^{\tilde f}}{\alpha n^{1/3}}
\stackrefpp{2ih5d}{rbulk,R_bulk}{\le} (1+\lambda)\sum_{\ell=0}^{L}\Big(\frac{R_\ell}{r_\ell}\Big)^{1/3}\big(p_\ell^{1/3}+n^{-1/9}\big) + \tfrac{2}{\alpha}(L+1)n^{-1/156}\\
&\stackrefpp{Rr_trivial,parameter_selection}{rbulk,R_bulk}{\le} (1+\lambda)\Big(\frac{\fC}{\fc}\Big)^{2/3}\big(p_0^{1/3}+p_L^{1/3}+(L+1)n^{-1/9}\big) + (1+\lambda)\sum_{\ell=1}^{L-1}\Big(\frac{R_\ell}{r_\ell}\Big)^{1/3}p_\ell^{1/3} + o(1) \\
&\stackrefpp{p_boundary_zero,parameter_selection}{rbulk,R_bulk}{=} o(1) + (1+\lambda)\sum_{\ell=1}^{L-1}\Big(\frac{R_\ell}{r_\ell}\Big)^{1/3}p_\ell^{1/3} \\
&\stackref{rbulk,R_bulk}{\le} o(1) + (1+\lambda)(1+o(1))\sum_{\ell=1}^{L-1}p_\ell^{1/3} \\
&\stackrefpp{psum_upper}{rbulk,R_bulk}{\le} o(1) + (1+\lambda)(1+o(1))\Big(\frac{J(\tilde f)}{2}+o(1)\Big)
\stackrel{\mbox{\footnotesize(Remark~\ref{rem_extension_bound})}}{=} (1+\lambda)\frac{J(\tilde f)}{2} + o(1).
\end{align*}
This verifies \eqref{eq_as_upper}.
\end{proofclaim}

\medskip

\noindent \textbf{Step 3c: probability of full convex chain near $\tilde f$.}
Continuing to use the notation from \eqref{tilde_triangle}, we define the following events:
\begin{gather}
E_n^{\tilde f} 
\coloneqq 
\bigg\{\sum_{\ell=0}^L \#(\cS_n\cap\triangle_\ell) = n\bigg\},
\label{tilde_E_def} \\
F_n^{\tilde f} 
\coloneqq \bigcap_{\ell=0}^L\big\{\text{$(\cS_n\cap\triangle_\ell)\cup\{\nA_\ell,\nB_\ell\}$ is in convex position}\big\}. \nonumber
\end{gather}
Claim~\ref{claim_capture} gives the following containment:
\eeq{ \label{step2_outcome}
\{\text{$\cS_n$ is convex and $\rho$-regular}\} \subseteq %\bigcup_{\tilde f\in\wt\cF_n} 
\big(E_n^{\tilde f_{\cS_n,n}}\cap F_n^{\tilde f_{\cS_n,n}}\big).
}
Our goal in this step is to prove the following estimate.

\begin{claim}[Probability upper bound] \label{claim_prob_upper}
For every $\tilde f\in\wt\cF_n$, we have
\eq{
\frac{1}{n}\log\Big[\frac{(3n)!}{n!}\P\Big(\{\text{$\cS_n$ is $(\eta,\delta)$-good}\}\cap E_n^{\tilde f}\cap F_n^{\tilde f}\Big)\Big] 
\le \log\frac{27}{4} + 3\log\Big(J(\tilde f)+o(1)\Big) + o(1),
}
where $o(1)\to 0$ as $n\to\infty$, uniformly in $\tilde f\in\wt\cF_n$.
\end{claim}

\begin{proofclaim}
Recall Definition~\ref{def_good} of $(\eta,\delta)$-goodness.
We have the following sequence of implications:
\begin{subequations} \label{goodness}
\begin{align}
\begin{split}
\text{$\cS_n$ is $(\eta,\delta)$-good} \quad
&\stackref{good_def}{\implies} \quad \#\big(\cS_n\cap([0,1]\times[0,\eta])\big)\le\delta n \\
&\stackref{y1_upper}{\implies} \quad \#\big(\cS_n\cap([0,1]\times[0,y_1])\big)\le\delta n \\
&\stackref{tri_in_tri}{\implies} \quad \#\big(\cS_n\cap\triangle_0\big) \le \delta n.
\end{split}
\intertext{Similarly,}
\begin{split} \label{goodness1}
\text{$\cS_n$ is $(\eta,\delta)$-good} \quad
&\stackref{good_def}{\implies} \quad \#\big(\cS_n\cap([1-\eta,1]\times[0,1])\big)\le\delta n \\
&\stackref{last_lower}{\implies} \quad \#\big(\cS_n\cap([x_L,1]\times[0,1])\big)\le\delta n \\
&\stackref{tri_in_tri}{\implies} \quad \#\big(\cS_n\cap\triangle_L\big) \le \delta n.
\end{split}
\end{align}
\end{subequations}
As before, denote the number of sample points that land in $\triangle_\ell$ by
\eq{ 
N_{n,\ell} \coloneqq \#(\cS_n\cap\triangle_\ell).
}
The random vector $(N_{n,0},\ldots,N_{n,L},n-\sum_{\ell=0}^{L}N_{n,\ell})$ has the multinomial distribution with $n$ trials and probability parameters $(p_0,\ldots,p_{L},1-\sum_{\ell=0}^{L}p_\ell)$.
Therefore, for any nonrandom vector $(n_0,\ldots,n_L)$ such that $n_0+\cdots+n_L=n$, we have
\eeq{ \label{393nlc}
\P\bigg(\bigcap_{\ell=0}^{L}\{N_{n,\ell}=n_\ell\}\bigg) = {n \choose n_0,\ldots,n_L}p_0^{n_0}\cdots p_{L}^{n_{L}}
= n!\prod_{\ell=0}^L\frac{p_\ell^{n_\ell}}{n_\ell!}.
}
Conditional on $\bigcap_{\ell=0}^{L}\{N_{n,\ell}=n_\ell\}$, the random set $\cS_n\cap\triangle_\ell$ has the same law as $n_\ell$ independent samples from $\fp$ conditioned to be inside $\triangle_\ell$, and these sets are independent across $\ell$.
Therefore, Proposition~\ref{prop_general}\ref{prop_general_likelihood} (with $\wh\fp$ equal to the conditional density on $\triangle_\ell$) implies
\eeq{ \label{27nku}
\P\givenp[\bigg]{F_n^{\tilde f}}{\bigcap_{\ell=0}^{L}\{N_{n,\ell}=n_\ell\}}
\le \prod_{\ell=0}^L R_\ell^{n_\ell}\frac{2^{n_\ell}}{n_\ell!(n_\ell+1)!}.
}
We now have
\eq{
&\P\Big(\{\text{$\cS_n$ is $(\eta,\delta)$-good}\}\cap E_n^{\tilde f}\cap F_n^{\tilde f}\Big) 
\stackref{tilde_E_def,goodness}{\le} \sum_{\substack{n_0+\cdots+n_{L}=n \\ n_0\le\delta n,\ n_L\le\delta n}}\P\bigg(\bigcap_{\ell=0}^{L}\{N_{n,\ell}=n_\ell\}\cap F_n^{\tilde f}\bigg) \\
&\stackref{393nlc,27nku}{\le}  \sum_{\substack{n_0+\cdots+n_{L}=n \\ n_0\le\delta n,\ n_L\le\delta n}}n!\prod_{\ell=0}^LR_\ell ^{n_\ell}\frac{(2p_\ell)^{n_\ell}}{(n_\ell!)^2(n_\ell+1)!} \\
&\stackrefpp{Rr_trivial,R_bulk}{27nku,393nlc}{\le}  \sum_{\substack{n_0+\cdots+n_{L}=n \\ n_0\le\delta n,\ n_L\le\delta n}}n!(1+o(1))^n(\fC/\fc)^{2\delta n}\prod_{\ell=0}^L\frac{(2 p_\ell)^{n_\ell}}{(n_\ell!)^3} \\
&\stackrefp{393nlc,27nku}{=}  \frac{2^nn!}{(3n)!}(1+o(1))^n(\fC/\fc)^{2\delta n}\sum_{\substack{n_0+\cdots+n_{L}=n \\ n_0\le\delta n,\ n_L\le\delta n}}{3n \choose n_0,n_0,n_0,\ldots,n_L,n_L,n_L}\prod_{\ell=0}^L p_\ell^{n_\ell}.
}
Since $p_\ell>0$ for every $\ell$ by \eqref{p_lower}, we can define
\eq{
\bar q_\ell \coloneqq \frac{p_\ell^{1/3}}{3\sum_{k=0}^Lp_k^{1/3}}.
}
We can now rewrite the previous estimate as
\eq{
&\frac{(3n)!}{n!}\P\Big(\{\text{$\cS_n$ is $(\eta,\delta)$-good}\}\cap E_n^{\tilde f}\cap F_n^{\tilde f}\Big) \\
&\le 2^n(1+o(1))^n(\fC/\fc)^{2\delta n}\Big(3\sum_{\ell=0}^L p_\ell^{1/3}\Big)^{3n}\sum_{\substack{n_0+\cdots+n_{L}=n \\ n_0\le\delta n,\ n_L\le\delta n}}{3n \choose n_0,n_0,n_0,\ldots,n_L,n_L,n_L}\prod_{\ell=0}^L(\bar q_\ell)^{3n_\ell}.
}
By considering the multinomial distribution with $3n$ trials, $3(L+1)$ types of outcomes, and probability parameters $(\bar q_0,\bar q_0,\bar q_0,\ldots,\bar q_L,\bar q_L,\bar q_L)$, we infer that
\eq{
\sum_{n_0+\cdots+n_L = n}{3n \choose n_0,n_0,n_0,\ldots,n_L,n_L,n_L}\prod_{\ell=0}^L(\bar q_\ell)^{3n_\ell} \le 1.
}
We conclude from the two previous displays that
\eq{
\frac{(3n)!}{n!}\P\Big(\{\text{$\cS_n$ is $(\eta,\delta)$-good}\}\cap E_n^{\tilde f}\cap F_n^{\tilde f}\Big)
&\stackrefp{psum_upper}{\le} 2^n(1+o(1))^n(\fC/\fc)^{2\delta n}\Big(3\sum_{\ell=0}^L p_\ell^{1/3}\Big)^{3n} \\
&\stackref{psum_upper}{\le} 2^n(1+o(1))^n(\fC/\fc)^{2\delta n}\Big(\frac{3J(\tilde f)}{2} + o(1)\Big)^{3n} \\
&\stackrefp{psum_upper}{=} \Big(\frac{27}{4}\Big)^n(1+o(1))^n(\fC/\fc)^{2\delta n}\Big(J(\tilde f)+o(1)\Big)^{3n}.
}
Now use the fact that $\delta\to0$ as $n\to\infty$ (see Step 0) to complete the proof.
\end{proofclaim}

\medskip

\noindent \textbf{Step 4: repairing defects.}
This step is purely deterministic.
We modify a given $\tilde f\in\wt\cF_n$ to obtain a fully convex function $g_{\tilde f}\in\cF$.
The goal to construct $g_{\tilde f}$ in such a way that
\eeq{ \label{step_4_goal}
\lim_{n\to\infty}\sup_{\tilde f\in\wt\cF_n}\|\tilde f-g_{\tilde f}\|_{[0,1]} = 0 \qquad \text{and} \qquad
\lim_{n\to\infty}\sup_{\tilde f\in\wt\cF_n}|J(\tilde f) - J(g_{\tilde f})| = 0.
}
The next paragraph will construct $g_{\tilde f}$ and prove \eqref{step_4_goal}.

Recall that $\tilde f$ is continuous on $[0,1]$.
On each subinterval $[x_\ell,x_{\ell+1}]$, $\tilde f$ is strictly convex and smooth, with boundary data given by \eqref{tilde_f_def}.
For $\ell\in\{1,\ldots,L\}$, the incoming slope at $x_\ell$ is $\partial^-\tilde f(x_\ell) = s_\ell^-$, while the outgoing slope is $\partial^+ \tilde f(x_\ell) = s_\ell^+$.
The only reason $\tilde f$ is not convex on all of $[0,1]$ is that $s_\ell^- > s_\ell^+$, so the first derivative experiences a negative jump at $x_\ell$.
Accounting for these jumps, we have
\eq{
\tilde f'(x) = \int_0^x \tilde f''(u)\ \dd u - \sum_{\ell=1}^L (s_\ell^- - s_\ell^+)\one\{x_\ell<x\} \quad \text{for all $x\in(0,1)\setminus\{x_1,\ldots,x_{L}\}$}.
}
Define a continuous convex function $\tilde g\colon[0,1]\to[0,\infty)$ by
\eq{
\tilde g(0) = 0 \quad \text{and} \quad \tilde g'(x) = \int_0^x\tilde f''(u)\ \dd u \quad \text{for all $x\in(0,1)$}.
}
Note that the second derivative is unchanged:
\eeq{ \label{second_same}
\tilde f''(x) = \tilde g''(x) \quad \text{for all $x\in(0,1)\setminus\{x_1,\ldots,x_{L}\}$}.
}
Furthermore, the following inequalities hold for every $x\in[0,1]$:
\eeq{ \label{tenbc0}
\tilde f(x) 
\le \tilde g(x)
\le \tilde f(x) + \sum_{\ell=1}^L (s_\ell^- - s_\ell^+)
\stackref{s_ell_jump}{\le} \tilde f(x) + L\cdot 3n^{-1}.
}
In particular, $\tilde g(1) \le \tilde f(1) + 3Ln^{-1} = 1 + 3Ln^{-1}$, so the following rescaled function takes values in $[0,1]$:
\eeq{ \label{g_no_tilde_def}
g(x) \coloneqq \frac{\tilde g(x)}{1+3Ln^{-1}}, \quad x\in[0,1].
}
Hence $g\in\cF$, and we set $g_{\tilde f} \coloneqq g$. 
The first limit in \eqref{step_4_goal} is obtained as follows:
\eeq{ \label{vjhpe}
\|\tilde f-g\|_{[0,1]} \le \|\tilde f - \tilde g\|_{[0,1]} + \|\tilde g - g\|_{[0,1]}
\stackref{tenbc0,g_no_tilde_def}{\le} 3Ln^{-1} + 3Ln^{-1} \stackref{parameter_selection}{=} o(1).
}
By continuity of $\fp$ from assumption \eqref{p_cont}, it follows from \eqref{vjhpe} that
\eeq{ \label{p_close}
\big\|\fp(\boldsymbol\cdot,\tilde f(\boldsymbol\cdot)) - \fp(\boldsymbol\cdot, g(\boldsymbol\cdot))\big\|_{[0,1]} = o(1).
}
Now the second limit in \eqref{step_4_goal} can be obtained as follows:
\eq{
J(\tilde f) 
&\stackrefpp{J_def1}{g_no_tilde_def,p_close}{=}\int_0^1\big[\tilde f''(x)\cdot\fp(x,\tilde f(x))\big]^{1/3}\ \dd x \\
&\stackrefpp{second_same}{g_no_tilde_def,p_close}{=} \int_0^1\big[\tilde g''(x)\cdot\fp(x,\tilde f(x))\big]^{1/3}\ \dd x \\
&\stackref{g_no_tilde_def,p_close}{=} \frac{1}{(1+3Ln^{-1})^{1/3}} \int_0^1\Big[g''(x)\cdot\big(\fp(x,g(x)) + o(1)\big)\Big]^{1/3}\ \dd x \\
&\stackrefpp{parameter_selection,p_lower}{g_no_tilde_def,p_close}{=}(1+o(1))\int_0^1\big[g''(x)\cdot\fp(x, g(x))\big]^{1/3}\ \dd x
\stackref{J_def1}{=} J(g) + o(1),
}
where the final equality uses the fact that $J$ is bounded.
This completes the construction of $g_{\tilde f} = g$ and the proof of \eqref{step_4_goal}.

A key consequence of the construction is the following.

\begin{claim}[Penalty for separation from optimizers] \label{claim_distance_consequence}
For every $\eps>0$, there exists $\theta>0$ such that the following implication holds for all large $n$ and every $\rho$-regular convex chain $\cC$ of cardinality at most $n$:
\eeq{ \label{force_subopt}
\inf_{h\in\argmax J}\max_{(x,y)\in\cC}|h(x)-y| \ge \eps 
\quad \implies \quad J\big(g_{\tilde f_{\cC,n}}\big) \le J_\star - \theta.
}
\end{claim}

\begin{proofclaim}
Fix $\eps>0$.
Before proceeding with the main argument, we make several assumptions about $n$.
By the first limit in \eqref{step_4_goal}, we may assume $n$ is large enough that
\eeq{ \label{g22cbx}
\|\tilde f-g_{\tilde f}\|_{[0,1]}\le \eps/6
\quad \text{for every $\tilde f\in\wt\cF_n$.}
}
Second, since $\tilde f(x_1)\le \eta$ by \eqref{y1_upper}, and $\eta\to0$ as $n\to\infty$ by \eqref{parameter_selection}, we may assume
\eeq{ \label{o9cg2x}
\tilde f(x_1) \le \eps/6
\quad \text{for every $\tilde f\in\cF_n$.}
}
Third, since $x_L \ge 1-\eta$ by \eqref{last_lower}, and $\eta\to0$, we may assume by Proposition~\ref{prop_equicontinuity}\ref{prop_equicontinuity_a} that
\eeq{ \label{3nl3x}
h(x_L) \ge 1 - \eps/2 
\quad \text{for every $h\in\argmax J$ and every $\tilde f\in\cF_n$}.
}
Fourth, recall the sequence $(c_n)_{n\ge1}$ from Claim~\ref{claim_triangles}, and assume $n$ is large enough that $c_n \le \eps/6$.
Fifth and finally, assume $n^{-3} \le \eps/6$.

By Proposition~\ref{prop_far_from_optimizer}, we may choose $\theta>0$ such that the following implication holds for all $g\in\cF$:
\eeq{ \label{b3xs42}
\inf_{h\in\argmax J}\|g-h\|_{[0,1]}\ge \eps/6 \quad \implies \quad J(g) \le J_\star - \theta.
}
Now let $\tilde f = \tilde f_{\cC,n}$ and $g = g_{\tilde f}$, and assume the hypothesis in \eqref{force_subopt}.
Consider any $h\in\argmax J$.
By assumption, there exists $(x,y)\in\cC$ such that $|h(x)-y| \ge \eps$.
There are two cases to consider:
\begin{itemize}

\item Suppose $x\in[x_\ell,x_{\ell+1}]$ for some $\ell\le L-1$.
We know
\eeq{ \label{gjnc5o}
(x,\tilde f(x))\stackref{graph_in_triangle}{\in}\triangle_\ell 
\quad \text{and} \quad 
(x,y)\stackrel{\mbox{\footnotesize(Claim~\ref{claim_capture}\ref{claim_capture_tri})}}{\in}\triangle_\ell.
}
If $\ell=0$, then $\triangle_\ell \subset [0,x_1]\times[0,\tilde f(x_1)]$ by Claim~\ref{claim_triangles}\ref{claim_triangles_valid}, so \eqref{gjnc5o} implies the first inequality below:
\eq{
|\tilde f(x)-y| \le \tilde f(x_1) \stackref{o9cg2x}{\le} \eps/6.
}
If $\ell\in\{1,\ldots,L-1\}$, then $\triangle_\ell\subset[x_\ell,x_\ell+c_n]\times[\tilde f(x_\ell),\tilde f(x_\ell)+c_n]$ by Claim~\ref{claim_triangles}\ref{claim_triangles_small}, so \eqref{gjnc5o} implies $|\tilde f(x)-y|\le c_n \le \eps/6$.
In either sub-case, we have
\eq{
|g(x)-h(x)| 
&\stackrefp{g22cbx}{\ge} |h(x) - y| - |y - \tilde f(x)| - |\tilde f(x) - g(x)| \\
&\stackref{g22cbx}{\ge} \eps - \eps/6 - \eps/6 = 2\eps/3.
}
In particular, $\|g-h\|_{[0,1]} \ge 2\eps/3$.

\item Suppose $x \ge x_L$. 
Then Claim~\ref{claim_capture}\ref{claim_capture_pt} justifies the first inequality below:
\eeq{ \label{ubo8n}
\tilde f(x_L) \le y+n^{-3} \le y + \eps/6.
}
Furthermore, since $|h(x)-y| \ge \eps$, we must have either
\eeq{ \label{hk4cd3}
y \ge h(x)+\eps
\quad\text{or}\quad
y \le h(x)-\eps.
}
But $h(x) \ge h(x_L)\ge 1-\eps/2$ by \eqref{3nl3x}, so only the latter scenario in \eqref{hk4cd3} is possible, and thus $y \le 1-\eps$.
Using this in \eqref{ubo8n} yields $\tilde f(x_L) \le 1-5\eps/6$, hence
\eq{
g(x_L) \stackref{g22cbx}{\le} \tilde f(x_L) + \eps/6 \le 1 - 4\eps/6.
}
Since $h(x_L)\ge 1-\eps/2$ by \eqref{3nl3x}, it follows that $\|g-h\|_{[0,1]} \ge \eps/6$.
\end{itemize}
We have now shown $\|g-h\|_{[0,1]} \ge \eps/6$ for every $h\in\argmax J$, so we can invoke \eqref{b3xs42} to obtain $J(g)\le J_\star - \theta$.
\end{proofclaim}

\medskip

\noindent \textbf{Step 5a: final synthesis for parts~\ref{prop_upper_bound_typical} and \ref{prop_upper_bound_far_typical}.}
By Claim~\ref{claim_suffice_typical}, we just need to verify \eqref{upper_reg}.
Fix $\lambda>0$.
With probability one, the following inequalities hold for all large $n$:
\eq{
\frac{\sL_n^\reg}{\alpha n^{1/3}} 
\stackrel{\mbox{\footnotesize (Claim~\ref{claim_deconcatenation})}}{\le} 
\sup_{\tilde f\in\wt\cF_n}\sum_{\ell=0}^L \frac{\sL_{n,\ell}^{\tilde f}}{\alpha n^{1/3}}
&\stackrel{\mbox{\footnotesize (Claim~\ref{claim_as_upper})}}{\le}
(1+\lambda)\sup_{\tilde f\in\wt\cF_n}\frac{J(\tilde f)}{2} + o(1) \\
&\stackrel{\parbox{\widthof{\footnotesize (Claim~\ref{claim_as_upper})}}{\centering\footnotesize \eqref{step_4_goal}}}{\le}
(1+\lambda)\sup_{g\in\cF}\frac{J(g)}{2} + o(1)
= (1+\lambda)\frac{J_\star}{2} + o(1).
}
As $\lambda>0$ is arbitrary, we can send $\lambda\searrow0$ to obtain \eqref{upper_reg1}.

For \eqref{upper_reg2}, we will follow the same reasoning but add an appeal to Claim~\ref{claim_distance_consequence} in the first inequality below.
Choosing $\theta>0$ from Claim~\ref{claim_distance_consequence}, and recalling the definition of $\sL_n^{\reg,\far}(\eps)$ from \eqref{regular_lengths}, we have
\eq{
\frac{\sL_n^{\reg,\far}(\eps)}{\alpha n^{1/3}} 
&\stackrel{\mbox{\footnotesize (Claims~\ref{claim_deconcatenation}, \ref{claim_distance_consequence})}}{\le} 
\sup_{\tilde f\in\wt\cF_n:\, J(g_{\tilde f}) \le J_\star-\theta}\sum_{\ell=0}^L \frac{\sL_{n,\ell}^{\tilde f}}{\alpha n^{1/3}} \\
&\stackrel{\parbox{\widthof{\footnotesize (Claims~\ref{claim_deconcatenation}, \ref{claim_distance_consequence})}}{\centering\footnotesize (Claim~\ref{claim_as_upper})}}{\le}
(1+\lambda)\sup_{\tilde f\in\wt\cF_n:\, J(g_{\tilde f})\le J_\star-\theta}\frac{J(\tilde f)}{2} + o(1)
\stackref{step_4_goal}{\le} (1+\lambda)\frac{J_\star-\theta}{2} + o(1)
}
for all large $n$, with probability one.
As $\lambda>0$ is arbitrary, we can send $\lambda\searrow0$ to obtain 
\eq{
\limsup_{n\to\infty}\frac{\sL_n^{\reg,\far}(\eps)}{\alpha n^{1/3}} \le \frac{J_\star-\theta}{2}.
}
This proves \eqref{upper_reg2}.

\medskip

\noindent \textbf{Step 5b: final synthesis for parts~\ref{prop_upper_bound_full} and \ref{prop_upper_bound_far_full}.}
By Claim~\ref{claim_suffice_full}, we just need to verify \eqref{29bx4r}.
Together with Claim~\ref{claim_prob_upper}, the second limit in \eqref{step_4_goal} implies
\eeq{ \label{step4_outcome}
\frac{1}{n}\log\Big[\frac{(3n)!}{n!}\P\Big(\{\text{$\cS_n$ is $(\eta,\delta)$-good}\}\cap E_n^{\tilde f}\cap F_n^{\tilde f}\Big)\Big]
\le \log\frac{27}{4} + 3\log\Big( J(g_{\tilde f}) + o(1)\Big) + o(1).
}
We now have
\eq{
&\frac{1}{n}\log\Big[\frac{(3n)!}{n!} \P(\text{$\cS_n$ is convex, $\rho$-regular, and $(\eta,\delta)$-good})\Big] \\
&\stackrefpp{step2_outcome}{step4_outcome,cardinality_subexp}{\le} \frac{1}{n}\log\bigg[\frac{(3n)!}{n!}\sum_{\tilde f\in\wt\cF_n}\P\Big(\{\text{$\cS_n$ is $(\eta,\delta)$-good}\}\cap E_n^{\tilde f}\cap F_n^{\tilde f}\Big)\bigg] \\
&\stackrefp{step4_outcome,cardinality_subexp}{\le} \frac{1}{n}\log\bigg[\frac{(3n)!}{n!}\sup_{\tilde f\in\wt\cF_n}\P\Big(\{\text{$\cS_n$ is $(\eta,\delta)$-good}\}\cap E_n^{\tilde f}\cap F_n^{\tilde f}\Big)\bigg] + \frac{\log \#\wt\cF_n}{n} \\
&\stackref{step4_outcome,cardinality_subexp}{\le} \log\frac{27}{4} + 3\log\Big( J_\star + o(1)\Big) + o(1).
}
We have proved \eqref{29bx4r_a}.

For \eqref{29bx4r_b}, choose $\theta>0$ from Claim~\ref{claim_distance_consequence}.
If $\cS_n$ is convex and $\rho$-regular, and the event $\Omega_\eps$ from \eqref{H_eps_def} occurs, then \eqref{force_subopt} yields $J(g_{\tilde f_{\cS_n,n}}) \le J_\star-\theta$.
This observation, together with \eqref{step2_outcome}, shows
\eeq{ \label{step2_outcome_far}
\{\text{$\cS_n$ is convex and $\rho$-regular}\}\cap\Omega_\eps \subseteq \bigcup_{\tilde f\in\wt\cF_n:\,J(g_{\tilde f})\le J_\star-\theta} \big(E_n^{\tilde f}\cap F_n^{\tilde f}\big).
}
We now have
\eq{
&\frac{1}{n}\log\Big[\frac{(3n)!}{n!} \P\big(\{\text{$\cS_n$ is convex, $\rho$-regular, and $(\eta,\delta)$-good}\}\cap \Omega_\eps\big)\Big] \\
&\stackrefpp{step2_outcome_far}{step4_outcome,cardinality_subexp}{\le} \frac{1}{n}\log\bigg[\frac{(3n)!}{n!}\sum_{\tilde f\in\wt\cF_n:\, J(g_{\tilde f})\le J_\star-\theta}\P\Big(\{\text{$\cS_n$ is $(\eta,\delta)$-good}\}\cap E_n^{\tilde f}\cap F_n^{\tilde f}\Big)\bigg] \\
&\stackrefp{step4_outcome,cardinality_subexp}{\le} \frac{1}{n}\log\bigg[\frac{(3n)!}{n!}\sup_{\tilde f\in\wt\cF_n:\, J(g_{\tilde f})\le J_\star-\theta}\P\Big(\{\text{$\cS_n$ is $(\eta,\delta)$-good}\}\cap E_n^{\tilde f}\cap F_n^{\tilde f}\Big)\bigg] + \frac{\log \#\wt\cF_n}{n} \\
&\stackref{step4_outcome,cardinality_subexp}{\le} \log\frac{27}{4} + 3\log\Big( J_\star - \theta + o(1)\Big) + o(1),
}
which proves \eqref{29bx4r_b}. \hfill $\blacksquare$

\section{Proofs of theorems} \label{sec_thm_proofs}

\begin{proof}[Proof of Theorem~\ref{thm_var_formula}]
The desired statement \eqref{variational_formula} claims both almost sure and $L^p$ convergence.
In light of Proposition~\ref{prop_general}\ref{prop_general_convergence} (with $\gamma=\tfrac13$), we just need to prove the almost sure part; that is, it suffices to show
\eeq{ \label{36rbck}
\lim_{n\to\infty}\frac{\sL_n}{n^{1/3}} =  \frac{\alpha}{2}J_\star \quad \text{almost surely}.
}
We assume \eqref{p_cont} throughout.
If we also assume \eqref{p_lower}, then the $\ge$ direction of \eqref{36rbck} follows from Proposition~\ref{prop_lower_bound}\ref{prop_lower_bound_typical} (take $f\in\argmax J$), while the $\le$ direction is Proposition~\ref{prop_upper_bound}\ref{prop_upper_bound_typical}.
The rest of the proof is to show why the assumption \eqref{p_lower} can be dropped.

Given $\eps\in(0,1]$, define the following modified density function:
\eeq{ \label{tilde_p_def}
\wt\fp(x,y) \coloneqq (1-\eps)\fp(x,y) + 2\eps,
\quad (x,y)\in\cT.
}
Since $0\le \fp(x,y)\le2\fC$ and $2\le2\fC$, we have
\eq{
\wt\fp(x,y) \le 2\fC \quad \text{for all $(x,y)\in\cT$.}
}
Associated to $\wt\fp$ is the optimal value
\eq{
\wt J_\star \coloneqq \sup_{f\in\cF}\wt J(f), \quad \text{where} \quad \wt J(f)\coloneqq \int_0^1\big[f''(x)\cdot\wt\fp(x,f(x))\big]^{1/3}\ \dd x.
}

\begin{claim}[Optimal values are close] \label{claim_star_tilde}
$|\wt J_\star - J_\star| \le 2(\eps\fC)^{1/3}$.
\end{claim}

\begin{proofclaim}
By inspection of \eqref{tilde_p_def}, we have the following inequalities:
\eq{
(1-\eps)^{1/3}\fp(x,y)^{1/3} \le \wt\fp(x,y)^{1/3} \le \fp(x,y)^{1/3} + (2\eps)^{1/3}.
}
Now subtract $\fp(x,y)^{1/3}$ throughout and use the fact that $1-(1-\eps)^{1/3} \le \eps^{1/3}$, to obtain
\eq{
-\eps^{1/3}\fp(x,y)^{1/3} \le \wt\fp(x,y)^{1/3} - \fp(x,y)^{1/3} \le (2\eps)^{1/3}.
}
By a suitable integration, we conclude that
\eq{
-\eps^{1/3}J(f) \le \wt J(f) - J(f) \le (2\eps)^{1/3}\int_0^1 f''(x)^{1/3}\ \dd x \quad \text{for every $f\in\cF$}.
}
Furthermore, \eqref{J_star_trivial} gives $-2(\eps\fC)^{1/3} \le  -\eps^{1/3}J_\star \le -\eps^{1/3} J(f)$, while $\int_0^1f''(x)^{1/3}\,\dd x \le 4^{1/3}$ by Proposition~\ref{prop_unif_case}\ref{prop_unif_case_a}.
So the previous display implies
\eq{
-2(\eps\fC)^{1/3} \le \wt J(f) - J(f) \le 2\eps^{1/3}.
}
Since $\fC\ge1$ and $f$ is arbitrary, the claim follows.
\end{proofclaim}

In what follows, $\wt\cS_n = \{(\wt X_1,\wt Y_1),\ldots,(\wt X_n,\wt Y_n)\}$ will denote a set of $n$ independent samples from $\wt\fp$, coupled to the original samples $\cS_n = \{(X_1,Y_1),\ldots,(X_n,Y_n)\}$ from $\fp$ as follows.
Let $\chi_1,\ldots,\chi_n$ be i.i.d.\ Bernoulli($\eps$) random variables that are independent of $\cS_n$.
For each $i\in\{1,\ldots,n\}$, define $(\wt X_i,\wt Y_i)$ in one of two ways:
\begin{itemize}
\item (Replicate) If $\chi_1=0$, then set $(\wt X_i,\wt Y_i) = (X_i,Y_i)$.
\item (Replace with uniform) If $\chi_1=1$, then let $(\wt X_i,\wt Y_i)$ be a uniformly random element of $\cT$, independent of everything else.
\end{itemize}
The resulting collection $\wt\cS_n = \{(\wt X_1,\wt Y_1),\ldots,(\wt X_n,\wt Y_n)\}$ is thus a set of $n$ independent samples from $\wt\fp$.
Denote the longest convex chain in $\wt\cS_n$ by
\eq{
\wt\sL_n \coloneqq \max\{\#\wt\cC:\, \text{$\wt\cC\subseteq\wt\cS_n$ and $\wt\cC$ is a convex chain}\}.
}
Since $\wt\fp$ is continuous and strictly positive, the first paragraph of the proof shows
\eeq{ \label{with_positivity}
\lim_{n\to\infty}\frac{\wt\sL_n}{n^{1/3}} = \frac{\alpha}{2}\wt J_\star \quad \text{almost surely}.
}
For the sake of the following claim, we also define
\eq{
\sL_n^{\mathsf{gain}} &\coloneqq \max\{\#\cC':\, \text{$\cC'\subseteq\wt\cS_n\setminus\cS_n$ and $\cC'$ is a convex chain}\}, \\
\sL_n^{\mathsf{loss}} &\coloneqq \max\{\#\cC':\, \text{$\cC'\subseteq\cS_n\setminus\wt\cS_n$ and $\cC'$ is a convex chain}\}.
}

\begin{claim}[Comparing maximum lengths] \label{claim_loss_gain}
$-\sL_n^{\mathsf{loss}} \le \wt\sL_n-\sL_n \le \sL_n^{\mathsf{gain}}$.
\end{claim}

\begin{proofclaim}
Let $\cC\subseteq\cS_n$ be maximizer for $\sL_n$.
Write $\cC = \wt\cC\uplus \cC'$, where $\wt\cC = \cC\cap\wt\cS_n$ and $\cC' = \cC\setminus\wt\cS_n$.
Since subsets of convex chains are convex chains, we have $\#\wt\cC \le \wt\sL_n$ and $\#\cC'\le\sL_n^{\mathsf{loss}}$.
Hence
\eq{
\sL_n = \#\cC = \#\wt\cC + \#\cC' \le  \wt\sL_n + \sL_n^{\mathsf{loss}}.
}
This proves the first inequality in the claim.
The second inequality is proved analogously by exchanging the roles of $\cS_n$ and $\wt\cS_n$.
\end{proofclaim}

\begin{claim}[Errors are small] \label{claim_small_error}
With probability one, we have
\eq{
\sL_n^{\mathsf{loss}} \le (9\fC\e^3\eps n)^{1/3} 
\quad \text{and} \quad 
\sL_n^{\mathsf{gain}} \le (9\fC\e^3\eps n)^{1/3} 
\quad \text{for all large $n$.}
}
\end{claim}

\begin{proofclaim}
Denote the number of ``replacements'' in our construction of $\wt\cS_n$ by
\eq{
T_n \coloneqq \chi_1+\cdots+\chi_n\sim\mathsf{Binomial}(n,\eps).
}
By Hoeffding's inequality, we have
\eeq{ \label{07ynic}
\P(T_n \ge 2\eps n) 
= \P(T_n - \E T_n \ge \eps n)
\le \exp(-2\eps^2n).
}
Conditional on $T_n$, the set $\cS_n\setminus\wt\cS_n$ consists of $T_n$ independent samples from $\fp$.
Therefore, Proposition~\ref{prop_general}\ref{prop_general_tail} gives
\eeq{ \label{3jg8nv}
\P\givenk[\big]{\sL_n^{\mathsf{loss}} \ge (4\fC\e^3 T_n)^{1/3} + n^{1/4}}{T_n}
\le 2^{-n^{1/4}}.
}
The only assumption about $\fp$ used here is that $\fp \le 2\fC = \fC/\Area(\cT)$.
We now have
\eq{
\P\Big(\sL_n^{\mathsf{loss}} \ge (4\fC\e^3 2\eps n)^{1/3}+n^{1/4}\Big)
&\stackrefp{3jg8nv,07ynic}\le \P\Big(\sL_n^{\mathsf{loss}} \ge (4\fC\e^3 T_n)^{1/3}+n^{1/4}\Big) + \P(T_n \ge 2\eps n) \\
&\stackref{3jg8nv,07ynic}{\le} 2^{-n^{1/4}} + \exp(-2\eps^2n).
}
Using the first Borel--Cantelli lemma, we conclude
$\sL_n^{\mathsf{loss}} \le (8\fC\e^3\eps n)^{1/3}+n^{1/4} \le (9\fC\e^3\eps n)^{1/3}$
for all large $n$.

Meanwhile, conditional on $T_n$, the set $\wt\cS_n\setminus\cS_n$ consists of $T_n$ independent uniform samples.
Since the uniform density $\fu$ satisfies $\fu \equiv 2 \le 2\fC$, the exact same argument applies to $\sL_n^{\mathsf{gain}}$.
\end{proofclaim}

We are now ready to complete the proof.
For the upper bound, we have
\eq{
\limsup_{n\to\infty}\frac{\sL_n}{n^{1/3}}
&\stackrel{\parbox{\widthof{\footnotesize (\eqref{with_positivity} and Claim~\ref{claim_small_error})}}{\centering\footnotesize (Claim~\ref{claim_loss_gain})}}{\le} \limsup_{n\to\infty}\frac{\wt\sL_n + \sL_n^{\mathsf{loss}}}{n^{1/3}} \\
&\stackrel{\mbox{\footnotesize (\eqref{with_positivity} and Claim~\ref{claim_small_error})}}{\le} \frac{\alpha}{2}\wt J_\star + (9\fC\e^3\eps)^{1/3} \\
&\stackrel{\parbox{\widthof{\footnotesize (\eqref{with_positivity} and Claim~\ref{claim_small_error})}}{\centering\footnotesize (Claim~\ref{claim_star_tilde})}}{\le} \frac{\alpha}{2}J_\star + \alpha (\eps\fC)^{1/3} + (9\fC\e^3\eps)^{1/3}
}
with probability one.
Sending $\eps\searrow0$ proves
\eq{
\limsup_{n\to\infty}\frac{\sL_n}{n^{1/3}} \le \frac{\alpha}{2}J_\star \quad \text{almost surely}.
}
By analogous reasoning using $\sL_n^{\mathsf{gain}}$ instead of $\sL_n^{\mathsf{loss}}$, we also have
\eq{
\liminf_{n\to\infty}\frac{\sL_n}{n^{1/3}} \ge \frac{\alpha}{2}J_\star \quad \text{almost surely}.
}
The two previous displays together prove \eqref{36rbck}.
\end{proof}

\begin{proof}[Proof of Theorem~\ref{thm_limit_curves}]
Fix $\eps>0$.
By Proposition~\ref{prop_upper_bound}\ref{prop_upper_bound_far_typical}, there exists $\delta>0$ such that
\eeq{ \label{gebkxz}
\limsup_{n\to\infty}\frac{\sL_n^\far(\eps)}{n^{1/3}} \le \frac{\alpha}{2}J_\star-5\delta \quad \text{almost surely.}
}
Since $\E\sL_n/n^{1/3}\to(\alpha/2)J_\star$ by Theorem~\ref{thm_var_formula}, there is $n_0$ large enough that
\eq{
\frac{\E\sL_n}{n^{1/3}} \ge \frac{\alpha}{2}J_\star - \delta \quad \text{for all $n\ge n_0$.}
}
By Proposition~\ref{prop_general}\ref{prop_general_concentration} (with $t = t_0 = \delta$ and $\beta=\tfrac13$), it follows that
\eeq{ \label{5xn2d}
\P\Big(\frac{\sL_n}{n^{1/3}} > \frac{\alpha}{2}J_\star - 2\delta\Big) \ge 1-C\exp\big(-\fa\delta n^{1/3}\big) \quad \text{for all $n\ge n_0$,}
}
where the constants $C$ and $\fa>0$ depend only on $\delta$ and $\fC$.

Meanwhile, we can replicate the proof of Proposition~\ref{prop_general}\hyperref[prop_general_concentration]{(d,e)} for $\sL_n^\far(\eps)$ in place $\sL_n$ (see Remark~\ref{rem_concentration}).
Combining Proposition~\ref{prop_general}\ref{prop_general_convergence} with \eqref{gebkxz} yields
\eq{
\limsup_{n\to\infty}\frac{\E\sL_n^\far(\eps)}{n^{1/3}} \le \frac{\alpha}{2}J_\star - 5\delta,
}
so we may assume $n_0$ is large enough that
\eq{
\frac{\E\sL_n^\far(\eps)}{n^{1/3}} \le \frac{\alpha}{2}J_\star - 4\delta \quad \text{for all $n\ge n_0$.}
}
Now Proposition~\ref{prop_general}\ref{prop_general_concentration} (again with $t = t_0 = \delta$ and $\beta=\tfrac13$) yields
\eeq{ \label{6xn2d}
\P\Big(\frac{\sL_n^\far(\eps)}{n^{1/3}} < \frac{\alpha}{2}J_\star - 3\delta\Big) \ge 1-C\exp\big(-\fa\delta n^{1/3}\big) \quad \text{for all $n\ge n_0$.}
}
We conclude from \eqref{5xn2d} and \eqref{6xn2d} that
\eq{
\P\big(\sL_n > \sL_n^\far(\eps)+\delta n^{1/3}\big)
\ge 1-2C\exp\big(-\fa\delta n^{1/3}\big) \quad \text{for all $n\ge n_0$.}
}
By taking complements, this proves \eqref{limit_curves_prob}.

Furthermore, by the first Borel--Cantelli lemma, we have $\sL_n > \sL_n^\far(\eps)+\delta n^{1/3}$ for all large $n$.
Whenever this inequality holds and $\cC_n\subseteq\cS_n$ is a convex chain of length $\#\cC_n \ge \sL_n-\delta n^{1/3}$, 
it must be that $\cC_n$ is \textit{not} a candidate for $\sL_n^\far(\eps)$ as defined in \eqref{L_neps_def}, hence
\eq{
\inf_{f\in\argmax J}\max_{(x,y)\in\cC_n} |f(x)-y| < \eps.
}
This proves \eqref{410bxw}.
\end{proof}

\begin{proof}[Proof of Theorem~\ref{thm_probability}]
The $\le$ direction of \eqref{thm_probability_eq} is stated in Proposition~\ref{prop_upper_bound}\ref{prop_upper_bound_full}, while the $\ge$ direction follows from Proposition~\ref{prop_lower_bound}\ref{prop_lower_bound_full} by choosing $f\in\argmax J$.
\end{proof}

\begin{proof}[Proof of Theorem~\ref{thm_conditional}]
The desired inequality \eqref{thm_conditional_eq} is obtained as follows:
\begin{align*}
&\limsup_{n\to\infty}\frac{1}{n}\log\P\givenp[\Big]{\inf_{f\in\argmax J}\max_{(x,y)\in\cS_n}|f(x)-y| \ge \eps}{\text{$\cS_n$ is a convex chain}} \\
&\stackref{L_neps_def}{=} \limsup_{n\to\infty}\frac{1}{n}\log\left(\frac{\frac{(3n)!}{n!}\P\big(\sL_n^\far(\eps) = n\big)}{\frac{(3n)!}{n!}\P(\sL_n = n)}\right)
\stackref{prop_upper_bound_eps,thm_probability_eq} < 0. \qedhere
\end{align*}
\end{proof}

\section*{Acknowledgments}
E.B. was partially supported by National Science Foundation grant DMS-2412473. 
A.S. was partially supported by Simons Foundation grant MP-TSM-00002716.
We thank \mbox{Ron Peled} for stimulating discussion.
We are also grateful to \mbox{Riddhipratim Basu} and \mbox{Duncan Dauvergne} for pointers to several references.
The simulations in Figure~\ref{fig_simulation} were written by \mbox{Nishant Ajitsaria}, who was an undergraduate student at North Carolina State University.

\settocdepth{section}
\subsection*{AI disclosure}
We did not use an LLM for any mathematical ideas, writing, or editing.
We did use Google Gemini and ChatGPT to find two references \cite{tolsa11, valtr95} that supplied technical ingredients to Lemmas~\ref{lem:middle_third_singular_mass} and \ref{lem_lozenge}.

%: BIBLIOGRAPHY
\bibliographystyle{myacm2}
\bibliography{erikbib}{}

@Article{aldous_diaconis95,
  author           = {Aldous, D. and Diaconis, P.},
  journal          = {Probab. Theory Related Fields},
  title            = {Hammersley's interacting particle process and longest increasing subsequences},
  year             = {1995},
  issn             = {0178-8051},
  number           = {2},
  pages            = {199--213},
  volume           = {103},
  creationdate     = {2025-05-28T18:54:29},
  doi              = {10.1007/BF01204214},
  fjournal         = {Probability Theory and Related Fields},
  modificationdate = {2025-05-28T18:55:00},
  mrclass          = {60C05 (60K35)},
  mrnumber         = {1355056},
  mrreviewer       = {Graham Brightwell},
  url              = {https://doi.org/10.1007/BF01204214},
}

@Article{ambrus20_arxiv,
  author           = {Ambrus, Gergely},
  title            = {Longest {$k$}-monotone chains},
  year             = {2020},
  note             = {Preprint, 15 pp},
  creationdate     = {2026-07-17T14:48:59},
  doi              = {10.48550/arXiv.2009.13887},
  modificationdate = {2026-07-17T14:59:43},
}

@InProceedings{ambrus17,
  author           = {Ambrus, Gergely},
  booktitle        = {Discrete {G}eometry and {C}onvexity in {H}onour of {I}mre {B}{\'a}r{\'a}ny},
  title            = {Longest convex chains and subadditive ergodicity},
  year             = {2017},
  pages            = {125--126},
  creationdate     = {2026-07-17T14:54:52},
  modificationdate = {2026-07-17T14:59:43},
  url              = {http://real.mtak.hu/id/eprint/86263},
}

@Article{ambrus_barany09,
  author           = {Ambrus, Gergely and B\'{a}r\'{a}ny, Imre},
  journal          = {Random Structures Algorithms},
  title            = {Longest convex chains},
  year             = {2009},
  issn             = {1042-9832},
  number           = {2},
  pages            = {137--162},
  volume           = {35},
  creationdate     = {2025-03-30T10:56:15},
  doi              = {10.1002/rsa.20269},
  fjournal         = {Random Structures \& Algorithms},
  modificationdate = {2025-03-30T11:14:12},
  mrclass          = {60D05 (52A22 60F15)},
  mrnumber         = {2544003},
  mrreviewer       = {Werner Nagel},
  url              = {https://doi.org/10.1002/rsa.20269},
}

@InCollection{baik23,
  author           = {Baik, Jinho},
  booktitle        = {I{CM}---{I}nternational {C}ongress of {M}athematicians. {V}ol. 6. {S}ections 12--14},
  publisher        = {EMS Press, Berlin},
  title            = {K{PZ} limit theorems},
  year             = {2023},
  isbn             = {978-3-98547-064-8; 978-3-98547-564-3; 978-3-98547-058-7},
  pages            = {4190--4211},
  creationdate     = {2026-07-13T20:02:13},
  doi              = {10.4171/ICM2022/8},
  modificationdate = {2026-07-16T17:47:01},
  mrclass          = {60K35 (82C22)},
  mrnumber         = {4680400},
  url              = {https://doi.org/10.4171/ICM2022/8},
}

@Article{baik_deift_johansson99,
  author           = {Baik, Jinho and Deift, Percy and Johansson, Kurt},
  journal          = {J. Amer. Math. Soc.},
  title            = {On the distribution of the length of the longest increasing subsequence of random permutations},
  year             = {1999},
  issn             = {0894-0347},
  number           = {4},
  pages            = {1119--1178},
  volume           = {12},
  creationdate     = {2022-08-29T10:10:21},
  doi              = {10.1090/S0894-0347-99-00307-0},
  fjournal         = {Journal of the American Mathematical Society},
  modificationdate = {2024-08-11T19:03:26},
  mrclass          = {05A05 (33D45 45E05 60C05)},
  mrnumber         = {1682248},
  mrreviewer       = {David J. Aldous},
  url              = {https://doi.org/10.1090/S0894-0347-99-00307-0},
}

@Article{barany99,
  author           = {B\'{a}r\'{a}ny, Imre},
  journal          = {Ann. Probab.},
  title            = {Sylvester's question: the probability that {$n$} points are in convex position},
  year             = {1999},
  number           = {4},
  pages            = {2020--2034},
  volume           = {27},
  creationdate     = {2026-01-29T12:23:26},
  doi              = {10.1214/aop/1022874826},
  fjournal         = {The Annals of Probability},
  issn             = {0091-1798,2168-894X},
  modificationdate = {2026-01-29T12:37:04},
  mrclass          = {60D05 (52A22)},
  mrnumber         = {1742899},
  mrreviewer       = {Rodney\ Coleman},
  url              = {https://doi.org/10.1214/aop/1022874826},
}

@Article{barany97,
  author           = {B\'{a}r\'{a}ny, Imre},
  journal          = {J. Reine Angew. Math.},
  title            = {Affine perimeter and limit shape},
  year             = {1997},
  issn             = {0075-4102},
  pages            = {71--84},
  volume           = {484},
  creationdate     = {2025-11-17T12:45:35},
  doi              = {10.1515/crll.1997.484.71},
  fjournal         = {Journal f\"{u}r die Reine und Angewandte Mathematik. [Crelle's Journal]},
  modificationdate = {2025-11-17T12:45:59},
  mrclass          = {52C05 (11H06 52B20)},
  mrnumber         = {1437299},
  mrreviewer       = {Martin Henk},
  url              = {https://doi.org/10.1515/crll.1997.484.71},
}

@Article{barany95,
  author           = {B\'{a}r\'{a}ny, I.},
  journal          = {Discrete Comput. Geom.},
  title            = {{The Limit Shape of Convex Lattice Polygons}},
  year             = {1995},
  number           = {3-4},
  pages            = {279--295},
  volume           = {13},
  creationdate     = {2025-11-17T13:00:51},
  doi              = {10.1007/BF02574045},
  fjournal         = {Discrete \& Computational Geometry. An International Journal of Mathematics and Computer Science},
  issn             = {0179-5376},
  modificationdate = {2026-08-09T14:19:28},
  mrclass          = {52C05 (11H06)},
  mrnumber         = {1318778},
  mrreviewer       = {Martin Henk},
  url              = {https://doi.org/10.1007/BF02574045},
}

@Article{barany_rote_steiger_zhang00,
  author           = {B\'{a}r\'{a}ny, I. and Rote, G. and Steiger, W. and Zhang, C.-H.},
  journal          = {Discrete Comput. Geom.},
  title            = {{A Central Limit Theorem for Convex Chains in the Square}},
  year             = {2000},
  number           = {1},
  pages            = {35--50},
  volume           = {23},
  creationdate     = {2025-11-17T12:16:35},
  doi              = {10.1007/PL00009490},
  fjournal         = {Discrete \& Computational Geometry. An International Journal of Mathematics and Computer Science},
  issn             = {0179-5376},
  modificationdate = {2026-08-09T14:01:36},
  mrclass          = {60D05 (60F05 60F15 60F17)},
  mrnumber         = {1727122},
  mrreviewer       = {Lothar Heinrich},
  url              = {https://doi.org/10.1007/PL00009490},
}

@Article{basdevant_gerin22,
  author           = {Basdevant, Anne-Laure and Gerin, Lucas},
  journal          = {Ann. Inst. Henri Poincar\'{e} Probab. Stat.},
  title            = {Longest increasing paths with {L}ipschitz constraints},
  year             = {2022},
  number           = {3},
  pages            = {1849--1868},
  volume           = {58},
  creationdate     = {2026-07-23T11:21:27},
  doi              = {10.1214/21-aihp1220},
  fjournal         = {Annales de l'Institut Henri Poincar\'{e} Probabilit\'{e}s et Statistiques},
  issn             = {0246-0203,1778-7017},
  modificationdate = {2026-07-23T11:23:05},
  mrclass          = {60K35 (60F15)},
  mrnumber         = {4452654},
  mrreviewer       = {Vladimir\ V.\ Fomichov},
  url              = {https://doi.org/10.1214/21-aihp1220},
}

@Article{basdevant_gerin19,
  author           = {Basdevant, Anne-Laure and Gerin, Lucas},
  journal          = {ALEA Lat. Am. J. Probab. Math. Stat.},
  title            = {Longest increasing paths with gaps},
  year             = {2019},
  number           = {2},
  pages            = {1141--1163},
  volume           = {16},
  creationdate     = {2026-07-23T14:18:09},
  doi              = {10.30757/alea.v16-43},
  fjournal         = {ALEA. Latin American Journal of Probability and Mathematical Statistics},
  issn             = {1980-0436},
  modificationdate = {2026-07-23T14:18:48},
  mrclass          = {60K35 (60F15)},
  mrnumber         = {4030532},
  mrreviewer       = {Irene\ Crimaldi},
  url              = {https://doi.org/10.30757/alea.v16-43},
}

@Article{basu_ganguly_hammond18,
  author           = {Basu, Riddhipratim and Ganguly, Shirshendu and Hammond, Alan},
  journal          = {Comm. Math. Phys.},
  title            = {{The Competition of Roughness and Curvature in Area-Constrained Polymer Models}},
  year             = {2018},
  number           = {3},
  pages            = {1121--1161},
  volume           = {364},
  creationdate     = {2026-07-23T14:05:34},
  doi              = {10.1007/s00220-018-3282-x},
  fjournal         = {Communications in Mathematical Physics},
  issn             = {0010-3616,1432-0916},
  modificationdate = {2026-08-09T14:03:30},
  mrclass          = {82D60 (60K35 82B24)},
  mrnumber         = {3875824},
  url              = {https://doi.org/10.1007/s00220-018-3282-x},
}

@Article{berger_torri21,
  author           = {Berger, Quentin and Torri, Niccol\`o},
  journal          = {Ann. Inst. Henri Poincar\'{e} D},
  title            = {Beyond {H}ammersley's {L}ast-{P}assage {P}ercolation: a discussion on possible local and global constraints},
  year             = {2021},
  number           = {2},
  pages            = {213--241},
  volume           = {8},
  creationdate     = {2026-07-23T14:00:33},
  doi              = {10.4171/aihpd/102},
  fjournal         = {Annales de l'Institut Henri Poincar\'{e} D. Combinatorics, Physics and their Interactions},
  issn             = {2308-5827,2308-5835},
  modificationdate = {2026-08-09T14:06:31},
  mrclass          = {60K35 (82B44)},
  mrnumber         = {4261671},
  mrreviewer       = {Xiaolin\ Zeng},
  url              = {https://doi.org/10.4171/aihpd/102},
}

@Article{berger_torri19b,
  author           = {Berger, Quentin and Torri, Niccol\`o},
  journal          = {Ann. Appl. Probab.},
  title            = {Entropy-controlled last-passage percolation},
  year             = {2019},
  number           = {3},
  pages            = {1878--1903},
  volume           = {29},
  creationdate     = {2026-07-23T14:12:27},
  doi              = {10.1214/18-AAP1448},
  fjournal         = {The Annals of Applied Probability},
  issn             = {1050-5164,2168-8737},
  modificationdate = {2026-07-23T14:18:48},
  mrclass          = {60K35 (60F05 60K37)},
  mrnumber         = {3914559},
  mrreviewer       = {Andrew\ R.\ Wade},
  url              = {https://doi.org/10.1214/18-AAP1448},
}

@Article{bogachev14,
  author           = {Bogachev, Leonid V.},
  journal          = {J. Combin. Theory Ser. A},
  title            = {Limit shape of random convex polygonal lines: {E}ven more universality},
  year             = {2014},
  pages            = {353--399},
  volume           = {127},
  creationdate     = {2026-07-16T19:12:05},
  doi              = {10.1016/j.jcta.2014.07.005},
  fjournal         = {Journal of Combinatorial Theory. Series A},
  issn             = {0097-3165,1096-0899},
  modificationdate = {2026-08-09T14:07:10},
  mrclass          = {60D05 (52A22 52C05 60F99)},
  mrnumber         = {3252668},
  mrreviewer       = {Matthias\ Schulte},
  url              = {https://doi.org/10.1016/j.jcta.2014.07.005},
}

@Article{bogachev_zarbaliev23,
  author           = {Bogachev, Leonid V. and Zarbaliev, Sakhavet M.},
  journal          = {Mathematics},
  title            = {Inverse {L}imit {S}hape {P}roblem for {M}ultiplicative {E}nsembles of {C}onvex {L}attice {P}olygonal {L}ines},
  year             = {2023},
  number           = {2},
  pages            = {article no. 385, 23 pp},
  volume           = {11},
  creationdate     = {2026-07-17T16:35:32},
  doi              = {10.3390/math11020385},
  issn             = {2227-7390},
  modificationdate = {2026-07-23T11:22:53},
  url              = {https://www.mdpi.com/2227-7390/11/2/385},
}

@Article{bogachev_zarbaliev11,
  author           = {Bogachev, Leonid V. and Zarbaliev, Sakhavat M.},
  journal          = {Ann. Probab.},
  title            = {Universality of the limit shape of convex lattice polygonal lines},
  year             = {2011},
  number           = {6},
  pages            = {2271--2317},
  volume           = {39},
  creationdate     = {2026-07-16T18:31:49},
  doi              = {10.1214/10-AOP607},
  fjournal         = {The Annals of Probability},
  issn             = {0091-1798,2168-894X},
  modificationdate = {2026-07-16T18:32:22},
  mrclass          = {60D05 (05A17 52A22 60F05)},
  mrnumber         = {2932669},
  mrreviewer       = {Andrew\ R.\ Wade},
  url              = {https://doi.org/10.1214/10-AOP607},
}

@Article{bogachev_zarbaliev99,
  author           = {Bogachev, L. V. and Zarbaliev, S. M.},
  journal          = {Dokl. Akad. Nauk},
  title            = {On the approximation of convex functions by random polygonal lines},
  year             = {1999},
  issn             = {0869-5652},
  number           = {3},
  pages            = {299--302},
  volume           = {364},
  creationdate     = {2026-07-16T22:17:13},
  fjournal         = {Rossi\u{\i}skaya Akademiya Nauk. Doklady Akademii Nauk},
  modificationdate = {2026-07-17T14:48:39},
  mrclass          = {60F05 (60B12 60K40 82B41)},
  mrnumber         = {1706217},
  mrreviewer       = {B.\ L.\ Granovsky},
}

@Book{bogachev07,
  author           = {Bogachev, V. I.},
  publisher        = {Springer-Verlag, Berlin},
  title            = {{Measure Theory}},
  year             = {2007},
  isbn             = {978-3-540-34513-8; 3-540-34513-2},
  creationdate     = {2025-11-09T22:02:28},
  doi              = {10.1007/978-3-540-34514-5},
  modificationdate = {2025-11-11T13:29:19},
  mrclass          = {28-02 (28Axx 28Cxx 46G12 60G42 60G44)},
  mrnumber         = {2267655},
  mrreviewer       = {Ren\'{e} L. Schilling},
  pages            = {Vol. I: xviii+500 pp., Vol. II: xiv+575},
  url              = {https://doi.org/10.1007/978-3-540-34514-5},
}

@PhdThesis{brosset25,
  author           = {Brosset, Fabien},
  school           = {{Universit{\'e} de Toulouse}},
  title            = {Large deviations for random sums and random convex polygons},
  year             = {2025},
  creationdate     = {2026-07-23T14:35:20},
  hal_id           = {tel-05674111},
  hal_version      = {v1},
  modificationdate = {2026-07-24T09:56:33},
  number           = {2025TLSES329},
  url              = {https://theses.hal.science/tel-05674111},
}

@Article{bureaux_enriquez17,
  author           = {Bureaux, Julien and Enriquez, Nathana\"{e}l},
  journal          = {Israel J. Math.},
  title            = {Asymptotics of convex lattice polygonal lines with a constrained number of vertices},
  year             = {2017},
  number           = {2},
  pages            = {515--549},
  volume           = {222},
  creationdate     = {2026-07-16T18:40:30},
  doi              = {10.1007/s11856-017-1599-3},
  fjournal         = {Israel Journal of Mathematics},
  issn             = {0021-2172,1565-8511},
  modificationdate = {2026-07-16T18:42:38},
  mrclass          = {52C05 (11P21 52A22 52B20 52C45 60D05 60F05)},
  mrnumber         = {3722260},
  mrreviewer       = {Martin\ Henk},
  url              = {https://doi.org/10.1007/s11856-017-1599-3},
}

@Article{calder_esedoglu_hero14,
  author           = {Calder, Jeff and Esedo\={g}lu, Selim and Hero, Alfred O.},
  journal          = {SIAM J. Math. Anal.},
  title            = {A {H}amilton--{J}acobi {E}quation for the {C}ontinuum {L}imit of {N}ondominated {S}orting},
  year             = {2014},
  number           = {1},
  pages            = {603--638},
  volume           = {46},
  creationdate     = {2026-07-24T16:29:15},
  doi              = {10.1137/13092842X},
  fjournal         = {SIAM Journal on Mathematical Analysis},
  issn             = {0036-1410,1095-7154},
  modificationdate = {2026-08-09T14:12:06},
  mrclass          = {35F21 (35D40 49L20 60C05 90C29)},
  mrnumber         = {3163240},
  mrreviewer       = {Messaoud\ Boulbrachene},
  url              = {https://doi.org/10.1137/13092842X},
}

@Article{corwin16,
  author           = {Corwin, I.},
  journal          = {Notices Amer. Math. Soc.},
  title            = {Kardar-{P}arisi-{Z}hang universality},
  year             = {2016},
  issn             = {0002-9920},
  number           = {3},
  pages            = {230--239},
  volume           = {63},
  creationdate     = {2022-08-29T10:23:51},
  doi              = {10.1090/noti1334},
  fjournal         = {Notices of the American Mathematical Society},
  modificationdate = {2024-08-12T10:29:48},
  mrclass          = {82B23 (35K59 35R60 60H15 60K35)},
  mrnumber         = {3445162},
  mrreviewer       = {Flora Koukiou},
  url              = {https://doi.org/10.1090/noti1334},
}

@Article{dauvergne_virag21_arxiv,
  author           = {Dauvergne, Duncan and Vir\'{a}g, B\'{a}lint},
  title            = {The scaling limit of the longest increasing subsequence},
  year             = {2021},
  creationdate     = {2022-09-13T18:02:36},
  doi              = {10.48550/arXiv.2104.08210},
  modificationdate = {2026-07-26T14:58:07},
  note             = {Preprint},
}

@Article{deuschel_zeitouni99,
  author           = {Deuschel, Jean-Dominique and Zeitouni, Ofer},
  journal          = {Combin. Probab. Comput.},
  title            = {On {I}ncreasing {S}ubsequences of {I}.{I}.{D}. {S}amples},
  year             = {1999},
  number           = {3},
  pages            = {247--263},
  volume           = {8},
  creationdate     = {2024-08-27T18:27:42},
  doi              = {10.1017/S0963548399003776},
  fjournal         = {Combinatorics, Probability and Computing},
  issn             = {0963-5483},
  modificationdate = {2026-08-09T14:14:16},
  mrclass          = {60F10 (60C05)},
  mrnumber         = {1702546},
  mrreviewer       = {Timo Sepp\"{a}l\"{a}inen},
  url              = {https://doi.org/10.1017/S0963548399003776},
}

@Article{deuschel_zeitouni95,
  author           = {Deuschel, Jean-Dominique and Zeitouni, Ofer},
  journal          = {Ann. Probab.},
  title            = {Limiting curves for i.i.d. records},
  year             = {1995},
  issn             = {0091-1798},
  number           = {2},
  pages            = {852--878},
  volume           = {23},
  creationdate     = {2024-08-27T18:28:24},
  doi              = {10.1214/aop/1176988293},
  fjournal         = {The Annals of Probability},
  modificationdate = {2026-02-11T17:31:03},
  mrclass          = {60G70 (60F10)},
  mrnumber         = {1334175},
  mrreviewer       = {Charles M. Goldie},
}

@Article{dey_joseph_peled24,
  author           = {Dey, Partha S. and Joseph, Mathew and Peled, Ron},
  journal          = {Israel J. Math.},
  title            = {Longest increasing path within the critical strip},
  year             = {2024},
  number           = {1},
  pages            = {1--41},
  volume           = {262},
  creationdate     = {2026-07-23T14:04:33},
  doi              = {10.1007/s11856-023-2603-8},
  fjournal         = {Israel Journal of Mathematics},
  issn             = {0021-2172,1565-8511},
  modificationdate = {2026-07-23T14:06:08},
  mrclass          = {60K35 (05C80 60C05 60G55)},
  mrnumber         = {4803421},
  url              = {https://doi.org/10.1007/s11856-023-2603-8},
}

@Book{friedli_velenik17,
  author           = {Friedli, S. and Velenik, Y.},
  publisher        = {Cambridge University Press, Cambridge},
  title            = {{Statistical Mechanics of Lattice Systems: A Concrete Mathematical Introduction}},
  year             = {2017},
  isbn             = {978-1-107-18482-4},
  creationdate     = {2025-11-16T21:54:26},
  doi              = {10.1017/9781316882603},
  modificationdate = {2025-11-17T12:16:28},
  mrclass          = {82-01 (82B05)},
  mrnumber         = {3752129},
  pages            = {xix+622},
}

@Article{ganguly_hegde_zhang23_arxiv,
  author           = {Ganguly, Shirshendu and Hegde, Milind and Zhang, Lingfu},
  title            = {Brownian bridge limit of path measures in the upper tail of {KPZ} models},
  year             = {2023},
  note             = {Preprint, 77 pp},
  creationdate     = {2026-07-17T15:19:31},
  doi              = {10.48550/arXiv.2311.12009},
  modificationdate = {2026-07-17T16:04:56},
}

@Book{janson_luczak_rucinski00,
  author           = {Janson, Svante and {\L}uczak, Tomasz and Rucinski, Andrzej},
  publisher        = {Wiley-Interscience, New York},
  title            = {Random {G}raphs},
  year             = {2000},
  isbn             = {0-471-17541-2},
  series           = {Wiley-Interscience Series in Discrete Mathematics and Optimization},
  creationdate     = {2022-08-29T10:23:15},
  doi              = {10.1002/9781118032718},
  modificationdate = {2026-07-16T17:59:05},
  mrclass          = {05C80 (60C05 82B41)},
  mrnumber         = {1782847},
  mrreviewer       = {Mark R. Jerrum},
  pages            = {xii+333},
  url              = {https://doi.org/10.1002/9781118032718},
}

@Article{johansson00b,
  author           = {Johansson, Kurt},
  journal          = {Probab. Theory Related Fields},
  title            = {Transversal fluctuations for increasing subsequences on the plane},
  year             = {2000},
  issn             = {0178-8051},
  number           = {4},
  pages            = {445--456},
  volume           = {116},
  creationdate     = {2022-08-29T10:23:50},
  doi              = {10.1007/s004400050258},
  fjournal         = {Probability Theory and Related Fields},
  modificationdate = {2024-08-12T13:52:40},
  mrclass          = {60K35 (82B24 82C24)},
  mrnumber         = {1757595},
  mrreviewer       = {Timo Sepp\"{a}l\"{a}inen},
  url              = {https://doi.org/10.1007/s004400050258},
}

@Article{logan_shepp77,
  author           = {Logan, B. F. and Shepp, L. A.},
  journal          = {Advances in Math.},
  title            = {A variational problem for random {Y}oung tableaux},
  year             = {1977},
  issn             = {0001-8708},
  number           = {2},
  pages            = {206--222},
  volume           = {26},
  creationdate     = {2026-06-13T14:20:31},
  doi              = {10.1016/0001-8708(77)90030-5},
  fjournal         = {Advances in Mathematics},
  modificationdate = {2026-06-13T14:20:52},
  mrclass          = {05E10 (49Q10 60C05)},
  mrnumber         = {1417317},
  mrreviewer       = {Graham\ Brightwell},
  url              = {https://doi.org/10.1016/0001-8708(77)90030-5},
}

@Book{robert_casella04,
  author           = {Robert, Christian P. and Casella, George},
  publisher        = {Springer-Verlag, New York},
  title            = {{Monte Carlo Statistical Methods}},
  year             = {2004},
  edition          = {Second},
  isbn             = {0-387-21239-6},
  series           = {Springer Texts in Statistics},
  creationdate     = {2026-04-12T14:46:25},
  doi              = {10.1007/978-1-4757-4145-2},
  modificationdate = {2026-08-09T14:16:18},
  mrclass          = {62-01 (60J10 62F15 65Cxx)},
  mrnumber         = {2080278},
  mrreviewer       = {Petru\ P.\ Blaga},
  pages            = {xxx+645},
  url              = {https://doi.org/10.1007/978-1-4757-4145-2},
}

@Book{romik15,
  author           = {Romik, Dan},
  publisher        = {Cambridge University Press, New York},
  title            = {The Surprising Mathematics of Longest Increasing Subsequences},
  year             = {2015},
  isbn             = {978-1-107-42882-9; 978-1-107-07583-2},
  series           = {Institute of Mathematical Statistics Textbooks},
  volume           = {4},
  creationdate     = {2025-04-09T09:42:24},
  doi              = {10.1017/CBO9781139872003},
  modificationdate = {2025-04-09T09:44:16},
  mrclass          = {05-01 (05A05 05D40 60B20 60C05 60K35 82B41 82C41)},
  mrnumber         = {3468738},
  mrreviewer       = {Sergi Elizalde},
  pages            = {xi+353},
}

@Article{sinai94,
  author           = {Sinai, Y. G.},
  journal          = {Funct. Anal. Appl.},
  title            = {{Probabilistic Approach to the Analysis of Statistics for Convex Polygonal Lines}},
  year             = {1994},
  pages            = {108--113},
  volume           = {28},
  creationdate     = {2025-11-17T13:40:27},
  doi              = {10.1007/BF01076497},
  fjournal         = {Functional Analysis and its Applications},
  modificationdate = {2026-08-09T14:21:06},
}

@Article{talagrand96,
  author           = {Talagrand, Michel},
  journal          = {Ann. Probab.},
  title            = {A new look at independence},
  year             = {1996},
  issn             = {0091-1798,2168-894X},
  number           = {1},
  pages            = {1--34},
  volume           = {24},
  creationdate     = {2026-06-26T15:31:56},
  doi              = {10.1214/aop/1042644705},
  fjournal         = {The Annals of Probability},
  modificationdate = {2026-06-26T15:33:10},
  mrclass          = {60E15 (28A35)},
  mrnumber         = {1387624},
  mrreviewer       = {Peter\ Eichelsbacher},
  url              = {https://doi.org/10.1214/aop/1042644705},
}

@InCollection{tolsa11,
  author           = {Tolsa, Xavier},
  booktitle        = {Nonlinear {A}nalysis, {F}unction {S}paces and {A}pplications. {V}ol. 9},
  publisher        = {Acad. Sci. Czech Repub. Inst. Math., Prague},
  title            = {Calder\'{o}n-{Z}ygmund theory with non doubling measures},
  year             = {2011},
  isbn             = {978-80-85823-59-2},
  pages            = {217--260},
  creationdate     = {2026-04-14T22:44:30},
  modificationdate = {2026-07-23T15:12:09},
  mrclass          = {42B20 (30H10 30H35 31A15 42B25)},
  mrnumber         = {3203660},
  mrreviewer       = {Javier\ Duoandikoetxea},
  url              = {https://dml.cz/handle/10338.dmlcz/702642},
}

@Article{valtr96,
  author           = {Valtr, Pavel},
  journal          = {Combinatorica},
  title            = {The probability that {$n$} random points in a triangle are in convex position},
  year             = {1996},
  number           = {4},
  pages            = {567--573},
  volume           = {16},
  creationdate     = {2026-07-16T17:46:18},
  doi              = {10.1007/BF01271274},
  fjournal         = {Combinatorica. An International Journal on Combinatorics and the Theory of Computing},
  issn             = {0209-9683},
  modificationdate = {2026-07-16T17:47:23},
  mrclass          = {60D05 (52A22)},
  mrnumber         = {1433643},
  mrreviewer       = {Rodney\ Coleman},
  url              = {https://doi.org/10.1007/BF01271274},
}

@Article{valtr95,
  author           = {Valtr, P.},
  journal          = {Discrete Comput. Geom.},
  title            = {{Probability that {$n$} Random Points are in Convex Position}},
  year             = {1995},
  number           = {3-4},
  pages            = {637--643},
  volume           = {13},
  creationdate     = {2026-06-10T15:35:35},
  doi              = {10.1007/BF02574070},
  fjournal         = {Discrete \& Computational Geometry. An International Journal of Mathematics and Computer Science},
  issn             = {0179-5376,1432-0444},
  modificationdate = {2026-08-09T14:18:11},
  mrclass          = {60D05 (52A22)},
  mrnumber         = {1318803},
  mrreviewer       = {Allen\ D.\ Rogers},
  url              = {https://doi.org/10.1007/BF02574070},
}

@Article{vershik_zeitouni99,
  author           = {Vershik, A. and Zeitouni, O.},
  journal          = {Israel J. Math.},
  title            = {Large deviations in the geometry of convex lattice polygons},
  year             = {1999},
  pages            = {13--27},
  volume           = {109},
  creationdate     = {2026-07-16T17:48:14},
  doi              = {10.1007/BF02775023},
  fjournal         = {Israel Journal of Mathematics},
  issn             = {0021-2172,1565-8511},
  modificationdate = {2026-07-16T17:48:53},
  mrclass          = {52A22 (60F10)},
  mrnumber         = {1679585},
  mrreviewer       = {Charles\ M.\ Goldie},
  url              = {https://doi.org/10.1007/BF02775023},
}

@Article{vershik94,
  author           = {Vershik, A. M.},
  journal          = {Funct. Anal. Appl.},
  title            = {{The Limit Shape of Convex Lattice Polygons and Related Topics}},
  year             = {1994},
  pages            = {13--20},
  volume           = {28},
  creationdate     = {2025-11-17T13:05:08},
  doi              = {10.1007/BF01079006},
  fjournal         = {Functional Analysis and its Applications},
  modificationdate = {2026-08-09T14:22:01},
}

@Article{vershik_kerov77,
  author           = {Vershik, A. M. and Kerov, S. V.},
  journal          = {Dokl. Akad. Nauk SSSR},
  title            = {Asymptotic behavior of the {P}lancherel measure of the symmetric group and the limit form of {Y}oung tableaux},
  year             = {1977},
  issn             = {0002-3264},
  number           = {6},
  pages            = {1024--1027},
  volume           = {233},
  creationdate     = {2026-06-13T14:19:10},
  fjournal         = {Doklady Akademii Nauk SSSR},
  modificationdate = {2026-07-17T16:04:17},
  mrclass          = {10J20 (20C30)},
  mrnumber         = {480398},
  mrreviewer       = {V.\ M.\ Maksimov},
  url              = {https://www.mathnet.ru/eng/dan40430},
}

\end{document}